%% file: main.tex
\documentclass{puthesis}

\usepackage{latexsym}
\usepackage{amsfonts}
\usepackage{amsmath}
\usepackage{amsthm}
\usepackage{amssymb}
\usepackage{enumerate}
\usepackage[shortlabels]{enumitem}
\usepackage{soul}
\usepackage{xcolor}
\usepackage{dsfont}
\definecolor{blau}{RGB}{36, 40, 168}
\def\longbox#1{\parbox{.85\textwidth}{\textit{#1}}}
\def\stmt#1{\vspace{0cm}\begin{equation}\longbox{#1}\end{equation}}
\def\sset#1{\{#1\}}
\def\F{\mathcal{F}}
\def\h{\mathcal{H}}
\def\q{\mathcal{Q}}
\def\n{\mathcal{N}}
\def\m{\mathcal{M}}
\def\p{\mathcal{P}}

\def\C{\mathcal{C}}
\def\z{\mathcal{Z}}
\def\q{\mathcal{Q}}
\def\R{\mathcal{R}}
\newcommand\dd{\hbox{-}}

\usepackage{etoolbox}
\AtBeginEnvironment{tabular}{\doublespacing}

\usepackage{cleveref}

\usepackage[style=numeric, backend=bibtex, maxbibnames=10 ]{biblatex}
\date{December 2, 2020; revised \today}
\newtheorem{theorem}{Theorem}[section]
\newtheorem{thm}[theorem]{Theorem}

\newtheorem{lem}[theorem]{Lemma}
\newtheorem{lemma}[theorem]{Lemma}
\newtheorem{corr}[theorem]{Corollary}
\newtheorem{fact}[theorem]{Fact}

\newcommand\kristina{Vu{\v{s}}kovi{\'{c}} }
\usepackage{comment}

\input{tikz_declarations}

\author{Linda J. Cook}
\adviser{Paul Seymour}
\title{On recognition algorithms and structure of graphs with restricted induced cycles}

\abstract{
 We call an induced cycle of length at least four a hole. The parity of a hole is the parity of its length.
Forbidding holes of certain types in a graph has deep structural implications.
In 2006, Chudnovksy, Seymour, Robertson, and Thomas famously proved that a graph is perfect if and only if it does not contain an odd hole or a complement of an odd hole.
In 2002, Conforti, Cornu\'{e}jols, Kapoor and Vu\v{s}kov\'{i}c provided a structural description of the class of even-hole-free graphs.
In Chapter \ref{chapter:monoholes}, we provide a structural description of all graphs that contain only holes of length $\ell$ for every $\ell \geq 7$.

Analysis of how holes interact with graph structure has yielded detection algorithms for holes of various lengths and parities.
In 1991, Bienstock showed it is NP-Hard to test whether a graph G has an even (or odd) hole containing a specified vertex $v \in  V(G)$. In 2002, Conforti, Cornu\'{e}jols, Kapoor and Vu\v{s}kov\'{i}c gave a polynomial-time algorithm to recognize even-hole-free graphs using their structure theorem. In 2003, Chudnovsky, Kawarabayashi and Seymour provided a simpler and slightly faster algorithm to test whether a graph contains an even hole. In 2019, Chudnovsky, Scott, Seymour and Spirkl provided a polynomial-time algorithm to test whether a graph contains an odd hole. Later that year, Chudnovsky, Scott and Seymour strengthened this result by providing a polynomial-time algorithm to test whether a graph contains an odd hole of length at least $\ell$ for any fixed integer $\ell \geq 5$. In Chapter \ref{chapter:LongEvenHole}, we provide a polynomial-time algorithm to test whether a graph contains an even hole of length at least $\ell$ for any fixed integer $\ell \geq 4$.
}

\dedication{To friends, family and their intersection.}

\acknowledgements{I would like to thank my PhD adviser Paul Seymour for being a wonderful supervisor.
Additionally, I benefited several mentors in the math community including Maria Chudnovsky, Marthe Bonamy, Sophie Spirkl and as well as supportive fellow PhD students.
I thank the administrative staff, especially Tina Dwyer and Audrey Mainzer, at PACM for their endless patience and guidance.
I thank Jan\'{o}s Koml\'{o}s for mentoring me as an undergraduate and introducing me to graph theory.
I also thank all my friends and family for all the support and understanding during my PhD; I could not have completed my PhD without you.
An incomplete list of all the lovely people who shared baked goods, and great advice with me during my PhD: Andy, Alex, Colleen, Em, Jessica, Keats, Lena, Lily, Jessica, Shelley. 
}

\begin{document}

\singlespacing	
	\chapter{Introduction}
	\input{chapter_intro.tex}
	\input{long_even_holes_overleaf.tex}

\input{chapter_monoholes.tex}

	\newpage
	\printbibliography	
\end{document}

%% file: tikz_declarations.tex
\usepackage{tikz, graphicx,subcaption}
\usepackage{xcolor}
\usetikzlibrary{snakes}
\usetikzlibrary{calc}
\definecolor{lila}{HTML}{7F00FF}
\definecolor{grun}{HTML}{00994C}
\definecolor{darkorange}{HTML}{D55E00}
\definecolor{rosa}{HTML}{CC79A7}
\definecolor{gelb}{HTML}{F0E442}
\definecolor{neongreen}{HTML}{00FF00}
\definecolor{overleafgray}{HTML}{333333}
\tikzstyle{snode}=[shape= circle, draw=red, ultra thick, fill=red!30]
\tikzstyle{special node} = [shape= circle, draw=blue, ultra thick, fill=blue!10]
\tikzstyle{blue node} = [special node]
\tikzstyle{yellow node} = [shape= circle, draw=gelb, ultra thick, fill=gelb!30]
\tikzstyle{red node} = [snode]
\tikzstyle{green node} =  [shape= circle, draw=grun, ultra thick, fill=grun!20]
\tikzstyle{dot} = [shape = circle, draw = black, fill = black, scale=0.7]  
\tikzstyle{maybe node} = [shape = circle, draw=black, dotted, thick]
\tikzstyle{maybe new node} =[shape = circle, draw=blue, dotted, ultra thick, fill=blue!10]
\tikzstyle{maybe equal} = [dotted, thick] 
\tikzstyle{normal node} =[shape=circle,draw=black, fill = white]
\tikzstyle{new edge} = [blue, ultra thick]
\tikzstyle{maybe new edge}=[new edge, decorate]
\tikzstyle{non edge} = [red, dashed]
\tikzset{decoration={snake,amplitude=.7mm,segment length=2.5mm,
                       post length=0mm,pre length=0mm}}
\tikzstyle{maybe edge} = [dashed]
\usetikzlibrary{shapes}
\usetikzlibrary{positioning}
\usepackage{ifthen}
\usetikzlibrary{shapes.geometric}
\usepackage{tikz-layers}
\usetikzlibrary{patterns}
\tikzstyle{bag} = [fill=blue!15]
\tikzstyle{densely dashdotted}=      [dash pattern=on 6pt off 1pt on \the\pgflinewidth off 1pt]
\tikzstyle{path}=[ultra thick, densely dashdotted]

%% file: chapter_intro.tex
\section{Prior Publication and Joint Work}
All results presented in this thesis are joint work with my advisor Paul Seymour.
Chapter \ref{chapter:LongEvenHole} contains results that have been published on the arXiv as \cite{cook2020detecting} and have been submitted to a journal.
I have also presented the results of Chapter \ref{chapter:LongEvenHole} at the Waterloo Graph Coloring Conference at University of Waterloo in Ontario, Canada on September 27, 2019, the New York State Regional Graduate Mathematics Conference on March 28, 2020 online and the Princeton Applied and Computational Math Graduate Student Seminar on September 17, 2019.

Chapter \ref{chapter:monoholes} is unpublished joint work with Paul Seymour and presents a structural description of graphs which are chordal or only have holes of some fixed length $\ell$ for some $\ell \geq 7$.
Another group consisting of Jake Horsfield, Myriam Preissmann, Ni Luh Dewi Sintiari, Cl{\'e}oph{\'e}e Robin, Nicolas Trotignon and Kristina Vu{\v{s}}kovi{\'{c}} has been working on the same problem independently. They have proved a complete structural description for the case where $\ell \geq 7$ and odd in an as yet unpublished manuscript \cite{nicolas_monoholes}.

Section \ref{section:definitions} gives generally known graph theory definitions.
Sections \ref{section:overview} and \ref{section:relatedwork} describe results related to the contents of this thesis.

\paragraph{Update from December 2023}
My thesis is also available at Princeton University Library.
\cite{cook2020detecting} is now published in European Journal of Combinatorics.
Jake Horsfield, Myriam Preissmann, Ni Luh Dewi Sintiari, Cl{\'e}oph{\'e}e Robin, Paul Seymour, Nicolas Trotignon, Kristina Vu{\v{s}}kovi{\'{c}} have since co-authored a paper \cite{cook2023graphs} which 
gives a complete structural description of graphs which only have holes of length $\ell$ for each $\ell \geq 7$ and includes the main results of chapter \ref{chapter:monoholes}.
Jake Horsfield, Myriam Preissmann, Ni Luh Dewi Sintiari, Cl{\'e}oph{\'e}e Robin, Nicolas Trotignon and Kristina Vu{\v{s}}kovi{\'{c}} have since published their version of the proof in a manuscript available on arXiv \cite{horsfield2022holes}.
Their approach is not the same and may be of interest to the readers of this thesis.

\section{Definitions} \label{section:definitions}
A \textit{finite simple graph} is a pair $(V, E)$ where $V$ is a finite set of \textit{vertices} and $E$ consists of subsets of $V$ of cardinality two called \textit{edges}. We will assume all graphs we discuss are finite and simple. 
Let $G$ be a graph. We denote the set of edges of $G$ as $E(G)$ and the set of vertices of $G$ as $V(G)$. We call $|V(G)|$ the \textit{order} of $G$ and denote it by $|G|$. We denote an edge $\{ x,y\}$ by $xy$. If $xy$ is an edge we say $x$ is \textit{adjacent} to $y$ and we say $x$ and $y$ are \textit{neighbors}. We say $x$ and the edge $xy$ are \textit{incident} to each other.
For $v \in V(G)$ the set of neighbors of $v$ is denoted by $N_G(v)$ and is called the \textit{neighborhood} of $v$. In cases where $G$ is not ambiguous we simply denote the neighborhood of $v$ by $N(v)$.  For a set of vertices $S \subseteq V(G)$ we define $N(S)$ to be the set $\cup_{v \in S} \ N(v) \setminus S$. If $G$ is a graph and $S$ is subset of $V(G)$ such that every vertex in $V(G) \setminus S$ has a neighbor in $S$ we call $S$ a \textit{dominating set} of $G$.
We reserve the symbol $G^c$ for the \textit{complement} of $G$ which is the graph defined as follows: $V(G) = V(G^c)$ and $E(G^c)$ satisfies the property that for every two distinct $x,y \in V(G)$, $xy \in E(G)$ if and only if $xy \not \in E(G^c)$.

If $G$ and $H$ are graphs we call $H$ a \textit{subgraph} of $G$ if $V(H) \subseteq V(G)$ and $E(H) \subseteq E(G)$.
If $H$ is a subgraph of $G$ and if for every two distinct $x,y \in V(H)$, $xy \in E(G)$ if and only if $xy \in E(H)$ we say $H$ is an \textit{induced subgraph} of $G$. In other words $H$ is an induced subgraph of $G$ if it can be obtained from $G$ by deleting vertices and any edges incident to deleted vertices. If $H$ is an induced subgraph of $G$ we say $G$ \textit{contains} $H$. In this case we say $H$ is the subgraph of $G$ induced by $V(H)$ and denote $H$ by $G[V(H)]$.
For $X \subseteq V(G)$ we denote the subgraph of $G$ induced by $V(G) \setminus X$ by $G \setminus X$. If $X$ consists of a single vertex $x$ we will abbreviate this notation to $G \setminus x$.

We say a graph $G$ and a graph $H$ are isomorphic if there is a bijective function $f: V(G) \to V(H)$ satisfying $uv \in E(G)$ if and only if $f(u)f(v) \in E(H)$. In this case we call $f$ an isomorphism.
We can now describe containment more formally. If $H, G$ are graphs and $G$ contains a subgraph isomorphic to $H$ we say $G$ contains $H$, otherwise we say $G$ is $H$-free. If $\mathcal{H}$ is a set of graphs we say $G$ is $\mathcal{H}$-free if $G$ does not contain any graph in $\mathcal{H}$.

Let $k \geq 1$ be an integer. Let $P$ be the graph consisting of a sequence of vertices $v_1, v_2, \dots, v_{k}$ and edges between consecutive vertices. Then we say $P$ is a \textit{path} of length $k-1$ and denote it by $P_{k}$. We call $v_1$ and $v_{k}$ the \textit{ends} of $P$. We say $V(P) \setminus \{ v_1, v_k\}$ is the \textit{interior} of $P$ and we denote it by $P^*$.  Thus if $P$ has length at most one, $P^* = \emptyset$. We denote $P$ by $v_1 \dd v_2 \dd \dots \dd v_{k}$ or $v_k \dd v_{k-1} \dd \dots \dd v_1$.
We call the graph consisting of $P$ and the edge $v_1v_k$ a cycle of length $k$ and denote it by $v_1 \dd v_2 \dd \dots v_k \dd v_1$.
We say the parity of a path or cycle is the parity of its length.

Let $x$ and $y$ be vertices. We call any path with ends $x,y$ an $xy$-path.
We say a graph $G$ is connected if for every $x,y \in V(G)$, $G$ contains some $xy$-path.
We call $C$ a connected component of a graph $G$ if it is a maximal connected subgraph of $G$.
If $D$ is a component of $G^c$, we call $D$ an \textit{anticomponent} of $G$.
For any $x,y$ in the same connected component of a graph $G$ we call the $xy$-path of minimum length a \textit{shortest $xy$-path} and denote its length by $d_G(x,y)$. In cases where $G$ is not ambiguous we will denote $d_G(x,y)$ by $d(x,y)$.

We call a graph \textit{complete} if it contains all possible edges. We denote the complete graph on $n$ vertices by $K_n$. We call a complete subgraph a \textit{clique}. If $G$ is a graph we say the \textit{clique number} of $G$ is the order of its largest clique and we denote it by $\omega(G)$.
We say $S \subseteq V(G)$ is a \textit{stable set} if there are no edges between any two elements of $S$. We call the cardinality of the largest stable set in $G$ the \textit{stability number} of $G$ and denote it by $\alpha(G)$. 
If $X, Y \subseteq V(G)$, we say $X$ is {\em anticomplete} to $Y$ if no vertex in $X$ is equal or adjacent to a vertex in $Y$.
We say $X,Y$ are \textit{complete} if $X$ and $Y$ are disjoint and $x$ is adjacent to $y$ for every $x \in X$ and $y \in Y$.
Let $G$ be a graph. 
If $G$ is a connected graph and $X$ is subset of $V(G)$ such that $G \setminus X$ is not connected we call $X$ a \textit{(vertex) cut-set}.
If $X$ is a cut-set of $G$ and $G[X]$ is a clique we call $X$ a \textit{clique cut-set}.
If the set consisting of a single vertex $v$ is a cut-set we call $v$ a \textit{cut-vertex}.
For every integer $k \geq 0$, we say a graph $G$ is $k$-connected if $|V(G)| \geq k +1$ and $G$ has no cut-set of cardinality $k$.

For any integer $n \geq 1$, we call an assignment $\phi: V(G) \to \{ 1, 2, 3 \dots, n\}$ a \textit{proper coloring using $n$ colors} if for every $xy \in E(G)$, $\phi(x) \neq \phi(y)$. We call the smallest $n$ for which $G$ has a proper coloring using $n$ colors the \textit{chromatic number} of $G$ and denote it by $\chi(G)$.
A graph with chromatic number two is called a \textit{bipartite} graph.

We will define few useful classes of graphs.
A connected graph containing no cycles is called a \textit{tree}.
We call an induced cycle of length at least four a \textit{hole} and we call the complement of a hole an \textit{antihole}.
A graph that does not contain any hole is called \textit{chordal}.

A {\em theta} is a graph $H$ consisting of two non-adjacent vertices $u,v$ and three paths $P_1, P_2, P_3$ joining $u,v$ with pairwise 
disjoint interiors. We call $\{ u\}, \{ v\}$ the terminating sets of $H$.
\begin{figure}
	\centering
	\input{Tikz/thetaprismpyramid.tex}
	\caption{An illustration of a theta (left), pyramid (center) and prism (right).
		The thick dashed lines represent paths of length at least one.}
	\label{fig:thetadef}
\end{figure}
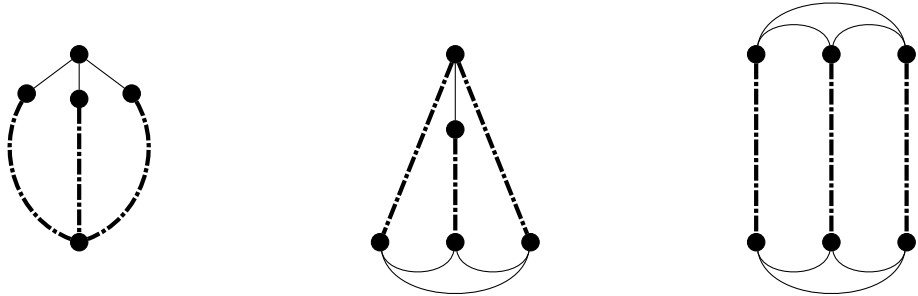
A {\em prism} is a graph $H$ consisting of two vertex disjoint triangles with vertex sets $\{a_1, a_2, a_3\}$ and $\{b_1, b_2, b_3\}$ and  three pairwise vertex-disjoint paths $P_1, P_2, P_3$, such that $P_i$ has ends $a_i$ and $b_i$ for 
$i \in \{1,2,3\}$. We call $\{a_1, a_2, a_3\}$, $\{b_1, b_2, b_3 \}$ the \textit{terminating sets of $H$}.
A \textit{pyramid} is a graph $H$ consisting of a triangle with vertex set $\{a_1, a_2, a_3\}$ and a vertex $r$ and three paths $P_1, P_2, P_3$ satisfying all of the following following:
\begin{itemize}
	\item For each $i \in \{ 1,2,3\}$, $P_i$ is an $ra_i$-path,
	\item At most one of $P_1, P_2, P_3$ has length one, and
	\item $P_1 \setminus a$, $P_2 \setminus a$, $P_3 \setminus a$ are pairwise vertex-disjoint paths.
\end{itemize}
We call $r$ the \textit{apex} of $H$ and we call $\{a_1, a_2, a_3\}, \{ r\}$ the \textit{terminating sets of $H$}.
If $H$ is a theta, prism, or pyramid and $P_1, P_2, P_3$ are as in the definition of $H$ we call $P_1, P_2, P_3$ the \textit{constituent} paths of $P$.
See Figure \ref{fig:thetadef} for an illustration. For further background on graph theory see \cite{diestel2005graph}, \cite{bondymurty} or \cite{berge1973graphs}.

\subsection{Notation}
For a positive integer $k$ let $[k]$ denote the set $\{1,2, \dots, k \}$.
For an integer $n \geq k$ let $[k,n]$ denote the set $\{k, k+1, k+2, \dots, n\}$.
\section{Overview} \label{section:overview}
In this thesis we present two results related to what holes a graph contains.
The first is a polynomial-time algorithm for every fixed integer $\ell \geq 4$ to determine whether an input graph contains a hole of length at least $\ell$ and is described in Chapter \ref{chapter:LongEvenHole}.
The second is the subject of Chapter \ref{chapter:monoholes} and is a structural characterization of the class of graphs that do not have a hole of any length other than $\ell$ for any integer $\ell \geq 7$. On first glance these may seem to be only tangentially related results. However, the algorithm to detect even holes of length at least $\ell$ relies on a structural analysis of graphs that do not contain ``obvious'' even holes, such as even holes of length at most $2\ell + 2$.
In this sense, both results are about describing the structure of graphs that do not contain holes of certain types (``graphs with restricted holes'').

Forbidding induced cycles of certain types in a graph has deep structural implications.
A simple example of this is the fact that graphs are bipartite if and only if they do not contain an odd induced cycle.
The most famous example of this is the class of perfect graphs.
The spirit of perfect graphs is due to Tibor Gallai and his study of linear programming duality theory's applications to combinatorics \cite{jensen_graph}.
A graph $G$ is perfect if for every induced subgraph $H$ of $G$, $\omega(H) = \chi(H)$. $G$ is called \textit{Berge} if $G$ and $G^c$ both do not contain any odd holes.
In 1961, Claude Berge conjectured that a graph $G$ is perfect if an only if $G$ is Berge \cite{Bergeconjectures}. 
He also conjectured that a graph $G$ is perfect if and only if its complement is perfect. 
After that, the study of perfect graphs became a field of its own, both due to their mathematical elegance and because of the two famous conjectures. The second conjecture was proven by Lov{\'a}sz in 1972 \cite{lovasz_weakperfectgraph}. The first conjecture stayed open for over forty years and was an active area of study. It was finally proved by Chudnovksy, Seymour, Robertson, and Thomas in 2006 and became known as the ``strong perfect graph theorem''. Their proof used a revolutionary technique of breaking graphs apart into more tractable pieces called graph decomposition \cite{chudnovsky2006strong}.

The quest to prove the perfect graph theorem motivated the first major work on the structure of even-hole-free graphs.
In 2002, Conforti, Cornu\'{e}jols, Kapoor and Vu\v{s}kov\'{i}c provided a structural description of even-hole-free graphs using graph decomposition \cite{conforti2002evenstructure}. They then used this description to develop a polynomial-time algorithm to test whether a graph contains an even hole  \cite{conforti2002even}. 
According Vu\v{s}kov\'{i}c's survey paper on even-hole-free graphs \cite{evenholefreegraphsurvey}, the work by Conforti et al. on even-hole-free graphs was motivated by the perfect graph conjecture. The group hoped that by studying even-hole-free graphs they would develop techniques that would be useful in the study of Berge graphs.
The complement of any hole of length at least six contains a hole of length four.
Thus, if $G$ is even-hole-free, then $C_5$ is the only possible odd hole in $G^c$.
This property makes even-hole-free graphs a reasonable proxy for Berge graphs. 

Since the proof of the strong perfect graph theorem and the work on even-hole-free graphs by Conforti et al., many polynomial-time algorithms have been found detecting an even (or odd) hole in a graph \cite{conforti2002even, chudnovsky2020oddhole} and related problems such as finding a shortest even hole in an input graph \cite{chudnovsky2020shortestodd} or finding an odd hole of length greater than some constant $\ell$ \cite{chudnovsky2019detectinglongodd}. (See Table \ref{table:hole-results}.)
Chapter \ref{chapter:LongEvenHole} provides a polynomial-time algorithm to determine whether an input graph contains an even hole of length greater than $\ell$ for some constant $\ell$. Section \ref{section:relatedwork} includes a survey of the main results and algorithms concerning even-hole-free graphs and odd-hole-free graphs.
\begin{table}[]
	\centering
	
	\begin{tabular}{c|c|c|c|c|}
		\cline{2-5}
		& \begin{tabular}[c]{@{}l@{}}Hole going through a \\ predetermined vertex\end{tabular} & Any length  & A shortest hole & Length $\geq \ell$ \\ \hline
		\multicolumn{1}{|l|}{Even}       &   NP-Hard \cite{bienstock_original, bienstock_correction} &   $\mathcal{O}(|G|^9)$ \cite{lai2019threeinatree}  &   $\mathcal{O}(|G|^{31})$ \cite{cheong2020finding}                  &    \begin{tabular}[c]{@{}l@{}} W[1]-hard \cite{w1hard}\\ 
			$\mathcal{O}(|G|^{9\ell + 3})$ \cite{cook2020detecting}\end{tabular}             \\ \hline
		\multicolumn{1}{|l|}{Odd}        &   NP-Hard \cite{bienstock_original, bienstock_correction}    &  $\mathcal{O}(|G|^8)$ \cite{lai2019threeinatree}          &   $\mathcal{O}(|G|^{14})$ \cite{chudnovsky2020shortestodd}                &      \begin{tabular}[c]{@{}l@{}} W[1]-hard \cite{w1hard} \\$\mathcal{O}(|G|^{20\ell + 40})$ \cite{chudnovsky2019detectinglongodd}\end{tabular}           \\ \hline
		\multicolumn{1}{|l|}{Any Parity} &  $\mathcal{O}(|G|^3)$                                      &  \begin{tabular}[c]{@{}l@{}} $\mathcal{O}(|G| + |E(G)|)$ \\ \cite{tarjan_chordal_1984, tarjan_chordal_1984_add} \cite{ rose_tarjan_lueker_1976_chordal, rose_tarjan_1978}  \end{tabular}  &    $\mathcal{O}(|G|*|E(G)|)$ \cite{Itai_girth_general_graph}                &   \begin{tabular}[c]{@{}l@{}} NP-Hard \\ $\mathcal{O}(|G|^{\ell})$ \end{tabular}             \\ \hline
	\end{tabular}
	\caption{A summary of algorithmic results for even/odd/general hole detection. The polynomial-time algorithm to determine whether a graph contains an even hole of length at least $\ell$ is joint work by Paul Seymour and me. It is described in Chapter \ref{chapter:LongEvenHole} and has been submitted for journal publication as \cite{cook2020detecting}. Section \ref{section:relatedwork} contains a more detailed overview of the algorithmic results for detecting odd and even holes.}
	\label{table:hole-results}
\end{table}
We have not yet examined the algorithmic implications of our structural description of monoholed graphs. However, I believe that the structural description of monoholed graphs in Chapter \ref{chapter:monoholes} could be exploited to answer algorithmic questions about monoholed graph as is done in the algorithm of Chapter \ref{chapter:LongEvenHole}. One particularly intriguing question is whether the class of monoholed graphs can be colored in polynomial-time. While coloring graphs is NP-Hard in general, Maffray, Penev, and Vu{\v{s}}kovi{\'{c}} gave a polynomial-time coloring algorithm for a subclass of monoholed graphs called ``rings'' \cite{maffray_coloringrings} which we define in Subsection \ref{subsec:structuralresults}. Since rings are important in our structural description of $\ell$-monoholed graphs it is possible some of the techniques of Maffray et al could be applied to the class of $\ell$-monoholed graphs in general for integers $\ell \geq 7$.

\section{Related Work} \label{section:relatedwork}
Both the results of Chapter \ref{chapter:LongEvenHole} and Chapter \ref{chapter:monoholes} concern themselves with the structure of graphs with restricted holes. In this section we summarize some of the structural results and recognition algorithms for graphs with restricted holes.

\subsection{Structural results} \label{subsec:structuralresults}
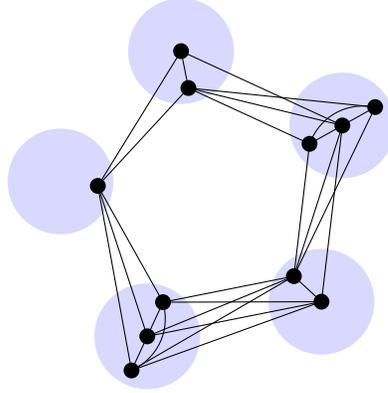
\begin{figure}[!h]
	\centering
	\input{Tikz/fat_C5.tex}
	\caption{An example of a ring on five sets. }
	\label{fig:inflated-C5}
\end{figure}
 One of the most relevant areas to this thesis is the study of rings.
Let $k \geq 3$ be an integer.
We say a graph $G$ is a \textit{ring on $k$ sets} if $V(G)$ can be partitioned into $k$ sets $X_0, X_2, \dots, X_{k-1}$ satisfying the following:
\begin{itemize}
	\item $G[X_0], G[X_1], \dots, G[X_{k-1}]$ are all cliques,
	\item For each $i \in \{0, 1, 2, \dots, k-1\}$ and $x, x' \in X_i$, $N(x) \subseteq N(x')$ or $N(x') \subseteq N(x)$,
	\item For each $i \in \{0, 1, 2, \dots, k-1\}$ and $x \in X_i$, $N(x) \subseteq X_{i-1} \cup X_i \cup X_{i+1}$ (where subscripts are taken modulo $k$) and
	\item For each $i \in \{0, 1, 2, \dots, k-1\}$ there is some $x \in X_i$ such that $x$ is complete to both $X_{i-1}$ and $X_{i+1}$.
\end{itemize}
(See Figure \ref{fig:inflated-C5} for an illustration.) Note that in Maffray, Penev, and Vu{\v{s}}kovi{\'{c}}'s paper on coloring rings, $k$ is assumed to be at least four.
By definition, if $G$ is a ring on $k$ sets, then $G$ does not contain a clique cut-set and every hole in $G$ has length $k$. In Chapter \ref{chapter:monoholes}, we show that if $G$ is a $\ell$-monoholed graph for some $\ell \geq 7$ then one of the following outcomes holds: $G$ contains a clique cut-set, $G$ contains a vertex $v$ that is adjacent to every other vertex in $G$, $G$ is chordal, $G$ is a ring or $G$ is a type of graph we call a ``crowned $k$-corpus''. Thus rings are one of the basic structural outcomes of our description of $\ell$-monoholed graphs for $\ell \geq 7$.
Interestingly, the study of rings originated because rings were also one of the basic structural outcomes in Boncompagni, Penev and Vu{\v{s}}kovi{\'{c}}'s characterization graphs without certain types of thetas, prisms, pyramids and wheels
as induced subgraphs \cite{kristina_truempergraphs}.\footnote{\color{black} In the context of \cite{kristina_truempergraphs} a wheel is a graph $H$ consisting of a cycle $C$ and another vertex $v$ adjacent to at least three elements of $V(C)$. $H$ is called a universal wheel if $v$ is complete to $V(C)$ and $H$ is called a twin wheel $v$ if $N(v) \cap V(C)$ consists of three consecutive vertices in $V(C)$. A wheel that is neither a twin wheel nor a universal wheel is called a proper wheel. Boncompagni et al. provide a characterization of graphs that contain no theta, pyramid, prism or proper wheels. Prisms, pyramids, thetas, and wheels are called Truemper configurations because they are important in a theorem by Truemper \cite{TRUEMPER1982112} that characterizes the graphs for which the edges can be labeled with integers in such a way that the sum of the labels on every induced cycle $C$ has a prescribed parity $f(C)$ for any assignment $f$ of parities to cycles. See \cite{evenholefreegraphsurvey} for a survey.}

Hoàng and Trotignon construct rings on $k$ sets with ``unbounded rank-width'' for any fixed integer $k \geq 3$ in forthcoming work \cite{hoang2020class}. Very informally, the rank-width of a graph is a way of measuring how complicated it is. Rank-width was introduced by Sang-il Oum and Paul Seymour in \cite{oum_rankwidth}.
Many problems that are NP-complete in general are polynomial-time for the class of graphs with bounded rank-width such as deciding whether a graph has chromatic number at most some constant \cite{kobler2003197} or determining whether a graph has a Hamiltonian cycle \cite{wanke1994251}. Subjectively, graph classes with unbounded rank-width, like the class of monoholed graphs, can be argued to be more ``interesting'' than those with bounded rank-width.

While the perfect graph theorem is the most famous result about the structure of graphs with restricted holes, there are several other notable results that relate to major open questions in graph theory. A key question in structural graph theory concerns the induced subgraphs of graphs with large chromatic number.
We need the following definitions.
We call a subset $\mathcal{F}$ of the set of all graphs an \textit{ideal} if $\mathcal{F}$ is closed under taking induced subgraphs.
We say an ideal $\mathcal{F}$ is \textit{$\chi$-bounded} if there exists some function $f$ such that for every $G \in \mathcal{F}$, $\chi(G) \leq f(\omega(G))$ and we call $f$ the \textit{$\chi$-bounding} function of $\mathcal{F}$.
Trivially, $\chi(G) \geq \omega(G)$ for every graph $G$.
The set of all graphs is not $\chi$-bounded. In fact, there exist triangle-free graphs with arbitrarily large chromatic number \cite{zykov1949some} \cite{mycielski1955coloriage}.

Many of the results on $\chi$-boundedness have to do with graphs with restricted holes.
Erd\H{o}s famously proved that for any cycle $C$, the ideal of $C$-free graphs is not $\chi$-bounded \cite{erdös_girth_chromatic_number_1959}. But what about $\mathcal{C}$-free graphs where $\mathcal{C}$ is an infinite family of cycles?
Berge graphs are $\chi$-bounded by the strong perfect graph theorem.
In 2016, Scott and Seymour showed that odd-hole-free graphs are $\chi$-bounded with $\chi$-bounding function $f(\kappa) = 2^{2^\kappa + 2}$ \cite{SCOTT_odd_holes_chibounded}, proving conjecture of Gyarfàs and Sumner.
Addario-Berry, Chudnovsky, Havet, Reed and  Seymour showed that every even-hole-free graph $G$ contains a vertex whose neighborhood consists of the vertex set of the union of two cliques and thus $\chi(G) \leq 2\omega(G)$ \cite{originalbisimplicial, chudnovsky2020evenholefree}.
Bonamy, Charbit and Thomassé proved that every graph with sufficiently large chromatic number contains an induced cycle of length $0 \mod 3$ \cite{bonamy2014graphs}, answering a question of Kalai and Meshulam. Scott and Seymour later proved that for any $p \geq 0$ and $q \geq 1$ the ideal of graphs with no induced cycle of length $p \mod q$ is $\chi$-bounded \cite{Scott2019}. Thus, the class of $\ell$-monoholed graphs for any fixed $\ell$ is $\chi$-bounded.
Maffray, Penev and \kristina give the optimal $\chi$-bounding function for the class of rings with at least four sets in \cite{maffray_coloringrings}.
Gyarfàs and Sumner conjectured that for any $\ell \geq 0$ the ideal of graphs not containing any hole of length greater than $\ell$ is $\chi$-bounded and they conjectured that the ideal of graphs not containing any odd hole of length greater than $\ell$ is $\chi$-bounded in \cite{long_oddholes_chi_boundedness}. The first conjecture was proven by Chudnovsky, Scott, and Seymour in \cite{longholes_chibounded} and the second stronger conjecture was proven by Chudnovsky, Scott, Seymour and Spirkl in \cite{long_oddholes_chi_boundedness}. See the survey by Scott and Seymour for more background on $\chi$-boundedness \cite{chi_bdd_survey}.

The Erd\H{os}-Hajnal conjecture states that for every graph $H$ there exists an $\epsilon > 0$ such that every $H$-free graph $G$ has a stable set or clique of cardinality at least $|G|^\epsilon$. This one of the most active open questions in structural graph theory. Recently, Chudnovsky, Scott, Seymour and Spirkl proved that the Erd\H{os}-Hajnal conjecture holds when $H$ is a hole of length five \cite{EH_C5}. See the survey by Maria Chudnovsky for further background on the Erd\H{os}-Hajnal conjecture \cite{EH_survey_chudnovsky}.

\subsection{Detecting (odd, even) Holes}
The main result of Chapter \ref{chapter:LongEvenHole} is an algorithm to determine whether an input graph $G$ contains a hole of length at least $\ell$ and even for some fixed $\ell \geq 4$.
In this section we will give an overview of prior work on problems related to detecting holes of specific parities. A summary of the results is given in Table \ref{table:hole-results}.

In 1991, Bienstock proved that it is NP-hard to determine whether $G$ contains a even (or odd) hole going through a specified vertex \cite{bienstock_original, bienstock_correction}, answering a question raised by Bruce Shepherd. Maffray and Trotignon extended this result to show that the problem remains NP-hard when we only consider triangle-free graphs as inputs \cite{maffray2005algorithms}. Note that it is trivial to test whether an input $G$ contains a hole through a specified vertex $v$ in time $\mathcal{O}(|G|^3)$: We enumerate all pairs of non-adjacent vertices $x,y \in N(v)$ and test if $G \setminus (N(v) \setminus \{ x, y\})$ contains an $xy$-path (e.g. by using breath first search).

In 2002, Conforti, Cornu\'{e}jols, Kapoor and Vu\v{s}kov\'{i}c  \cite{conforti2002even} gave an approximately $\mathcal{O}(|G|^{40})$ algorithm to test whether a graph contains an even hole by their using their structural decomposition theorem from \cite{conforti2002evenstructure}.
In 2003, Chudnovsky, Kawarabayashi, and Seymour \cite{chudnovsky2005detectingeven} provided a simpler algorithm that searches for even holes without the use of a structural decomposition theorem for even-hole-free graphs in time $\mathcal{O}(|G|^{31})$. In forthcoming work \cite{cheong2020finding}, Cheong and Lu show that techniques of \cite{chudnovsky2005detectingeven} can be used to find the shortest even hole in an input graph $G$ or determine that $G$ is even-hole-free in time $\mathcal{O}(|G|^{31})$.

Significantly faster algorithms have been found using decomposition theorems for even-hole-free graphs based on \cite{conforti2002evenstructure}. In 2008, da Silva and Vu\v{s}kov\'{i}c published a strengthening of the decomposition theorem of \cite{conforti2002evenstructure} along with an algorithm using the new decomposition theorem to test whether a graph is even-hole-free in time $\mathcal{O}(|G|^{19})$ \cite{daSilvaEvenHoleAlg}. In 2015, Chang and Lu \cite{Chang2015EvenHole} gave an $\mathcal{O}(|G|^{11})$  algorithm to determine whether a graph contains an even hole using the decomposition theorem of \cite{daSilvaEvenHoleAlg}. Lai, Lu and Thorup improved this running time to $\mathcal{O}(|G|^9)$ in 2020 \cite{lai2019threeinatree} by modifying the algorithm of \cite{Chang2015EvenHole} to improve the running-time of its subroutines. 

Detecting an odd hole remained open until 2020 when Chudnovksy, Scott, Seymour and Spirkl provided an algorithm to detect an odd hole in $G$ in time $\mathcal{O}(|G|^9)$ \cite{chudnovsky2020oddhole}. In 2020, Lai, Lu and Thorup improved this running time to $\mathcal{O}(|G|^8)$ \cite{lai2019threeinatree}. In the same year, Chudnovsky, Scott, and Seymour \cite{chudnovsky2020shortestodd} gave an algorithm that determines whether a graph $G$ has an odd hole and returns the minimum length of an odd hole in $G$ if one exists in time $\mathcal{O}(|G|^{14})$. 

Chudnovsky, Scott and Seymour give a $\mathcal{O}(|G|^{20\ell + 40})$ algorithm to test whether $G$ contains an odd hole of length at least $\ell$, where $\ell \geq 5$ is given as a constant, in 2019 \cite{chudnovsky2019detectinglongodd}. Paul Seymour and I give an algorithm to test whether $G$ contains an even hole of length at least $\ell$ in time $\mathcal{O}(|G|^{9\ell + 3})$ in forthcoming work \cite{cook2020detecting}. Chapter \ref{chapter:LongEvenHole} describes a variant of this algorithm.

It would be nice for long even hole (and long odd hole) detection to remain polynomial-time when $\ell$ is considered to be part of the input rather than a constant. Unfortunately, this seems highly unlikely. Sepehr Hajebi provided a proof in private communication \cite{w1hard} that detecting long holes with specific residues is W[1]-hard and thus not fixed parameter tractable unless the central conjecture of parameterized complexity theory is false. More precisely, for all integers $r,m$ with $m \geq 2$  and $0 \leq r < m$  if there was an algorithm that on input $G, \ell$ determined whether $G$ contains a hole $C$ of length at length at least $\ell$ and $|E(C)| \cong r \mod m$ in time $\mathcal{O}(f(\ell)*p(|G|))$ where $f$ is some computable function and $p$ is a polynomial, then the central conjecture of parameterized complexity theory (that FPT $\neq$ W[1]) would be false.

\subsection{Other related  algorithmic results}
Prior to the discovery of the first odd-hole detection algorithm by Chudnovsky et al. in 2020, polynomial-time algorithms had been found for detecting odd holes in certain restricted graph classes.
In 1987, Hsu presented an algorithm for detecting odd holes in planar graphs in time $\mathcal{O}(|G|^3)$ \cite{hsu1987planaroddholealg}. In 2009, Schrem, Stern and Golumbic provided an algorithm for detecting odd holes in claw-free\footnote{The claw is the graph consisting of four vertices $v_1, v_2, v_3, v_4$ such that $v_2, v_3, v_4$ are pairwise non-adjacent and $v_1$ is adjacent to each of $v_2, v_3, v_4$.} graphs in time $\mathcal{O}(|G|*|E(G)|^2)$  using an approach based on breadth-first-search \cite{shrem2010clawoddalg}. This result was improved four years later when Kennedy and King provided an                                                                                                                                                                                                                                                                                                                                                                                                                                                                                                                  algorithm to detect odd holes in claw-free graphs in time $\mathcal{O}(|E(G)|^2 + |G|^2\log(|G|))$ \cite{kENNEDY20132492} using structural results of Fouquet \cite{FOUQUET199335} and Chudnovsky and Seymour \cite{chudnovsky_seymour_clawfree_2005}. In 2006, Conforti, Cornu{\'e}jols, Liu, Xinming, Vu{\v{s}}kovi{\'c}, and Zambelli provided a polynomial-time algorithm to test for odd holes in graphs with bounded clique number \cite{conforti2006oddholeboundedclique}.

Porto provided an algorithm to test whether a planar graph $G$ is even-hole-free in time $\mathcal{O}(|G|^3)$ \cite{porto1992even}. Itah and Rodeh provided an $\mathcal{O}(|G||E(G)|)$ algorithm to find the girth of a planar graph \cite{itah-rodeh-girth-alg-1978} in 1978. 
This result was subsequently improved by Djidev \cite{djidev-2010-girth-planar-alg}, by the min cut algorithm of Chalermsook, Fakcharoenphol and Nanongkai \cite{chalermsook-Fakcharoenphol-Nanongkai-min-cut-girth} and by Weimann and Yuster \cite{weimann-yuster-girth}.
Chang and Lu provided a linear time algorithm to determine the girth of a planar graph in 2011 \cite{Chang-Lu-2011-linear-time-girth}.
The complexity of determining whether $G$ contains a hole of length at least five is $\mathcal{O}(|E(G)|^2 + |G|)$ by an algorithm of Nikolopoulos, and Palios \cite{Nikolopoulos2007, Nikolopoulos2004Proceedings}.

In his 1992 paper Bienstock also showed that it is NP-hard to determine whether an input graph contains a hole through two prespecified vertices \cite{bienstock_original}. However, when the problem is restricted to planar graphs it becomes solvable in polynomial-time: 
In fact, for any fixed integer $k \geq 0$,  Kawarabayashi and Kobayashi give a linear time algorithm to test whether a planar graph contains a hole going through $k$ fixed vertices. Moreover, for every $\epsilon \geq 0$ they provide an algorithm that on input a planar graph $G$ and $L \subseteq V(G)$ of cardinality $o((\frac{\log|G|}{\log \log |G|})^{2/3})$ tests whether $G$ contains a hole going through every vertex in $L$ in time $\mathcal{O}(|G|^{2 + \epsilon})$ \cite{kawarabayashi2010}.

The previous results were all concerned with algorithms to decide whether $G$ contained various types of cycles as an induced subgraph. There has also been work on algorithms to decide whether $G$ has a cycle of length $k$ as subgraph for some input $k$, to find a cycle of length $k$ if one exists and to count the number of cycles of length $k$ in $G$. For instance, Alon, Yuster and Zwick present several results of this type in \cite{Alon1997}.

%% file: Tikz/thetaprismpyramid.tex
\begin{tikzpicture}
	\def\height{-2.5}
	\def\width{1}
	\def\spacing{5}
	
	\node[dot] (a1) at (1,0) {};
	\node[dot] (a11) at (1, -1) {};
	\foreach \x in {0,1,2}{
		\pgfmathparse{int(\x + 1)}
		\edef\z{\pgfmathresult}
		\node[dot] (b\x) at (\x*\width, \height){};
		\foreach \y in {0, ..., \x}{
			\draw[bend left = 80] (b\x) to  (b\y);
		}
	}
	\draw[path] (a1) to (b0);
	\draw[path] (a1) to (b2);
	\draw (a1) to (a11);
	\draw[path] (a11) to (b1);

	%
	%
	%
	\node[dot] (a) at (1 -\spacing,0) {};
	\node[dot, below left=1em and 1.5em of a] (c1) {};
	\node[dot, below=1em of a] (c2) {};
	\node[dot, below right=1em and 1.5em of a] (c3) {};
	\node[dot] (b) at (1 -\spacing,\height) {};
	\draw[path] (c2) to (b);
	\draw[path, bend right = 50] (c1) to  (b);
	\draw[path, bend left = 50] (c3) to  (b);
	\foreach \x in {1,2,3}{\draw (a) to (c\x); }
	%
	%
	%
	%
	%
	%
	\foreach \x in {0,1,2}{
		\pgfmathparse{int(\x + 1)}
		\edef\z{\pgfmathresult}
		\node[dot](c\x) at (\x*\width + \spacing, 0){};
		\node[dot] (d\x) at (\x*\width + \spacing, \height){};
		\foreach \y in {0, ..., \x}{
			\draw[bend left = 80] (d\x) to  (d\y);
			\draw[bend right = 80] (c\x) to  (c\y);
		}
	}
	\draw[path] (c0) to  (d0);
	\draw[path] (c1) to  (d1);
	\draw[path] (c2) to (d2);
\end{tikzpicture}

%% file: Tikz/fat_C5.tex
\begin{tikzpicture}
	\newcounter{count}
	\foreach \i in {1,2,3}{
		\node[
		regular polygon,
		regular polygon sides=5,
		minimum size= \i * 1cm + 2cm,
		rotate=180/16,
		] (r\i) {};
	}
	
	\fill[radius=3pt] \foreach \i in {1, ..., 5} { (r1.corner \i) circle[] };
	\fill[radius=3pt] \foreach \i in {1,3,4,5} { (r2.corner \i) circle[] };
	\fill[radius=3pt] \foreach \i in {3,5} { (r3.corner \i) circle[] };
	
	\draw (r1.corner 1) to (r2.corner 1);
	\draw (r1.corner 4) to (r2.corner 4);
	
	\foreach \i in {3,5}{
		\draw (r1.corner \i) to (r2.corner \i);
		\draw (r2.corner \i) to (r3.corner \i);
		\draw[bend left = 40] (r1.corner \i) to (r3.corner \i);
	}
	
	\draw (r1.corner 1) to (r1.corner 2);
	\draw (r1.corner 2) to (r1.corner 3);
	\draw (r1.corner 3) to (r1.corner 4);
	\draw (r1.corner 4) to (r1.corner 5);
	\draw (r1.corner 5) to (r1.corner 1);
	
	\draw (r1.corner 2) to (r2.corner 1);
	
	\draw (r1.corner 2) to (r2.corner 3);
	\draw (r1.corner 2) to (r3.corner 3);

	\draw (r1.corner 1) to (r2.corner 5);
	\draw (r1.corner 1) to (r3.corner 5);
	
	\draw (r1.corner 4) to (r2.corner 5);
	\draw (r1.corner 4) to (r3.corner 5);
	
	\draw (r1.corner 4) to (r2.corner 3);
	\draw (r1.corner 4) to (r3.corner 3);
	
	\draw (r2.corner 1) to (r2.corner 5);
	\draw (r2.corner 5) to (r2.corner 4);
	
	\draw (r2.corner 4) to (r2.corner 3);
	\draw (r1.corner 3) to (r2.corner 4);
	
	\draw (r3.corner 3) to (r2.corner 4);
	
	\begin{scope}[on behind layer]
	\fill[radius=20 pt, bag] \foreach \i in {1,...,5} {(r2.corner \i) circle[] };
	\end{scope}

\end{tikzpicture}

%% file: long_even_holes_overleaf.tex
\chapter{Detecting a Long Even Hole}\label{chapter:LongEvenHole}

\section{Technical Overview}
The main result of this chapter is the following:

\begin{theorem}\label{mainthm}
For each integer $\ell \geq 4$, there is an algorithm with the following specifications:
\begin{description}
\item [Input:] A graph $G$.
\item [Output:] Decides whether $G$ has an even hole of length at least $\ell$.
\item [Running time:] $\mathcal{O}(|G|^{108\ell-22})$
\end{description}
\end{theorem}


Our algorithm combines approaches described in \cite{chudnovsky2005detectingeven} and \cite{chudnovsky2019detectinglongodd}.The new algorithm uses a technique called ``cleaning'', as do the algorithms of \cite{chudnovsky2005detectingeven},\cite{chudnovsky2019detectinglongodd} and many other algorithms to detect induced subgraphs. We test for the existence of long even holes without the use of a decomposition theorem as is done in the algorithm of \cite{chudnovsky2005detectingeven}.

Here is an outline of the method. Throughout this paper $\ell \geq 4$ is a fixed integer and a {\em long} hole or path is a hole or path of length at least $\ell$. If $C$ is a hole in $G$, a vertex $v$ of $V(G) \setminus V(C)$ is {\em $C$-major} if there is no subpath of $C$ of length three containing all neighbors of $v$ in $V(C)$. A hole $C$ is {\em clean} if it has no $C$-major vertex.

\begin{itemize}
\item First, we test for the presence in the input graph $G$ of certain kinds of induced subgraphs (``short'' long even holes, ``long jewels of bounded order'', ``long thetas'', a type of wheel called ``long ban-the bombs' and ``long near-prisms'' that are detectable in polynomial time and whose presence would imply that $G$ contains a long even hole. We call these kinds of subgraphs ``easily-detected configurations.'' We may assume these tests are unsuccessful.
\item
Second, we generate a {\em cleaning list}, a list of polynomially many subsets of $V(G)$ such that if $C$ is a long even hole of minimum length in $G$ ({\em a shortest long even hole}) then for some set $X$ in the list, $X$ contains every $C$-major vertex and no vertex of $C$. This process depends on the absence of easily-detected configurations.
\item
Third, for every $X$ in our cleaning list we check whether $G \setminus X$ contains a clean shortest long even hole. This depends on the absence of easily-detected configurations and major vertices.
\end{itemize}

We remark that we are calling long near-prisms easily detectable configurations, because they are detectable in polynomial time in graphs without long thetas as an induced subgraph. However for a general graph $G$, deciding whether $G$ contains a long near-prism is NP-complete;  Maffray and Trotignon's proof \cite{maffray2005algorithms} that deciding whether $G$ contains a prism is NP-complete can easily be adjusted to prove that deciding whether $G$ contains a long near-prism is NP-complete. We are able to detect long thetas by invoking the ``three-in-a-tree'' algorithm given in \cite{chudnovsky2010three}. The detection of long near-prisms makes up the bulk of what is novel in this paper and is the computationally most expensive step of our algorithm.

The approach of determining whether $G$ contains an even hole by first testing whether $G$ contains a prism or a theta was outlined in \cite{chudnovsky2005detectingeven}. Moreover, Chudnovsky and Kapadia gave an algorithm to decide whether $G$ contains a theta or a prism in \cite{chudnovsky2008thetaprism}. Their algorithm does not directly translate to long theta and long near-prism detection, but we were able to use a similar algorithm structure to detect long near-prisms when $G$ contains no long theta. Finally, when $G$ has no easily detectable configurations, we detect a clean shortest long even hole $C$ by guessing three evenly spaced vertices along $C$ and taking shortest paths between them as in \cite{chudnovsky2019detectinglongodd}.

\section{The easily-detected configurations}

We begin with a test for what we call ``short'' long even holes:
\begin{theorem} \label{alg:shortlongevenholes}
For each integer $k \geq \ell$, there is an algorithm with the following specifications:
\begin{description}
\item[Input:] A graph $G$.
\item[Output:] Decides whether $G$ has a long even hole of length at most $k$.
\item[Running Time:] $\mathcal{O}(|G|^k)$.
\end{description}
\end{theorem}

\begin{proof}
We enumerate all vertex sets of size $\ell, \ell+1, \dots, k$ and for each one, check whether it induces a long even hole.
\end{proof}

We need the following modification of an easily-detected configuration of \cite{chudnovsky2019detectinglongodd}.
Let $u,v \in V(G)$ and let $Q_1, Q_2$ be induced paths between $u,v$ of different parity. Let $P$ be an induced path between $u,v$ of length at least $\ell$, such that $P^*$ is disjoint from and anticomplete to $Q_1^*$, $Q_2^*$. 
We say the subgraph $H$ induced by $V(P \cup Q_1 \cup Q_2)$ is a {\em jewel of order $\max{(|V(Q_1)|, |V(Q_2)|)}$ formed by $Q_1, Q_2, P$}.
If every hole in $H$ is long then we call $H$ a {\em jewel of order $\max{(|V(Q_1)|, |V(Q_2)|)}$.}
Note $H$ is long if and only if $P$ has length at least $\ell - \min \{ |E(Q_1)|, |E(Q_2) \}$.

We need the following (slight modification) of an easy result given as Theorem 2.2 of \cite{chudnovsky2019detectinglongodd}.

\begin{theorem} \label{alg:longjewels}
There is an algorithm with the following specifications.
\begin{description}
\item[Input:] A graph $G$ and an integer $k \geq 0$.
\item[Output:] Decides whether $G$ has a long jewel of order at most $k$.
\item[Running Time:] $\mathcal{O}(|G|^n)$ where $n =  k + 1 + \max \{ k , \ell -1\}$.
\end{description} 
\end{theorem}
\begin{proof}
We a triple of induced paths $Q_1, Q_2, R$ a \textit{jewel fragment} if all of the following conditions hold:
\begin{itemize}
	\item $Q_1, Q_2$ have the same two ends, say $u, v$,
	\item $Q_1$ is odd and $Q_2$ is even,
	\item $Q_1, Q_2$ each have length at most $k -1$,
	\item $R$ has length $\max \{0,  \ell - \min \{ |E(Q_1)|, |E(Q_2)| \} \}$)
	\item $R$ has ends $u, w$ for some $w \in V(G)$ and
	\item $V(R \setminus u)$ is anticomplete to $V(Q_1 \cup Q_2) \setminus \{ u\}$.
\end{itemize}
We will use the notation introduced in the definition of jewel fragment to state our algorithm and prove its correctness.
Our algorithm is as follows:
We enumerate all jewel fragments in $G$.
For each jewel fragment $Q_1, Q_2, R$ in $G$ we perform the following: We compute the set $X$ of all vertices that are anticomplete to $V(Q_1 \cup Q_2 \cup R) \setminus \{ u, v\}$. Then we test whether $G[X \cup \{ w, v\}]$ contains a $wv$-path $Z$, eg. by using breadth-first search. If so, we output that $G$ contains a long jewel.
If we have checked every triple without finding a long jewel we output that $G$ contains no long jewel.

We show that the output is correct.
Suppose that for some jewel fragment  $Q_1, Q_2, R$, the path $Z$ exists.
Then by construction, $Q_1, Q_2, R \cup Z$ forms a long jewel of order at most $k$.
Thus if the algorithm will not output that $G$ contains a long jewel of order at most $k$ unless it actually contains one.

Suppose $G$ contains a long jewel of order at most $k$ formed by paths $Q_1, Q_2, P$.
Let $R$ be a subpath of $P$ of length $\max \{0,  \ell - \min \{ |E(Q_1)|, |E(Q_2)| \} \}$ with one end equal to $u$.
Then $Q_1, Q_2, R$ is a jewel fragment. Hence the algorithm will output that $G$ contains a long jewel of order at most $k$ if $G$ contains one.

Checking a jewel fragment takes time $\mathcal{O}(|G|^2)$.
Let $n=k - 1 + \max \{ k, \ell -1\}$
Since every jewel fragment contains at most $n$ vertices there are at most $|G|^n$ of them and it takes $\mathcal{O}(|G|^n)$ time to find them all.
Hence the running time is as claimed.
\end{proof}
A {\em theta} is a graph consisting of two non-adjacent vertices $u,v$ and three paths $P_1, P_2, P_3$ joining $u,v$ with pairwise disjoint interiors and we say $P_1, P_2, P_3$ {\em form} a theta. A {\em long theta} where for every two distinct $i,j \in \{1,2,3\}$, $|E(P_i)| + |E(P_j)| \geq \ell$. If $G$ contains a long theta it contains a long hole because for every distinct $i,j \in \{1,2,3\}$, $V(P_i) \cup V(P_j)$ induces a long hole and at least two of $P_1, P_2, P_3$ must have the same parity. We use the ``three-in-a-tree'' algorithm given as the main result of \cite{chudnovsky2010three} to detect long thetas:

\begin{theorem} \label{alg:threeinatree}
There is an algorithm with the following specifications:
\begin{description}
\item[Input:] A graph $G$ and three vertices $v_1,v_2,v_3$ of $G$.
\item[Output:] Decides whether there is an induced subgraph $T$ of $G$ with $v_1,v_2,v_3 \in V(T)$ such that $T$ is a tree.
\item[Running Time:] $\mathcal{O}(|G|^4)$.
\end{description} 
\end{theorem}

Chudnovsky and Seymour's algorithm in \cite{chudnovsky2010three} to detect a theta in a graph $G$ can easily be adjusted to detect a long theta.
We need the following definition:
We call a the graph $Z$ consisting of the union of three paths $Q_1, Q_2, Q_3$ with a common end $a$ and otherwise vertex-disjoint a \textit{long claw} if $Q_1, Q_2, Q_3$ have lengths $k_1, k_2, k_3$, respectively satisfying the following:
\begin{itemize}
	\item $k_1, k_2, k_3 \geq 2$,
	\item $k_1 + k_2, k_2 + k_3, k_3 + k_1 \geq \ell -2$ and
	\item $k_1 + k_2 + k_3 \leq 2 \ell - 6$.
\end{itemize}
We say $Q_1, Q_2, Q_2$ form $Z$.

\begin{lemma}
	Let $H$ be a theta. Then $H$ is a long theta if and only if $H$ contains a long claw. \label{lem:theta:has-a-claw}
\end{lemma}
\begin{proof}
	Let $P_1, P_2, P_3$ be the constituent paths fo $H$ and let $x, y$ be the two vertices of degree to in $H$.
	\stmt{If $H$ is a long theta then $H$ contains a long claw. \label{bleh:one}}
	Suppose $H$ is a long theta.
	Suppose each of $P_1, P_2, P_3$ have length at least $\frac{\ell}{2}$.
	Then for each $i \in \{1,2,3\}$, let $Q_i$ be the subpath of $P_i$ of length $\ell -1$ with one end equal to $x$.
	Thus $Q_1 , Q_2, Q_3$ form a long claw.
	
	Hence we may assume $|E(P_1)| < \frac{\ell}{2}$.
	By definition of long theta $|E(P_2)|, |E(P_3)| \geq \ell - |E(P_1)|$.
	Let $Q_2, Q_3$ be the subpaths of $P_2, P_3$, respectively, of length $\ell - |E(P_1)| -1 $ with one end equal to $x$.
	Then $Q_1, Q_2, Q_3$ form a long theta.
	This proves (\ref{bleh:one}).
	\stmt{If $H$ contains a long claw, then $H$ is a long theta. \label{bleh:two}}
	Suppose $H$ is a theta with constituent paths $P_1, P_2, P_3$.
	Suppose $Q_1, Q_2, Q_3$ form a long claw contained in $H$.
	Without loss of generality $x$ is the common end of $Q_1, Q_2, Q_3$.
	For each $i \in \{1,2,3\}$ let $r_i$ denote the other end of $Q_i$.
	By of long claw, $r_1, r_2, r_3$ are pairwise non-adjacent.
	Hence for any two distinct $i, j \in \{1,2,3\}$, the hole $P_i \cup P_2$ has length at least $|E(Q_i)| + |E(Q_j)| + 2 \geq \ell$.
	This proves (\ref{bleh:two}).
\end{proof}

\begin{theorem}\label{alg:longtheta}
There is an algorithm with the following specifications:
\begin{description}
\item[Input:] A graph $G$.
\item[Output:] Decides whether $G$ contains a long theta.
\item[Running Time:] $\mathcal{O}(|G|^{2\ell+7})$.
\end{description}
\end{theorem}
\begin{proof}
Let $Z$ be a long claw in $G$ and let $r_1, r_2, r_3$ be the vertices of degree one in $Z$.
Let $G_Z$ be the graph obtained from $G$ by deleting all vertices other than $r_1, r_2, r_3$ that belong to or have a neighbor in $V(B) \setminus \{r_1, r_2, r_3\}$. Then it follows from Lemma \ref{lem:theta:has-a-claw}, $Z$ is the induced subgraph of a long theta in $G$ if and only if $G_Z$ contains some tree $T$ with $r_1, r_2, r_3 \in V(T)$.

The algorithm is as follows:
We enumerate all induced claws in $G$.
For every induced claw $Z$ in $G$ we compute $G_Z$ and test whether $G_Z$ has an induced tree containing the three vertices of degree one in $B$ by using the algorithm of \cite{chudnovsky2010three}.
Long claws have at most $2\ell - 5$ vertices so there are at most $|G|^{2\ell -5}$ of them.
Hence the running time is $\mathcal{O}(|G|^{2 \ell -1})$.
\end{proof}

Lai, Lu and Thorup provide a faster algorithm for the three-in-a-tree problem in \cite{lai2019threeinatree}. Using their $\mathcal{O}(|E(G)|(\log |G|)^2)$ algorithm we can reduce the running time for detecting a long theta to $\mathcal{O}(|G|^{2\ell-3}(\log|G|)^2)$. This improvement does not affect the asymptotic running time of our long even holes detection algorithm.

For brevity, it is convenient to describe enumerating all subgraphs of a certain type as ``guessing'' subgraphs of that type. In this language the three-in-a-tree algorithm can be written as follows: We guess the paths $Q_1$, $Q_2$ and $Q_3$ and test whether $r_1, r_2, r_3$ are contained in some induced tree of $G_Z$.

\begin{figure}[!h]
	\centering
	\input{Tikz/banthebomb.tex}
	\caption{An illustration of a ban-the-bomb.}
	\label{fig:banthebomb}
\end{figure}
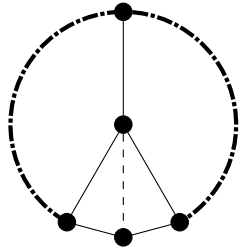

Let us say a \textit{ban-the-bomb} is the graph consisting of a hole $C$ and a vertex $v$ satisfying the following property: There is some path $x \dd y \dd z$ of $C$ and a vertex $w \in V(C)$ such that $w$ is non-adjacent to $x,y$ and $v$ is adjacent to $w, x, z$ and $v$ has no neighbors in $V(C) \setminus \{w, x, y, z\}$. (See Figure \ref{fig:banthebomb}).
We say a ban-the-bomb $B$ is \textit{long} if every hole in $B$ is long.
We will reduce detecting long ban-the-bombs in graphs without thetas to the three-in-a-tree algorithm.
We need the following definition.

A \textit{long bomb} is a graph consisting of a path $R$ of length $2\ell -6$ with three center vertices $x \dd y \dd z$ in order and two more vertices $w, v$ where $v$ is adjacent to $w, x, y, z$ and there are no other edges. See Figure \ref{fig:bomb}.
\begin{figure}
	\centering
	\input{Tikz/bomb.tex}
	\caption{An illustration of a long bomb. The thick dashed paths are both of length $\ell -4$.}
	\label{fig:bomb}
\end{figure}
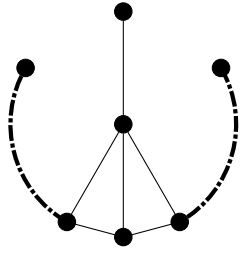

\begin{lemma}\label{lem:banbombalg}
	Let $G$ be a graph containing no long theta or hole of length four.
	If $B$ is a bomb in $G$ and $L$ is the set of vertices of degree one in $B$, let $G_B$ denote the graph obtained from $G$ by deleting every vertex in $V(G) \setminus L$ that belong to or have a neighbor in $V(B) \setminus L$.
	Then $G$ contains a long ban-the-bomb if and only if $G$ contains a long bomb $B$ such that $G_B$ contains a tree $T$ satisfying the following condition: $L \subseteq V(T)$ where $L$ is the set of degree vertices of degree one in $B$.
\end{lemma}
\begin{proof}
	Suppose $G$ contains a long ban-the-bomb $H$. Then since $G$ is $C_4$-free and every hole in $H$ has length at least $\ell$, $B$ contains a long bomb $B$.
	Then by definition, $G_B$ contains the desired tree from the statement of the Lemma.
	
	Suppose $G$ contains a long bomb $B$ where $L$ is the set of vertices of degree one in $B$.
	Suppose $G_B$ contains a tree $T$ with $L \subseteq V(T)$. Choose $T$ to be minimal. Then if $T$ is a path $B \cup T$ is a long ban-the-bomb in $G$. Hence we may assume $L$ is the set of leaves of $T$.
	Then $T$ has exactly one vertex of degree one $t$. Moreover, $T$ is the union of three paths $P_1, P_2, P_3$ each with one end equal to $t$ and the other equal to an element of $L$. By assumption each of $P_1, P_2, P_3$ has length at least one.
	
	Let $R$ be the path of length $2\ell - 6$ contained in $B$. Let $v$ be the vertex of degree four in $B$.
	Let $x,y,z$ be the neighbors of $v$ in $R$ such that $x \dd y \dd z$ is a subpath of $R$.
	Then $B \cup P_1 \cup P_2 \cup P_3 \setminus y$ is a theta and it is long, a contradiction.
\end{proof}

\begin{theorem}\label{alg:banthebomb}
	For each integer $\ell \geq 4$, there is an algorithm with the following specifications:
	\begin{description}
		\item[Input:] A graph $G$ containing no long theta or hole of length four.
		\item[Output:] Decides whether $G$ contains a long ban-the-bomb.
		\item[Running Time:] $\mathcal{O}(|G|^{2\ell + 1})$.
	\end{description}
\end{theorem}
\begin{proof}
	The algorithm is as follows:
	We enumerate all bombs $B$ in $G$.
	For each bomb $B$ in $G$ we identify the set $L$ of vertices of degree one in $B$.
	We construct the graph $G_B$ obtained from $G$ by deleting all vertices in $V(G) \setminus L$ that are equal or adjacent to a vertex in $V(B) \setminus L$.
	We test whether $G_B$ contains a tree $T$ with $L \subseteq V(T)$ using the algorithm of \cite{chudnovsky2010three}.
	
	Correctness follows from Lemma \ref{lem:banbombalg}.
	A bomb contains $2 \ell -3$ vertices so there are at most $|G|^{2 \ell -3}$ of them.
	Hence the running time is $\mathcal{O}(|G|^{2\ell + 1})$.
\end{proof}

\begin{figure}[!h]
	\centering
	\input{Tikz/longnearprism.tex}
	\caption{An illustration of long near-prisms. In each drawing two of the thick dashed lines must represent paths of length at least $\ell$.}
	\label{fig:longnearprism}
\end{figure}
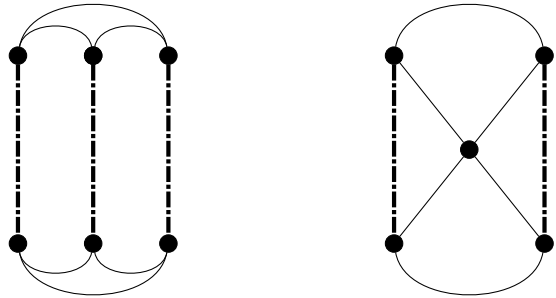
A {\em long near-prism} is a graph $K$ consisting of two triangles on $\{a_1, a_2, a_3\}$ and $\{b_1,b_2,b_3\}$ called {\em bases} and three pairwise vertex-disjoint paths $P_1, P_2, P_3$ such that all of the following conditions hold:
\begin{itemize}
    \item The bases of $K$ are vertex disjoint or $a_i = b_i$ for exactly one $i \in \{1,2,3\}$.
    \item For every $i \in \{1,2,3\}$, $P_i$ is has ends $a_i$ and $b_i$.
    \item At most one of $P_1, P_2, P_3$ has length less than $\ell$.
\end{itemize}
We call $P_1, P_2, P_3$ the {\em constituent paths of $K$}.
(See Figure \ref{fig:longnearprism}.)
Note by definition, every hole in a long near prism is long and every long near prism contains a long even hole.
The next section will describe our algorithm to test whether $G$ contains a long near-prism when $G$ contains no long thetas.
\section{Long near-prisms}
 
\begin{theorem}\label{alg:longprisms}
For each integer $\ell \geq 4$, there is an algorithm with the following specifications:
\begin{description}
\item[Input:] A graph $G$ containing no long theta.
\item[Output:] Decides whether $G$ contains a long near-prism.
\item[Running Time:] $\mathcal{O}(|G|^{108\ell -22})$.
\end{description}
\end{theorem}

We will use the same notation as in the definition of long near-prisms.
The outline of the algorithm is similar to that of \cite{chudnovsky2008thetaprism}.
For a constituent path $P_i$ of $K$ we define $P_i'$, $P_i''$ to be the subpaths of $P_i$ whose vertex sets consist of all vertices with $P_i$-distance at most $\ell-1$ from $a_i, b_i$, respectively. Thus $P_i'$ and $P_i''$ each have one end equal to $a_i$, $b_i$, respectively.
We denote the other ends of $P_i'$ and $P_i''$ by $s_i$ and $t_i$ respectively.
We define a {\em frame} $F$ of a long near-prism $K$ to be the graph obtained by taking the union of the following graphs: 
\begin{itemize}
    \item The triangle induced by $\{a_1, a_2, a_3 \}$.
    \item The triangle induced by $\{b_1, b_2, b_3 \}$.
    \item The paths $P_i'$, $P_i''$ for each $i \in \{1,2,3\}$.
\end{itemize}
Note that for every $i \in \{1,2,3 \}$ if $P_i$ has length less than $2\ell-3$, then $P_i \subseteq F$.
We call the frame of $K$ {\em tidy} if no vertex in the frame has a neighbor in $V(G) \setminus V(K)$ except for vertices equal to $s_i$ or $t_i$ for some $i \in \{1,2, 3\}$ such that $P_i \not \subseteq F$. We call a long near-prism $K'$ {\em shorter} than a long near-prism $K$, if $|V(K')| < |V(K)|$. Let $K$ be a shortest long near-prism in $G$ with paths $P_1, P_2, P_3$. Let $u, v$ be distinct, non-adjacent vertices in $V(P_i)$ for some $i \in \{1,2,3\}$. We call a shortest $uv$-path $Q$ {\em good} if no vertex of $Q^*$ has neighbors in $V(K)\setminus (V(uP_iv) \setminus \{u, v \})$. If $Q$ is not good it is {\em bad}. Hence, for any $i \in \{1,2,3\}$, if $P_i \subseteq F$ and $F$ is tidy, then $s_iP_it_i$ is the unique $s_it_i$-path and thus all shortest $s_it_i$-paths are good.

Suppose all shortest $s_it_i$-paths are good for each $i \in \{1,2,3\}$. Then the problem becomes easy because of the following: 
we first guess the frame $F$ of $K$. If $V(P_1) \subseteq F$, we set $Q_1$ to be the empty graph. Otherwise, we set $Q_1$ to be a shortest $s_1t_1$-path.
We delete all vertices not in $V(F) \cup V(Q_1)$ with neighbors in $Q_1^*$. Since $Q_1$ is good we have not deleted any vertex of $V(K) \setminus V(P_1)$. 
We repeat this process on $P_2, P_3$ to obtain $Q_2$, $Q_3$.
Then, $V(F) \cup_{i = 1}^3 V(Q_i)$ induces a shortest long near-prism. 

We call a vertex $q \in V(G) \setminus V(K)$ {\em $K$-major} if there is no three-vertex subpath of $K$ containing all neighbors of $q$ in $V(K)$.  In order to arrange that all shortest $s_it_i$-paths are good we generate a ``path-cleaning'' list of polynomially many sets of vertices such that for some $X$ in the list $X$ is disjoint from $K$ and $X$ contains a vertex from every bad shortest $s_it_i$-path. This process is described in Subsection \ref{pathcleaning}. 

In order to generate this cleaning list we require that $G$ contains no $K$-major vertices. Thus we need another phase of cleaning where we generate a cleaning list of polynomially many sets of vertices such that for some $X$ in the list $X$ is disjoint from $K$ and $X$ contains all $K$-major vertices. This phase is described in Subsection \ref{kmajor}.

\subsection{Path-cleaning}\label{pathcleaning}
We use the notation from the definition of prism and frame throughout this section. Let $i \in \{1,2,3\}$, let $u, v$ be non-adjacent vertices in $V(P_i)$, and suppose $a_i, u, v, b_i$ occur in order along $P_i$. For a bad shortest $uv$-path $Q$ we define $\zeta_Q$ to be the vertex in $Q^*$ with minimum $Q$-distance to $v$ with a neighbor in $V(K) \setminus V(uP_iv)$. 

In this section we will provide a cleaning algorithm that generates a list of polynomially many sets of vertices such that for any shortest long near-prism $K$ with a tidy frame and vertices $u,v$ in the same constituent path of $K$ for some $X$ in the list, $X$ contains a $\zeta_Q$ for every bad shortest $uv$-path $Q$ for which $\zeta_Q$ is not $K$-major and $X \cap V(K) = \emptyset$. This algorithm depends on $G$ containing no long thetas. 
We will apply this algorithm twice, once to help clean major vertices and once to clean a vertex from all bad shortest $s_it_i$-paths for each $i \in \{ 1,2,3 \}$. Note that in \cite{cook2020detecting} we are able to skip this step entirely through a more careful structural analysis.

For a vertex $q$ with a neighbor in $V(P_i)$ for some $i \in \{1,2,3\}$, let $\alpha_i(q)$ denote the element of $N(q) \cap V(P_i)$ with minimum $P_i$-distance to $a_i$. Similarly, let $\beta_i(q)$ denote the element of $N(q) \cap V(P_i)$ with minimum $P_i$-distance to $b_i$. If $F$ is tidy, it follows that $q \not \in V(F)$ and the paths $a_iP_i\alpha_i(q)$ and $b_iP_i\beta_i(q)$ both have length at least $\ell$.

We need the following lemma:

\begin{lem}\label{lem:prismpathsets}
Let $G$ be a graph without long thetas and $K$ be a shortest long near-prism in $G$. Let $P_1, P_2, P_3$ be the constituent paths of $K$. Suppose $K$ has a tidy frame.
Let $u,v$ be distinct non-adjacent vertices in $V(P_1)$. Suppose $Q, Q'$ are bad shortest $uv$-paths such that $\zeta_Q$, $\zeta_{Q'}$ each have neighbors in $V(P_2)$ and are not $K$-major. Then there is a $\zeta_Q\zeta_{Q'}$-path of length at most $\ell+1$ with interior contained in $V(P_2)$.
\end{lem}

\begin{proof}
Suppose not. Then the $P_2$-distance between any neighbor of $\zeta_Q$ in $V(P_2)$ and any neighbor of $\zeta_{Q'}$ in $V(P_2)$ is at least $\ell$. Without loss of generality suppose $a_1,u, v, b_2$ occur in order along $P_1$.
\stmt{The set of neighbors of $\zeta_{Q}$ in $V(K)\setminus V(uP_1v)$ does not consist of exactly two adjacent vertices. The same statement holds for $\zeta_{Q'}.$\label{dumpling:1}}
Suppose $\zeta_{Q}$ has exactly two neighbors $y_1, y_2$ in $V(P_2)$ and they are adjacent. Then we obtain a shorter long near-prism induced by $(V(K) \setminus V(a_1P_1v)) \cup V(\zeta_QQv)$, a contradiction. This proves (\ref{dumpling:1}).
\stmt{ There is an induced $\zeta_Q\zeta_{Q'}$-path $W$ disjoint from $V(K) \setminus V(uP_1v)$ of length greater than one. \label{dumpling:2}}
There is a $\zeta_Q\zeta_{Q'}$-path $W \subseteq \zeta_QQv \cup vQ'\zeta_{Q'}$, so we only need to show that $W$ is not a single edge. Suppose that $\zeta_Q$ and $\zeta_Q'$ are adjacent.
Without loss of generality suppose that $a_2, \alpha_2(\zeta_Q), \alpha_2(\zeta_{Q'}), b_2$ occur in order on $P_2$.
Let $K'$ denote the graph from $K$ obtained by replacing the path $\alpha_2(\zeta_Q)P_2\beta_2(\zeta_{Q'})$ with $\alpha_2(\zeta_Q)\dd \zeta_Q \dd \zeta_{Q'}\dd \beta_2(\zeta_{Q'})$.
Since $K$ has a tidy frame, $K'$ is a long near-prism and it is shorter than $K$. Since $\zeta_Q, \zeta_{Q'}$ are not $K$-major, $G$ contains $K'$ as an induced subgraph, a contradiction. This proves (\ref{dumpling:2}).
\\
\\
Since $\zeta_Q$ is not $K$-major, 
$\zeta_Q$ has at most one neighbor in $V(K)$ not equal to $\alpha_2(\zeta_Q)$, $\beta_2(\zeta_Q)$. Similarly, $\zeta_{Q'}$ has at most one neighbor in $V(K)$ not equal to $\alpha_2(\zeta_{Q'}), \beta_2(\zeta_{Q'})$. Let $H = ((N(\zeta_Q) \cup N(\zeta_{Q'})) \cap V(K)) \setminus \{\alpha_2(\zeta_Q), \alpha_2(\zeta_{Q'}, \beta_2(\zeta_Q), \beta_2(\zeta_{Q'} \}$. Then $V(W) \cup V(P_2) \cup V(P_3) \setminus H$ induces a long theta, a contradiction.
\end{proof}

\begin{theorem} \label{thm:prismpathcleaning}
There is an algorithm with the following specifications:
\begin{description}
\item[Input:] A graph $G$ containing no long theta.
\item[Output:] A list of $\mathcal{O}(|G|^{4\ell+2})$ subsets of $V(G)$ with the following property: for every shortest long near-prism $K$, if $K$ has a tidy frame and $u,v$ are distinct non-adjacent vertices in the same constituent path of $K$ then there is a set $X$ in the list such that:
\begin{itemize}
    \item $X$ is disjoint from $V(K)$ and
    \item $X \cap V(Q) \neq \emptyset$ for every bad shortest $uv$-path $Q$ such that $\zeta_Q$ is not $K$-major.
\end{itemize}
\item[Running Time:] $\mathcal{O}(|G|^{4\ell+3})$.
\end{description}
\end{theorem}

\begin{proof}
The algorithm is as follows:

We enumerate all pairs $(R_2, R_3)$ of disjoint paths of length at most $2\ell$. For each $i \in \{2, 3 \}$, we set $X_i$ to be the set of all vertices in $V(G) \setminus V(R_i)$ with neighbors in $R_i^*$.
We output $X_2 \cup X_3$ and move on to the next pair of paths.

Let $K$ be a shortest long near-prism with constituent paths $P_1, P_2, P_3$. Suppose $K$ has a tidy frame and that $u,v \in V(P_1)$. We claim there is a choice of $(R_2, R_3)$ such that $X_1 \cup X_2$ is disjoint from $V(K)$ and $X_1 \cup X_2$ contains $\zeta_Q$ for every bad shortest $uv$-path $Q$ such that $\zeta_Q$ is not $K$-major.
Suppose there is a bad $uv$-path $Q$ such that $\zeta_Q$ is not $K$-major and $\zeta_Q$ has a neighbor in $V(P_2)$. 
Then for some choice of $R_2$, we have that $R_2 \subseteq P_2$ and $R_2^*$ contains all vertices with $P_2$-distance at most $\ell-1$ from some neighbor of $\zeta_Q$ in $V(P_2)$. Hence, $\zeta_Q \in X_2$. By Lemma \ref{lem:prismpathsets}, $\zeta_S \in X_2$ for every bad shortest $v_1v_2$-path $S$ such that $\zeta_S$ is not $K$-major and $\zeta_S$ has a neighbor in $V(P_2)$. By construction, $X_2$ is disjoint from $V(K)$. Since the case where $\zeta_Q$ has a neighbor on $P_3$ is symmetric this completes the proof of correctness. 

There are $\mathcal{O}(|G|^{4\ell+2})$ possibilities for $(R_1, R_2)$ and constructing $X_1$ and $X_2$ takes $\mathcal{O}(|G|)$ so the running time and list length are as claimed.
\end{proof}

\begin{corr} \label{corr:fullprismpathcleaning}
There is an algorithm with the following specifications:
\begin{description}
\item[Input:] A graph $G$ containing no long theta.
\item[Output:] A list of $\mathcal{O}(|G|^{12\ell + 6})$ subsets of $V(G)$ with the following property: If $K$ is a shortest long near-prism, $K$ has a tidy frame and there are no $K$-major vertices, then there is a set $X$ in the list such that:
\begin{itemize}
    \item $X$ is disjoint from $V(K)$ and
    \item $X$ contains a vertex of every bad shortest $s_it_i$-path for each $i \in \{1,2,3\}$.
\end{itemize}
\item[Running Time] $\mathcal{O}(|G|^{12\ell+6})$.
\end{description}
\end{corr}
\begin{proof}

We run the algorithm of Theorem \ref{thm:prismpathcleaning} on input $G$ to generate a list $\mathcal{L}$. For each choice of $X, Y, Z \in \mathcal{L}$, we output $X \cup Y \cup Z$.
\end{proof}

\subsection{Major vertices on prisms}\label{kmajor}

For a vertex $x \not \in V(K)$ with a neighbor in $V(P_i)$ for some $i \in \{1,2,3\}$ we denote by $A_i(x)$, $B_i(x)$, the paths $\alpha_i(x)P_ia_i$ and $\beta_i(x)P_ib_i$, respectively. 
 
\begin{lem} \label{2paths}
Let $K$ be a shortest long near-prism in $G$ with constituent paths $P_1, P_2, P_3$. Suppose $K$ has a tidy frame. If $q$ is a $K$-major vertex, then $q$ has neighbors in at least two different constituent paths of $K$.
\end{lem}
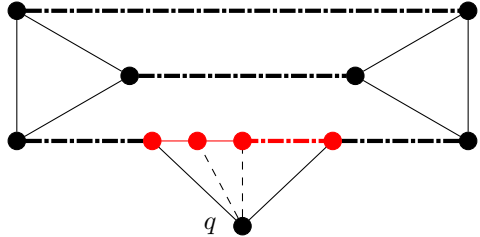
\begin{figure}[!h]
	\centering
	\input{Tikz/prism1path.tex}
	\caption{An illustration of the proof of Lemma \ref{2paths}.}
	\label{fig:2paths}
\end{figure}
\begin{proof}
Suppose all neighbors of $q$ in $V(K)$ are contained in $V(P_1)$. Then $\alpha_1(q)P_1\beta_1(q)$ has length strictly greater than two. We obtain a shorter prism by replacing $\alpha_1(q)P_1\beta_1(q)$ in $P_1$ with the path $\alpha_1(q)q\beta_1(q)$. Since $K'$ has the same frame as $K$, it follows that $K'$ is a long near-prism, a contradiction.
(See Figure \ref{fig:2paths}.)
\end{proof}

\begin{lem}\label{3pairwisenonadjacent}
Let $K$ be a shortest long near-prism. Suppose $K$ has a tidy frame.
If $q$ is a $K$-major vertex, then $q$ has three pairwise non-adjacent neighbors in $V(K)$.
\end{lem}
\begin{proof}
Suppose not.
By Lemma \ref{2paths}, we may assume $q$ has a neighbor in $V(P_1)$ and in $V(P_2)$. Since $K$ has a tidy frame, we may assume $q$ has no neighbors in $V(P_3)$. If $q$ has exactly one neighbor in $V(P_1)$ and exactly one neighbor in $V(P_2)$, then $V(P_1) \cup V(P_2) \cup \{q\}$ induces a long theta. So we may assume $q$ has exactly two neighbors in $V(P_1)$ and that they are adjacent. Then we obtain a shorter long near-prism with bases $\{a_1, a_2, a_3\}$ and $\{\alpha_1(q), \beta_1(q), q \}$ and constituent paths $A_1(q)$, $A_2(q)$ and $B_1(q) \dd P_3$, a contradiction.
\end{proof}
We will show that for any shortest long near-prism $K$, if $K$ has a tidy frame, there is a bounded size set of $K$-major vertices with certain structural properties that can be exploited to clean $K$-major vertices. We first need to state a few definitions. For a $K$-major vertex $q$ we say a vertex $v \in V(K)$ is {\em $q$-internal} if $v \in V(\alpha_i(q)P_i\beta_i(q))$ for some $i \in \{1, 2, 3\}$ such that $q$ has a neighbor in $V(P_i)$. Otherwise, $v \in V(K)$ is called {\em $q$-external}.

\begin{lem}\label{Kexternal}
Let $G$ be a graph containing no long thetas and let $K$ be a shortest long near-prism in $G$. Let $P_1, P_2, P_3$ be the constituent paths of $K$. Suppose $K$ has a tidy frame.
Let $x$ and $y$ be distinct non-adjacent $K$-major vertices. Suppose $y$ has two non-adjacent $x$-external vertices.
Then, there exists an $i \in \{1,2, 3 \}$ such that $N(x) \cap V(P_i) \neq \emptyset$ and $y$ has a neighbor $w \in V(P_i)$ satisfying $\min \{ d_{P_i}(w, \alpha_i(x)), \ d_{P_i}(w, \beta_i(x)) \} \leq \ell-3$.
\end{lem}

\begin{proof}
Suppose that for each $i \in \{1,2,3\}$, either $x$ has no neighbor in $V(P_i)$ or $y$ has no neighbors in $V(P_i)$ with $P_i$-distance at most $\ell-3$ from $\alpha_i(x)$ or $\beta_i(x)$.
\stmt{For all distinct $i,j \in \{1,2,3\}$, if $x$ has neighbors in both $V(P_i)$ and $V(P_j)$, $y$ does not have $x$-external neighbors in both $V(P_i)$ and $V(P_j)$. \label{rabbit:1}}
Suppose $x$ has neighbors in both $V(P_i)$ and $V(P_j)$ and $y$ has $x$-external neighbors in both $V(P_i)$ and $V(P_j)$. Then for all $k  \in \{ i, j \}$, there is a subpath $Q_k$ of $P_k$ of length at least $\ell-2$ with one end an $x$-external neighbor of $y$ and the other end $\alpha_k(x)$ or $\beta_k(x)$ such that $x \dd Q_k \dd y$ is an induced $xy$-path. 
Let $M_k$ denote the subpath of $P_k$ with interior equal to $V(Q_k)$ for each $k \in \{i,j \}$. By definition of $\alpha_i(x)$, $\alpha_j(x)$, $\beta_i(x)$, $\beta_j(x)$ and Lemma \ref{3pairwisenonadjacent}, $x$ has a neighbor in $V(K) \setminus (V(M_i) \cup V(M_j))$.
Since $y$ has no neighbors adjacent to $\alpha_i(x)$, $\alpha_j(x)$, $\beta_i(x)$ or $\beta_j(x)$, it follows from Lemma \ref{3pairwisenonadjacent} that $y$ has a neighbor in $V(K) \setminus (V(M_i) \cup V(M_j))$.

Since $K$ has a tidy frame, $K \setminus (V(M_i) \cup V(M_j))$ is connected. Hence there is an $xy$-path $B$ with interior in $V(K) \setminus (V(M_i) \cup V(M_j))$.
But then, $x \dd Q_i \dd y$, $x \dd Q_j \dd y$ and $B$ form a long theta, a contradiction. This proves (\ref{rabbit:1}).
\stmt{$x$ does not have neighbors in all of three of $V(P_1), V(P_2), V(P_3)$. \label{rabbit:2}}
Suppose $x$ has neighbors in all three of $V(P_1), V(P_2), V(P_3)$. By (\ref{rabbit:1}) we may assume that the $x$-external neighbors of $y$ are contained in $V(P_1)$.
Then, we may assume $y$ has an $x$-internal neighbor in $V(P_2)$. Hence, $\alpha_2(x)$ and $\beta_2(x)$ are not equal or adjacent.
Let $M_1$ denote the induced $xy$-path with interior contained in $V(\beta_2(y)P_2\beta_2(x))$.
We may assume that either $y$ has a neighbor in $V(A_1(x))$ and a neighbor in $V(B_1(x))$ or that $y$ has two non-adjacent neighbors in $V(A_1(x))$.
In the first case, $M_1$, $y \dd A_1(y) \dd A_2(x) \dd x$, $y \dd B_1(y) \dd B_3(x) \dd x$ form a long theta, a contradiction.
Hence, the second case holds. There are two long induced $xy$-paths $M_2$ and $M_3$ with interiors contained in $V(A_1(x) \cup A_2(x))$ and such that $M_2^*$ is disjoint from and anticomplete to $M_3^*$. Then $M_1, M_2, M_3$ form a long theta, a contradiction.
This proves (\ref{rabbit:2}).
\\
\\
By (\ref{rabbit:2}) we may assume that $x$ has no neighbors in $V(P_3)$. By Lemma \ref{2paths}, it follows that $x$ has neighbors in both $V(P_1)$ and $V(P_2)$.
\stmt{ If $\{i,j\} = \{1,2\}$, $y$ does not have two non-adjacent $x$-external neighbors in $V(P_i)$ and an $x$-internal neighbor in $V(P_j)$.\label{rabbit:3}}
Suppose $y$ has two non-adjacent $x$-external neighbors in $V(P_i)$ and an $x$-internal neighbor in $V(P_j)$. By (\ref{rabbit:1}), $y$ has no $x$-external neighbors in $V(P_j)$. By definition of $x$-internal, $\alpha_j(x)$ is not equal or adjacent to $\beta_j(x)$.
Let $M_1$ and $M_2$ denote the induced $xy$-paths with interiors in $V(\beta_j(y)P_1\beta_j(x))$ and $V(\beta_i(x)P_i\beta_i(y))$, respectively.
Then $M_1, M_2$, and $x \dd A_j(x) \dd A_i(y) \dd y$ form a long theta, a contradiction. This proves (\ref{rabbit:3}).
\\
\\
By (\ref{rabbit:1}), (\ref{rabbit:3}), Lemma \ref{2paths} and Lemma \ref{3pairwisenonadjacent}, $y$ has a neighbor in $V(P_3)$.
By Lemma \ref{2paths} and Lemma \ref{3pairwisenonadjacent}, we may assume that $x$ has two non-adjacent neighbors in $V(P_1)$ and at least one neighbor in $V(P_2)$.
\stmt{Either $y$ has exactly one neighbor in $V(P_3)$ or $y$ has exactly two neighbors in $V(P_3)$ and they are adjacent. \label{rabbit:4}}
Suppose $y$ has two non-adjacent neighbors in $V(P_3)$. There are two long induced $xy$-paths $M_1$, $M_2$ with interiors in $V(P_1 \cup P_3)$ such that $M_1^*$ is disjoint from and anticomplete to $M_2^*$. If $y$ has a neighbor in $V(P_2)$, there is an $xy$-path $M_3$ with interior in $V(P_2)$, such that $M_1$, $M_2$ and $M_3$ form a long theta, a contradiction. So by Lemma \ref{2paths} and \ref{3pairwisenonadjacent}, we may assume $y$ has a neighbor in $V(P \setminus B_1(x))$. Let $Q_1$ denote the long induced $xy$-path whose interior is a subset of $V(B_1(x) \cup B_3(y))$. Let $Q_2$ denote the induced $xy$-path with interior contained in $V(\alpha_1(x)P_1\alpha_1(y))$. Then $Q_1, Q_2$ and $x \dd A_2(x) \dd A_3(y) \dd y$ form a long theta, a contradiction. This proves (\ref{rabbit:4}).
\stmt{Either $y$ has exactly one $x$-external neighbor in $V(P_1 \cup P_2)$ or exactly two $x$-external neighbors in $V(P_1 \cup P_2)$ and they are adjacent. \label{rabbit:5}}
By (\ref{rabbit:4}), $y$ has at least one $x$-external neighbor in $V(P_1 \cup P_2)$. Suppose $y$ has two non-adjacent $x$-external neighbors in $V(P_1 \cup P_2)$.
We may assume that either $y$ has two non-adjacent neighbors in $V(A_1(x) \cup A_2(x))$ or $y$ has a neighbor in $V(A_1(x) \cup A_2(x))$ and a neighbor in $V(B_1(x) \cup B_2(x))$.
In the first case, there are two long induced $xy$-paths $M_1, M_2$ with interiors in $V(A_1(x) \cup A_2(x))$ such that no vertex of $M_1^*$ is disjoint from and anticomplete to $M_2^*$. Then $M_1, M_2$ and $x \dd B_1(x) \dd B_3(y) \dd y$ form a long theta, a contradiction. 

Hence, the second case holds. By (\ref{rabbit:1}), we have that one of $V(P_1)$, $V(P_2)$ contains no $x$-external neighbors of $y$. If $V(P_2)$ contains no $x$-external neighbors of $y$, there is a long theta formed by the two induced $xy$-paths with interiors contained in $V(A_1(x))$ and $V(B_1(x))$, respectively, and the path $x \dd A_2(x) \dd A_3(y) \dd y$, a contradiction. Hence, $V(P_1)$ contains no $x$-external neighbors of $y$. Let $Q_1$ denote the $xy$-path with interior contained in $V(\beta_2(x)P_2\beta_2(y))$. Then $Q_1$, $x \dd A_1(x) \dd A_2(y) \dd y$ and $x \dd B_1(x) \dd B_3(y) \dd y$ form a long theta, a contradiction. This proves (\ref{rabbit:5}).
\\
\\
By (\ref{rabbit:4}) and (\ref{rabbit:5}), there is an edge $e \in E(P_1 \cup P_2)$ and $e' \in E(P_3)$ such that every $x$-external neighbor of $y$ in $V(K)$ is incident with $e$ or $e'$ and $e, e'$ each contain at least one $x$-external neighbor of $y$. By Lemma \ref{3pairwisenonadjacent}, $y$ has an $x$-internal neighbor in $V(P_i)$ for some $i \in \{1,2 \}$. Since $x$ has two non-adjacent neighbors in $V(P_1)$ we may assume without loss of generality that $y$ has an $x$-internal neighbor in $V(P_1)$. Let $M_1$ denote the induced $xy$-path with interior in $V(\alpha_1(x)P_1\alpha_1(y))$ and let $M_2$ denote the induced $xy$-path with interior in $V(\beta_1(x)P_1\beta_1(y))$.
If $e \in E(A_1(x))$, there is a long theta formed by $M_1, M_2$ and $y \dd B_3(y) \dd B_2(x) \dd x$. Hence, we may assume that $e \in E(A_2(x))$. Let $M_3$ denote the long induced $xy$-path with interior in $V(\alpha_2(y)P_2\alpha_2(x))$. Then $M_1, M_3$ and $x \dd B_1(x) \dd B_3(y) \dd y$ form a long theta, a contradiction.
\end{proof}

We use the results from this section to prove that there exists pair consisting of a set of $K$-major vertices and a set of paths each of which is contained in $K$ with certain helpful properties for cleaning. We begin with some definitions. For a set of paths $\mathcal{P}$ we denote $\cup_{P \in \mathcal{P}}P^*$ by $\mathcal{P}^*$.
Let $G$ be a graph containing no long thetas and let $K$ be a shortest long near-prism in $G$.
Let $P_1, P_2$ be distinct constituent paths of $K$ and let $F$ be the frame of $K$.
Let $S$ be the set of $K$-major vertices that have a neighbor in $V(P_1)$ and a neighbor in $V(P_2)$.
Let $H \subseteq S$ and let $\mathcal{Q}$ be a set of paths such that each $Q \in \mathcal{Q}$ is a subpath of $P_1$, $P_2$ or $P_3$.
We call the ordered pair $(H, \mathcal{Q})$ a {\em $(K, P_1, P_2)$-contrivance} if it satisfies the following two properties:
\begin{itemize}
    \item If $S$ is non-empty, then $H$ contains a vertex $v \in S$ maximizing $|E(A_1(v))|$ over all $v \in S$, $\alpha_1(v) \in \mathcal{Q}^*$, and $N(H) \cap V(A_1(v)) \subseteq \mathcal{Q}^*$ and
    \item Every vertex $w \in S \setminus H$ has a neighbor in $H \cup \mathcal{Q}^*$.
\end{itemize}

We will show that if $G$ has no long thetas, $K$ is a shortest long near-prism in $G$ and $K$ has a tidy frame, then for every choice of two distinct constituent paths $P_1, P_2$ of $K$ there is a $(K, P_1, P_2)$-contrivance $(H, \mathcal{Q})$ with $|H|$ and $\sum_{Q \in \mathcal{Q}} |V(Q)|$ bounded. Thus we will be able to guess a $(K, P_1, P_2)$-contrivance in our cleaning algorithm. We need the following lemma:

\begin{lem}\label{lem:prismtriangleexternal}
Let $G$ be a graph containing no long thetas and $K$ be a shortest long near-prism in $G$. Suppose $K$ has a tidy frame. Let $P_1, P_2,P_3$ be the constituent paths of $K$. Suppose $x$ and $y$ are $K$-major vertices satisfying all of the following:
\begin{itemize}
    \item $x$ and $y$ are non-adjacent,
    \item $x$ has a neighbor in $V(P_1)$ and $y$ has a neighbor in $V(A_1(x))$ and
    \item If $v \in V(K)$ is adjacent to $y$, then $d_{K}(v, \alpha_i(x)), d_{K}(v, \beta_i(x)) \geq \ell-1$ for any $i \in \{1,2,3\}.$
\end{itemize}
Then $y$ has exactly two $x$-external neighbors in $V(K)$ and they are adjacent elements of $V(A_1(x))$.
\end{lem}

\begin{proof}
By Lemma \ref{Kexternal}, we need only show that $y$ does not have a unique neighbor in $V(A_1(x))$. Suppose that $y$ has a unique neighbor $w \in V(A_1(x))$. Then by Lemma \ref{Kexternal}, $w$ is the unique $x$-external neighbor of $y$ in $V(K)$. By Lemma \ref{3pairwisenonadjacent}, we may assume $x$ has neighbor in $V(P_2)$. 
Let $M_1, M_2$ denote the two paths of the cycle $x \dd A_1(x) \dd A_2(x) \dd x$ with ends $x$ and $w$. Since $K$ has a tidy frame and $d_{P_1}(w, \alpha_1(x)) \geq \ell-1$, the paths $M_1$ and $M_2$ each have length at least $\ell$.
Let $J$ be the set of vertices in $V(K) \setminus V(M_1 \cup M_2)$ that are anticomplete to $V(M_1 \cup M_2)$.
By Lemma \ref{3pairwisenonadjacent}, $x$ and $y$ each have a neighbor in $J$. $G[J]$ is connected so there is an induced $xy$-path $M_3$ with interior in $J$. But then $M_1, M_2$ and $M_3 \dd w$ form a long theta, a contradiction.
\end{proof}

\begin{lem} \label{prismset}
Let $G$ be a graph containing no long thetas and $K$ be a shortest long near-prism in $G$. Suppose $K$ has a tidy frame. Let $P_1$, $P_2$ be two constituent paths of $K$. Then, there is a $(K, P_1, P_2)$-contrivance $(H, \mathcal{Q})$ in $G$ satisfying
$|H| \leq 3$, $|\mathcal{Q}| \leq 14$ and $\sum_{Q \in \mathcal{Q}} |V(Q)| \leq 28\ell -12.$
\end{lem}
\begin{proof}
Let $S$ be the set of $K$-major vertices with both a neighbor in $V(P_1)$ and a neighbor in $V(P_2)$. We may assume without loss of generality that $S \neq \emptyset$.
Let $x \in S$ maximize $|E(A_1(x))|$.
For any $i \in \{1,2,3\}$, vertex $v \in V(P_i)$ and non-negative integer $k$, we denote the set consisting of all vertices $b$ satisfying $d_{P_i}(a,b) \leq k$  as $N_{P_i}^k(a)$.

Let $\mathcal{Q}_x$ denote the set of paths contained in $K$ whose interiors are equal to $N_{P_i}^{\ell-1}(\alpha_i(x))$ or $N_{P_i}^{\ell-1}(\beta_i(x))$ for some $i \in \{1,2,3\}$. Hence, $|\mathcal{Q}_x| \leq 6$.
Let $S_x$ denote the set of vertices $s \in S$ that have no neighbor in $\mathcal{Q}_x^*$.
We may assume without loss of generality that $(\{x\}, \mathcal{Q}_x)$  is not a $(K, P_1, P_2)$-contrivance. Thus, $S_x \neq \emptyset$. 

Let $y \in S_x$ maximize $|E(A_2(y))|$ over all $y \in S_x$.
Let $\mathcal{Q}_y$ denote the set of paths contained in $K$ whose interiors are equal to $N_{P_i}^{\ell}(\alpha_i(y))$ or $N_{P_i}^{\ell-1}(\beta_i(y))$ for some $i \in \{1,2,3\}$.
Hence, $|\mathcal{Q}_y| \leq 6$.
By Lemma \ref{lem:prismtriangleexternal}, $y$ has exactly two neighbors in $V(A_1(x))$ and they are adjacent. Hence $N(y) \cap V(A_1(x)) \subseteq \mathcal{Q}_y^*$.
Let $S_{xy}$ be the set of vertices $s \in S_x$ that have no neighbor in $\mathcal{Q}^*_x \cup \mathcal{Q}^*_y \cup \{x,y\}$. We may assume without loss of generality that $( \{x,y\}, \mathcal{Q}_x \cup \mathcal{Q}_y )$ is not a $(K, P_1, P_2)$-contrivance. Hence, $S_{xy} \neq \emptyset$.
\stmt{Let $v \in S_{xy}$. Then $v$ has exactly two $x$-external neighbors in $V(K)$ and they are adjacent elements of $V(A_1(x))$ and $v$ has exactly two $y$-external neighbors and they are adjacent elements of $V(A_2(y))$. \label{sad:1}}
Apply Lemma \ref{lem:prismtriangleexternal} to $x$, $v$ and to $y$, $v$. This proves (\ref{sad:1}).
\\
\\
Let $z$ be a vertex in $S_{xy}$.
Let $p_1, q_1$ denote the $x$-external neighbors of $z$ in $V(P_1)$. Let $p_2, q_2$ denote the $y$-external neighbors of $z$ in $V(P_2)$.
Let $\mathcal{Q}_z$ denote the set consisting of the two paths contained in $K$ whose interior is equal to $N_{P_1}^{\ell-2}(p_i) \cup N_{P_1}^{\ell-2}(q_i)$ for some $i \in \{1,2\}$. Hence, $N(z) \cap V(A_1(x)) \subseteq \mathcal{Q}^*_z$.

Let $H = \{x,y,z\}$. Let $\mathcal{Q} = \mathcal{Q}_x \cup \mathcal{Q}_y \cup \mathcal{Q}_z$.
We claim that $(H, \mathcal{Q})$ is a ($K$, $P_1$, $P_2$)-contrivance. By choice of $H$, $\mathcal{Q}$ it is enough to show that every $t \in S \setminus H$ has a neighbor in $H \cup \mathcal{Q}^*$. Suppose some $t \in S \setminus H$ has no neighbors in $H \cup \mathcal{Q}^*$.
Then, by (\ref{sad:1}) there is a long induced $tz$-path $M_1$ with interior in $V(A_1(x))$ and a long induced $tz$-path $M_2$ with interior in $V(A_2(y))$.
Let $h_1, h_2$ denote the two vertices of $V(P_1 \cup P_2) \setminus V(A_1(x) \cup A_2(y))$ with a neighbor in $V(A_1(x) \cup V(A_2(y))$. Let $J$ denote the graph $K \setminus (V(A_1(x) \cup A_2(y)) \cup \{h_1, h_2, a_3\})$.  Since $t,z \in S_{xy}$, both $t$ and $z$ are non-adjacent to each of $h_1, h_2, a_3$. Hence, by Lemma \ref{3pairwisenonadjacent}, $t,z$ each have a neighbor in $V(J)$.
Then since $J$ is connected there is an induced $tz$-path $M_3$ with interior in $V(J)$. Hence, $M_1, M_2, M_3$ form a long theta, a contradiction. It follows that $(H, \mathcal{Q})$ is a ($K$, $P_1$, $P_2$)-contrivance.
\\
\\
By construction, we have that $|H| = 3$, $|\mathcal{Q}|   \leq 14$. The number of vertices in the paths in $\mathcal{Q}$ are as follows:
\begin{equation*}
  |V(Q_i)| \leq
    \begin{cases}
      2\ell-1 & \text{if $Q \in \mathcal{Q}_x \cup \mathcal{P}_y.$}\\
      2\ell-2 & \text{if $Q \in \mathcal{Q}_z$}.\\
    \end{cases}       
\end{equation*}
Since $|\mathcal{Q}_x \cup \mathcal{Q}_y| \leq 12$ and $|\mathcal{Q}_z| = 2$, it follows that $\sum_{Q 
\in \mathcal{Q}} V(Q) \leq 28\ell -16$ as claimed.
\end{proof}


We apply the results from this section to obtain the following cleaning algorithm for major vertices of shortest long near-prisms.

\begin{theorem} \label{cleanprism}
There is an algorithm with the following specifications:
\begin{description}
\item[Input:] A graph $G$ containing no long theta.
\item[Output:] A list of $\mathcal{O}(|G|^{32\ell-10})$ subsets of $V(G)$, with the following property: For every shortest long near-prism $K$ and choice of two distinct constituent paths $P_1, P_2$ of $K$, if $K$ has a tidy frame, then there is some $X$ in the list such that:
\begin{itemize}
    \item $X$ is disjoint from $V(K)$ and 
    \item $X$ contains all $K$-major vertices with neighbors in both $V(P_1)$ and $V(P_2)$.
\end{itemize}
\item[Running Time:] $\mathcal{O}(|G|^{32\ell -8})$.
\end{description}
\end{theorem}

\begin{proof}
We enumerate all triples $(H, \mathcal{Q})$ satisfying all of the following:
\begin{itemize}
    \item $H$ is a set of at most three vertices in $G$, 
    \item $\mathcal{Q}$ is a set of at most 14 paths of $G$, and
    \item $\sum_{Q \in \mathcal{Q}}|V(Q)| \leq 28\ell - 16$.
\end{itemize}
For each such pair we perform the following.
If $H$ is empty, we output the empty set.
Otherwise, we guess a vertex $a \in V(G)$ and a vertex $\alpha$ in $\mathcal{Q}^*$.
For each $Q \in \mathcal{Q}$, let $X_Q$ be the set of vertices in $V(G) \setminus V(Q)$ with neighbors in $Q^*$. Let $X = \cup_{Q \in \mathcal{Q}}X_Q$. We run the algorithm of Theorem \ref{thm:prismpathcleaning} on $G \setminus X$ to generate a list $Y_1, Y_2, \dots , Y_k$ of subsets of $V(G)$.

For each $i \in \{1,2, \dots, k \}$, let $G_i$ be the graph $G \setminus  ((X \cup Y_i \cup H \cup N(H)) \setminus \mathcal{Q}^*)$. 
We compute the union of the interiors of all shortest $a\alpha$-paths in $G_i$ and denote it by $R_i$. Let $Z_i$ be the set of vertices of $V(G)\setminus \mathcal{Q}^*$ with a neighbor in $V(R_i)$ and a neighbor in $H$. We output $H \cup X \cup Z_i$.

We prove the algorithm is correct. Let $K$ be a shortest long near-prism in $G$ and let $P_1$, $P_2$ be distinct constituent paths of $K$. Suppose $K$ has  a tidy frame. 
Then it follows from Lemma \ref{prismset} that for some choice of $(H, \mathcal{Q})$ the pair $(H, \mathcal{Q})$ is a ($K$, $P_1$, $P_2$)-contrivance. Let $S$ denote the set of $K$-major vertices with a neighbor in $V(P_1)$ and a neighbor in $V(P_2)$. By definition of $(K, P_1, P_2)$-contrivance, $H$ is empty if and only if $S$ is empty.
We may assume $S$ and $H$ are both non-empty.
Thus, every vertex in $S \setminus (H \cup N(H))$ has a neighbor in $X$.
Let $A$ denote the path $aP\alpha$.
Let $x$ be a vertex in $S$ maximizing $|E(A_1(x))|$ over all $x \in S$. For some choice of $a, \alpha$, the paths $A_1(x)$ and $A$ are equal so we may assume $A = A_1(x)$. It follows that every vertex in $S$ has a neighbor in $V(A)$.
By construction, $X$ is disjoint from $V(K)$.
Hence, by Theorem \ref{thm:prismpathcleaning}, there exists an $i \in \{1,2, \dots k \}$ such that $Y_i$ is disjoint from $V(K)$ and $Y_i$ contains a vertex of every bad shortest $a\alpha$-path $Q$ such that $\zeta_Q$ is not $K$-major.
By definition of $(K, P_1, P_2)$-contrivance, $H \cup N(H) \cup \mathcal{Q}^*$ contains all vertices in $S$. Hence all shortest $a\alpha$-paths in $G_i$ are good.
Since $N(H) \cap V(A) \subseteq \mathcal{Q}^*$, it follows that $A^* \subseteq R_i$. Thus, $Z_i$ contains all vertices in $S \setminus (H \cup Y_i)$. Since $K$ has a tidy frame, $a_1$ has no neighbors in $H$ and by definition of $(K, P_1, P_2)$-contrivance, $N(H) \cap V(A) \subseteq \mathcal{Q}^*$. Since all shortest $a\alpha$-paths in $G_i$ are good, it follows that $Z_i$ is disjoint from $V(K)$. Hence, the list satisfies the properties from the claim.

There are $\mathcal{O}(|G|^{28\ell -12})$ choices for $(H, \mathcal{Q})$ and $a$.
For each choice we find $X$ in time $\mathcal{O}(|G|)$ and run the algorithm of Theorem \ref{thm:prismpathcleaning} to generate a list of $\mathcal{O}(|G|^{4\ell+2})$ subsets $Y_1, Y_2, \dots, Y_k$ of $V(G)$ in time $\mathcal{O}(|G|^{4\ell+3})$.
For each set in the list we compute $Z_i$ in $\mathcal{O}(|G|^2)$.
Hence the total running is $\mathcal{O}(|G|^{32\ell-8 })$ and the length of the output is $\mathcal{O}(|G|^{32\ell -10})$.

\end{proof}

\begin{corr}\label{corr:fullprismcleaner}
There is an algorithm with the following specifications:
\begin{description}
\item[Input:] A graph $G$ containing no long theta.
\item[Output:] A list of $\mathcal{O}(|G|^{96\ell-30})$ subsets of $V(G)$, with the following property: For every shortest long near-prism $K$ and choice of two distinct constituent paths $P_1$, $P_2$ of $K$, if $K$ has a tidy frame, then there is some $X$ in the list such that:
\begin{itemize}
    \item $X$ is disjoint from $V(K)$ and
    \item $X$ contains all $K$-major vertices with neighbors in both $V(P_1)$ and $V(P_2)$.
\end{itemize}
\item[Running Time:] $\mathcal{O}(|G|^{96\ell - 30})$.
\end{description}
\end{corr}
\begin{proof}
We apply the algorithm of Theorem \ref{cleanprism} to obtain a cleaning list $X_1, X_2, \dots ,X_k$ of length $\mathcal{O}(|G|^{32\ell -6})$ in time $\mathcal{O}(|G|^{32\ell-4})$. We output $X_a \cup X_b \cup X_c$ for each choice $a,b,c \in \{1,2 \dots, k\}$. Correctness follows from Lemma \ref{3pairwisenonadjacent}.
\end{proof}

\subsection{The long near-prism detection algorithm} \label{section:longprismalg}

We can now prove the main result of this section, which we restate. 
\\ \\
\noindent \textbf{Theorem \ref{alg:longprisms}} \textit{For each integer $\ell \geq 4$ there is an algorithm with the following specifications:}
\begin{description}
\item[Input:] \textit{A graph $G$ containing no long theta.}
\item[Output:] \textit{Decides whether $G$ contains a long near-prism.}
\item[Running Time:] $\mathcal{O}(|G|^{108\ell-22})$.
\end{description}

\begin{proof}
The algorithm is as follows:
We guess a set $J$ of at most $6\ell-6$ vertices and a set $D$ of at most $6$ vertices.
We construct the set $X$ of all vertices in $V(G) \setminus (J \cup D)$ with neighbors in $J$.

We apply the algorithm described in Corollary \ref{corr:fullprismcleaner} to $G \setminus X$ and obtain a cleaning list $Y_1, Y_2 \dots Y_p$.
For each $i \in \{1, 2, \dots, p \}$, we apply the algorithm described in Corollary \ref{corr:fullprismpathcleaning} to generate another cleaning list $Z^i_1, Z^i_2, \dots Z^i_{k_i}$.
We guess two vertices $x_1, y_1$ in $J \cup D$.
We search for a shortest $x_1y_1$-path $Q_1$ in $G \setminus (X \cup Y_i \cup Z^i_j)$. We construct the set $A$ of all vertices in $V(G) \setminus (J \cup X \cup Y_i \cup Z^i_j)$ with neighbors in $Q_1^*$. Then we guess two vertices $x_2, y_2$ in $J$ and we search for a shortest $x_2y_2$-path $Q_2$ in $G \setminus (X \cup Y_i \cup Z^i_j \cup A)$ and construct the set $B$ of all vertices in $V(G) \setminus (J \cup X \cup Y_i \cup Z^i_j \cup A)$ with neighbors in $Q_2^*$. Finally, we guess two vertices $x_3, y_3$ in $J$ we search for a shortest $x_3y_3$-path $Q_3$ in $G \setminus (X \cup Y_i \cup Z^i_j \cup A \cup B)$. We test whether $S \cup V(Q_1) \cup V(Q_2) \cup V(Q_3)$ induces a long near-prism.

Now, we prove the output is correct. Let $K$ be a shortest long near-prism in $G$ and $F$ be the frame of $K$. Then for some guess of $J$ and $D$, the set $J \cup D$ is equal to $V(F)$ and $D$ is the set of vertices in $V(F)$ with a neighbor in $V(K) \setminus V(F)$. Hence, $G \setminus X$ contains $K$ and $K$ has a tidy frame in $G \setminus X$.
By Corollary \ref{corr:fullprismcleaner} it follows that for some choice of $i \in \{1, 2, \dots, p\}$, there are no $K$-major vertices in $G \setminus (X \cup Y_i)$.
Therefore by Corollary \ref{corr:fullprismpathcleaning}, for some choice of $j \in \{1,2, \dots , k_i\}$, all shortest $s_it_i$-paths are good in $G \setminus (X \cup Y_i \cup Z^i_j)$ for every $i \in \{1,2,3\}$ such that $P_i \not \subseteq F$. 
For each $i \in \{1,2,3 \}$, we may assume $x_i, y_i$ equal $s_i, t_i$ since $s_i, t_i \in V(F)$. 
Since $Q_1$ is good, there is a long near-prism induced by $V(F \cup Q_1 \cup P_2 \cup P_3)$. Thus $Q_2$ exists and is a good shortest $s_2t_2$-path. Similarly, $Q_3$ exists and is a good shortest $s_3t_3$-path for $K$. By choice of $A, B$ it follows that, $Q_1, Q_2, Q_3$ are pairwise vertex disjoint and their interiors are pairwise anticomplete.
Thus since $F$ is tidy, $J \cup V(Q_1 \cup Q_2 \cup Q_3)$ induces a long near-prism $K'$.
Since $Q_p$ is good for each $p \in \{1,2,3\}$, it follows that $K'$ is a shortest long near-prism in $G$.

There are $\mathcal{O}(|G|^{6\ell})$ guesses to check for $J \cup D$. For each of these we obtain a cleaning list $Y_1, Y_2 \dots Y_p$ of length $\mathcal{O}(|G|^{96\ell-30})$ in time $\mathcal{O}(|G|^{96\ell -30})$. For each $Y_i$ in the list we generate another cleaning list $Z_1, Z_2, \dots, Z_t$ of length $\mathcal{O}(|G|^{12\ell + 6})$ in time $\mathcal{O}(|G|^{12\ell + 6})$. Finding $Q_i$ for each $i \in \{1,2,3\}$ and testing whether $J \cup V(Q_1 \cup Q_2 \cup Q_3)$ is a long near-prism takes $\mathcal{O}(|G|^2)$. Hence, the running time is $\mathcal{O}(|G|^{108\ell-22}$).
\end{proof}

\section{Detecting a clean shortest long even hole}
In this section we provide an algorithm to detect a clean shortest long even hole in a ``candidate'', a graph that contains no easily detectable configurations. More rigorously, $G$ is a \textit{candidate} if it contains no long even hole of length at most $2\ell +2$, no long jewel of order at most $\ell+3$, no long theta, no long ban-the-bomb. 

\begin{lem}\label{lem:Cmajor3pairwisenonadjacent}
	Let $G$ be graph with no long near-prism or long theta and let $C$ be shortest long even hole in $G$.
	Then every $C$-major vertex has three pairwise non-adjacent neighbors in $V(C)$.
\end{lem}

\begin{proof}
	Let $x$ be a $C$-major vertex and suppose that there exist vertices $a_1,a_2, b_1, b_2 \in V(C)$ such that $N(x) \cap C \subseteq \{a_1, a_2, b_1, b_2 \}$, $a_1$, $b_1$ are adjacent and $a_2, b_2$ are adjacent. Without loss of generality $a_1, b_1, b_2, a_2$ occur in order along $V(C)$. Let $A$ denote the path of $C$ with ends $a_1, a_2$ that does not contain $b_1$ or $b_2$. Let $B$ denote the path of $C$ with ends $b_1, b_2$ that does not contain $a_1$ or $a_2$. By \ref{lem:jewelpaths}, $A$, $B$ each have length at least $\ell$.
	Since $x$ is $C$-major, $x$ must have at least one neighbor in $\{a_1, b_1\}$ and at least one neighbor in $\{a_2, b_2\}$.
	If $x$  is adjacent to all of $a_1, a_2, b_1, b_2$, it follows that $V(C) \cup \{x\}$ induces a long near-prism, a contradiction.
	If $x$ is adjacent to exactly one of $\{a_1, b_1\}$ and $x$ is adjacent to exactly one of $\{a_2, b_2\}$ then $V(C) \cup \{x \}$ induces a long theta, a contradiction.
	Hence we may assume $x$ is adjacent to $a_1$, $a_2$, $b_2$ and $x$ is non-adjacent to $b_1$. Then $V(A) \cup \{x\}$ and $V(B) \cup \{a_1, x\}$ both induce long holes that are shorter than $C$, so both holes must both be odd. But then $|E(A)|$ and $|E(B)| + 1$ are both odd, contradicting that $C$ is even.
\end{proof}

We will need the following analogue of 3.4 in \cite{chudnovsky2019detectinglongodd}:

\begin{lem}\label{lem:jewelpaths}
	Let $G$ be a graph containing no long jewel of order at most $k$ and no long even hole of length less than $k + \ell$. Let $C$ be a shortest long even hole in $G$ and let $v \in V(G)$ be a $C$-major vertex. Then every path of $C$ that contains all neighbors of $v$ in $V(C)$ has length greater than $k$.
\end{lem}

\begin{proof}
	Suppose that $P$ is a path of $C$ of length at most $k$ and that $P$ contains all of neighbors of $v$ in $V(C)$. Choose $P$ to be minimal. Denote the ends of $P$ as $a,b$. Let $Q$ be the other path of $C$ with ends $a$ and $b$. We have $|E(Q)| \geq \ell$ and $|E(P)| \geq 3$. So $V(Q) \cup \{ v \}$ induces a long hole shorter than $C$. So $Q$ is odd and thus $P$ is odd. But then $P$, $a \dd v \dd b$ and $Q$ form a long jewel of order at most $k$, a contradiction.
\end{proof}

\begin{lemma}
	Let $G$ be a candidate and let $C$ be a shortest long even hole in $G$.
	Let $v$ be a $C$-major vertex. Then for every three vertex path $Q$ of $C$, $v$ has at least two neighbors in $V(G) \setminus V(Q)$.
	\label{lem:nobombonsleh}
\end{lemma}
\begin{proof}
	\begin{figure}[!h]
		\centering
		\input{Tikz/bombonsleh.tex}
		\caption{An illustration of the proof of Lemma \ref{lem:nobombonsleh}. (Note the edge $vy$-exists because $G$ is a candidate, but it does not matter for the sake of the argument.)}
		\label{fig:slehbomb}
	\end{figure}
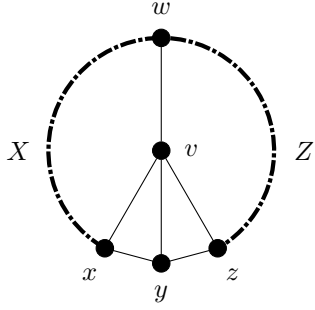
	Suppose for some three vertex path $P$ of $C$, $v$ has at most one neighbor in $V(C) \setminus V(P)$.
	Let the vertices of $P$ be $x \dd y \dd z$ in order.
	By Lemma \ref{lem:Cmajor3pairwisenonadjacent}, $v$ has a neighbor $w \in V(C) \setminus V(P)$ and $v$ is adjacent to both $x$ and $z$.
	Let $X$ be the $xw$-path of $C$ not containing $z$ and let $Z$ be the $zw$-path of $C$ not containing $x$. (See Figure \ref{fig:slehbomb}).

	$V(C) \cup \{ v\}$ induces a bomb, so it is not long. Hence, we may assume $|E(X)| \leq \ell -3$.
	Since $G$ contains no hole of length at most $2\ell +2$, it follows that $|E(Z)| \geq \ell -2$.
	Then $V(Z) \cup \{ v\}$ induces a long hole and it is shorter than $C$. Hence, $Z$ is an odd path.
	Thus $X$ is an odd path. But then $x \dd v \dd w, P \cup X$ and $Z$ form a long jewel of order at most $\ell -3$, a contradiction.
\end{proof}

\begin{theorem}\label{thm:cleanshortest}
Let $C$ be a clean shortest long even hole in a candidate $G$.
Let $u,v$ be distinct, non-adjacent vertices in $V(C)$. Let $L_1$, $L_2$ be the two paths of $C$ with ends $u$ and $v$ where $|E(L_1)| \leq |E(L_2)|$. Then:
\begin{enumerate}[(i)]
\item $L_1$ is a shortest $uv$-path in $G$ and
\item for every shortest $uv$-path $P$ in $G$, either $P \cup L_2$ is a clean shortest long even hole in $G$ or $|E(L_1)| = |E(L_2)|$ and $P \cup L_1$ is a clean shortest long even hole in $G$.
\end{enumerate}
\end{theorem}

We begin by proving the first statement of Theorem \ref{thm:cleanshortest} in a more general form that can be used for cleaning major vertices of shortest long even holes.
For $u,v$ distinct and non-adjacent vertices in $V(C)$ we call an induced $uv$-path $Q$ a {\em shortcut} if $V(Q)$ contains no $C$-major vertices and $Q$ has length less than $d_C(u,v)$.
We will need the following properties of shortest long even holes of a candidate. We will prove that a clean shortest long even hole in a candidate does not have a shortcut.
In this language the first statement of Theorem \ref{cleanshortest} is ``$G$ does not contain a shortcut for $C$''.

\begin{lem} \label{lem:P_1P_2}
Let $G$ be a candidate and let $C$ be a shortest long even hole in $G$.
Let $Q$ be a shortcut of $C$. Denote the vertices of $Q$ that are adjacent to an end of $Q$ by $q_1, q_k$. Suppose $Q^* \setminus \{q_1, q_k \}$ is disjoint from and anticomplete to $V(C)$. Then, one of $q_1, q_k$ has two non-adjacent neighbors in $V(C)$.
\end{lem}
\begin{proof}
Since $q_1$ is not $C$-major, there is a path $P_1$ of $C$ of length at most two such that all neighbors of $q_1$ in $V(C)$ are elements of $V(P_1)$. Choose $P_1$ to be minimal. Define $P_2$ similarly for $q_k$. Suppose for a contradiction that both $P_1$ and $P_2$ have length at most one.
If $P_1$ and $P_2$ both have length $0$ then $G[V(C \cup Q)]$ forms either a long theta or a long jewel of order less than $\ell$, contradicting that $G$ is a candidate. Hence, we may assume $P_1$ has length one.

Suppose $P_2$ has length one. If $|E(L_1)| \geq \ell +2$, then $G[V(C \cup Q)]$ is a long near-prism, a contradiction. It follows that $|E(Q)|, |E(L_1)| \leq \ell +1$. $L_2$ has length greater than $\ell$ since $|E(C)| > 2\ell$. Thus, $L_2 \cup Q$ is a long hole and it is shorter than $C$, so it is odd. Therefore $Q$ and $L_1$ have different parities. But then $L_1, Q$ and $L_2$ form a long jewel of order at most $\ell+1$, a contradiction.

Hence, we may assume that $P_2$ has length 0. Denote the two vertices of $P_1$ as $a_1, b_1$ where $a_1 \in V(L_1)$ and $b_1 \in V(L_2)$. Denote the end of $Q$ adjacent to $q_k$ by $v$. Since $C$ has length at least $2\ell + 3$, the path $b_1L_2v$ has length greater than $\ell$. Then $b_1L_2v \dd q_kQq_1 \dd b_1$ is a long hole and it is shorter than $C$, so it is odd. Hence, $b_1L_2v$ has a different parity than $q_kQq_1 \dd b_1$ and thus $b_1L_2v$ has a different parity than $q_kQq_1 \dd a_1$. Since $C$ is even, $b_1L_2v$ has a different parity than $a_1L_1v$, so $a_1L_1v$ and $q_kQq_1 \dd a_1$ have the same parity. Then, $a_1L_1v \dd q_kQq_1 \dd a_1$ is an even hole and it is shorter than $C$, so it is not long. Thus, $|E(a_1L_1v)|, |E(q_1Qv)| < \ell$. Hence, $b_1 \dd a_1L_1v$, $b_1 \dd q_1Qv$, and $b_1 L_2v$ form a long jewel of order less than $\ell+1$, a contradiction. 
\end{proof}

\begin{lem} \label{lem:cycledistance}
Let $C$ be a shortest long even hole in a graph $G$.
Let $u,v$ be distinct and non-adjacent vertices in $V(C)$.
Let $Q$ be an induced $uv$-path of length at most $d_C(u,v)$ such that $V(Q)$ contains no $C$-major vertices. Suppose no proper subpath of $Q$ is a shortcut for $C$. Suppose there is some $q \in Q^*$ such that $q$ is not adjacent to an end of $Q$ and $q$ has a neighbor $w \in V(C)$. Let $R$ denote the path of $C$ with ends $u,w$ whose interior does not contain $v$. Let $S$ denote the path of $C$ with ends $w,v$ whose interior does not contain $u$. Then $d_C(u,w) = |E(R)|$ and $d_C(w,v) = |E(S)|.$
\end{lem}
\begin{figure}
	\centering
	\input{Tikz/shortcut_distance_lemma.tex}
	\caption{An illustration of Lemma \ref{lem:cycledistance}. Note this is a simplification;   vertices in $Q$ may be equal or adjacent to vertices in $C$.}
	\label{fig:lem:cycledistance}
\end{figure}
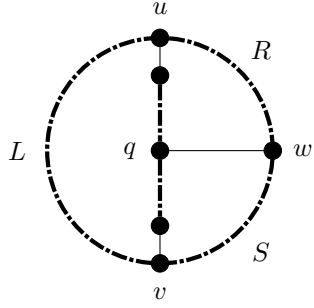
\begin{proof}
Suppose the lemma does not hold. (See Figure \ref{fig:lem:cycledistance}.)
By symmetry we may assume, $|E(R)| > d_C(u,w)$. 
Since $Q$ is induced, $w \neq v$ and so $|E(S)|\geq 1$.
Let $L$ denote the path of $C$ with ends $u,v$ that does not go through $w$. Then since $uQq \dd w$ is not a shortcut for $C$, it follows that $|E(L)| + |E(S)| \leq d_Q(u, q) + 1 \leq |E(Q)|$. However, $|E(L)| + |E(S)| \geq |E(Q)| + 1$, a contradiction.
\end{proof}

\begin{theorem} \label{thm:shortcuts}
Let $G$ be a candidate and let $C$ be a shortest long even hole in $G$.
Then $C$ has no shortcut.
\end{theorem}

\begin{proof}
Suppose $G$ is a minimal counterexample. Then $G$ has a clean shortest long even hole $C$ with a shortcut $Q$.
We may assume $C$ and $Q$ are chosen to minimize $|E(Q)|$. Let $u,v$ be the ends of $Q$. Let $L_1$ and $L_2$ be the two paths of $C$ joining $u,v$ with $|E(L_1)| \leq |E(L_2)|$.
Denote the vertices of $Q$ by  $u \dd q_1 \dd q_2 \dd \dots \dd q_k \dd v$ in order. It follows that $|E(L_1)|, |E(L_2)| > k+1$. Since $Q$ contains no major vertices, $k>1$. Since $q_1$ is not $C$-major, there is a path $P_1$ of $C$ of length at most two such that all neighbors of $q_1$ in $V(C)$ lie in $P_1$. Choose $V(P_1)$ to be minimal. Define $P_2$ similarly for $q_k$.
\stmt{$P_1$ and $P_2$ are vertex-disjoint. \label{pie:1}}
Suppose not. Then $|E(L_1)| \leq 4$. Since $|E(L_1)| > k + 1 \geq 3$, it follows that $k=2$ and $|E(L_1)| =4$. Thus $P_1$ and $P_2$ both have length two. Hence, $L_1$, $u \dd q_1 \dd q_2\dd v$ and $L_2$ form a long jewel of order four, a contradiction. This proves (\ref{pie:1}).
\stmt{One of $q_2, \dots, q_{k-1}$ has a neighbor in $V(C)$. \label{pie:2}}
Suppose not. By Lemma \ref{lem:P_1P_2} we may assume $P_1$ has length two. Let $C'$ be the hole obtained by replacing the central vertex of $P_1$ with $q_1$. (See Figure \ref{fig:interiornbrs}). If $\{q_2, \dots q_k\}$ contains any $C'$-major vertices, $k=2$.
But then $q_2$ has no neighbor in $V(C') \setminus (V(P_2) \cup \{ q_1\})$, contradicting Lemma \ref{lem:nobombonsleh}. Hence $C'$ is a clean shortest long even hole.

If $u$ is the middle vertex of $P_1$, the path $q_1 Q q_k \dd v$ is a shortcut of $C'$ and it is shorter than $Q$, a contradiction. Hence, we may assume $u$ is an end of $P_1$. Denote the other end of $P_1$ by $z$. 

Suppose $z \in L_2^*$. Then $vL_2z \dd q_1$ has length greater than $k$. The path $L_1 \dd q_1$ has length at least $k + 3$ so $q_1Qq_k \dd v$ is a shortcut of $C'$, a contradiction. Therefore, we may assume $z \in L_1^*$. Since $|E(zL_1v)| > k-1$ and $|E(L_2)| >k$, it follows that $q_1Qq_k \dd v$ is a shortcut of $C'$, a contradiction.
This proves (\ref{pie:2}).
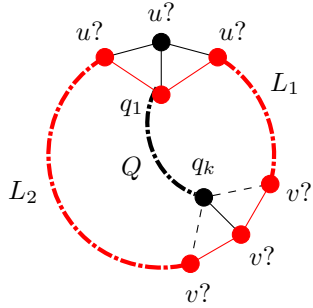
\begin{figure}[h!]
	\centering
	\input{Tikz/interiornbrs.tex}
	\caption{An illustration of the case considered in statement (\ref{pie:2}). $C$ is drawn as the outer face. $C'$ is drawn in red. The vertices labeled with ``$u$?'' and ``$v$''? might be equal to $u$ or $v$, respectively.
	}
	\label{fig:interiornbrs}
\end{figure}
\\
\\
We will show that none of $q_2, \dots, q_{k-1}$ has neighbors in $V(C)$ for a contradiction.
\stmt{ Suppose $q_i \in \{q_2, q_3 \dots , q_{k-1} \}$ has a neighbor $w$ in $V(C)$. Let $R$ denote the path of $C$ with ends $u,w$ that does not go through $v$. Let $S$ denote the path of $C$ with ends $w,v$ that does not go through $u$. Let $x = i+ 1 - |E(R)|$, let $y = k-i+2 - |E(S)|$. Then $x,y \in \{0,1 \}$ and at at most one of $x,y$ is equal to zero. \label{pie:3}}
See Figure \ref{fig:RSdist} for an illustration of this case.
By Lemma \ref{lem:cycledistance} and the fact that $Q$ is a shortest shortcut it follows that $|E(R)| \leq i + 1$ and $|E(S)| \leq k-i+2$. Since $Q$ is a shortcut, $|E(R)| + |E(S)| > k+1$, and the claim follows. This proves (\ref{pie:3}).
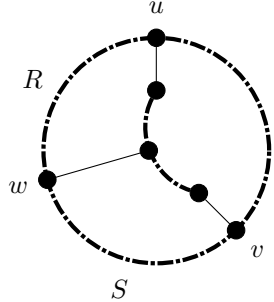
\begin{figure}[h!]
	\centering
	\input{Tikz/RSdist.tex}
	\caption{An illustration of the case considered in statement (\ref{pie:3}). $C$ is drawn as the outer face. $Q$ is the $uv$-path not contained in $C$. Note this is a simplified image, vertices in $V(Q)$ may be equal or adjacent to vertices in $V(C)$.
	}
	\label{fig:RSdist}
\end{figure}
\stmt{ None of $q_2, \dots, q_{k-1}$ have a neighbor in $V(L_1)$. \label{pie:4}}
Suppose that for some $i \in \{2, 3, \dots, k-1\}$, $q_i$ has a neighbor $w \in V(L_1)$. Since $Q$ is a shortest $uv$-path, $w \neq u,v$. Let $R_1$, $S_1$ be the subpaths of $L_1$ with ends $u,w$ and $w,v$, respectively.
Suppose $q_j \in \{q_2, q_3, \dots, q_{k-1} \}$ has a neighbor $z \in V(L_2)$. Let $R_2, S_2$ denote the subpaths of $L_2$ with with ends $u, z$ and $z,v$, respectively.
\begin{figure}[!h]
	\centering
	\input{Tikz/L1L2nbrs.tex}
	\caption{An illustration of the case where $\{q_2, q_2, \dots, q_{k-1}\}$ has neighbors in both $V(L_1)$ and $V(L_2)$ as analyzed in (\ref{pie:4}). $C$ is drawn as the outer face. $Q$ is the $uv$-path not contained in $C$. $Q'$ is drawn in red. Note that this is a simplified drawing, more vertices in $V(Q)$ may be equal or adjacent to vertices in $V(C)$.
	}
	\label{fig:L1L2nbrs}
\end{figure}
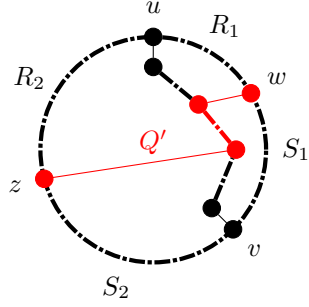 

Then, $d_C(w,z) = \min \{ |E(R_1)| + |E(R_2)|, \ |E(S_1) + |E(S_2)| \}$. Let  $Q'$ denote the path $w \dd q_iQq_j \dd z$.
By (\ref{pie:3}) it follows that $|E(R_1)| \geq i$ and $|E(R_2)| \geq j$.
Thus, $|E(R_1)| + |E(R_2)| > |j-i| + 3 > |E(Q')|$ since $i,j \geq 2$. Similarly, $|E(S_1)| + |E(S_2)| > |E(Q')|$.
But then $Q'$ is a shortcut, contradicting that $Q$ is a shortest shortcut. Hence, none of $q_2, q_3, \dots, q_{k-1}$ has a neighbor in $V(L_2)$.

We denote the vertices of $L_2$ by $u \dd b_1 \dd b_2 \dd \dots \dd b_m \dd v$ in order. Let $x = i +1 - |E(R_1)|$ and let $y = k-i+2 - |E(S_1)|$. Suppose $x = y = 0$. Then $|E(L_1)| = k+3$, and $L_1, L_2$ and $Q$ all have the same parity. Since $G[V(Q \cup L_2)]$ does not contain a long even hole, we may assume $b_1$ is adjacent to $q_1$. But then $b_1 \dd q_1 Q q_i \dd w$ is shortcut since:
$$d_C(b_1,w) = \min \{|E(L_2)| -1 + |E(S_1)|, \ |E(R_1)| + 1\} \geq \min \{ 2k -i + 4, \ i + 2 \} > i+1.$$
We reach a contradiction as $b_1 \dd q_1Qq_i \dd w$ is shorter than $Q$.
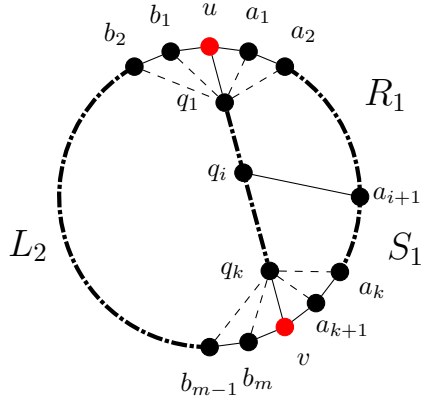
\begin{figure}[!h]
	\centering
	\input{Tikz/RSend.tex}
	\caption{An illustration for the proof of Statement (\ref{pie:4})  of Theorem \ref{thm:shortcuts}. $C$ is drawn as the outer face and $Q$ is the $uv$-path not contained in $Q$.  Note this figure is simplified and that $i$ could be any value in $[2, k-1]$.
	}
	\label{fig:RSend}
\end{figure} 
Thus by (\ref{pie:3}), we may assume $x_1 = 0$ and $y_1 = 1$. Hence, $|E(L_1)| = k+2$. Denote the vertices of $L_1$ by $u \dd a_1 \dd a_2 \dd  \dots \dd a_{k+1} \dd v$ in order. Then, $w = a_{i+1}$. (See Figure \ref{fig:RSend}.)
For all $j \in \{1, 2\}$, $q_1$ is not adjacent to $b_j$, because otherwise $b_j \dd q_1 Q q_i \dd a_{i+1}$ would be a shortcut with length less than $|E(Q)|$. 
It follows from (\ref{pie:3}) that if $q_h \in Q^*$ is adjacent to $a_j \in L_1^*$, then $j \in \{h, h+1 \}$.
Let $C'$ denote the hole induced by $V(L_2) \cup \{q_1, q_2, \dots q_i, a_{i+1}, a_{i+2}, \dots a_{k+1} \}$. It follows that $C'$ is a clean shortest long even hole in $G$. If $i \neq k-1$, then $C'$ is a clean shortest long even hole in $G[V(C \cup Q)]$. But in that case, $q_iQq_k \dd v$ is a shortcut for $C'$ in $G[V(C \cup Q)]$ and $q_iQq_k \dd v$ is shorter than $Q$, a contradiction.

Hence, $i = k-1$. Since $i\geq 2$, it follows that $k \geq 3$. The vertices $q_k$ and $b_{m-1}$ are non-adjacent because otherwise $a_k \dd q_{k-1} \dd q_k \dd b_{m-1}$ is a shortcut and, since $k > 2$, it is a shorter than $Q$, a contradiction. Suppose $q_k$ is not adjacent to $b_m$. Then $L_2 \cup Q$ is a hole. Since $G$ contains no long even holes of length less than $2\ell$, it follows that $|E(L_2)| \geq \ell$. But then $q_{k-1} \dd q_k \dd v, q_{k-1} \dd a_{k} \dd a_{k+1} \dd v$ and $q_{k-1} Q u \dd L_2$ form a long jewel of order 3, a contradiction. Hence $q_k$ is adjacent to $b_m$. We may assume $Q$ was chosen to maximize the distance along $C$ between its ends. It follows that $m \leq k+2$. Hence $m = k + 1$, because $L_1$ and $L_2$ have the same parity.

If $q_{k-1}$ is adjacent to $a_{k-1}$, since $k \geq 3$ the paths $b_{k+1} \dd v \dd a_{k+1} \dd a_k \dd a_{k-1}$, $b_{k+1} \dd q_k \dd q_{k-1} \dd a_{k-1}$ and $L_2\dd u L_1 a_{k-1}$ form a long jewel of order four, a contradiction. Thus, $q_{k-1}$ is non-adjacent to $a_{k-1}$.
Consider the cycle $C''$ obtained from $C$ by replacing the path $b_{k+1} \dd v \dd a_{k+1} \dd a_k$ with $b_{k+1} \dd q_k \dd q_{k-1} \dd a_k$. It follows that $C''$ is a shortest long even hole in $G$. Since $G$ has no long even hole of length at most $2\ell + 2$ and $|E(C)| = 2k + 2$, it follows that $k > 3$. Consequently, $C''$ is a clean shortest long even hole in $G[V(C \cup Q)]$. But $uQq_{k-1}$ is a shortcut for $C''$ and it is shorter than $Q$, a contradiction. This proves (\ref{pie:4}).
\\
\\
By (\ref{pie:2}) and (\ref{pie:4}), some $q_i \in \{q_2, q_3, \dots, q_k \}$ has a neighbor $w \in V(L_2)$ and we may assume that $|E(L_1)| < |E(L_2)|$. Since $C$ is even, it follows that $|E(L_1)| +2 \leq |E(L_2)|$. Moreover, $|E(L_1)| > k + 2$, so $|E(L_2)| \geq k + 4$. Let $R_2, S_2$ be the subpaths of $L_2$ between $u,w$ and $w,v$, respectively.
Since $Q$ is a shortest shortcut it follows from Lemma \ref{lem:cycledistance} that $|E(R_2)| \leq i+1$ and $|E(S_2)| \leq k -i +2$ so $|E(L_2)| \leq k+3$, a contradiction.
\end{proof}

We will now prove the second statement of Theorem \ref{thm:cleanshortest}:

\begin{theorem}
Let $G$ be a candidate. Let $C$ be a clean shortest long even hole in $G$. Let $u,v$ be distinct and non-adjacent vertices in $V(C)$. Let $L_1$ and $L_2$ be the two paths of $C$ with ends $u$ and $v$. Let $Q$ be a shortest $uv$-path. Then $L_2 \cup Q$ or $L_1 \cup Q$ is a clean shortest long even hole.
\label{thm:newcleansleh}
\end{theorem}
\begin{proof}
We may assume $|E(L_1)| \leq |E(L_2)|$. By Theorem \ref{thm:shortcuts}, $L_1$ and $Q$ have the same length.
Denote the vertices of $L_1$ by $u \dd a_1 \dd a_2 \dd \dots \dd a_k \dd v$ in order, denote the vertices of $Q$ by $u \dd q_1 \dd q_2 \dd  \dots \dd q_k \dd v$ in order and denote the vertices of $L_2$ by $u \dd b_1 \dd b_2 \dd \dots \dd b_m \dd v$ in order.
 We proceed by induction on $k$.

We show that the theorem holds if $k = 1$. Consider the cycle $C'$ obtained by replacing the middle vertex of $L_1$ with $q_1$. Since $C$ is clean, $C'$ is a shortest long even hole. Suppose $x$ is a $C'$-major vertex. Let $P$ be a minimum length path of $C$ containing all neighbors of $x$ in $V(C')$. Then by Lemma \ref{lem:jewelpaths}, $P$ has length at least $\ell +3 \geq 7$. Since $x$ is not $C$-major, it follows that $x$ is adjacent to $q_1$ and $x$ is adjacent to some $w \in V(L_2)$ with $d_{C'}(q_1, w) \geq \ell + 3 \geq 7$. Hence $d_C(u,w) > 6$. But then $u \dd q_1 \dd x \dd w$ is a shortcut, contradicting Theorem \ref{thm:shortcuts}. Hence, we may assume $k \geq 2$.
We begin by proving $L_1 \cup Q$ or $L_2 \cup Q$ is a shortest long even hole.
\stmt{If $k=2$, then $L_1 \cup Q$ or $L_2 \cup Q$ is a shortest long even hole.\label{dog:1}}
Suppose $k=2$, and thus $|E(L_1)| = 3$. Since $|E(C)| \geq 9$ and $C$ has no shortcuts, $q_1$ and $q_2$ have no neighbors in $L_2^*$. Thus $L_2 \cup Q$ is a shortest long even hole. This proves (\ref{dog:1}).
\stmt{If there exists some $q_i \in \{q_1, q_2, \dots, q_k\}$ such that $q_i \in V(C)$, then $L_1 \cup Q$ or $L_2 \cup Q$ is a shortest long even hole. \label{dog:2}}
Suppose for some $i \in \{1,2, \dots, k \}$, $q_i \in V(C)$.
Let $R$ denote the path of $C$ with ends $u, q_i$ that does not contain $v$ and let $S$ denote the path of $C$ with ends $q_i, v$ that does not contain $u$. By Theorem \ref{thm:shortcuts} and Lemma \ref{lem:cycledistance}, it follows that $R$, $S$ have lengths $d_C(u, q_i)$ and $d_C(q_i,v)$ respectively.
Thus $R$ has length at most $i$ and $S$ has length at most $k-i+1$ by Theorem \ref{thm:shortcuts}. So $|E(R \cup S)| = k+1$ and we may assume $R \cup S = L_1$.
By induction, the cycle $C'$ obtained from $C$ by replacing $R$ with the path $u \dd q_1 \dd q_2 \dd \dots \dd q_i$ and the cycle obtained from $C'$ by replacing $S$ with the path $q_i \dd q_{i+1} \dd \dots \dd q_k \dd v$ are both clean shortest long even holes. But then $Q \cup L_2$ is a shortest long even hole. This proves (\ref{dog:2}).
\stmt{The vertex set $\{q_2, q_3, \dots, q_{k-1} \}$ is anticomplete to $V(L_1)$ or it is anticomplete to $V(L_2)$. \label{dog:3}}
Suppose that it isn't. Hence, for some $i,j \in \{2,3\dots, k-1\}$, $q_i$ has a neighbor $x \in L_1^*$ and $q_j$ has a neighbor $y \in L_2^*$ and we may assume $i \leq j$. By (\ref{dog:2}), we may assume $Q^* \cap V(C) = \emptyset$.
Let $R_1, S_1$ be the subpaths of $L_1$ between $u$ and $x$ and between $x$ and $v$, respectively. Let $R_2, S_2$ be the subpaths of $L_2$ between $u$ and $y$ and between $y$ and $v$, respectively. Then by Theorem \ref{thm:shortcuts} and Lemma \ref{lem:cycledistance}, $|E(R_1)| \leq i +1$, $|E(S_1)| \leq k-i +2$, $|E(R_2)| \leq j+1$ and $|E(S_2)| \leq k-j + 2$. Since $|E(L_1)|$, $|E(L_2)| \geq k + 1$, it follows that $|E(R_1)| \geq i$, $|E(S_1)| \geq k-i+1$, $|E(R_2)| \geq j$ and $|E(S_2)| \geq k - j +1$.  By Theorem \ref{thm:shortcuts}, the distance between $x$ and $y$ in $G$ is equal to the length of $S_1 \cup S_2$ or $R_1 \cup R_2$. So $d_G(x,y) \geq \min \{i+j +2, 2k - i -j +4 \} > j-i +2$.  
But the path $x \dd q_i \dd q_{i+1} \dd \dots \dd q_{j-1} \dd q_j \dd y$ has length $j-i +2$, so it is a shortcut, a contradicting Theorem \ref{thm:shortcuts}. This proves (\ref{dog:3}).
\stmt{If  none of $q_2, q_3, \dots, q_{k-1}$ have neighbors in $V(C)$, then $L_1 \cup Q$ or $L_2 \cup Q$ is a shortest long even hole. \label{dog:4}}
Suppose none of $q_2, q_3, \dots, q_{k-1}$ have neighbors in $V(C)$. By (\ref{dog:1}) we may assume $k \geq 2$ and by (\ref{dog:2}) we may assume $V(C) \cap Q^* = \emptyset$.
Since $C$ is clean there is a path $P_1$ of $C$ with length at most two containing all neighbors of $q_1$ in $V(C)$. Choose $P_1$ to be minimal. Define $P_2$ similarly for $q_k$.
Suppose $P_1$ has length two. Denote the ends of $P_1$ by $w,z$. Since the theorem holds if $k=1$, the cycle $C'$ obtained by replacing the middle vertex of $P_1$ with $q_1$ is a clean shortest long even hole. By Theorem \ref{thm:shortcuts}, $C'$ has no shortcut, so we may assume $d_{C}(v, w) \leq k-1$. Then $z = u$ and we may assume $w = a_2$. 

If $q_k$ is adjacent to $b_m$, it follows that $q_1 \dd q_2 \dd \dots \dd q_k \dd b_m$ is a shortcut of $C'$, a contradiction. Suppose $q_k$ is adjacent to $b_{m-1}$. Denote the hole $u \dd b_1 \dd b_2 \dd \dots \dd b_{m-1} \dd q_k \dd q_{k-1} \dd \dots \dd q_1 \dd u$ by $C''$. Since $|E(C)| \geq 2\ell + 3$, it follows that $C''$ is long. Since $L_1, L_2, Q$ all have the same parity, $C''$ is even. But $C''$ is shorter than $C$, a contradiction. Hence, $q_k$ has no neighbor in $L_2^*$. But then $L_2 \cup Q$ is a shortest long even hole. Thus, we may assume $P_1$ does not have length two. Similarly, $P_2$ does not have length two.

We may assume $q_1$ is adjacent to $b_1$, because otherwise $L_2 \cup Q$ is a shortest long even hole. It follows that $d_G(b_1, v) \leq k+1$, so $d_C(b_1, v) \leq k+1$ by Theorem \ref{thm:shortcuts}. Therefore, $|E(L_2)| = k+1$. We may assume $q_k$ is adjacent to $a_k$, because otherwise $L_1 \cup Q$ is a shortest long even hole. But $|E(C)| \geq 2\ell +3$, so $G[V(C \cup Q)]$ is a long near-prism, a contradiction. This proves (\ref{dog:4}).
\stmt{ Suppose that $|E(L_2)| \geq |E(L_1)| + 2$ and for some $q_i \in \{q_2, q_3, \dots, q_{k-1} \}$, $q_i$ has a neighbor $w \in L_2^*$. Then, $L_1 \cup Q$ or $L_2 \cup Q$ is a shortest long even hole \label{dog:5}}
By (\ref{dog:2}) we may assume $V(C) \cap Q^* = \emptyset$. By (\ref{dog:3}) none of $q_2, q_3, \dots, q_{k-1}$ have neighbors in $V(L_1)$. Let $R_2$ and $S_2$ denote the subpaths of $L_2$ between $u$ and $w$ and between $w$ and $v$ respectively. By Theorem \ref{thm:shortcuts} and Lemma \ref{lem:cycledistance}, $|E(R_2)| \leq i+1$ and $|E(S_2)| \leq k-i+2$. Since $|E(L_2)| \geq k+3$, it follows that $|E(R_2)| = i+1$ and $|E(S_2)| = k-i+2$. Hence, $|E(C)| = 2k + 4$ and $|E(L_2)| = |E(L_1)| + 2$.
Since $|E(C)| \geq 2\ell+3$, $L_1 \cup Q$ is a long even cycle, so we may assume $L_1 \cup Q$ is not an induced subgraph of $G$. Hence, we may assume $q_1$ is adjacent to $a_j$ for some $j \in \{1,2\}$. Then  $d_C(w,a_j) \geq \min \{ i+2, 2k-i + 1\} > i + 1$. But $b_j \dd q_1 \dd q_2 \dd \dots \dd q_i \dd w$ has length $i+1$, a contradiction. This proves (\ref{dog:5}).
\stmt{$L_1 \cup Q$ or $L_2 \cup Q$ is a shortest long even hole. \label{dog:6}}
Suppose neither $L_1 \cup Q$ nor $L_2 \cup Q$ is a shortest long even hole.
By (\ref{dog:2}), $Q^* \cap V(C) = \emptyset$.
By (\ref{dog:3}) and (\ref{dog:5}), we may assume $V(L_2)$ is anticomplete to $\{q_2, q_3, \dots q_{k-1} \}$.
Hence, by (\ref{dog:4}), we may assume that some $q_i \in \{q_2, q_3, \dots, q_{k-1}\}$ has a neighbor $a_j \in V(L_1^*)$. 
Since $G[V(Q \cup L_2)]$ contains no long even hole we may assume $q_1$ is adjacent to $b_1$ and non-adjacent to $b_2$.
Hence $d_G(b_1, v) \leq k+1$, so by Theorem \ref{thm:shortcuts} $d_C(b_1, v) \leq k+1$. Thus $|E(L_2)| \leq k + 2$. Since $L_1$ and $L_2$ have the same parity, $|E(L_2)| = k + 1$.
Since $b_1 \dd q_1 \dd q_2 \dd \dots \dd q_i \dd a_j$ is not a shortcut, it follows from Lemma \ref{lem:cycledistance} that $j \leq i+1$.
Since $a_j \dd q_i \dd q_{i+1} \dd \dots q_{k} \dd v$ is not a shortcut, it follows from Lemma \ref{lem:cycledistance} that $j \geq i$.
Suppose $i = j$. Let $C'$ denote the cycle obtained from $C$ by replacing the path $a_i \dd a_{i-1} \dd \dots a_1 \dd u \dd b_1$ with the path $a_i \dd q_i \dd q_{i-1} \dd \dots \dd q_1 \dd b$. Then $C'$ is a clean shortest long even hole by induction. But then $q_i \dd q_{i+1} \dd \dots \dd q_k \dd v$ is a shortcut for $C'$, contradicting Theorem \ref{thm:shortcuts}.
Hence, we may assume $i = j-1$ and that for all $c, d \in \{2, 3, \dots, k -1 \}$, if $a_c$ is adjacent to $q_d$ then $c = d - 1$.
Without loss of generality we may assume $i$ is chosen to be the smallest element of $\{2, 3, \dots, k-1 \}$ such that $q_i$ is adjacent to $a_{i-1}$.
The paths $b_1 \dd u \dd a_1 \dd a_2 \dd \dots \dd a_{i-1}$, $b_1 \dd q_1 \dd q_2 \dd \dots \dd q_i \dd a_{i-1}$ and $b_1 \dd b_2 \dd \dots \dd b_m \dd v \dd a_k \dd a_{k-1} \dd \dots \dd a_{i-1}$ form a long jewel of order $i$. Hence, $i > \ell +3$. Then, if $a_1$ is not adjacent to $q_1$, the cycle $u \dd a_1 \dd a_2 \dd \dots \dd a_{i-1} \dd q_i \dd q_{i-1} \dd \dots \dd q_1 \dd u$ is long even hole shorter than $C$, a contradiction. Thus, $q_1$ is adjacent to $a_1$. But then, $q_1 \dd L_2 \dd vQq_i$, $q_1Qq_i$ and $q_1 \dd a_1L_1a_{i-1} \dd q_i$ form a long theta, a contradiction. 
This proves (\ref{dog:6}).
\\
\\
\begin{figure}[!h]
	\centering
	\input{Tikz/newcleansleh.tex}
	\caption{An illustration of a case considered at the end of the proof of Theorem \ref{thm:newcleansleh}. From left to right the $uv$-paths are $L_2, Q, L_1$ in the drawing. The $V(M_1)$ and $V(M_2)$ are both contained in the red vertices.}
		\label{fig:newcleansleh}
\end{figure}
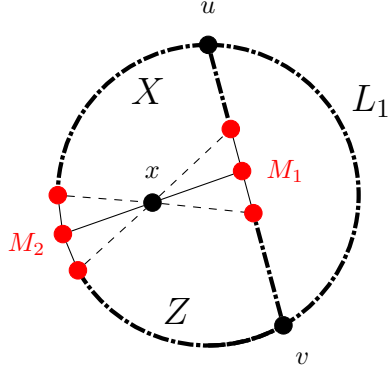
Let $C'$ denote $L_2 \cup Q$. By (\ref{dog:6}) we may assume $C'$ is a shortest long even hole. It remains to show $C'$ is clean.
Suppose there is a $C'$-major vertex $x$. Since $x$ is not $C$-major, $x$ has a neighbor in $Q^*$. Since $Q^*$ is a shortest path there is a subpath $M_1$ of $Q$ of length at most two containing all neighbors of $x$ in $V(Q)$. Choose $M_1$ to be minimal. Thus $x$ has a neighbor in $L_2^*$. Since $x$ is not $C$-major, there is a subpath $M_2$ of $L_2$ of length at most two containing all neighbors of $x$ in $V(L_2)$. Choose $M_2$ to be minimal. Let $M_1$ have ends $y_1,z_1$ where $u,y_1,z_1,v$ are in order on $Q$. Let $M_2$ have ends $y_2, z_2$ where $u, y_2, z_2, v$ are in order in $L_2$.  Let $Y$ denote the path of $C'$ with ends $y_1$ and $y_2$ that contains $u$ and let $Z$ denote the path of $C'$ with ends $z_1$ and $z_2$ that contains $v$. (See Figure \ref{fig:newcleansleh}.)

By Lemma \ref{lem:jewelpaths}, $M_1 \cup Y \cup M_2$ and $M_1 \cup Z \cup M_2$ both have length at least $\ell+3$. Thus $M_1$ and $M_2$ are vertex disjoint and do not contain $u$ or $v$. Since $|E(C)| \geq 2\ell+3$ one of $Y$, $Z$ has length at least $\ell$, say $Y$. Hence the hole obtained by adding $y_1 \dd x \dd y_2$ to $Y$ is long and shorter than $C$, so $Y$ is odd. Thus the path $M_1 \cup Z \cup M_2$ is odd.

Suppose $M_1$ has length two and denote its vertices by $q_{i-1}\dd q_i \dd q_{i+1}$ in order. Let $Q'$ be the path obtained from $Q$ by replacing $q_i$ with $x$. Since $L_2 \cup Q'$ is not an induced subgraph of $G$, by (\ref{dog:6}) we have that $L_1 \cup Q$ is an induced of $G$ and $|E(L_1)| = |E(L_2)| = k+1$. Then $k + 1 \geq \ell$. Since there are no long thetas, $L_1 \cup Q$ is not a hole. Thus $q_i$ has a neighbor $w$ in $L_1^*$. Then $d_G(w, y_2), d_G(w, z_2) \leq 4$. So by Theorem \ref{thm:shortcuts}, it follows that $d_C(w, y_2),d_C(w,z_2) \leq 4$. Since $|E(C)| > 9$, we may assume $y_2 \dd z_2 \dd v \dd w$ is a path of $C$. Hence, $w = a_k$. By Lemma \ref{lem:jewelpaths}, the path $y_2 \dd z_2 \dd vQq_i$ has length at least $\ell +3$. Thus, $i \leq k-\ell$ and it follows that $uQq_i\dd a_k$ has length less than $k$. But by Lemma \ref{lem:cycledistance}, $d_C(u, a_k) = k$ so $uQq_i \dd a_k$ is a shortcut for $C$, contradicting Theorem \ref{thm:shortcuts}.

Thus $M_1$ has length at most one, and so $Z$ has length at least $\ell$. Hence $Z$ has odd length. Since $C'$ is an even hole and $Y,Z$ both have odd length, it follows that $M_1$ and $M_2$ have the same parity.
If $M_1$ and $M_2$ both have length equal to zero, $G[V(C' \cup Q)]$ is a long theta, a contradiction.
If $M_1$ and $M_2$ both have length equal to one, $G[V(C' \cup Q)]$ is a long near-prism, a contradiction.
So, $M_1$ has length equal to zero and $M_2$ has length equal to two. Let $C''$ be the cycle obtained from $C$ by replacing the middle vertex of $M_2$ by $x$. Since $k \geq 2$, $C''$ is a clean shortest long even hole. Then by (\ref{dog:6}), $L_1 \cup Q$ is a long even hole. But then $G[V(C \cup Q)]$ is a long theta, a contradiction.
\end{proof}
We can now give the main result of the section.

\begin{theorem}\label{thm:detectingCleanSLEH}
There is an algorithm with the following specifications:
\begin{description}
\item[Input:] A candidate $G$.
\item[Output:] Decides either that $G$ has a long even hole  or that there is no clean long even hole in $G$.
\item[Running Time:] $\mathcal{O}(|G|^5)$.
\end{description}
\end{theorem}

\begin{proof}
If $C$ is a clean shortest long even hole let $u,v,w \in V(C)$ be chosen so that each of $d_C(u,v)$, $d_C(w,v)$, $d_C(w,z)$ is equal to $\lceil |C|/3 \rceil$ or $\lfloor |C|/3 \rfloor$.
Here is the algorithm: Guess $u,v,w$, find a shortest path between each pair of them and test whether these three paths form a long even hole. If so, output that $G$ has a long even hole, otherwise output $G$ has no long even hole.
Correctness follows from Theorem \ref{thm:cleanshortest}.
\end{proof}

\section{Cleaning a shortest long even hole}
Let $C$ be a shortest long even hole in $G$.
For a $C$-major vertex $x$, we call a subpath $P$ of $C$ of length at least two an {\em $x$-gap} if both ends of $P$ are neighbors of $x$ and no interior vertex of $P$ is adjacent to $x$. Thus, adding $x$ to $P$ yields a hole.
For a pair of non-adjacent $C$-major vertices $x,y$ we call a path $P$ of $C$ an {\em $xy$-gap} if $V(P)$ is the interior of an $xy$-path. 
For a path $P$ with ends $a,b$ we call $v \in V(P)$ a {\em midpoint} of $P$ if it maximizes the value $\min \{ d_P(v,a), d_P(v,b) \}$ among all vertices in $V(P)$. We begin with the following observations about gaps of major vertices on shortest long even holes.

\begin{lem}\label{lem:nestedgaps}
Let $G$ be a graph and let $C$ be a shortest long even hole in $G$.
Let $x,y$ be a pair of non-adjacent $C$-major vertices. Suppose the neighbors of $x$ in $V(C)$ are contained in a $y$-gap $P$. Then there is an $xy$-gap of length at most $\lceil \frac{\ell}{2} \rceil -3$ contained in $P$.
\end{lem}

\begin{proof}
Suppose not. Let $v_1$, $v_2$ denote the ends of $P$. For $i \in \{1, 2 \}$, let $P_i$ denote the $xy$-gap contained in $P$ with one end equal to $v_i$. Then $P_1$, $P_2$ each have length at least $\lceil \frac{\ell}{2} \rceil -2$. Let $P_3$ be the subpath of $C$ with interior equal to $V(P) \setminus (V(P_1 \cup P_2)$. Thus $P = P_1 \cup P_2 \cup P_3$. Since $x$ is $C$-major, $P_3$ has length at least three and since $y$ is $C$-major the path $C \setminus P^*$ has length at least three. It follows that $V(P) \cup \{y\}$ induces a long hole and it is shorter than $C$. Hence, $|E(P_1)| + |E(P_2)| + |E(P_3)|$ is odd. Since $V(P_1) \cup V(P_2) \cup \{x, y\}$ induces a long hole shorter than $C$, it follows that $|E(P_1)| + |E(P_2)|$ is odd. Hence, $|E(P_3)|$ is even. So $(V(C) \setminus P_3^*) \cup \{x\}$ induces a long even hole and it is shorter than $C$, a contradiction.
\end{proof}

We need the following Lemma illustrated in Figure \ref{fig:trianglegap}.

\begin{lem} \label{lem:trianglegap}
Let $G$ be a candidate and let $C$ be a shortest long even hole in $G$.
Let $x,y$ be non-adjacent $C$-major vertices. Let $P$ be a path of $C$ such that $P$ is a $y$-gap, every neighbor of  $x$ in $V(P)$ has $P$-distance at least $\ell-1$ from an end of $P$ and $x$ has no neighbor in $V(C)$ that is adjacent to an end of $P$. Let $x$ have a neighbor in $V(P)$. Then $x$ has exactly two neighbors in $V(P)$ and they are adjacent.
\end{lem}
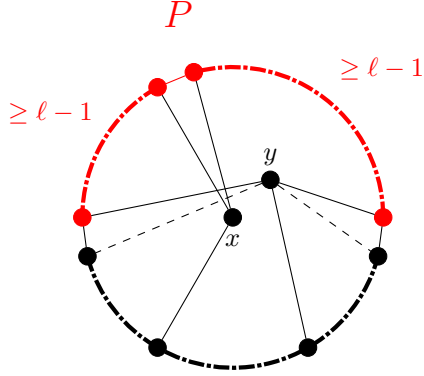
\begin{figure}
	\centering
	\input{Tikz/cleaningasleh.tex}
	\caption{An illustration of the statement of Lemma \ref{lem:trianglegap}. $C$ is drawn as the outer face. The two paths drawn in thick red dashed lines are both of length at least $\ell -1$.
	$P$ is the path of $C$ drawn in red. Note this is a simplified drawing, $x,y$ may have more edges in $V(C) \setminus V(P)$. }
	\label{fig:trianglegap}
\end{figure}
\begin{proof}
Denote the ends of $P$ by $a,b$.
By Lemma \ref{lem:nestedgaps}, it follows that $x$ has a neighbor in $V(C) \setminus V(P)$. Since $x$ is not adjacent to any vertex at $C$-distance at most 1 from an end of $P$ and $y$ has three pairwise non-adjacent neighbors in $V(C)$ by Lemma \ref{lem:Cmajor3pairwisenonadjacent}, there is an $xy$-path $M_1$ with $V(M_1)$ disjoint from and anticomplete to $V(P)$.
Suppose $N(x) \cap V(P)$ consists of a single vertex $w$. Then there is a long theta formed by $wPa \dd y, wPb \dd y, w \dd xM_1y$, a contradiction.
So we may assume that $x$ has two non-adjacent neighbors in $V(P)$. Then there are two long $xy$-paths $M_2, M_3$ with interior in $V(P)$ such that $M_2$ and $M_3$ have disjoint and anticomplete interiors. But then, $M_1, M_2$ and $M_3$ form a long theta, a contradiction.
\end{proof}

We will prove the existence of a bounded-sized collection of paths and vertices of $G$ with useful properties for cleaning $C$-major vertices of a shortest long even hole $C$ called a ``$C$-contrivance''. A $C$-contrivance is analogous to a $(K, P_1, P_2)$-contrivance where $K$ is a shortest long near-prism and $P_1, P_2$ are distinct constituent paths of $K$. The techniques used to prove that $G$ contains a $C$-contrivance are similar to those used in the proof of Lemma \ref{prismset}.
More formally, for a graph $G$ and a shortest long even hole $C$ in $G$, we call a triple  $(A,\mathcal{B}, m)$ a {\em $C$-contrivance} if the following conditions all hold:
\begin{enumerate}
    \item \label{C:Adef}$A$ is a set of $C$-major vertices.
    \item \label{C:Bdef}$\mathcal{B}$ is set of paths of $C$.
    \item  \label{C:Ddef} $m$ is a vertex in $V(C)$.
    \item \label{C:Nbd}Every $C$-major vertex has a neighbor in $A \cup (\cup_{B \in \mathcal{B}} B^*)$.
    \item \label{C:Path} There is a path $P$ of $C$ with both ends in $\cup_{B \in \mathcal{B}} B^*$ such that $m$ is a midpoint of $P$, all $C$-major vertices have a neighbor in $V(P)$ and $N(A) \cap V(P) \subseteq \cup_{B \in
    \mathcal{B}} B^*$.
    \item \label{C:Anbrs} Every vertex in $A$ has a neighbor in $\cup_{B \in \mathcal{B}} B^*$. 
\end{enumerate}

If $(A, \mathcal{B}, m)$ is a $C$-contrivance, we say that $A$, $\mathcal{B}$ and $m$ {\em form} a $C$-contrivance.
We guess a $C$-contrivance as a part of our cleaning algorithm, so it is critical to prove the existence of a $C$-contrivance with some bound on $|A|, |\mathcal{B}|$ and the lengths of the paths in $\mathcal{B}$. We don't need to have a bound on the length of $P$.

We call a $C$-contrivance $(A, \mathcal{B}, m)$ {\em useful} if $|A| \leq 3$, $|\mathcal{B}| \leq 6$, all paths in $\mathcal{B}$ have length at most $2\ell-5$ and at most two paths in $\mathcal{B}$ have length greater that $\ell + 2$. By definition of $C$-contrivance, $\cup_{B \in \mathcal{B}} B^* \neq \emptyset$ so some path in $\mathcal{B}$ must have length at least two.

\begin{lem} \label{lem:Ccontrivance}
Let $G$ be a candidate and let $C$ be a shortest long even hole in $G$.
Then there is a there is a useful $C$-contrivance.
\end{lem}

\begin{proof}
We may assume that $C$ is not clean.
Let $a_1$ be a $C$-major vertex with an $a_1$-gap $P$ such that $P$ has maximum length among all paths of $C$ that form an $x$-gap for some $C$-major vertex $x$.
Denote the ends of $P$ by $v_1, v_2$. It follows that every $C$-major vertex has a neighbor in $V(P)$.
For $i \in \{1,2 \}$ let $B_i$ be the path of $C$ whose vertex set consists of all vertices of $V(P)$ with $P$-distance at most $\ell-1$ from $v_i$ and the two vertices of $V(C)\setminus V(P)$ with $C$-distance at most two from $v_i$. 

Let $S$ be the set of $C$-major vertices with no neighbor in $B_1^* \cup B_2^* \cup \{a_1 \}$. Let $m$ be a midpoint of $P$.
We may assume that $S \neq \emptyset$, because otherwise the triple consisting of $(\{a_1 \}, \{B_1, B_2 \}, m )$ is a useful $C$-contrivance.
For each $w \in S$, we define $P_w$ to be the $w$-gap with $v_1 \in P_w^*$.
Let $a_2$ be an element of $S$ maximizing $|E(P_{a_2}) \setminus E(P)|$. Let $v_3$ denote the end of $P_{a_2}$ that is not contained in $V(P)$. Since $a_2$ has no neighbor in $B_1^* \cup B_2^* \cup \{a_1 \}$ and $a_2$ has at least one neighbor in $V(P)$, it follows from Lemma \ref{lem:trianglegap} that $a_2$ has exactly two neighbors in $V(P)$ and they are adjacent. Denote them by $v_4, v_5$, where $v_4$ is an end of $P_{a_2}$.
Let $R$ denote the path $v_1P_{a_2}v_3$. It follows that every $C$-major vertex has a neighbor in $B_1^* \cup B_2^* \cup \{a_1 \}$ or a neighbor in $V(R)$. Let $B_3$ be the path of $C$ whose vertex set consists of all vertices of $V(P_{a_2})$ with $P_{a_2}$-distance at most $\ell-1$ from $v_3$ and the two vertices of $V(C) \setminus V(P_{a_2})$ with $C$-distance at most two from $v_3$. 
Let $B_4$ be the path of $C$ whose vertex set consists of all vertices of $V(P_{a_2})$ with $P_{a_2}$-distance at most $\ell-1$ from $v_4$ and the two vertices of $V(C) \setminus V(P_{a_2})$ with $C$-distance at most two from $v_4$. Then $v_5 \in B_4^*$.

Let $T$ denote the set of $C$-major vertices with no neighbor in $\cup_{i=1}^4 B_i^* \cup \{a_1, a_2\}$. 
Let $t \in T$. Then $t$ is not equal to $a_1$ or $a_2$.
Since $t$ has a neighbor in $V(R)$ it follows from Lemma \ref{lem:trianglegap} applied to $a_2, t$ and $P_{a_2}$ that $a_3$ has exactly two neighbors in $V(P_{a_2})$ and they are adjacent. Denote them by $x_1, x_2$. Since $t$ has a neighbor in $V(R)$ and $t$ is not adjacent to $v_1, v_3$, it follows that $x_1, x_2 \in V(R)$.
Since $t$ has a neighbor in $V(P)$, it follows from Lemma \ref{lem:trianglegap} applied to $a_1, t$ and $P$ that $t$ has exactly two neighbors in $P^*$. Denote them by $x_3, x_4$. Since $x_1, x_2$ are the only neighbors of $t$ in $V(P_{a_2})$ it follows that $x_3, x_4 \in P^* \setminus V(P_{a_2})$.

We may assume $T \neq \emptyset$ because otherwise $\{a_1, a_2 \}$, $\{B_1, B_2, B_3, B_4\}$ and $m$ form a useful $C$-contrivance. Let $a_3$ be an arbitrary element of $T$. Denote the neighbors of $a_3$ in $V(R)$ by $v_6, v_7$ and the denote the neighbors of $a_3$ in $P^* \setminus V(P_{a_2})$ by $v_8, v_9$. 
Let $B_5$ be the path of $C$ containing all vertices of $V(C)$ with $C$-distance at most $\ell-3$ from $v_6$ or $v_7$.
Let $B_6$ be the path of $C$ containing all vertices of $V(C)$ with $C$-distance at most $\ell-3$ from $v_8$ or $v_9$.

Let $A = \{a_1, a_2, a_3 \}$. Let $\mathcal{B} = \{ B_1, B_2, \dots, B_6 \}$.
We claim $(A, \mathcal{B}, m)$ is a useful $C$-contrivance.
It follows from the choice of $B_1, B_2, \dots, B_6$ that $v_i \in \cup_{j=1}^6 B_j^*$ for every $i \in \{1, 2, \dots, 9 \}$. Hence, $(A, \mathcal{B}, m)$ satisfies condition \ref{C:Anbrs}.
Moreover, $(A, \mathcal{B}, m)$ satisfies conditions \ref{C:Adef}, \ref{C:Bdef} and \ref{C:Ddef} and the conditions of usefulness by construction, so we need only prove it satisfies \ref{C:Nbd} and \ref{C:Path}.

 Every $C$-major vertex has a neighbor in $V(P)$ and  $N(A) \cap V(P) = \{v_1, v_2, v_4, v_5, v_8, v_9\}$. Since $P$ has ends $v_1, v_2$ and $m$ is a midpoint of $P$, condition \ref{C:Path} is satisfied. Suppose there is some $C$-major vertex $w$ that has no neighbor in $A \cup (\cup_{i=1}^6 B_i^*)$. Then $w \in T$. Hence, $w$ has exactly two neighbors in $V(R)$ and they are adjacent and $w$ has exactly two neighbors in $V(P) \setminus V(P_{a_2})$ and they are adjacent.
 Since $w$ has no neighbors in $B_5^* \cup B_6^*$ it follows that there are two long induced $a_3w$-paths $M_1, M_2$ with $M_1^* \subseteq R$ and $M_2^* \subseteq V(P)\setminus V(R)$. Moreover, since $a_3, w$ have no neighbors in $B_1^*$, it follows that $M_1^*$ is anticomplete to $M_2^*$. Since $w$ and $a_3$ each have no neighbor in $B_2^* \cup B_3^*$, it follows that $w$, $a_3$ have no neighbors in $V(C)$ that are adjacent to $v_2$ or $v_3$. Hence, it follows from Lemma \ref{lem:Cmajor3pairwisenonadjacent} that there is a $wa_3$-path $M_3$ with $M_3^*$ disjoint from and anticomplete to $V(P)\cup V(R)$. But then $M_1, M_2, M_3$ form a long theta, a contradiction.
\end{proof}

We can now prove the main result of this section.

\begin{theorem} \label{alg:cleaningSLEH}
There is an algorithm with the following specifications:
\begin{description}
\item[Input:] A candidate $G$.
\item[Output:] A list of $\mathcal{O}(|G|^{8\ell + 2})$ sets with the following property: For every shortest long even hole $C$ there is some $X$ in the list such that $X$ contains all $C$-major vertices and $X \cap V(C) = \emptyset$.
\item[Running Time:] $\mathcal{O}(|G|^{8\ell + 4})$
\end{description}
\end{theorem}
\begin{proof}
We guess a set $A$ of at most three vertices in $V(G)$ and a vertex $m$ in $V(G)$.
We guess a list $\mathcal{B}$ of at most $6$ paths of $G$, $B_1, B_2, \dots B_k$ such that all paths in $\mathcal{B}$ have length at most $2\ell-5$, at least one path in $\mathcal{B}$ has length greater than one, and at most two paths in $\mathcal{B}$ have length greater than $\ell + 2$. Let $\mathcal{B}^*$ denote $\cup_{i=1}^kB_i^*$. By definition, $\mathcal{B}^* \neq \emptyset$. Let $Y$ be the set of vertices in $V(G) \setminus (\cup_{i=1}^k V(B_i))$ with neighbors in $\mathcal{B}^*$.  Guess two vertices $d_1, d_2$ in $\mathcal{B}^*$.
Let $R, S$ be union of the vertex sets of all shortest $d_1m$-paths and $d_2m$-paths in $G \setminus ( (Y \cup N(A)) \setminus \mathcal{B}^*)$ respectively. Let $Z$ be the set of vertices in $V(G) \setminus (Y \cup \mathcal{B}^*)$ with a neighbor in $A$ and a neighbor in $R \cup S \setminus \{d_1, d_2 \}$. Output $Y \cup Z$.

We will now prove the output is correct. Suppose $C$ is a shortest long even hole in $G$.
Then by Lemma \ref{lem:Ccontrivance}, $G$ contains a useful $C$-contrivance. Thus, for some guess of $(A, \mathcal{B}, m)$, the triple $(A, \mathcal{B}, m)$ is a useful $C$-contrivance. By construction, $Y$ is disjoint from $V(C)$ and $Y$ contains every $C$-major vertex in $V(G)$ with a neighbor in $\mathcal{B}^*$. By condition \ref{C:Anbrs} in the definition of $C$-contrivance, $A$ is contained in $Y$. Let $P$ denote the path of $C$ with ends $d_1, d_2$ that contains $m$. By condition \ref{C:Path} in the definition of $C$-contrivance, we may assume that $d_1, d_2$ are chosen such that every $C$-major vertex has a neighbor in $V(P)$, $m$ is a midpoint of $P$ and $N(A)\cap V(P) \subseteq \mathcal{B}^*$.
Let $P_1$ denote the subpath of $P$ with ends $d_1, m$ and let $P_2$ denote the subpath of $P$ with ends $d_2, m$. Since $m$ is a midpoint of $P$, it follows that $P_1$ and $P_2$ have lengths $d_C(d_1, m)$ and $d_C(d_2, m)$, respectively. Since $N(A) \cap V(P) \subseteq \mathcal{B}^*$, the paths $P_1, P_2$ are both subgraphs of $G \setminus ((Y \cup N(A)) \setminus \mathcal{B}^*)$.
By condition \ref{C:Nbd} in the definition of $C$-contrivance, $Y \cup N(A)) \setminus \mathcal{B}^*$ contains all $C$-major vertices. Hence, it follows from Theorem \ref{thm:shortcuts} that $P_1$ and $P_2$ are shortest paths between $d_1, m$ and $d_2, m$, respectively, in $G \setminus ((Y \cup N(A)) \setminus \mathcal{B}^*)$. Hence, $P^* \subseteq (R \cup S) \setminus \{d_1, d_2 \}$.

We prove $Z$ contains all $C$-major vertices in $V(G) \setminus Y$. Suppose $w \in V(G) \setminus Y$ is a $C$-major vertex. By condition \ref{C:Nbd}, it follows that $w$ has a neighbor in $A$. Since $d_1, d_2 \in \mathcal{B}^*$ and all $C$-major vertices have a neighbor in $V(P)$, it follows that $w$ has a neighbor in $(R \cup S) \setminus \{d_1, d_2 \}$.

We prove $Z$ is disjoint from $V(C)$.
Suppose there exists some $z \in V(C) \cap Z$. Then since $N(A) \cap V(P) \subseteq \mathcal{B}^*$ it follows that $z \not \in V(P)$. Since $m$ is a midpoint of $P$, it follows that $d_C(z, m) \geq \frac{|E(P)|}{2} + 1$.
Since $z \in Z$, we may assume there is some shortest $d_1m$-path in $G \setminus ((Y \cup N(A)) \setminus \mathcal{B}^*)$ and $q \in V(Q) \setminus \{d_1 \}$ such that $z$ is adjacent to $q$. Then, the path $qQm$ has length strictly less than $\frac{|E(P)|}{2}$ so $z \dd qQm$ has length strictly less than $d_C(z,m)$.
Since $Y \cup N(A)) \setminus \mathcal{B}^*$ contains all $C$-major vertices, $z \dd qQm$ contains no $C$-major vertices. Hence, $z \dd qQm$ is a shortcut, contradicting Theorem \ref{thm:shortcuts}.
This completes the proof of correctness.

We will now prove the bounds on the running time and list length. There are $\mathcal{O}(|G|^{8\ell + 2})$ guesses of $(A, \mathcal{B}, m)$ to check. For each such guess, we compute $Y, Z$ in time $\mathcal{O}(|G|^2)$. Hence the list outputted has length $\mathcal{O}(|G|^{8\ell+2})$ and the total running time is $\mathcal{O}(|G|^{8\ell+4})$.
\end{proof}

\section{The Algorithm}

We can now prove our main result which we restate: 
\\ \\
\noindent \textbf{Theorem \ref{mainthm}} \textit{For each integer $\ell \geq 4$ there is an algorithm with the following specifications:}
\begin{description}
\item[Input:] \textit{A graph $G$.}
\item[Output:] \textit{Decides whether $G$ has an even hole of length at least $\ell$.}
\item[Running Time:] $\mathcal{O}(|G|^{108\ell-22})$.
\end{description}

\begin{proof}
Our algorithm is as follows.
We begin by testing if $G$ is a candidate by performing the following steps.
\begin{itemize}
	\item We apply the algorithm of Theorem \ref{alg:shortlongevenholes} to test whether $G$ contains a long even hole of length at most $2\ell + 2$ in time $\mathcal{O}(|G|^{2\ell+2})$, we apply the algorithm of Theorem \ref{alg:longjewels} to test whether $G$ contains a long jewel of order at most $\ell + 3$ in time $\mathcal{O}(|G|^{3\ell +6})$ and we apply the algorithm of Theorem \ref{alg:longtheta} to test whether $G$ contains a long theta in time $\mathcal{O}(|G|^{2\ell + 7})$. We may assume these tests fail, because otherwise $G$ contains a long even hole.
	\item We apply the algorithm of Theorem \ref{alg:banthebomb} to test whether $G$ contains a long ban-the-bomb in time $\mathcal{O}(|G|^{2\ell + 5})$. We may assume this test fails.
	\item Then, we apply the algorithm of Theorem \ref{alg:longprisms} to test whether $G$ contains a long near-prism in time $\mathcal{O}(|G|^{108\ell-22})$. We may assume this test fails. Consequently, $G$ is a candidate.
	\item Thus we are able to apply the algorithm given in Theorem \ref{alg:cleaningSLEH} to obtain a cleaning list for major vertices of length $\mathcal{O}(|G|^{8\ell+2})$ in time $\mathcal{O}(|G|^{8\ell+4}).$
	\item For every $X$ in the list we use the algorithm of Theorem \ref{thm:detectingCleanSLEH} to test whether $G \setminus X$ has a clean shortest long even hole in time $\mathcal{O}(|G|^5)$. If we have not found a clean shortest long even hole in $G \setminus X$ for any $X$ in the list, we output that $G$ has no long even hole.
\end{itemize}

Correctness follows from Theorems \ref{thm:detectingCleanSLEH} and \ref{alg:cleaningSLEH}. For the running time, testing whether $G$ is a candidate takes time $\mathcal{O}(|G|^{108\ell-22})$ and determining whether a candidate contains a shortest long even hole takes time $\mathcal{O}(|G|^{8\ell+7})$. Hence, the total running time is $\mathcal{O}(|G|^{108\ell-22})$.
\end{proof}




%% file: Tikz/banthebomb.tex
\begin{tikzpicture}
	\foreach \x in {3,8,9,10}{
		\path (0,0) ++(30*\x:1.5cm) coordinate (a\x);
		\node[dot] at (a\x){};
	}
	
	\draw (a8) -- (a9) -- (a10);
	\node[dot] at (0,0) (center){};
	
	\draw[path] (a8) arc[start angle = 240, end angle= -60, radius = 1.5cm];
	
	\draw (a10) -- (center) -- (a8);
	\draw (center) -- (a3);
	\draw[maybe edge] (center) -- (a9);
\end{tikzpicture}

%% file: Tikz/bomb.tex
\begin{tikzpicture}
	\foreach \x in {3,8,9,10}{
		\path (0,0) ++(30*\x:1.5cm) coordinate (a\x);
		\node[dot] at (a\x){};
	}
	\foreach \x in {0,...,11}{
		\path (0,0) ++(30*\x:1.5cm) coordinate (a\x);
	}
	
	\node[dot] at (a5) {};
	\node[dot] at (a1) {};
	\draw (a8) -- (a9) -- (a10);
	\node[dot] at (0,0) (center){};
	
	\draw[path] (a8) arc[start angle = 240, end angle= 150, radius = 1.5cm];
	\draw[path] (a1) arc[start angle = 30, end angle= -60, radius = 1.5cm];
	
	\draw (a10) -- (center) -- (a8);
	\draw (center) -- (a3);
	\draw (center) -- (a9);
\end{tikzpicture}

%% file: Tikz/longnearprism.tex
\begin{tikzpicture}
	\def\height{-2.5}
	\def\width{1}
	\def\spacing{5}
	
	\node[dot] (a1) at (1,0) {};
	
	\foreach \x in {0,1,2}{
		\pgfmathparse{int(\x + 1)}
		\edef\z{\pgfmathresult}
		\node[dot](c\x) at (\x*\width, 0){};
		\node[dot] (d\x) at (\x*\width, \height){};
		\foreach \y in {0, ..., \x}{
			\draw[bend left = 80] (d\x) to  (d\y);
			\draw[bend right = 80] (c\x) to  (c\y);
		}
		\draw[path] (c\x) to (d\x);
		
	}

	\foreach \x in {0,2}{
		\pgfmathparse{int(\x + 1)}
		\edef\z{\pgfmathresult}
		\node[dot](c\x) at (\x*\width + \spacing, 0){};
		\node[dot] (d\x) at (\x*\width + \spacing, \height){};
		\draw[path] (c\x) to (d\x);
	}
	
	\node[dot] (c1) at (\width + \spacing, .5*\height){};

	\draw[bend left = 80] (d2) to  (d0);
	\draw[bend right = 80] (c2) to  (c0);
	
	\draw (c2) -- (c1) -- (c0);
	\draw (d2) -- (c1) -- (d0);
	
\end{tikzpicture}

%% file: Tikz/prism1path.tex
\begin{tikzpicture}
	
	\coordinate (A) at (0,0);
	\coordinate (B) at (5,0);
	
	\foreach \x in {1,2,0}{
		\path (A)  ++(120*\x:1cm) node[dot] (A\x) {};
		\path (B)  ++(120*\x++180:1cm)node[dot] (B\x) {};
	}
	
	\draw (A1) -- (A2) -- (A0) -- (A1);
	\draw (B1) -- (B2) -- (B0) -- (B1);
	\draw[path] (A1) -- (B2);
	\draw[path] (A0) -- (B0);
	
	\node[dot] (q) at (2.5, -2){};
	
	\foreach \x in {1, ..., 9}{
		\coordinate (z\x) at ($(A2)!.\x!(B1)$);
	}
	\draw[maybe edge] (z4) -- (q) -- (z5);
	\draw (z3) -- (q) -- (z7);
	\draw[path] (A2) -- (z3);
	\draw[path] (B1) -- (z7);
	\foreach \x in {3,4,5,7}{
		\node[dot, color = red] at (z\x){};
	}
	\draw[color = red] (z3) -- (z4) -- (z5);
	\draw[path, color = red] (z5) to (z7);
	\node[left=.1 of q] {$q$};

\end{tikzpicture}

%% file: Tikz/bombonsleh.tex
\begin{tikzpicture}
	\node at (.4, 0) {$v$};
	\foreach \x in {0, ...,11}{
		\path (0,0) ++(30*\x:1.5cm) coordinate (a\x);
		\path (0,0) ++(30*\x:1.9cm) coordinate (b\x);
	}
	
	\foreach \x in {3,8,9,10}{	\node[dot] at (a\x){};	}
	
	\node at (b8) {$x$};
	\node at (b9) {$y$};
	\node at (b10) {$z$};
	\node at (b3) {$w$};
	\node at (b6) {$X$};
	\node at (b0) {$Z$};
	
	\draw (a8) -- (a9) -- (a10);
	\node[dot] at (0,0) (center){};
	
	\draw[path] (a8) arc[start angle = 240, end angle= -60, radius = 1.5cm];
	
	\draw (a10) -- (center) -- (a8);
	\draw (center) -- (a3);
	\draw (center) -- (a9);
\end{tikzpicture}

%% file: Tikz/shortcut_distance_lemma.tex
\begin{tikzpicture}
	\node at (-.4, 0) {$q$};
	\foreach \x in {0, ...,11}{
		\path (0,0) ++(30*\x:1.5cm) coordinate (a\x);
		\path (0,0) ++(30*\x:1.9cm) coordinate (b\x);
	}
	\foreach \x in {0,..., 23}{
		\path (0,0) ++(15*\x:1.9cm) coordinate (c\x);
	}
	
	\foreach \x in {3,9,0}{	\node[dot] at (a\x){};	}
	\def\step{1};
	\node[dot] (q1) at (0, \step){};
	\node[dot] (qk) at (0, -\step){};
	\node[dot] (q) at (0,0) {};
	
	\draw (a3) -- (q1);
	\draw (a9) -- (qk);
	\draw[path] (q1) -- (q) -- (qk);
	\draw (q) -- (a0);
	
	\draw[path] (0,0) circle[radius=1.5cm];
	\node at (c6) {$u$};
	\node at (c18) {$v$};
	\node at (c0) {$w$};
	\node at (c3){$R$};
	\node at (c21) {$S$};
	\node at (c12) {$L$};
\end{tikzpicture}

%% file: Tikz/interiornbrs.tex
\begin{tikzpicture}
	
	\foreach \x in {0, ...,23}{
		\path (0,0) ++(15*\x:1.5cm) coordinate (a\x);
		\path (0,0) ++(15*\x:1.9cm) coordinate (b\x);
		\path (0, 0) ++ (15 *\x:.8cm) coordinate (i\x);
	}
	\draw (a4) -- (a6) -- (a8);
	\node[dot] (qk) at (i21){};
	\draw[maybe edge] (a19) -- (qk) -- (a23);
	\draw (qk) -- (a21);
	\draw[path, bend right= 50] (q1) to (qk);
	
	\node[below left=.05 and .025 of q1] (qq1) {$q_1$};
	\node[above =.05 of qk] {$q_k$};
	\node[below= .3 of qq1]  {$Q$};
	\foreach \x in {8, 6, 4} {\node at (b\x) {$u$?};}
	\foreach \x in {23, 21, 19} {\node at (b\x) {$v$?};}
	\node at (b2) {$L_1$};
	\node at (b13) {$L_2$};
	\draw (i6) to (a6);
	\begin{scope}[color=red]{
			\draw (a23) -- (a21)-- (a19);
			\draw[path] (a23) arc[start angle = -15, end angle= 60, radius = 1.5cm];
			\draw[path] (a8) arc[start angle = 120, end angle= 285, radius = 1.5cm];
			\node[dot, color = red] (q1) at (i6){};
			\foreach \x in {4,8,23,21,19}{ \node[dot, color = red] at (a\x){};}
			\foreach \x in {4,8} {\draw (q1) -- (a\x);}
		}
	\end{scope}
	\node[dot] at (a6) {};
\end{tikzpicture}

%% file: Tikz/RSdist.tex
\begin{tikzpicture}
	
	\foreach \x in {0, ...,23}{
		\path (0,0) ++(15*\x:1.5cm) coordinate (a\x);
		\path (0,0) ++(15*\x:1.9cm) coordinate (b\x);
		\path (0, 0) ++ (15 *\x:.8cm) coordinate (i\x);
	}
	\node[dot] (qk) at (i21){};
	\draw (qk) -- (a21);
	\draw[path, bend right= 50] (q1) to (qk);
	
	\foreach \x in { 6} {\node at (b\x) {$u$};}
	\foreach \x in {21} {\node at (b\x) {$v$};}
	\draw (i6) to (a6);
	
	\draw[path] (0,0) circle[radius = 1.5cm];
	\node[dot] (q1) at (i6){};
	\foreach \x in {21}{ \node[dot] at (a\x){};}
	\node[dot] at (a6) {};
	
	\node[dot] (qi) at (-.1,0) {};
	\draw (qi) to (a13);
	\node[dot] at (a13){};
	\node at (b13) {$w$};
	\node at (b10) {$R$};
	\node at (b17) {$S$};
\end{tikzpicture}

%% file: Tikz/L1L2nbrs.tex
\begin{tikzpicture}
	
	\foreach \x in {0, ...,23}{
		\path (0,0) ++(15*\x:1.5cm) coordinate (a\x);
		\path (0,0) ++(15*\x:1.9cm) coordinate (b\x);
		\path (0, 0) ++ (15 *\x:1.1cm) coordinate (i\x);
		\path (0, 0) ++ (15 *\x:.85cm) coordinate (r\x);
	}
	\node[dot] (qk) at (i21){};
	\draw (qk) -- (a21);
	\node[dot] (q1) at (i6){};

	\foreach \x in { 6} {\node at (b\x) {$u$};}
	\foreach \x in {21} {\node at (b\x) {$v$};}
	\draw (q1) to (a6);
	
	\draw[path] (0,0) circle[radius = 1.5cm];
	\foreach \x in {21}{ \node[dot] at (a\x){};}
	\node[dot] at (a6) {};
	
	\node at (b2) {$w$};
	\node at (b13) {$z$};
	\node at (b10) {$R_2$};
	\node at (b17) {$S_2$};
	\node at (b4) {$R_1$};
	\node at (b0) {$S_1$};
	
	\begin{scope}[color = red]
		\node[dot, color = red] (qi) at (r3) {};
		\node[dot, color = red] (qj) at (i0) {};
		\draw[path] (qi) to (qj);
		\draw (qj) to (a13);
		\node[dot, color = red] at (a13){};
		\draw (qi) to (a2);
		\node[dot, color = red] at (a2){};
		\node at (0,.1) {$Q'$};
	\end{scope}
	\draw[path] (q1) to (qi);
	\draw[path] (qk) to (qj);
\end{tikzpicture}

%% file: Tikz/RSend.tex
\begin{tikzpicture}
	
	\def\crad{2cm};
	\def\orad{2.5cm};
	
	\foreach \x in {0, ...,23}{
		\path (0,0) ++(15*\x:\crad) coordinate (c\x);
		\path (0,0) ++(15*\x:\orad) coordinate (o\x);
		\path (0, 0) ++ (15 *\x:2.7cm) coordinate (i\x);
	}
	\node[dot, color= red] (u) at (c6) {};
	\node[dot, color = red] (v) at (c20){};
	\node[dot] (a1) at (c5){};
	\node[dot] (a2) at (c4){};
	\node[dot] (ak1) at (c22){};
	\node[dot] (ak) at (c21){};
	\node[dot] (b1) at (c7){};
	\node[dot] (b2) at (c8){};
	\node[dot] (bk) at (c19){};
	\node[dot] (bk1) at (c18){};
	\node[dot] (w) at (c0){};
	
	\coordinate (qi) at ($(c6)!.45!(c20)$);
	\node[dot] at (qi){};
	
	\foreach \x in {1, ..., 9}{
		\coordinate (z\x) at ($(c6)!.\x!(c20)$);
	}
	
	\node[dot] (q1) at (z2) {};
	\node[dot] (qk) at (z8){};
	
	\draw (b2) -- (b1) -- (u) -- (a1) -- (a2);
	\draw (bk1) -- (bk) -- (v) -- (ak) -- (ak1);
	\draw (u) -- (q1);
	\draw (v) -- (qk);
	\draw[path] (q1) -- (qk);
	\draw[path] (c8) arc[start angle = 120, end angle= 270, radius = \crad];
	\draw[path] (c4) arc[start angle = 60, end angle= -30, radius = \crad];
	
	\node at (o6) {$u$};
	\node at (o20) {$v$};
	\node at (o7) {$b_1$};
	\node at (o8){$b_2$};
	\node at (o18){$b_{m-1}$};
	\node at (o19){$b_m$};
	\node at (o13){\Large $L_2$};
	\node at (o5) {$a_1$};
	\node at (o4) {$a_2$};
	\node at (o21) {$a_{k+1}$};
	\node at (o22) {$a_k$};
	\node at (o0){$a_{i+1}$};
	\node at (i2){\Large $R_1$};
	\node at (i23){\Large $S_1$};
	
	\node[left=.08 of q1]{$q_1$};
	\node[left=.08 of qk]{$q_k$};
	\node[left=.08 of qi]{$q_i$};
	
	\draw[maybe edge] (bk1) -- (qk) -- (ak);
	\draw[maybe edge] (bk) -- (qk) -- (ak1);
	\draw[maybe edge] (b2) -- (q1) -- (a2);
	\draw[maybe edge] (b1) -- (q1) -- (a1);
	
	\draw (qi) -- (w);
\end{tikzpicture}

%% file: Tikz/newcleansleh.tex
\begin{tikzpicture}
	
	\def\crad{2cm};
	\def\orad{2.5cm};
	
	\foreach \x in {0, ...,23}{
		\path (0,0) ++(15*\x:\crad) coordinate (c\x);
		\path (0,0) ++(15*\x:\orad) coordinate (o\x);
		\path (0, 0) ++ (15 *\x:1.6cm) coordinate (i\x);
	}
	
	\draw[path] (c14) arc[start angle = 210, end angle= 300, radius = \crad];
	\draw[path] (c6) arc[start angle = 90, end angle= 180, radius = \crad];
	\draw[path] (c6) arc[start angle = 90, end angle = -90, radius = \crad];

	\coordinate (qi) at ($(c6)!.45!(c20)$);
	\node[dot, color = red] (qmid) at (qi){};
	
	\foreach \x in {1, ..., 9}{
		\coordinate (z\x) at ($(c6)!.\x!(c20)$);
	}
	\node[dot, color =red] (qbefore) at (z3){};
	\node[dot, color = red](qafter) at (z6){};
	\node[dot] (u) at (c6) {};
	\node[dot] (v) at (c20) {};
	\draw[path] (u) to (qbefore);
	\draw[path] (v) to (qafter);
	\draw (qbefore) to (qmid);
	\draw (qmid) to (qafter);
	
	\node[dot, color=red] (before) at (c12) {};
	\node[dot, color = red] (mid) at (c13) {};
	\node[dot, color = red] (after) at (c14) {};
	\draw (before) -- (mid) -- (after);
	
	\node[dot] (x) at ($(qi)!.5!(c13)$){};
	\draw[maybe edge] (before) -- (x) -- (qbefore);
	\draw[maybe edge] (after) -- (x) -- (qafter);
	\draw (qmid) -- (x) -- (mid);
	

	\node[right=.08 of qmid]{\textcolor{red}{$M_1$}};
	\node at (o13){\textcolor{red}{$M_2$}};
	
	\node at (i17) {\Large $Z$};
	\node at (i8) {\Large $X$};
	\node at (o2) {\Large $L_1$};
	\node at (o6) {$u$};
	\node at (o20){$v$};
	\node [above = .08 of x]{$x$};
	
\end{tikzpicture}

%% file: Tikz/cleaningasleh.tex
\begin{tikzpicture}
	
	\def\crad{2cm};
	\def\orad{2.2cm};
	
	\foreach \x in {0, ...,11}{
		\path (0,0) ++(30*\x:\crad) coordinate (c\x);
		\path (0,0) ++(30*\x:\orad) coordinate (o\x);
		\path (0, 0) ++ (30 *\x:1.6cm) coordinate (i\x);
	}

	\foreach \x in {0, ...,23}{
		\path (0,0) ++(15*\x:\crad) coordinate (cc\x);
		\path (0,0) ++(15*\x:\orad) coordinate (oo\x);
		\path (0, 0) ++ (15 *\x:2.8cm) coordinate (ii\x);
	}
	\begin{scope}[color = red]
		\draw[path] (c0) arc[start angle = 0, end angle= 105, radius = \crad];
		\draw[path] (cc8) arc[start angle = 120, end angle = 180, radius = \crad];
	\end{scope}
	\draw[path] (c13) arc[start angle = 195, end angle= 345, radius = \crad];
	
	\node[dot, color = red] (ge) at (c0) {};
	\node[dot, color = red] (gs) at (c6){};
	\node[dot] (y) at (.5,.5){};
	\draw (y) to (gs);
	\draw (y) to (ge);
	\node[dot] (gee) at (cc23) {};
	\node[dot] (gss) at (cc13){};
	\draw (gss) -- (gs);
	\draw (gee) -- (ge);
	
	\node[dot] (x) at (0,0){};
	\node[dot, color = red] (x1) at (cc8){};
	\node[dot, color = red]  (x2) at (cc7){};
	\draw (x1) -- (x) -- (x2);
	
	\draw[maybe edge] (y)  to (gss);
	\draw[maybe edge] (y) to (gee);
	\node[dot] (xnbr) at (cc16){};
	\draw (x) -- (xnbr);
	\node[dot] (ynbr) at (cc20){};
	\draw (y) -- (ynbr);
	\draw[color = red] (x1) -- (x2);
	\node at (0, -.3) {$x$};
	\node at (.5, .8) {$y$};
	\node at (ii10)  {\textcolor{red}{$\geq \ell -1$}};
	\node at (ii3)  {\textcolor{red}{$\geq \ell -1$}};
	\node at (ii7) {\Large \textcolor{red}{$P$}};
\end{tikzpicture}

%% file: chapter_monoholes.tex
	\chapter{Monoholed Graphs}\label{chapter:monoholes}
	In the next several chapters we will describe the structure of graphs where every hole has length $\ell$ for some integer $\ell \geq 7$. We call $G$ \textit{$\ell$-monoholed} if every hole in $G$ has length $\ell$. When the value of $\ell$ is not ambiguous we will refer to $G$ as \textit{monoholed}.
	We need the following easy fact about monoholed graphs:
	
	\begin{fact} \label{fact:hole-and-vertex-complete-anticomplete-or-P3}
		Let $G$ be an $\ell$-monoholed graph for some $\ell \geq 5$. Suppose $C$ is a hole in $G$. Then for every $v \in V(G) \setminus V(C)$, either $v$ is complete to $V(C)$, $v$ is anticomplete to $V(C)$ or there is a path $P$ of $C$ of length at most such that $N(v) \cap V(C) = V(P)$.
	\end{fact}
	\begin{proof}
		Suppose some $v \in V(G) \setminus V(C)$ has both a neighbor and a non-neighbor in $V(C)$. Let $P$ be a path of $C$ containing all neighbors of $v$ in $V(C)$ and choose $P$ to be minimal. We may assume $P$ has length at least two. Then $V(C) \cup \{ v\} \setminus P^*$ induces a hole of length $|E(C)| - |E(P)| + 2$. Since $G$ is $\ell$-monoholed and $\ell \geq 5$, it follows that $P$ has length two. Since $\ell \geq 5$, $V(P) \cup \{ v\}$ does not induce a hole so $v$ is adjacent to the interior vertex of $P$.
	\end{proof}

	\section{Introducing the structure of $\ell$-monoholed graphs}
	A well-known class of bipartite graphs called half-graphs comes up frequently in our analysis. This class was first named by Erd\H{o}s and Hajnal in \cite{EH-halfgraph}.
	For an integer $n \geq 1$ we say $H_n$ is the bipartite graph on with vertices $\{x_1, \dots, x_n\} \cup \{y_1, \dots , y_n\}$ and edge set $\{x_iy_j \ | \ i,j \in [k], i \geq j\}$. (See Figure \ref{fig:H_n}).  We call a graph $G$ a {\em half-graph} if $G$ is contained in $H_n$ for some $n$. (Note that in \cite{EH-halfgraph}, half-graphs only referred to the set of graphs consisting of $H_n$ for every integer $n \geq 1$.)
	It follows from the definition of half-graph, that a graph $G$ is a half-graph if and only if it contains no induced two edge matching. 
	\begin{figure}[!h]
		\centering
		\resizebox{.6\textwidth}{!}{
			\input{Tikz/halfgraph.tex}
		}
		\caption{An illustration of $H_{10}$}
		\label{fig:H_n}
	\end{figure}
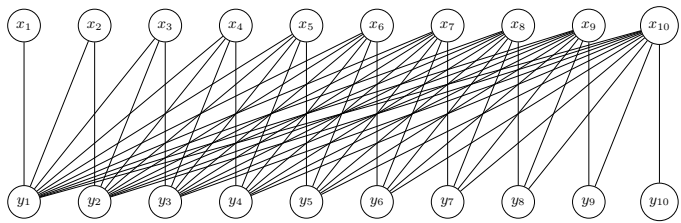
	
	Let $X, Y, Z$ be disjoint sets of vertices and suppose there is a half-graph $H_1$ between $X$ and $Y$ and a half-graph $H_2$ between $Y$ and $Z$. 
	Let $H$ denote the graph with vertex set $X \cup Y \cup Z$ and edge set $E(H_1 \cup H_2)$.
	We say $H_1$ and $H_2$ are \textit{compatible with respect to $X$} if for every $x, x' \in X$, $N_H(x) \subseteq N_H(x')$ or $N_H(x') \subseteq N_H(x)$. In other words, $H_2$ and $H_1$ are compatible with respect to $X$ if and only if $H$ is a half-graph.
	
	The following class of graphs is important for our analysis. We call $H$ a \textit{threshold graph} if the vertex set of $H$ can be partitioned into a stable set $S$ and the vertex set of a clique $K$ and the edges between $S$ and $K$ form a half-graph. 
	Threshold graphs were first introduced by Chvátal and Hammer in \cite{chvatalthresholdgraphs}. For further background on threshold graphs see Chapter 10 of Golumbic's \textit{Algorithmic Graph Theory and Perfect Graphs} \cite{golumbic2004algorithmic} or Mahadev and Peled's book on the subject \cite{mahadev1995threshold}.
	We will need a theorem of \cite{chvatalthresholdgraphs}:
	\begin{thm} [Chvátal and Hammer, 1973]
		A graph $H$ is a threshold graph if and only if it contains no $P_4, C_4,$ or two-edge matching.
	\end{thm} 
	\begin{proof}
		Let $K$ be a maximal clique in $H$. Let $S$ denote $V(H) \setminus V(K)$. Suppose some $s_1, s_2 \in S$ are adjacent. Since $K$ is a maximal clique there are distinct vertices $k_1, k_2 \in V(K)$ such that $k_1$ is not adjacent to $s_1$ and $k_2$ is not adjacent to $s_2$. Since $s_1s_2$ and $k_1k_2$ are not an induced two edge matching, we may assume $s_1$ is adjacent to $k_2$. If $k_1$ and $k_2$ are adjacent $s_1 \dd s_2 \dd k_1 \dd k_2 \dd s_1$ is an induced $C_4$, a contradiction. So $k_1$ is not adjacent to $k_2$. But then, $k_1 \dd k_2 \dd s_1 \dd s_2$ is an induced $P_4$, a contradiction. 
		
		Thus, $S$ is a stable set. Let $s_1, s_2 \in S$ and suppose there is some $k_1 \in N (s_1) \setminus N(s_2)$. Then $N(s_2) \subseteq N(s_1)$ for if there is a $k_2 \in N(s_2) \setminus N(s_1)$, then $s_1 \dd k_1 \dd k_2 \dd s_2$ is an induced $P_4$, a contradiction.
		Hence there is a half graph between $V(K)$ and $S$.
	\end{proof}

	\subsection{Inflated graphs}
	We use an object called an ``inflated graph'' throughout our analysis. We call $\h$ an inflated graph if $\h$ is obtained from a graph $G$ by replacing every $v \in V(G)$  by a non-empty clique $K_v$ and for every edge $xy \in E(G)$ we add a connected half-graph between $K_x$ and $K_y$ so that for every $v \in V(G)$ and $x,y \in N_G(v)$ the half graphs between $K_x, K_v$ and $K_y, K_v$ are compatible with respect to $K_v$.
	Informally, inflated graphs can be thought of as a generalization of rings.\footnote{Recall in this thesis, rings are a class of graphs. See Section \ref{subsec:structuralresults}}
	
	We say $V(K_v)$ for $v \in V(G)$ are the \textit{bags} of $\h$.
	We say the bags $V(K_x), V(K_y)$ are \textit{adjacent bags} or \textit{neighboring bags} in $\h$ if and only if $xy \in E(G)$. 	We denote the set of neighboring bags of a bag $B$ as $\mathcal{N}(B)$.
	If $\mathcal{X}$ is a set of bags of an inflated graph $\h$ we denote the union of all bags in $\mathcal{X}$ as $V(\mathcal{\h})$.
	
	We call an inflated graph $\mathcal{M}$ a \textit{sub-inflated graph} if there is some injective function $f$ from the bags of $\mathcal{M}$ to the bags of $\mathcal{H}$ such that for any bag $B$ of $M$, $B \subseteq f(B)$ and for any two bags $M_1, M_2$ of $\mathcal{M}$, $f$ satisfies the property that: $M_1$ and $M_2$ are adjacent bags in $\mathcal{M}$ if and only if $f^{-1}(M)$ and $f^{-1}(M)$ are adjacent bags in $\h$.
	We say $\mathcal{M}$ is an \textit{underlying inflated graph} of $\h$ if $\mathcal{M}$ is a sub-inflated graph of $H$ and $f$ is a bijection. We say $M$ is an \textit{underlying graph} of $\h$ if $M$ is an underlying inflated graph of $\h$ and every bag of $M$ has size exactly one. For any $v \in V(M)$ we say the bag $f(B)$ and $v$ \textit{correspond} to each other.	
	It follows from the definition that $G$ is an underlying graph of $\h$.
	
	Let $a, b \in V(\h)$. Then there exist bags $A, B$ of $\h$ such that $a \in A$ and $b \in B$. Let $v_a, v_b$ denote the vertices in $G$ corresponding to $A$ and $B$, respectively. Then we say the \textit{$\h$-underlying distance} between $a$ and $b$ is $d_G(v_a, v_b)$.
	
	We call $\mathcal{H}$ a \textit{inflated path} or \textit{inflated cycle} if the graph underlying $\mathcal{H}$ is a path or a cycle, respectively.
	In this case, we say the \textit{length} of $H$ is the length of its underlying graph. We call an inflated cycle of length $\ell$ an \textit{inflated $\ell$-hole of an inflated $C_\ell$.}
	For an inflated path $\mathcal{P}$ we call the union set of all bags corresponding to interior vertices of its underlying graph \textit{interior bags} of $\mathcal{P}$ and denote their union as $\mathcal{P}^*$. We call the bags corresponding to ends of the underlying graph of $\mathcal{P}$ \textit{end bags} of $\mathcal{P}$.
	Note a ring on $n$ sets is an inflated $C_n$ for any integer $n \geq 3$.
	
	\begin{lemma}
		Let $\h$ be an inflated graph. Then the following statements all hold:
		\begin{enumerate}[(a)]
			\item \label{fgf:complete} For every bag $B$ of $\h$ there is a $v \in V(B)$ that is complete to $V(\mathcal{N}(B))$
			\item \label{fgf:weak}For any two vertices $u, v \in V(H)$ if $u$ and $v$ are contained in non-adjacent bags of $\h$ then there is an underlying graph of $\h$ containing both $u$ and $v$. 
			\item \label{fgf:strong} Let $\mathcal{S}$ be an sub-inflated graph of $\h$ and suppose that every bag of $S$ has size one. Then there is an underlying graph $G$ of $\h$ such that $\mathcal{S}$ is contained in $G$ as an induced subgraph.
		\end{enumerate}
				\label{lem:inflated-graphs-facts}
	\end{lemma}
	\begin{proof}
		Let $B$ be a bag of $\h$ and let $v$ be the vertex in $B$ of maximum degree.
		Then by definition of compatible half graphs, $N(v) \supseteq N(w)$ for every $w \in B$.
		Let $D$ be a neighboring bag of $B$.
		By definition of connected half graph, some $v$ is complete to $D$.
		This proves \ref{fgf:complete}).
		
		Suppose $\mathcal{S}$ is a sub-inflated graph of $\h$. By \ref{fgf:complete} for each bag $B$ of $\h$ there is a vertex $v_B$ that is complete to all neighboring bags of $B$ in $\h$. Let $\mathcal{B}$ be the set of bags of $\h$ that do not contain any vertices of $\mathcal{S}$. By \ref{fgf:complete} every $B \in \mathcal{B}$ contains a vertex $v_B$ complete to all neighboring bags of $B$. Then by definition of sub-inflated graph $V(S) \cup \{ v_B \ | \ B \in \mathcal{B} \}$ induces an underlying graph of $\h$. This proves \ref{fgf:strong}.
		
		This completes the proof because \ref{fgf:strong} implies \ref{fgf:weak} since the graph induced by any two vertices in non-adjacent bags of $\h$ is a sub-inflated graph of $\h$ with bags of size one.
	\end{proof}
	
	We will now state the main result of this chapter.

	\begin{theorem}\label{thm:khole-main-result}
		
		Let $G$ be an $\ell$-monoholed graph for some $\ell \geq 7$  one of the following conditions holds:
		\begin{enumerate}[(a)]
			\item $G$ contains a vertex that is adjacent to every other vertex in $V(G)$. 
			\item$G$ contains a clique cutset, 
			\item $G$ is chordal,
			\item $G$ is an inflated $\ell$-hole, 
			\item $G$ is a type of inflated graph we call a ``crowned $k$-corpus''.
		\end{enumerate}
	\end{theorem}
	
	Note, the inflated graphs in case (e) are not yet defined. As the definition is somewhat technical we will describe them later in this chapter. 
	
	\section{Spines and Spiders}
	In this section we discuss a special kind of graph called a ``mated $k$-spider'' with a ``nice'' structure. We will prove that for any $\ell$-monoholed graph $G$, $G$ contains a $k$-mated spider for some $k \geq 3$ or $G$ satisfies one of conditions (a), (b), (c) or (d) of Theorem \ref{thm:khole-main-result}. We will fully characterize the structure of mated spiders in Section \ref{sec:mated-spiders-structure}. In future sections we go on to fully characterize $\ell$-monoholed graphs by choosing a mated $k$-spider $S$ in $G$ with $k$ maximum and seeing how $G \setminus S$ attaches to $S$ when $G$ does not satisfy conditions (a), (b), (c) or (d) of Theorem \ref{thm:khole-main-result}.
	
	For $k \geq 3$, we call a graph $S$ a \textit{$k$-spider} if it vertices of degree one are $t_1, t_2, \dots, t_k$ called its \textit{toes} and it is minimally connected under deleting vertices with these toes.
	For $i,j \in [k]$, let $d_S(i,j)$ denote $d_S(t_i, t_j)$. When the choice of spider $S$ is not ambiguous we will simply write $d(i,j)$ for $d_S(i,j)$.
	We call a path $P$ of $S$ an \textit{leg} if one end of $P$ is a toe and $P$ is a maximal path satisfying that all internal vertices of $P$ have degree two. It follows that for each $i \in [k]$ there is a unique leg $L_i$ with one end equal to $t_i$ and $L_i$ has length at least one. For each $i \in [k]$, we refer to $L_i$ as the \textit{$t_i$-leg} of $S$. For each $i \in [k]$, let $a_i$ denote the end of $L_i$ that is not equal to $t_i$. We refer to $a_1, a_2, \dots, a_k$ at the \textit{joints} and for each $i \in [k]$ we call $a_i$ the \textit{$t_i$-joint} of $S$. Let $A$ be the graph obtained from $S$ by deleting $V(L_i \setminus a_i)$ for every $i \in [k]$. By definition, for all $i,j \in [k]$, if $i \neq j$ then $V(L_i \setminus a_i)$ is anticomplete to $V(L_j \setminus a_j)$.  It follows that $A$ is connected. We call $A$ the \textit{body} of $S$. 
	\begin{figure}[!h]
		\centering
		\input{Tikz/fig_spider.tex}
		\caption{An example of  a $5$-spider. The vertices of the body are drawn in red. Note $a_1 =a_2$ and $a_4 = a_5$. }
		\label{fig:spider}
	\end{figure}
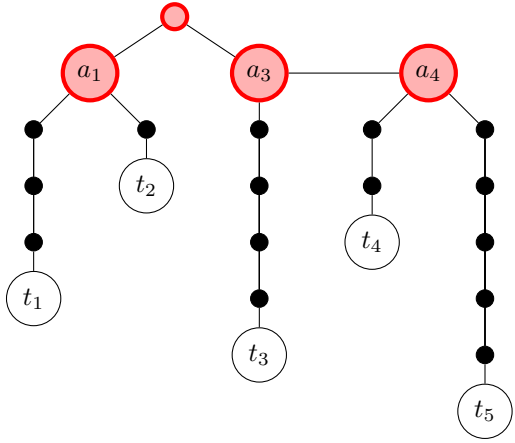
	(See Figure \ref{fig:spider}.)
	We call two $k$-spiders $S, S'$ \textit{mates} if they have the same set of toes, their vertex sets are anticomplete except for their toes and for every $i,j \in [k]$ with $i \neq j$, $d_S(i,j) + d_{S'}(i,j) = \ell$.
	We call a spider \textit{mated} if it has a mate.
	We need the following lemma:
	\begin{lemma} \label{lem:inflated-hole-vertex-neighbors-facts}
		Let $G$ be a $\ell$-monoholed graph. Suppose $G$ contains an inflated $C_\ell$ $\h$. Let $w \in V(G) \setminus V(\h)$ and suppose $w$ has a neighbor in $V(\h)$. Then either,  $w$ is complete to $V(C)$ or the following conditions all hold:
		\begin{enumerate}[(a)]
			\item \label{fhn:a}There are some three consecutive bags of $B_1, B_2, B_3$  of $\h$ such that the neighbors of $v$ in $V(\h)$ are contained in $V(B_1 \cup B_2 \cup B_3)$, 
			\item \label{fhn:middle} If $w$ has neighbors in both $V(B_1)$ and $V(B_3)$, $w$ is complete to $V(B_2)$,  
			\item \label{fhn:b} The graphs between $B_1, B_2 \cup \{ w\}$ and $B_3, B_2 \cup \{v\}$ are half graphs and they are compatible with respect to $B_2 \cup \{ w\}$, and
			\item \label{fhn:cutsetset} $G[ V(\h) \cup \{ w \}]$ has a clique cut-set or $G[ V(\h) \cup \{ w \}]$ induces an inflated $C_\ell$ whose bags can be obtained from the bags of $\h$ by adding $w$ to $B_2$.
		\end{enumerate}
	\end{lemma}
	\begin{proof}
		Suppose $w \in V(G) \setminus V(H)$ has both a neighbor and a non-neighbor in $V(\h)$.
		\stmt{\label{fhn:one}If $w$ has a non-neighbor in a bag $X$ of $\h$, then $w$ cannot have neighbors in both adjacent bags of $X$ in $\h$.}
		Let $Y, Z$ denote the neighboring bags of $X$ in $\h$.
		Suppose $w$ has a neighbor $y \in Y$ and  neighbor $z \in Z$. By Lemma \ref{lem:inflated-graphs-facts}, there is some $y' \in Y$ and some $z' \in Z$ is complete to $V(X)$.
		Hence, $G[\sset{y,y', x, z, z', w}]$ includes a hole of length at most 6, a contradiction. This proves (\ref{fhn:one}).
		\\
		\\
		It follows that $w$ does not have a neighbor in every bag of $\h$.
		Let $\q$ be a sub-inflated graph of $\h$ containing all neighbors of $w$ in $V(\h)$. Choose $\q$ to contain a minimum number of bags. Then $\q$ is an inflated-path.
		
		Suppose $\q$ has length greater than three. Since $\q$ is minimal, $w$ has a neighbor $u, v$ in the end bags of $\h$. Then by Lemma \ref{lem:inflated-graphs-facts}, there is an underlying graph $C$ of $\h$ with $u, v \in V(C)$. It follows from the definition of $\q$ that $C \setminus \q^* \cup \{ w\}$ is a hole of length $\ell -m +2$ where $m$ denotes the length of $\q$.  Hence,  \ref{fhn:a} holds and \ref{fhn:middle} hold by applying (\ref{fhn:one}).
		
		Let $B_1, B_2, B_3$ be the bags of $\q$ in order. Suppose there is some $b_1 \in V(B_1)$, $b_2 \in V(B_2)$ and $b_3 \in V(B_3)$ such that $b_1$ is adjacent to $b_2$, $w$ is adjacent to $b_3$, $w$ is not adjacent to $b_2$ and $b_1$ is not adjacent to $b_3$. By Lemma \ref{lem:inflated-graphs-facts} there is an underlying hole $C$ of $\h$ with $b_1, b_3 \in V(\h)$. Let $P$ be the path from $C$ obtained by deleting the vertex corresponding to $B_2$ in $C$. Then $V(C) \cup \{ b_2, w\}$ induces a hole of length $\ell$ + 1. Hence, \ref{fhn:b} holds.
		
		Suppose $G[ V(\h) \cup \{ w \}$ has no clique cutset. By definition of inflated graph, we need only show that $w$ has a neighbor in both $V(B_1)$ and $V(B_3)$ to prove \ref{fhn:cutsetset}. Suppose the neighbors of $w$ in $\h$ are contained in $V(B_1 \cup B_2)$. Then since the half-graphs between $B_1, B_2$ and $B_3, B_2$ are compatible with respect to $B_2$, if $b_1 \in V(B_1)$ is adjacent to $w$ then  $b_1$ is complete to $V(B_2)$. Hence, $N(w) \cap V(\h)$ is the vertex set of a clique, a contradiction. Thus, \ref{fhn:cutsetset} holds.
	\end{proof}

	\begin{theorem}\label{thm:G-has-a-mated-spider}
		Let $G$ be a $\ell$-monoholed graph for some $\ell \geq 6$. Then one of the following holds:
		\begin{enumerate}[(a)]
			\item $G$ contains a vertex that is adjacent to every other vertex in $V(G)$. \label{ms:complete-vertex}
			\item $G$ contains a clique cutset, \label{ms:clique-cutset}
			\item $G$ is chordal, \label{ms:chordal}
			\item $G$ is an inflated $\ell$-hole,  \label{ms:hole}
			\item $G$ contains a pair of mated $k$-spiders for some $k \geq 3$. \label{ms:mated-spider}
		\end{enumerate}
	\end{theorem}
	\begin{proof}
		We may assume that \ref{ms:complete-vertex}, \ref{ms:clique-cutset}, \ref{ms:chordal}, and \ref{ms:hole} do not hold.
		$G$ contains an inflated $\ell$-hole $\C$ because it is not chordal. Choose $\C$ to maximize $|V(\C)|$. Let $W$ be the set of vertices in $V(G) \setminus V(C)$ that are complete to $V(\C)$. Since $\C$ is not a clique and G is $C_4$-free, $G[W]$ must be a clique. $V(G) \neq V(\C) \cup W$ because \ref{ms:complete-vertex} and \ref{ms:hole} do not hold. Since $G$ does not contain a clique cut-set there is some connected graph $X$ contained in $G \setminus (V(\C) \cup W)$ such that $V(X)$ has two non-adjacent neighbors in $V(\C) \cup W$.
		Since $W$ is complete to $V(\C)$ it follows that $V(X)$ must have two non-adjacent neighbors in $V(\C)$. Choose $X$ to be minimal.
		Then, $X$ is a path. 
		$G[ V(\C \cup X)]$ does not contain a clique cutset. So by Lemma \ref{lem:inflated-hole-vertex-neighbors-facts} if $X$ is a single vertex $x$, then $x$ can be added to a bag of $\C$ to obtain a larger inflated hole, contradicting the maximality of $\C$. Thus $|X| \geq 2$. Let the vertices of $X$ be $x_1 \dd x_2 \dd \dots \dd x_n$ in order.

		Then by minimality of $X$, there are some two non-adjacent vertices $v,w \in V(\C)$ such that $x_1$ is adjacent to $v$ and $x_n$ is adjacent to $w$. Then $v$ and $w$ can't be in the same bag of $\C$.
		By Lemma \ref{lem:inflated-graphs-facts} there is a hole $C$ underlying $\C$ with $x, y \in V(C)$. 
		 Let $P_1, P_2$ be paths of $C$ containing all neighbors of $x_1$ and $x_n$ in $V(C)$,  respectively. Choose $P_1, P_2$ to be minimal. By minimality of $X$, we may assume $P_1, P_2$ each have length at most one.  For $i \in \{1,2\}$, let the ends of $P_i$ be $a_i, b_i$. We may assume $a_1, b_1, b_2, a_2$ occur in order in $C$. Let $A$ be the path of $C$ with ends $a_1, a_2$ that does not contain $b_1$ or $b_2$ and let $B$ be the path of $C'$ with ends $b_1, b_2$ that does not contain $a_1$ or $a_2$. 
		\stmt{If some vertex $x_i \in X^*$ has a neighbor in $z \in V(C)\setminus V(B)$ then $z \in V(A)$ and $A$ is a path of length two, $a_1 = b_1$ and $a_2 = b_2$. The same statement holds with $A$ and $B$ exchanged. \label{inflatedholes:A-is-short-if-xi-has-a-nbr-in-it}}
		Suppose $x_i$ has a neighbor in $z\in V(C)\setminus V(B)$ for some $i \in [2, n-1]$.  By minimality of $X$, $z$ is adjacent to all of $a_1, a_2, b_1, b_2$. Hence, $a_1 = b_1$, $a_2 = b_2$. Then $a_1$ and $a_2$ are not adjacent, so $A$ is the path $a_1 \dd z \dd a_2$. (See Figure \ref{fig:G-has-a-mated-spider} for an illustration). This proves (\ref{inflatedholes:A-is-short-if-xi-has-a-nbr-in-it}).
		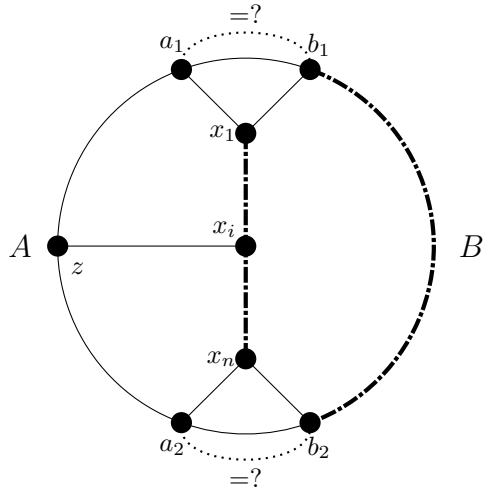
\begin{figure}[!h]
			\centering
			\input{Tikz/path_with_edges_to_hole.tex}
			\caption{An illustration of the proof of Statement \ref{inflatedholes:A-is-short-if-xi-has-a-nbr-in-it} of
				Theorem \ref{thm:G-has-a-mated-spider}. Thick dashed lines indicate a path of length greater than zero. The vertices $a_1$ and $b_1$ are connected by dots and an $= ?$ to indicate that $a_1$ may be equal to $b_1$. }
			\label{fig:G-has-a-mated-spider}
		\end{figure}
		\stmt{No vertex in $X^*$ has a neighbor in $V(C)$. \label{inflatedholes:X-interior-is-anticomplete}}
		Suppose $x_i$ has a neighbor in $z\in V(C)$ for some $i \in [2, n-1]$. By (\ref{inflatedholes:A-is-short-if-xi-has-a-nbr-in-it}) we may assume $z \in V(A)$. Then by (\ref{inflatedholes:A-is-short-if-xi-has-a-nbr-in-it}), $A$ is the path $a_1 \dd z \dd a_2$. Since $C$ has length at least $7$, (\ref{inflatedholes:A-is-short-if-xi-has-a-nbr-in-it}) implies that no vertex in $V(X)\setminus \{x_1, x_n\}$ has a neighbor in $V(C) \setminus V(A)$. It follows that $V(C \cup X)\setminus \{z \}$ induces a hole. Hence $|E(X)| +2 = |E(A)| = 2$. But $X$ is a path of length at least one, a contradiction. This proves (\ref{inflatedholes:X-interior-is-anticomplete}).
		\\
		\\
		Let $M$ denote the graph $G[V(C \cup X)]$.
		It follows that $M$ is a prism, pyramid, or theta depending on the lengths of $P_1, P_2$. 
		Moreover since $G$ is $\ell$-monoholed the constituent paths of $M$ will have lengths in $\{\frac{\ell}{2}-1, \frac{\ell-1}{2}, \frac{\ell}{2}\}$. Since $\ell \geq 7$ every constituent path of $M$ has length at least three. Let $T$ be a set consisting of one vertex from the interior of each constituent path of $M$. Then $M$ is the union of a pair of mated spiders with toes $T$.
		
	\end{proof}

	\section{The structure of mated $k$-spiders} \label{sec:mated-spiders-structure}
	In this section we will fully characterize the structure of graphs consisting of the union of a pair of  mated $k$-spiders in an $\ell$-monoholed graph. In particular, we will prove that the union of mated $k$-spiders is a $k$-theta or a ``generalized $k$-prism'' when $\ell$ is even and it is a $k$-pyramid when $\ell$ is odd.
	
	Let us begin with the definition of a generalized $k$-prism:
	Let $H$ be a graph such that $V(H)$ is the the union of the vertex sets of paths $P_1, P_2, \dots, P_k$ for some $k \geq 3$.  For each $i \in [k]$ let $a_i$ and $b_i$ denote the ends of $P_i$. Let $\ell = 2n + 2$. We call $H$ an \textit{generalized $k$-prism} if both of the following conditions hold:
	\begin{itemize}
		\item For every two distinct  $i,j \in [k]$, $V(P_i)$ is anticomplete to $V(P_j)$ except possibly $a_i$ is equal or adjacent to $a_j$ or $b_i$ is equal or adjacent to $b_j$.
		\item $[k]$ can be partitioned into (possibly empty) sets $Q, R, S, T$ satisfying all of the following:
		\begin{itemize}
			\item The graphs $G[\{a_i \ | \ i \in [k] \}]$ and $G[\{b_i \ | \ i \in [k] \}]$ both do not contain cut-vertices. 
			\item For any two distinct $i,j \in [k]$, $a_i = a_j$ if and only if $i$ and $j$ are both in $T$ and $b_i = b_j$ if and only if $i$ and $j$ are both in $S$, 
			\item For every $i \in Q$, $P_i$ has length $n-1$,
			\item For every $i \in R \cup S \cup T$, $P_i$ has length $n$,
			\item The sets $\{a_i \ | \ i \in Q \cup S \}$ and $\{b_i \ | \ i \in Q \cup T \}$ are stable, 
			\item The sets $\{a_i \ | \ i \in R \}$ and $\{b_i \ | \ i \in R \}$ induce cliques,
			\item For any $i \in S$, $j \in R$ and $k \in T$, $a_i$ is adjacent to $a_j$ and and $b_j$ is adjacent to $b_k$,
			\item If $t \in T$, $a_	t$ is complete to every other vertex in $\{a_1,a_2, \dots, a_k \}$ and if $s \in S$, $b_s$ is complete to every other vertex in $\{b_1, b_2, \dots, b_k \}$,
			\item There is a half graph between $\{a_i \ | \ i \in Q \}$ and $\{a_j \ | \ j \in R \}$ and a half graph between $\{b_i \ | \ i \in Q \}$ and $\{b_j  \ | \ j \in R \}$, and for any $i \in Q$ and $j \in R$, $a_i$ is adjacent to $a_j$ if and only if $b_i$ is not adjacent to $b_j$.
		\end{itemize}
	\end{itemize}
	See Figure \ref{fig:orderly-spine} for an illustration.
	We call $Q, R, S, T$ a \textit{defining partition} of $H$. Note $H$ may have multiple defining partitions. In general we work with defining partitions that maximize the cardinality of $R$ and thus $S$ is non-empty if and only if $|S| \geq 2$ and $T$ is non-empty if and only if $|T| \geq 2$.
	We call the paths $P_1, P_2, \dots, P_k$ the \textit{constituent paths} of $H$ and we call the sets $\{ a_1, a_2, \dots, a_k\}$ and $\{b_1, b_2, \dots, b_k\}$ the \textit{terminating sets} of $H$.
	Note that by definition $k$-prisms with constituent paths of length $\frac{\ell}{2} -1$ are generalized $k$-prism, but $k$-thetas and $k$-pyramids are not generalized $k$-prisms.
	
	Let us define an odd analog of a generalized $k$-prism.
	Let $H$ be a graph such that $V(H)$ is the the union of the vertex sets of paths $P_1, P_2, \dots, P_k$ for some $k \geq 3$.  For each $i \in [k]$ let $a_i$ and $b_i$ denote the ends of $P_i$. Let $\ell = 2n +1$. We call $H$ an \textit{generalized $k$-pyramid} if both of the following conditions hold:
	\begin{itemize}
		\item For every two distinct  $i,j \in [k]$, $V(P_i)$ is anticomplete to $V(P_j)$ except possibly $a_i$ is equal or adjacent to $a_j$ or $b_i$ is equal or adjacent to $b_j$,
		\item Each of $H[\{a_1, a_2, \dots, a_k\}]$ is either a 2-connected threshold graph or $a_i = a_j$ for every $i, j  \in [k]$. The same statement holds for $H\{b_1, b_2, \dots, b_k\}]$,
		\item For every $i \in [k]$, $P_i$ has length $n$ or $n-1$,
		\item And, in particular, there is a partition $Q, R$ of $[k]$ satisfying all of the following:
		\begin{itemize}
			\item $a_i = a_j$ and $b_i$ is adjacent to $b_j$ for any two distinct $i,j \in Q$.
			\item For any two distinct $i,j \in R$, $a_i \neq a_j$ and $b_i \neq b_j$.
			\item For any two distinct $i, j \in R$, $a_i$ is adjacent to $a_j$ if and only if $a_i'$ is not adjacent to $b_j'$.
			\item $\{ a_i \ | \ i \in Q\}$ and $\{ a_i \ | \ i \in R\}$ are complete to each other and $\{ b_i \ | \ i \in Q\}$ and $\{ b_i \ | \ i \in R\}$.
			\item If $i \in Q$, then $P_i$ has length $n$.
			\item If $i \in R$, then $P_i$ has length $n-1$.
		\end{itemize}
	\end{itemize}
	
	We call a graph $H$ a \textit{$k$-spine} if it is a generalized $k$-pyramid, $k$-theta, or generalized $k$-prism.
	Recall, that we say the terminating sets of a pyramid is the vertex set of its base and the set consisting of its apex. Let $H$ be a theta and let $u, v$ be the two vertices of degree at least three in $H$. 
	Then, the terminating sets of $H$ are $\{ u\}$ and $\{ v\}$.
	We are now ready to state the main result of this chapter.

	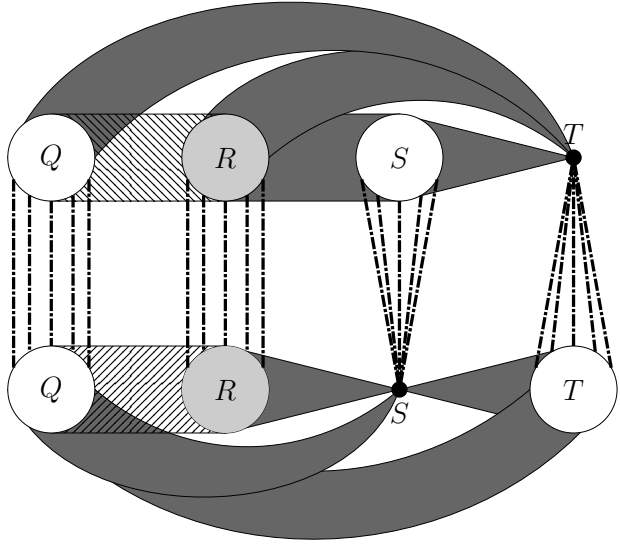
\begin{figure}[!h]
		\centering
		\resizebox{.55\textwidth}{!}{
			\input{Tikz/standardconstruction.tex}
		}
		\caption{ An illustration of a generalized $k$-pyramid. The large white circles represent stable sets and the large gray circles represent cliques. We add huge dark gray lines to indicate that sets are complete to each other. We use stripes between $Q$ and $R$ to illustrate that there is a half graph between them. We draw the stripes in opposite directions on top and bottom to indicate that the half graphs are complementary. Note some of the sets may be empty and the vertices drawn for $S$ and $T$ need not exist.}
		\label{fig:orderly-spine}
	\end{figure}

	\begin{thm}\label{thm:orderly-spines-exist}
		Let $G$ be an $\ell$-monoholed graph for some $\ell \geq 7$. Then one of the following conditions holds:
		\begin{enumerate}[(a)]
			\item $G$ contains a vertex that is adjacent to every other vertex in $V(G)$. 
			\item$G$ contains a clique cutset, 
			\item $G$ is chordal,
			\item $G$ is an inflated $\ell$-hole, 
			\item $G$ contains a $k$-spine for some $k \geq 3$ and in particular if $\ell$ is odd $G$ contains a $k$-pyramid.
		\end{enumerate}
	\end{thm}

	\subsection{Proving Theorem \ref{thm:orderly-spines-exist}}
	For the rest of this chapter $G$ will be an $\ell$-monoholed graph. 
	We will now prove a series of lemmas about the structure of spiders and pairs of mated spiders in $G$ in order to prove Theorem \ref{thm:orderly-spines-exist}.
	
	\begin{lem}\label{lem:A-not-2-connecteded-4-spider-outcomes}
		Let $S$ be a $4$-spider in $G$ and let $A$ be the body of $S$. If $A$ has a cut-vertex then two of $d(1,2) + d(3,4)$, $d(1,3) + d(2,4)$ and $d(1,4) + d(2,3)$ are equal and at least two more than the third.
	\end{lem}
	\begin{proof}
		Let $a_1, a_2, a_3, a_4$ be as in the definition of $k$-spider.
		Let $v$ be a cut vertex of $A$ and let $A', A''$ be two different components of $A \setminus \{ v \}$. Let $B', B''$ denote $S[V(A') \cup \{v \}]$ and $S[V(A'') \cup \{v \}]$, respectively. 
		Suppose none of $a_2, a_3, a_4 \in V(A')$. Then it follows from the minimality of $H$ that $B'$ is a $a_1v$-path, contradicting the definition of a leg. It follows that two of $a_1, a_2, a_3, a_4$ are in each of $V(A')$ and $V(A'')$. Hence, $v$ does not equal $a_1, a_2, a_3, a_4$. We may assume $a_1, a_2 \in V(A')$ and $a_3, a_4 \in V(A'')$.  Thus, $d(1,3) + d(2,4) = d(1,4) + d(2,3)$.
		Since $B'$ is connected it contains a shortest $a_1v$-path $Q_1$ and a shortest $a_2v$-path $Q_2$. Since $S$ is minimal, $V(B') = V(Q_1 \cup Q_2)$. 
		Moreover, since $A'$ is connected some vertex in $V(Q_1)$ is equal or adjacent to some vertex of $V(Q_2)$. Thus, $d(1,2) < d_S(a_1, v) + d_S(a_2, v)$. By symmetry, $d(3,4) < d_S(a_3, v) + d_S(a_4, v)$. Hence, $d(1,2) + d(3,4) -2 \leq  d(1,3) + d(2,4) = d(1,4) + d(2,3)$.
	\end{proof}
	
	We make extensive use of the following easy observation: 
	\begin{fact}\label{lem:leg-fact}
		Let $H$ be a $k$-spider for some $k \geq 3$. Let $A$ be the body of $H$. Let $a_1, \dots, a_k$ be as in the definition of $k$-spider. Then for each $i \in [k]$ either $a_i$ has degree at least two in $A$ or there is some $a_i = a_j$ for some $j \in [k] \setminus \{i \}$. 
	\end{fact}
	\begin{proof}
		The fact follows immediately from the definition of leg. 
	\end{proof}
	
	\begin{lem} \label{lem:A-2-connecteded-4-spider-outcomes}
		Let $S$ be a 4-spider and let $A$ be the body of $S$.  Suppose $A$ contains no cut-vertex, then $V(A)=\{a_1, a_2, a_3, a_4\}$, and either:
		\begin{enumerate}[(a)]
			\item $A$ consists of a single vertex, \label{A-is-a-vertex}
			\item $A$ consists of a single edge,  \label{A-is-an-edge}
			\item $A$ is a $K_3$, \label{A-is-a-k3}
			\item $A$ is a $K_4$, or \label{A-is-k4}
			\item $A$ is a diamond. \label{A-is-diamond}
		\end{enumerate}
	\end{lem}
	\begin{proof}
		Suppose $A$ has no cut-vertex. Then by minimality of $A$, $V(A)=\{a_1, a_2, a_3, a_4\}$. If $|V(A)| \leq 2$, \ref{A-is-a-vertex} or \ref{A-is-an-edge} hold trivially. Suppose $|V(A)| = 3$. Without loss of generality, $a_1 = a_2$.  By Fact \ref{lem:leg-fact}, $a_3$ and $a_4$ have degree at least two in $A$. Thus, $A$ is a $K_3$ and \ref{A-is-a-k3} holds. Suppose all of $a_1, a_2, a_3, a_4$ are distinct. Then since $A$ is 2-connected, we may assume $a_1\dd a_2 \dd a_3 \dd a_4 \dd a_1$ is a cycle. Since $G$ contains no hole of length four, \ref{A-is-k4} or \ref{A-is-diamond} holds.
	\end{proof}

	\begin{lem} \label{lem:mated-spider-top-shapes}
		Let $S$ be a mated $4$-spider with the notation as in the definition of spider. Then $A$ has no cut vertex, $V(A) = \{a_1, a_2, a_3, a_4\}$ and one of the following holds:
		
		\begin{enumerate}
			\item [\ref{A-is-a-vertex}]$A$ consists of a single vertex,
			\item [\ref{A-is-a-k3}] $A$ is a $K_3$,
			\item [\ref{A-is-k4}] $A$ is a $K_4$, or
			\item [\ref{A-is-diamond}]$A$ is a diamond. 
		\end{enumerate}
		
	\end{lem}
	
	\begin{proof}
		Let $S'$ be a mate of $S$. Then by Lemma  \ref{lem:A-not-2-connecteded-4-spider-outcomes} and Lemma \ref{lem:A-2-connecteded-4-spider-outcomes}, two of $d_{S'}(1,2) + d_{S'}(3,4)$, $d_{S'}(1,3) + d_{S'}(2,4)$ and $d_{S'}(1,4) + d_{S'}(2,3)$ are equal and the third is at most one more than the other two.
		Thus, two of $d_{S}(1,2) + d_{S}(3,4)$, $d_{S}(1,3) + d_{S}(2,4)$ and $d_{S}(1,4) + d_{S}(2,3)$ are equal and the third is at most one less than the other two. Hence $A$ cannot be an edge and by Lemma \ref{lem:A-not-2-connecteded-4-spider-outcomes} $A$ cannot contain a cut-vertex. The result now follows from applying Lemma \ref{lem:A-2-connecteded-4-spider-outcomes} to $S$.
	\end{proof}
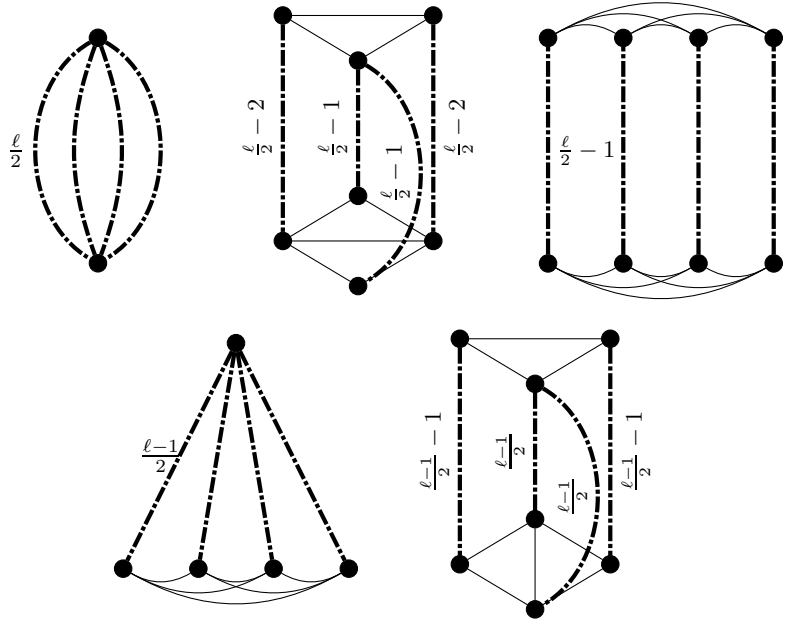
\begin{figure}[!h]
	\centering
	\input{"fig_mated_four_spiders.tex"}
	\caption{An illustration of Lemma \ref{lem:mated4spider-structure}. The image shows the five possible structures of mated $4$-spiders in an $\ell$-monoholed graph for $\ell \geq 7$. The top row depicts the cases when $\ell$ is even and the bottom row depicts the cases when $\ell$ is odd. The numbers next to the paths indicate the path lengths.}
	\label{fig:mated4spiders}
\end{figure}
	
	Mated four structures have a well defined structure (see Figure \ref{fig:mated4spiders} for an illustration).
	\begin{lem} \label{lem:mated4spider-structure}
		Let $S, S'$ be mated $4$-spiders in an $\ell$-monoholed graph for some even $\ell \geq 8$. Let $A, A'$ be the body of $S, S'$, respectively. Let the toes of $A$ be $t_1, \dots, t_4$ and for each $i \in [4]$ let $a_i$ and $a_i'$ be the $t_i$-joint of $S$, $S'$, respectively. For each $i \in [k]$, let $P_i$ be the union of the $t_i$-leg of $S$ and $S'$. Then (by renumbering $t_1, \dots, t_4$ or exchanging $A$ and $A'$ if necessary) either
		\begin{enumerate}[(a)]
			\item $A, A'$ are single vertices, $\ell$ is even, and $P_i$ has length $\frac{\ell}{2}$ for each $i \in [k]$, \label{c:1-1}
			\item $A$ is a single vertex and $A'$ is a $K_4$, $\ell$ is odd and $P_i$ has length $\frac{\ell -1}{2}$ for each $i \in [k]$, \label{c:1-4}
			\item $A$ is a triangle, $a_1=a_2$ and $A'$ is a diamond with $a_1'$ non-adjacent to $a_2'$, $\ell$ is even and $P_1, P_2$ have length $\frac{\ell}{2} -1$ and $P_3, P_4$ have length $\frac{\ell}{2} -2$. \label{c:3-4}
			\item $A$ is a triangle $a_1 = a_2$ and $A'$ is a diamond with $a_3'$ non-adjacent to $a_4'$, $\ell$ is odd, $P_1, P_2$ have length $\frac{\ell -1}{2}$ and $P_3, P_4$ have length $\frac{\ell -3}{2}$. \label{c:newoddcase}
			\item $A$ and $A'$ are both equal to $K_4$, $\ell$ is even and $P_i$ has length $\frac{\ell}{2} -1$ for each $i \in [k]$.\label{c:4-4}
		\end{enumerate}
	\end{lem}
	
	\begin{proof}
		Suppose that $d_S(1,2) + d_S(3,4) = d_S(1,3) + d_S(2,4) + d_S(1,4) + d_S(2,3)$. Then by Lemma \ref{lem:mated-spider-top-shapes}, $A$ must be a single vertex or a $K_4$. It follows from the definition of mated spiders that, $d_{S'}(1,2) + d_{S'}(3,4) = d_{S'}(1,3) + d_{S'}(2,4) + d_{S'}(1,4) + d_{S'}(2,3)$ so $A'$ must be a single vertex or a $K_4$. Since $V(P_i \cup P_j)$ induces a hole in $G$ for any two distinct $i,j \in [k]$, it follows that \ref{c:1-1}, \ref{c:1-4} or \ref{c:4-4} holds.
		
		Hence by Lemma \ref{lem:mated-spider-top-shapes}, we may assume $d_S(1,2) + d_S(3,4)  \neq d_S(1,3) + d_S(2,4) = d_S(1,4) + d_S(2,3)$.
		Moreover,  $d_S(1,2) + d_S(3,4) = d_S(1,3) + d_S(2,4) \pm 1$ by Lemma \ref{lem:mated-spider-top-shapes}. By  exchanging $S$ and $S'$ if necessary, we may assume that $d_S(1,2) + d_S(3,4) = d_S(1,3) + d_S(2,4) - 1$. Then by another application of Lemma \ref{lem:mated-spider-top-shapes} we obtain that $A$ is a $K_3$ and $a_1 = a_2$ and that $A'$ is a diamond with $a_1'$ non-adjacent to $a_2'$. Thus, \ref{c:3-4} or \ref{c:newoddcase} holds (by renumbering if necessary). 
	\end{proof}
	
	\begin{lem}\label{lem:mated-k-spiders-are-2-connected}
		Let $S$ be a mated $k$-spider for some integer $k \geq 3$. Let $A$ be the body of $S$ and $a_1, \dots, a_k$ be the joints of $S$. Then $A$ has no cut-vertex and $V(A) = \{a_1, a_2, \dots, a_k\}$.
	\end{lem}
	
	\begin{proof}
		Suppose $A$ has a cut-vertex $v$ and let $A_1, A_2$ be two components of $A \setminus v$.  Choose some $a_1 \in V(A_1)$. Suppose $a_1$ is not a $t_i$-joint for any $i \neq 1$. Let $P$ be a shortest $a_1v$-path in $A$.  By Fact \ref{lem:leg-fact}, $a_1$ has degree at least two in $A$.  Hence some  $u \in V(A_1) \setminus V(P)$ is adjacent to $a_1$. Since $A$ is minimally connected with respect to vertex deletion, there exists some $i \in [k]$ such that $u$ belongs to every $a_iv$-path in $A$. Since $i \neq 1$, we may assume $i=2$. It follows that $d_S(t_1, t_2) < d_S(t_1, v) + d_S(t_2, v)$. By symmetry we may assume that $a_3, a_4 \in A$ and $d_S(t_3, t_4) < d_S(t_3, v) + d_S(t_4, v)$.
		But then, $d_S(1,2) + d_S(3,4) -2 \leq d_S(1,3) + d_S(2,4)$, contradicting Lemma \ref{lem:A-not-2-connecteded-4-spider-outcomes}. Hence $A$ has no cut-vertex. Since $A$ is minimally connected with respect to vertex deletion, $V(A) = \{a_1, a_2, \dots, a_k\}$.
	\end{proof}

	\begin{lem}\label{lem:if-even-body-of-mated-spider-is-a-threshold-graph}
		Suppose $A$ is the body of a mated $k$-spider in $G$ for some integer $k \geq 3$. Then $A$ is a threshold graph. In particular some vertex in $ V(A)$ is adjacent to every other vertex in $V(A)$ and so $A$ has diameter at most two. 
	\end{lem}
	\begin{proof}
		Suppose $a_1 \dd a_2 \dd a_3 \dd a_4$ is a $P_4$ in $A$. Then using the notation from the definition of a spider, $d(1,2) + d(3,4) = d(1,4) + d(2,3) -2$, contradicting Lemma \ref{lem:mated-spider-top-shapes}. Suppose $a_1a_2$ and $a_3a_4$ are edges of $A$ and $\{a_1, a_2\}$ is anticomplete to $\{a_3, a_4 \}$.  Then $d(1,2) + d(3,4) \leq d(1,4) + d(2,3) -2$, contradicting Lemma \ref{lem:mated-spider-top-shapes}.
	\end{proof}

	We can now prove the structure of mated $k$-spiders.

	\begin{theorem}\label{thm:matedspiderstructure-odd-case}
		Let $G$ be an $\ell$-monoholed graph for some $\ell \geq 7$ and odd.
	Suppose	$S,S'$ are mated $k$-spiders in $G$ for some $k \geq 3$. Then $S \cup S'$ is a generalized $k$-pyramid.
	\end{theorem}
	\begin{proof}
		Since $\ell \geq 6$ there is some integer $n \geq 2$ such that $\ell = 2n +1$.
		Let $t_1, t_2, \dots, t_k$ be the toes of $S,S'$. For each $i \in [k]$, let $a_i, a_i'$ be the $t_i$-joints of $S$ and $S'$, respectively. Let $A, A'$ be the bodies of $S$ and $S'$, respectively. For each $i \in [k]$ let $P_i$ be the union of the $t_i$-leg of $S$ and $S'$. 
		By Lemma \ref{lem:mated4spider-structure}, if $|A| \neq 1$ then $|A| \geq 3$ and if $|A'| \geq 1$ then $|A'| \geq 3$.
		By Lemma \ref{lem:if-even-body-of-mated-spider-is-a-threshold-graph}, $A$ and $A'$ are threshold graphs and by \ref{lem:mated-k-spiders-are-2-connected} if $A, A'$ are not single vertices they are two connected threshold graphs. Moreover by \ref{lem:mated-k-spiders-are-2-connected}, $V(A) = \{ a_1, \dots, a_k\}$ and $V(A') = \{ a_1', \dots, a_k'\}$.
		\stmt{Each of $P_1, P_2, \dots, P_k$ has length at most $n$. \label{patch:len-at-most-n}}
		Suppose $P_1$ has length at least $n + 1$. By definition, for each $i \in [2,k]$ there is a $a_1a_1'$-path $R_i$ containing $P_i$ such that $P_1 \cup R_i$ is a hole.
		Hence, $R_2, R_3, \dots, R_k$ all have length at most $n$.
		But $G[V(R_2 \cup R_3)]$ contains a hole and it has length at most $2n$, a contradiction. This proves (\ref{patch:len-at-most-n}).
		\stmt{Each of $P_1, P_2, \dots, P_k$ has length at least $n-1$. \label{patch:at-least-n-minus-one}}
		Suppose $P_1$ has length at most $n-2$.
		By Fact \ref{lem:leg-fact}, we may assume $a_1'$ is equal or adjacent to $a_2'$. 
		Let $Z_{12}$ be a shortest $a_1a_2$-path in $G[A]$. Then $V(P_1 \cup P_2 \cup Z_{12})$ induces a hole of length at most $n-1 + |E(P_1)| + |E(P_2)|$. 
		Since $G$ is $\ell$-monoholed it follows from (\ref{patch:len-at-most-n}) that $|E(P_2)| = n$ and $|E(Z_{12})| = 2$. Thus we may assume $Z_{12}$ has vertices $a_1 \dd a_3 \dd a_2$ in order.
		Since $P_3$ has length at most $n$, $a_3'$ is not adjacent to $a_1'$.
		Since $G[V(A')]$ has diameter at most two by \ref{lem:mated-k-spiders-are-2-connected}, $P_3$ must have length at $n$ or there is a hole of length less than $\ell$ containing $P_1, P_3$, a contradiction.
		Hence $a_2' \neq a_3'$.
		It follows that there is a hole of length greater than $\ell$ containing $P_2, P_3$, a contradiction. This proves (\ref{patch:at-least-n-minus-one}).
		\\
		\\
		We call $a_i$ a multi-purpose joint if $a_i = a_j$ for some $i \neq j$, $i,j \in [k]$.
		We define multi-purpose joint similarly for $A'$.
		\stmt{If $a_i$ is a multipurpose joint, $a_i'$ is not a multipurpose joint. If $a_i'$ is a multipurpose joint, $a_i$ is not a multipurpose joint. \label{patch:not apexes on both sides}}
		This follows from the fact that if $\ell$ is odd $G$ cannot contain a theta. This proves (\ref{patch:not apexes on both sides}).
		\stmt{If $a_i$ or $a_i'$ is a multipurpose joint $P_i$ has length $n$. \label{patch:multipurpose-len-n}}
		We may assume $S \cup S'$ is not a $k$-pyramid.
		Suppose $a_1 = a_2$. 
		Since $S \cup S'$ is not a $k$-pyramid, $|A| \geq 1$.
		Hence $G[A]$ is a 2-connected threshold graph. So there exist distinct $a_3, a_4 \in V(A)$ such that $a_1$ is adjacent to $a_3, a_4$.
		Thus there is a four spider $R$ contained in $S$ with body contained in $G[a_1, a_2, a_3, a_4]$ and $R$ has a mate $R'$ contained in $G[A']$.
		Then by Lemma \ref{lem:mated4spider-structure}, $a_1, a_2, a_3$ are pairwise adjacent and the body of $R'$ is a diamond.
		Let $t_1', t_2, \dots, t_4$ be the toes of $R, R'$.
		For each $i \in [4]$, let $Z_i$ be union of the two $t_i$-legs of $R$ and $R'$, respectively.
		Then in particular Lemma \ref{lem:mated4spider-structure} implies $R_1, R_2$ each have length $n$.
		By definition $P_1 \subseteq R_1$ and $P_2 \subseteq R_2$ so $P_1, P_2$ both have length $n$ by (\ref{patch:len-at-most-n}).
		This proves (\ref{patch:multipurpose-len-n}).
		\stmt{$A \cup A'$ contains at most one multipurpose joint. \label{patch:one-multi-total}}
		Since $G$ is $\ell$-monoholed, the statement follows from the fact that (\ref{patch:not apexes on both sides}) and (\ref{patch:multipurpose-len-n}). 
		This proves (\ref{patch:one-multi-total}).
		\\
		\\
		We may assume $A'$ does not contain a multipurpose joint.
		Let $Q, R$ be a partition of $[k]$ such that $a_{q_1} = a_{q_2}$ for every $q_1, q_2 \in Q$ and $a_{r_1}  \neq a_{r_2}$ for any two distinct $r_1, r_2 \in R$.
		For a set $S \subseteq[k]$ let $A(S)$, $A'(S)$ denote the sets $\{a_i \ | \ i \in S\}$ and $\{a_i' \ | \ i \in S\}$, respectively.
		\stmt{If $Q \neq \emptyset$, then $A'(Q)$ is a clique. \label{patch:mini-pyramid}}
		Since every hole in $G$ has length $2n + 1$ and $G[A']$ is connected, the statement follows from (\ref{patch:multipurpose-len-n}).
		This proves (\ref{patch:mini-pyramid}).
		\stmt{$A(Q)$ is complete to $A(R)$ and $A'(Q)$ is complete to $A'(R)$ and $P_i$ has length $n-1$ for every $i \in R$. \label{patch:last}}
		Since $G[A]$, $G[B]$ are connected this follows from (\ref{patch:at-least-n-minus-one}) and (\ref{patch:multipurpose-len-n}).
		This proves (\ref{patch:last}).
		I$G[A]$, $G[B]$ are threshold graphs, so $G[A(R)]$ and $G[B(R)]$ are threshold graphs.
		For any two distinct $i,j \in R$ it follows from (\ref{patch:last}) that $a_i$ is adjacent to $a_j$ if and only if $a_i'$ is not adjacent to $a_j'$.
		Thus $S \cup S'$ is a generalized $k$-pyramid.
	\end{proof}

	\begin{theorem}\label{thm:orderly-spines-even-case}
		Let $G$ be an $\ell$-monoholed graph for some $\ell \geq 8$ and even.
	Suppose	$S,S'$ are mated $k$-spiders in $G$ for some $k \geq 3$. Then $S \cup S'$ is a $k$-theta or a generalized $k$-prism.
	\end{theorem}
	\begin{proof}
		Since $\ell \geq 6$ there is some integer $n \geq 2$ such that $\ell = 2n +2$. 
		Let $t_1, t_2, \dots, t_k$ be the toes of $S,S'$. For each $i \in [k]$, let $a_i, a_i'$ be the $t_i$-joints of $S$ and $S'$, respectively. Let $A, A'$ be the bodies of $S$ and $S'$, respectively. For each $i \in [k]$ let $P_i$ be the union of the $t_i$-leg of $S$ and $S'$. We begin with a series of observations about the structure of $A,A'$ and its relationship to the lengths of the paths $P_1, P_2, \dots, P_k$. 
		\stmt{\label{gc:P_i-has-length-at-most-n-or-theta}
			For every $i \in [k]$, $P_i$ has length at most $n$ or $S \cup S'$ is a $k$-theta.	}
		Suppose that $P_1$ has length at least $n +1$. For every $i \in [2, k]$, there is a $a_ia_i'$-path $R_i$ of $S \cup S'$ including $P_i$.
		Since $R_i \cup P_1$  is a hole, $R_i$ has length $\ell - |E(P_1)| \leq n + 1$ for every $i \in [2,k]$. Since $G[V(R_1\cup R_2)]$ includes a hole, $|E(R_1)|, |E(R_2)| \geq n +1 $.  Hence, $P_1$ has length $n+1$ and $R_i$ has length $n+1$ for every $i \in [2,k]$. 
		It follows that every $V(R_i \cup R_j)$ induces a hole for any two distinct $i, j \in [2,k]$. 
		
		Let $a_1, \dots, a_k$ be the joints of $S$. Suppose $S \cup S'$ is not a $k$-theta. Then we may assume the body of $S$ is not a single vertex and  $a_2 \neq a_1$. By Fact \ref{lem:leg-fact} we may assume $a_3 \neq a_1$ and is equal or adjacent to $a_2$.  But then $V(R_2 \cup R_3)$ does not induce a hole, a contradiction. This proves (\ref{gc:P_i-has-length-at-most-n-or-theta}).
		\\
		\\
		We may assume $S \cup S'$ is not a $k$-theta. Then, $P_i$ has length at most $n$ for every $i \in [k]$. We now prove a lower bound on the lengths of the paths $P_i$ for $i \in [k]$.
		\stmt{For every $i \in [k]$, $P_i$ has length at least $n-1$.\label{gc:P_i-has-length-at-least-n-1}}
		Consider the path $P_1$. 
		By Fact \ref{lem:leg-fact}, we may assume $a_2$ is equal or adjacent to $a_1$. 
		Let $Q$ be an shortest $a_1'a_2$-path in the body of $S'$.  By \ref{lem:if-even-body-of-mated-spider-is-a-threshold-graph}, $Q$ has length at most two. Since $P_2$ has length at most $n$, it follows that $G[V(P_1 \cup P_2 \cup Q)]$ has a hole of length at most $|E(P_1)| + n +3$. Hence, $|E(P_1)| \geq n-1$. This proves (\ref{gc:P_i-has-length-at-least-n-1}).
		\\
		\\
		We call a vertex $v \in A \cup A'$  a \textit{multi-purpose joint} if it is the $t_i$-joint for more than one distinct $i \in [k]$. 
		\stmt{There is at most one multi-purpose joint in each of $A, A'$. If there is a multi-purpose joint $v \in V(A)$, then for every $i \in [k]$ for which $v$ is the $t_i$-joint of $S$, $P_i$ has length $n$ and $a_i'$ is not a multi-purpose joint of $S'$. 		
			If $a_i = a_j$ for distinct $i,j \in [k]$ then $a_i'$ is not adjacent to $a_j'$.
			Moreover, the same statements hold after exchanging $A$ and $A'$. 
			\label{gc:only-one-mulitpurpose-joint}
		}
		Suppose $A$ contains more than one multi-purpose joint.
		Without loss of generality, $a_1 = a_2$ and $a_3=a_4 \neq a_1$. Consider the spider $R$ contained in $S$ with toes $t_1, t_2, t_3, t_4$. Then the body of $R$ is an $a_1a_3$-path. But since $R$ has a mate contained in $S'$, this contradicts Lemma \ref{lem:mated-spider-top-shapes}. Hence $A$ has at most one multi-purpose joint.
		
		Suppose $a_1$ is a multi-purpose joint of $S$. Suppose $a_1$ has a non-neighbor, say $a_3$. Without loss of generality $a_1 = a_2$.
		We may assume $a_4$ is equal or adjacent to $a_3$ by Fact \ref{lem:leg-fact}. Again, consider the spider $R$ contained in $S$ with toes $t_1, t_2, t_3, t_4$.
		Let $B$ be the body of $R$. Since $R$ has a mate contained in $S'$, it follows that $V(B) \subseteq \{ a_1, a_3, a_4\}$ by Lemma \ref{lem:mated-spider-top-shapes}.
		But then $2 \leq |V(B)| \leq 3$. Hence by Lemma \ref{lem:mated-spider-top-shapes}, $B$ must be a triangle and $V(B) = \{ a_1, a_3, a_4\}$. But $a_1$ is not adjacent to $a_3$, a contradiction. Thus, $a_1$ must be adjacent to every other vertex in $V(A)$.
		
		Suppose  $P_1$ has length less than $n$. By Lemma \ref{lem:if-even-body-of-mated-spider-is-a-threshold-graph}, there is an $a_1'a_2'$-path $Q$ in $A'$ of length at most two. But then since $P_2$ has length at most $n$, the graph $G[V(P_1 \cup P_2 \cup Q)]$ contains a hole of length at most $2n+1$, a contradiction. Hence, $P_i$ has length $n$ for every $i \in [k]$ such that $a_1=a_i$.
		Since there are no holes of length $2n+1$ it follows that $a_i'$ and $a_j'$ are non-adjacent for every two distinct $i,j \in [k]$ such that $a_1 = a_i, a_j$.
		
		Finally, suppose $a_1'$ is a multi-purpose joint of $S'$. Then $a_1'$ is distinct from and non-adjacent to $a_2'$, for otherwise $G[V(P_1 \cup P_2)]$ is a hole of length at most $2n+1$, a contradiction. So we may assume $a_1' = a_3'$. Then, $a_1$ is distinct from and non-adjacent to $a_3$. But then $P_2, P_3$ and the union of shortest paths between $a_1$ and $a_3$ in $A$ and $a_2'$ and $a_3'$ in $A'$ form a hole of length at $2n + 4$, a contradiction. 
		By the symmetry between $S$ and $S'$, this proves (\ref{gc:only-one-mulitpurpose-joint}).
		
		\stmt{For any two distinct $i,j \in [k]$, if $P_i$ and $P_j$ both have length $n$, then one of the following holds:
			\begin{itemize}
				\item $a_ia_j$ and $a_i'a_j'$ are edges or
				\item $a_i = a_j$ and $a_i'$ and $a_j'$ are non-adjacent (or vice versa.)
			\end{itemize}
			\label{gc:if-Pi-and-Pj-both-have-length-n}
		}
		This follows immediately from the fact that there is a hole in $S\cup S'$ containing $P_i$ and $P_j$. This proves (\ref{gc:if-Pi-and-Pj-both-have-length-n}).
		\\
		\\		
		For a set $J \subseteq [k]$, let $A(J)$ and $A'(J)$ denote the sets $\{ a_j \ | \ j \in J\}$ and $\{a_j' \ | \ j \in J\}$, respectively. 
		\stmt{$A$, $A'$ have minimum degree at least two \label{gc:min-degree-2}}
		By Fact \ref{lem:leg-fact} we need only show that multi-purpose joints in $A, A'$ have degree at least two in $A$ and $A'$, respectively. Suppose $a_1$ is a multi-purpose joint of degree at most one in $A$. Let $J$ be the set of integers $i \in [k]$ such that $a_1 = a_i$. By (\ref{gc:only-one-mulitpurpose-joint}), $A'(J)$ is a stable set. Since $A'$ is connected, it follows that $J \neq [k]$. Then since $A$ is connected, $a_1$ has a neighbor in $A$. Without loss of generality $a_1 = a_2$ and $a_1$ is adjacent to $a_3$. By (\ref{gc:only-one-mulitpurpose-joint}) and Fact \ref{lem:leg-fact}, we may assume $a_4 \neq a_1$ and $a_4$ is adjacent to $a_3$.
		Consider the spider $R$ contained in $S$ with toes $t_1, t_2, t_3, t_4$. Let $B$ be the body of $R$. Then $G[\{a_1, a_3, a_4\}]$ contains $B$. By definition of joint, $a_1$ is a joint of $R$ and at least one of $a_3, a_4$ is a joint of $R$. So the $B$ is a path of length one or two, contradicting Lemma \ref{lem:mated4spider-structure}. This proves (\ref{gc:min-degree-2}).
		\\
		\\
		We now have enough tools to complete the proof.
		Let $Q$ be the set of indices $i \in [k]$ such that $P_i$ has length $n-1$.
		$A(Q)$ is a stable set for if any two $a_i, a_j \in A(Q)$ are adjacent the hole in $S\cup S'$ containing $P_i \cup P_j$ has length at most $2n+1$, a contradiction. 
		At least two vertices in $A$ have no non-neighbors in $V(A)$ since $A$ is a threshold graph of minimum degree at least two by Lemma \ref{lem:if-even-body-of-mated-spider-is-a-threshold-graph} and (\ref{gc:min-degree-2}). Hence, there is some vertex in $a_T \in V(A) \setminus A(Q)$ that is adjacent to every other vertex in $V(A)$. Choose $a_T$ to multi-purpose if possible. Let $T = \{ i \ | \ a_T = a_i\}$. By (\ref{gc:only-one-mulitpurpose-joint}) none of the vertices in $A(Q) \cup A'(Q)$ are multi-purpose since the constituent paths ending in $A(Q), A(Q')$ all have length $n-1$. Thus for any $i \in [k] \setminus Q$, it follows that $a'_i \not \in A'(Q)$ and in particular $A'(Q)$ and $A'(T)$ are disjoint.
		By symmetry, there is some vertex $a'_S \in V(A) \setminus A'(Q)$ that is adjacent to every other vertex in $V(A')$. Choose $a'_S$ to be multi-purpose if possible. Let $S = \{i \ | \ a'_S = a'_i\}$. By (\ref{gc:only-one-mulitpurpose-joint}), $Q, S, T$ are pairwise disjoint.
		$A'(T)$ and $A(S)$ are stable sets for otherwise there is a hole of length $2n+1$, a contradiction.
		$A'(T)$ is anticomplete to $A'(Q)$ for otherwise we obtain a hole of length less than $2n +2$ since $a_T$ is complete to $A(Q)$, a contradiction. Similarly, $A(S)$ is anticomplete to $A(Q)$. 
		Let $R = [k] \setminus (Q \cup S \cup T)$. Then by (\ref{gc:if-Pi-and-Pj-both-have-length-n}), $A(R)$ is complete to $A(S)$ and $A'(R)$ is complete to $A'(S)$. 
		The edges between $A(Q)$ and $A(R)$ form a half graph since $A$ is a threshold graph by Lemma \ref{lem:if-even-body-of-mated-spider-is-a-threshold-graph}. Similarly, the edges between $A'(Q)$ and $A'(R)$ are a half-graph. Moreover since every hole must have length $2n+2$, for any $i \in Q$ and $j \in R$, we have that $a_i$ is adjacent to $a_j$ if and only if $a_i'$ is non-adjacent to $a_j'$.
		\stmt{No joint in $A(R) \cup A(R')$ is multi-purpose. \label{gc:R-single-purpose-joints}}
		Suppose some $a_z \in A(R)$ is multipurpose. Then by \ref{gc:only-one-mulitpurpose-joint}, $a_T$ is not multi-purpose. By choice of $a_T$ it follows that $a_z$ must have a non-neighbor in $a_i \in A$. Let $Z = \{ j \in [k] \ | \ a_j = a_z\}$. Since no joint in $A'(Z)$ is multi-purpose by (\ref{gc:only-one-mulitpurpose-joint}) and $P_z$ has length $n$, it follows that $P_i$ cannot have length $n$ by (\ref{gc:if-Pi-and-Pj-both-have-length-n}). Thus, $i \in Q$. Let $t \in A'(T)$ and let $u, v$ be distinct vertices in $A'(Z)$. Since there are complementary half-graphs between $A(Q), A(R)$ and $A'(Q), A(R)$, it follows that $u,v$ are both adjacent to $a'_i$. But then since $t$ is complete to $A'(R)$ and anticomplete to $A'(Q)$, the graph $G[\{ u,v, a_i', t\}]$ is a $C_4$, a contradiction. This proves (\ref{gc:R-single-purpose-joints}). 
		\\
		\\
		It follows that $A(R)$ and $A(R')$ are vertex sets of cliques by (\ref{gc:if-Pi-and-Pj-both-have-length-n}). 
		Hence, $S' \cup S$ is a $k$-spine.
	\end{proof}
	
	We are now ready to prove the Theorem \ref{thm:orderly-spines-exist}
	which we restate:
	\begin{thm}\label{thm:orderly-spines-exist}
		Let $G$ be an $\ell$-monoholed graph. Then one of the following conditions holds:
		\begin{enumerate}[(a)]
			\item $G$ contains a vertex that is adjacent to every other vertex in $V(G)$. 
			\item$G$ contains a clique cut-set, 
			\item $G$ is chordal,
			\item $G$ is an inflated $\ell$-hole, 
			\item $G$ contains a $k$-theta with paths of length $\frac{\ell}{2}$ (and $\ell$ is even), or
			\item There is some $k \geq 3$ such that $G$ contains $k$-spine.
		\end{enumerate}
	\end{thm}
	
	\begin{proof}
		Suppose none of (a), (b), (c) or (d) holds. Then by Theorem \ref{thm:G-has-a-mated-spider} $G$ contains a pair of mated $k$-spiders $S, S'$. If $\ell$ is odd, (e) holds by Theorem \ref{thm:matedspiderstructure-odd-case}. If $\ell$ is even, (e) holds by Theorem \ref{thm:orderly-spines-even-case}.
	\end{proof}
	
	\section{On Corpora and Crowns}

	For technical reasons, our analysis would be easier if we could assume that any two constituent paths $k$-spine are vertex disjoint. However, this clearly is false; $k$-thetas, $k$-pyramids and some generalized $k$-prisms have multiple constituent paths ending at the same vertex.  Our solution is to consider only the subpaths of $k$-spines that are vertex disjoint. For example, if $v$ apex of a $k$-pyramid with constituent paths $P_1, P_2, \dots, P_k$, we will analyze the paths $P_1 \setminus v$,  $P_2 \setminus v \dots, P_k \setminus v$ instead of analyzing $P_1, \dots, P_k$.
	
	Let $F$ be a $k$-spine for some $k \geq 3$. We call any vertex $v \in V(F)$ belonging to multiple constituent paths of $F$ an apex. Note by definition of $k$-spine every apex is an end of a constituent path of $F$.
	Let $J$ be the set of apexes of $F$ and let $P_1, P_2, \dots, P_k$ denote the constituent paths of $F$. Then we call $P_1 \setminus J$, $P_2 \setminus J$, $\dots$, $P_k \setminus J$ the \textit{elemental paths} of $F$. 
	We define an analogue of terminating sets for elemental paths as follows:
	Let $A, B$ be the terminating sets of $F$. For any $i \in [k]$ and end $v$ of $P_i \setminus J$, if $v$ is equal or adjacent to a vertex in $A$ we call $v$ the \textit{$A$-end} of $P_i$. Otherwise, $v$ is the \textit{$B$-end} of $P_i$.
	We call the set of $A$-ends of elemental paths of $F$ and the set of $B$-ends of elemental paths of $F$ the \textit{elemental sides} of $F$.
	
	We call the graph obtained from $F$ by removing apexes in $F$ the \textit{core} of $F$.
	We call an inflated graph $\F$ a \textit{$k$-corpus} if the graph $F$ underlying $\F$ is a $k$-spine and the bags of $\F$ corresponding to vertices in an elemental side of $F$ or apexes of $F$ are either complete or anticomplete to each other. 
	The elemental sides of $\F$ are the sets of bags corresponding to the elemental sides of $F$.
	The apexes of $\F$ are the bags corresponding to apexes of $F$. We say the core of $\F$ is the sub-inflated-graph of $\F$ containing all bags in $\F$ corresponding the core of $F$.
	Note by definition any $k$-spine is also a $k$-corpus.	
	By definition of $k$-spine, for any two constituent paths of a $k$-corpus there is a inflated $C_\ell$ containing both of them.
	
	In this section we will consider a $k$-corpus $\F$ chosen to maximize $k$ and with respect to that maximize the number of vertices in $\F$ and how vertices in $V(G) \setminus V(\F)$ interact with $\F$.
	We will make repeated use of the following easy fact.
	\begin{fact}\label{fact:corpus-two-elemental-vertices-ell-minus-2-path-between-them}
		Let $\F$ be a $k$-corpus. Suppose $X_1$ and $X_2$ are bags in the same elemental side of $\F$. Let $\q_1, \q_2$ be the elemental paths of $\F$ containing $X_1$ and $X_2$, respectively. Let $Y$ be the vertex set of the elemental side of $\F$ not containing $X_1, X_2$. Suppose $x_1 \in X_1$ and $x_2 \in X_2$ are non-adjacent. Then $\F$ contains an $x_1x_2$-path $P$ of length $\ell -2$ such that $V(P) \subseteq V(\q_1 \cup \q_2 \cup Y)$.
	\end{fact}
	\begin{proof}
		For any bag $B$ in an inflated graph there is a vertex $b \in B$ such that $b$ is complete to every neighboring bag of $B$ by definition of inflated graph.
		Hence, the result follows from the definition of $k$-spine.
	\end{proof}
	
	\subsection{Crowns}
	In this subsection we will begin to consider the structure of vertices outside of a maximal $k$-corpus $\mathcal{\F}$ with $k$-maximum that have that neighbors in an elemental side of a maximal $k$-body.
	In order to do this, we will introduce a new object we call a ``crown'' and prove some properties about it.
	
	Suppose $G$ is a graph where $I, J$ is a partition of $V(G)$.
	Let $P$ be a induced path of $G$  with vertices $p_1 \dd p_2 \dd p_3 \dd p_4 \dd p_5$.
	We call $P$ a \textit{mean $P_5$} if $p_1, p_3, p_4 \in I$ and $p_3, p_4 \in J$.
	Suppose $i_1, i_2, i_3, i_4, j \in V(G)$ such $j$ is adjacent to $i_1, i_2$ and $i_2$ is adjacent to $i_3, i_4$ and $G[\{ i_1, i_2, i_3, i_4, j\}]$ contains no further edges. Then if $i_1, i_2, i_3, i_4 \in I$ and $j \in J$ we call the graph induced by $\{i_1, i_2, i_3, i_4, j\}$ a \textit{mean fork}. See Figure \ref{fig:mean-crown-graphs}.
	
	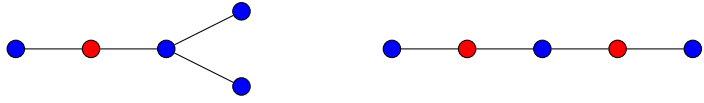
\begin{figure}[!h]
		\centering
		\input{Tikz/crown_mean_graphs.tex}
		\caption{A mean fork is drawn at left and a mean $P_5$ is drawn at right. Blue vertices are in $I$ and red vertices are in $J$.}
		\label{fig:mean-crown-graphs}
	\end{figure}
	
	We call a graph $G$ a \textit{crown} if there is a partition of $V(G)$ into non-empty sets $I, J$ satisfying all of the following axioms:
	
	\begin{enumerate}[(i)]
		\item There is no clique $Z$ such that $G \setminus Z$ has two parts $X, Y$ such that $Y$ is non-empty and $I \subseteq V(X \cup Z)$, \label{crown-axiom:no-clique-cutset}. 
		\item For every induced $u, v \in I$, every $uv$-path $P$ with $P^* \subseteq J$ has length two, \label{crown-axiom:paths-length-two}
		\item $H[I]$ does not contain a $P_4$, \label{crown-axiom:I-is-co-graph}
		\item $H$ does not contain a mean $P_5$. \label{crown-axiom:no-mean-P5}
		\item  \label{crown-axiom:no-mean-fork}$H$ does not contain a mean fork.
	\end{enumerate}
	
	\begin{lemma}
		Let $G$ be a crown with partition $I,J$. Suppose $G$ is $\ell$-monoholed for some $\ell \geq 6$. Then for every $j \in J$, the set $N(j) \cap I$ contains two non-adjacent vertices. 
		\label{lem:crown-tops-have-two-no-adjacent-nbrs-in-bottom}
	\end{lemma}
	\begin{proof}
		Let $J'$ be the set of vertices $j \in J$ where $N(j) \cap I$ contains two non-adjacent vertices. Suppose $J \neq J'$ and let $H$ be a component of $G[J \setminus J']$. 
		By Axiom \ref{crown-axiom:no-clique-cutset}, the $N(V(H)) \cap (I \cup J')$ is not the vertex set of a clique. Hence there exist non-adjacent $u, v \in I \cup J$ such that $u, v$ both have neighbors in $V(H)$.
		Let $P$ be an induced $uv$-path with $P^* \subseteq J'$.
		\stmt{$u,v$ are not both elements of $I$ \label{diadem:caseII}}
		Suppose $u, v \in I$. Then by Axiom \ref{crown-axiom:paths-length-two}, $P$ has length two. But then $P^*$ is a single vertex $j \in J'$. But then $N(j) \cap V(I)$ is not the vertex set of a clique, a contradiction. This proves (\ref{diadem:caseII}).
		\stmt{$u,v$ are not both elements of $J'$.\label{diadem:caseJ'J'}}
		Suppose $u, v \in J'$. 
		Let $I_{uv}$ denote the set of common neighbors of $u$ and $v$ in $I$. Let $I_u$ and $I_v$ denote the sets of neighbors of $u$ and $v$ in $I \setminus I_{uv}$, respectively.
		Suppose there are two non-adjacent vertices $i, i' \in I_{uv}$. Then $v \dd i \dd u \dd i' \dd v$ is a hole of length four, a contradiction. Hence $G[I_{uv}]$ is a clique.
		
		By definition of $V(H)$, $I_u$ and $I_v$ are not empty.
		Suppose there is some $i_u \in I_u$ and $i_v \in I_v$ such that $i_u$ and $i_v$ are not adjacent. By definition $T$ does not contain a common neighbor of $i_u$ and $i_v$. But then $G[V(T) \cup \{ u,v\}]$ contains an $i_ui_v$-path of length greater than two, contradicting Axiom \ref{crown-axiom:paths-length-two}. Hence $I_u$ and $I_v$ are complete to each other. 
		
		Suppose there exist non-adjacent $i, i' \in I_u$. Then there is a hole of length four induced by $i, i', u$ and some vertex in $I_v$. Hence, $G[I_u]$ is a clique. Similarly, $G[I_v]$ is a clique.
		Since $N(u) \cap I$ contains two non-adjacent vertices, it follows that $I_{uv} \neq \emptyset$.
		Since $u,v \in J'$, there is exists $i_u \in I_u$, $i_{uv}, i_{uv}' \in I_{uv}$ and $i_v \in I_v$ such that $i_u$ is not adjacent to $i_{uv}$ and $i_v$ is not adjacent to $i_{uv}'$.
		If $i_{uv} = i_{uv}'$ then $u \dd i_{uv} \dd v \dd i_v \dd i_u \dd u$ is a hole of length five, a contradiction.
		It follows that every $i' \in I_{uv}$ is adjacent to one of $i_u$ or $i_v$. Hence $i_{uv} \neq i_{uv}'$. Then $i_u \dd i_{uv}' \dd i_{uv} \dd i_u$ is a $P_4$ in $G[I]$ contradicting Axiom \ref{crown-axiom:I-is-co-graph}. This proves (\ref{diadem:caseJ'J'}).
		\\
		\\
		By (\ref{diadem:caseII}) and (\ref{diadem:caseJ'J'}),  we may assume that $u \in I$ and $v \in J'$. Then there exist non-adjacent $i, i' \in N(v) \cap I$. 
		By definition, $u \neq i, i'$. Thus, if $u$ is adjacent to both $i$ and $i'$, then $u \dd i \dd v \dd i' \dd u$ is a hole of length four, a contradiction. Hence we may assume $u$ is not adjacent to $i'$. Then the path $uPv\dd i'$ must contain a common neighbor of $u$ and $i'$ by Axiom \ref{crown-axiom:paths-length-two}. Hence, $V(P) \setminus \{ u\}$ must contain a common neighbor of $u$ and $i'$, contradicting that $P^* \subseteq J \setminus J'$. 
	\end{proof}
	
	\begin{fact}
		Let $G$ be a crown with partition $I, J$. Suppose $G$ is $\ell$-monoholed for some $\ell \geq 6$. Then for every $j \in J$, the complement of the graph induced by $N(j) \cap I$ contains exactly one non-trivial component.
		\label{fact:crown-j-has-unique-non-trivial-anticomponent-in-I}
	\end{fact}
	\begin{proof}
		Let $j \in J$ and let $H$ denote $G[N(j) \cap I]$. By Axiom \ref{crown-axiom:no-clique-cutset}, $H$ is not a clique so $H^c$ contains at least one non-trivial component.
		Suppose $H^c$ contains two nontrivial components $C_1, C_2$. Then there are some $w,x \in V(C_1)$ and $y,z \in V(C_2)$ such that $wx, yz \not \in E(H)$. Since $C_1$ and $C_2$ are different components of $H^c$, it follows that $w \dd y \dd x \dd z \dd w$ is a hole of length four in $G$, a contradiction.
	\end{proof}
	
	Suppose $G$ is an $\ell$-monoholed graph with partition $I, J$ for some $\ell \geq 6$ and $G$ is a crown. Then for every $j \in J$ let $g(j)$ denote the vertices of the nontrivial anticomponent of $N(j) \cap I$ guaranteed by Fact \ref{fact:crown-j-has-unique-non-trivial-anticomponent-in-I}. We call the set $g(j)$ the \textit{good children} of $j$.
	We call the set neighbors of $j$ in $I \setminus g(j)$ the \textit{bad children} of $j$ and denote it by $b(j)$.
	
	\begin{lemma}
		Let $G$ be an $\ell$-monoholed for some $\ell \geq 6$. Suppose $G$ contains a crown with partition $I, J$.
		Then for every two distinct $u, v \in J$ the following statements both hold:
		\begin{enumerate}[(a)]
			\item If $u, v$ are adjacent then either $N(u) \cap I \subseteq N(v) \cap I$  or $N(v) \cap N(I) \subseteq N(u) \cap I$.
			\item  If $u, v$ are not adjacent then $u$ is anticomplete to $g(v)$ and $v$ is anticomplete to $g(u)$.
		\end{enumerate}
	\label{lem:crown-good-bag-children-properties}
	\end{lemma}
	\begin{proof}
		\stmt{(a) holds. \label{tiara:a}}
		Let $u, v \in J$ be adjacent vertices. Suppose there exists $i, i' \in I$ such that $iu, i'v \in E(G)$ and $iv, i'u \not \in E(G)$. Then $i,i'$ are not adjacent because otherwise $u \dd i \dd i' \dd v \dd u$ is a hole of length four, a contradiction. But then $i \dd u \dd v \dd i'$ is a path contradicting Axiom \ref{crown-axiom:paths-length-two}. 
		Thus, $N(u) \cap I \subseteq N(v) \cap I$.
		This proves (\ref{tiara:a}).
		\\
		\\
		Thus, it only remains to show that statement (b) holds.
		Let $u, v \in J$ be adjacent vertices. Since $G$ does not contain a hole of length four the set of common neighbors of $u$ and $v$ is either empty or the vertex set of a clique. 
		\stmt{$g(u)$ and $g(v)$ are disjoint. \label{tiara:disjoint-good-nbrs}}
		Suppose there exists some $w \in g(u) \cap g(v)$.
		By definition, $w$ is not adjacent to some $w_u \in g(u)$ and some $w_v \in g(v)$. It follows that $v$ is not adjacent to $w_u$ and $u$ is not adjacent to $w_v$.
		If $w_u$ and $w_v$ are adjacent then $w_u \dd u \dd w \dd v \dd w_v \dd w_u$ is a hole of length five, a contradiction. so $w_uw_v \not \in E(G)$. But then $w_u \dd u \dd w \dd v \dd w_v$ is a mean $P_5$, contradicting that $G$ is a crown. This proves (\ref{tiara:disjoint-good-nbrs})
		\\
		\\
		Suppose some good child $w$ of $u$ is a neighbor of $v$.
		By definition $w$ has a non-neighbor $x \in g(u)$. Then $x$ is not adjacent to $v$ because $G$ does not contain a hole of length four. Choose non-adjacent vertices $r,s \in g(v)$.
		By (\ref{tiara:disjoint-good-nbrs}), $w$ is a bad child of $v$. Thus, $w$ is adjacent to both $r$ and $s$. Since $G[\{ r,s, w, x\}]$ does not induce a $C_4$, we may assume $x$ is not adjacent to $r$. If $x$ is adjacent to $s$ the path $x \dd s \dd w \dd r$ violates Axiom \ref{crown-axiom:I-is-co-graph} from the definition of crown. So $x$ is not adjacent $s$.
		It follows that $u$ is not adjacent to $r,s$ because $r,s \not \in g(u)$ by (\ref{tiara:disjoint-good-nbrs}). But then $G[\{x, u, w, r,s\}]$ is a mean fork, contradicting that $G$ is a crown.
	\end{proof}
	
	\subsection{Defining a crowned $k$-corpus}
	Let $G$ be an $\ell$-monoholed graph.
	Let $\F$ be a $k$-corpus and let $X,Y$ denote the vertex sets of the two-elemental sides of $\F$.
	Let $A_X, A_Y$ contain all vertices of $\F$ in apexes adjacent to vertices in $X$ and $Y$ respectively. Possibly $A_X, A_Y$ are empty.
	Let $J_X$ be a set of vertices in $V(G) \setminus V(\F)$ with the property that $N(J_X) \cap V(\F) \subseteq X \cup A_X$ and for every $j \in J_X$ $N(j) \cap X$ contains two non-adjacent vertices.
	Define $J_Y$ similarly for $Y$.
	Let $\mathcal{R}$ be the inflated graph induced by $G[V(\F) \cup J_X \cup J_Y]$ where bags of $\F$ are bags of $\mathcal{R}$ and $\{j \}$ is a bag for every $j \in J_X \cup J_Y$.
	We call any such graph $\mathcal{R}$ a \textit{crowned $k$-corpus}.
	
	For simplicity, we will equate vertices in $J_X \cup J_Y$ with the bags containing them in our analysis.
	We refer to $\F$ as the $k$-corpus of $\mathcal{R}$. We call the apexes, elemental sets, and elemental paths of $\mathcal{F}$ the apexes, elemental sets and elemental path of $\mathcal{R}$, respectively.
	The following lemma proves $G[X \cup J_X]$ and $G[Y \cup J_Y]$ are both crowns under the partitions $X, J_X$ and $Y, J_Y$, respectively. We will refer to them as the \textit{crowns} of $\mathcal{R}$.
	
	\begin{lemma}\label{lem:the-crowns-of-a-crowned-corpus-are-crowns}
		Let $G$ be an $\ell$-monoholed graph and $\mathcal{R}$ be a crowned $k$-corpus in $G$.
		Let $X, Y, J_X, J_Y$ be as in the definition of crowned $k$-corpus.
		Then $G[X \cup J_X]$ is a crown with partition $X, J_X$ and $G[Y \cup J_Y]$ is a crown with partition $Y, J_Y$.
	\end{lemma}
	\begin{proof}
		Let $\{X_1, X_2, \dots, X_k\}$ and $\{Y_1, Y_2, \dots Y_k\}$ denote the elemental sides of $\mathcal{R}$. For each $i \in [k]$ denote the elemental path of $\mathcal{R}$ with ends $X_i, Y_i$ as $\q_i$. 
		\stmt{For every $u, v \in X$, every induced $uv$-path $P$ with interior in $J_X$ has length two. \label{coronet:path-axiom}}
		We may assume $u$ is not adjacent to $v$. Then $u,v$ are in different bags of $\F$. Then there is a $uv$-path $M$ in $\F$ with interior anticomplete to $J_X$ of length $\ell -2$. Since $G$ is $\ell$-monoholed the statement follows. This proves (\ref{coronet:path-axiom}).
		\stmt{$G[X]$ is $P_4$-free. \label{coronet:elemental-sides-are-cographs}} 
		Suppose $v_1 \dd v_2 \dd v_3 \dd v_4$ is an induced path in $G[X]$.
		Since each bag of $\F$ induces a clique, we may assume $v_1 \in X_1$ and $v_4 \in X_2$.
		By definition of $k$-corpus, $v_2 \not \in X_2$ and $v_3 \not \in X_4$. 
		By Fact \ref{fact:corpus-two-elemental-vertices-ell-minus-2-path-between-them} there is an $v_2v_3$-path $P$ such that $V(P) \subseteq V(\q_1 \cup \q_2) \cup Y$.
		But then the union of $P$ and the path $v_1 \dd v_2 \dd v_3 \dd v_4$ is a hole of length at least $\ell +1$, a contradiction.
		This proves (\ref{coronet:elemental-sides-are-cographs}).
		\stmt{$G[X \cup J_X]$ does not contain a mean $P_5$ under the partition $X, J_X$. \label{coronet:meanP5}}
		Suppose there exist $v_1, v_2, v_3$ and $z_1, z_2 \in J_X$ such that $v_1 \dd z_1 \dd v_2 \dd z_2 \dd v_3$ is an induced path.
		Then since $v_1, v_2, v_3$ are pairwise non-adjacent we may assume $v_1 \in X_1, v_2 \in X_2, v_3 \in X_3$. By Fact \ref{fact:corpus-two-elemental-vertices-ell-minus-2-path-between-them} $v_1v_3$-path $P$ of length $\ell -2$ with $V(P) \subseteq V(\q_1 \cup \q_3 \cup J_Y)$. Hence the union of $v_1 \dd z_1 \dd v_2 \dd z_2 \dd v_3$ and $P$ is a hole and it is longer than $\ell$, a contradiction. This proves (\ref{coronet:meanP5}).
		\stmt{$G[X \cup J_X]$ does not contain a mean fork under the partition $X, J_X$. \label{coronet:mean-fork}}
		Suppose there exist $v_1, v_2, v_3, v_4 \in X$ and $z \in J_X$ such that $v_2$ adjacent to each of $z, v_3, v_4$ and $v_1$ is adjacent to $z$ and suppose there are no further edges in $G[\{ v_1, v_2, v_3, v_4, z\}]$. By definition of $k$-corpus none of $v_1, v_2, v_3, v_4$ are in the same bag. Hence, we may assume $v_i \in X_i$ for $i \in [4]$. By Fact \ref{fact:corpus-two-elemental-vertices-ell-minus-2-path-between-them}, there is a $v_1v_4$-path $P$ of length $\ell -2$ such that $V(P) \subseteq V(\q_1 \cup \q_3 \cup J_Y)$. But then the union of $P$ and the path $v_1 \dd z \dd v_2 \dd v_4$ is a hole and it has length greater than $\ell$, a contradiction. This proves (\ref{coronet:meanP5}).
		\\
		\\
		It follows that $G[X \cup J_X]$ is a crown under the partition $X, J_X$. By symmetry, $G[Y \cup J_Y]$ is a crown under the partition $Y, J_Y$.
	\end{proof}
	
	We will make repeated use of the following consequence of the definition of crowned $k$-corpus.
	\begin{fact}
		Let $G$ be an $\ell$-monoholed graph for some $\ell \geq 5$.
		Suppose $G$ contains a $k$-corpus $\F$ for some $k \geq 3$.
		The set of vertices $H$ in $V(G) \setminus V(\F)$ complete to $V(\F)$ induces a clique.
		If $\mathcal{R}$ is a crowned $k$-corpus such that the corpus of $\R$ is $\F$, the vertex set $V(\R) \setminus V(\F)$ is complete to $H$.
	\end{fact}
	\begin{proof}
		By definition a $k$-corpus contains a stable set of size at least two and for every $j \in V(\R) \setminus V(\F)$, $N(j) \cap V(\F)$ contains two non-adjacent vertices.
		The result follows from the fact that $G$ is $C_4$-free.
	\end{proof}
	
	\section{Analyzing a maximal crowned corpus}

	\subsection{Vertices with neighbors in a maximal crowned corpus}
	
	\begin{thm}\label{thm:crowned-k-corpus-extra-vertex-only-can-have-nbrs-in-one-end}
		Let $G$ be an $\ell$-monoholed graph. Suppose $G$ does not contain a clique cut-set and suppose $G$ contains a $k$-spine for some $k \geq 3$. Let $\R$ be a crowned $k$-corpus in $G$ chosen to maximize $k$ and with respect to that to maximize $V(\R)$. Let $\mathcal{Z}$ be the core of $\R$.
		Let $v \in V(G) \setminus V(\mathcal{Z})$. Then either $v$ is complete to $V(\R)$, $N(v) \cap V(\mathcal{Z})$ is the vertex set of a clique or the neighbors of $v$ in $V(\mathcal{Z})$ are contained in a single elemental side of $\mathcal{R}$.
	\end{thm}
	\begin{proof}
		Let $\{ X_1, X_2, \dots, X_k\}$ and $\{ Y_1, Y_2, \dots, Y_k\}$ be the elemental side of $\R$.
		For each $i \in [k]$ let $\q_i$ denote the elemental path of $\R$ with ends $X_i, Y_i$.
		Suppose for some $v \in V(G) \setminus V(\z)$, $N(v) \cap V(\mathcal{Z})$ contains two non-adjacent vertices and $v$ is not complete to $V(\mathcal{R})$.
		Suppose for a contradiction $v$ contains a neighbor in an interior bag of an elemental path of $\R$. Let $\F$ be the $k$-corpus of $\R$.
		\stmt{$v$ is not complete to $V(\F)$ \label{corp:v-not-complete-to-F}}
		Suppose $v$ is complete to $V(\F)$. Then $v$ is not adjacent to some $r \in V(\R) \setminus V(\F)$. By definition of crowned $k$-corpus, we may assume $r$ has a neighbors $x, x' \in X_1 \cup X_2 \dots \cup X_k$ and $x_1$ and $x_2$ are non-adjacent. But then $r \dd x \dd v \dd x' \dd r$ is a hole of length four, a contradiction. This prove (\ref{corp:v-not-complete-to-F}).
		\stmt{The underlying distance of any two neighbors of $v$ in $V(\F)$ is at most two. 
			\label{corp:nbrs-in-F-in-a-P3}}
		Suppose $x,y$ are neighbors of $v$ in $V(\F)$ and suppose that the underlying distance between $x$ and $y$ is greater than two.
		By definition of $k$-corpus, we may assume there is some inflated hole $\mathcal{C}$ of length $\ell$ containing $x,y$ and the elemental paths $\q_1, \q_2$. 
		Then by Lemma \ref{lem:inflated-hole-vertex-neighbors-facts}, $v$ is complete to $V(\mathcal{C})$. By definition of $k$-corpus for any $w \in V(\F)$ there is an inflated hole $\mathcal{C'}$ of length $\ell$ containing $\q_1$ and $w$. It follows from Lemma \ref{lem:inflated-hole-vertex-neighbors-facts} that $v$ is complete to $V(\mathcal{C}')$. Hence $v$ is complete to $V(\F)$, contradicting (\ref{corp:v-not-complete-to-F}).
		This proves (\ref{corp:nbrs-in-F-in-a-P3}).
		\stmt{There is some bag $J$ of $\F$ such that every neighbor of $v$ in $V(\F)$ is contained in $J$ or a neighboring bag of $\F$. \label{corp:nbrs-or-v-in-F-are-in-some-bag-and-itsnbrs}}
		Suppose $v$ has neighbors in three pairwise non-adjacent bags $A, B, C$ such that no bag of $\F$ is adjacent to each of $A, B, C$.
		By (\ref{corp:nbrs-in-F-in-a-P3}) there are distinct bags $D_{AB}, D_{BC}, D_{AC}$ in $\F$ such that $A$ and $B$ are adjacent to $D_{AB}$, $B$ and $C$ are adjacent to $D_{BC}$ and $A, C$ are adjacent to $D_{AC}$.
		Let $H$ be the graph with vertex set $\{A, B, C, D_{AB}, D_{BC}, D_{AC} \}$ where two vertices in $V(H)$ are adjacent if and only if they are adjacent bags in $\F$. Then by definition $\F$ contains $H$.
		Consider the cycle $A \dd D_{AB} \dd B \dd D_{BC} \dd C \dd D_{CA} \dd A$. Since $H$ does not contain a hole of length at most 6, $D_{AB}$, $D_{BC}$ and $D_{AC}$ must be pairwise adjacent. See Figure \ref{fig:corpus-forbideen-isg}.
		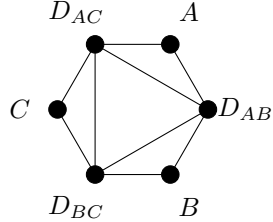
\begin{figure}[!h]
			\centering
			\input{Tikz/forbidden_induced_subgraph_of_a_corpus.tex}
			\caption{An illustration of the graph $H$ from the proof of (\ref{corp:nbrs-or-v-in-F-are-in-some-bag-and-itsnbrs}).}
			\label{fig:corpus-forbideen-isg}
		\end{figure}
		
		Since $D_{AB}. D_{BC}, D_{CA}$ have degree four in $H$ they cannot be contained in the interior of any constituent path of $\F$. 
		Hence $D_{AB}, D_{BC}, D_{CA}$ are the ends bags of some distinct constituent paths $\p_1, \p_2, \p_3$ of $\F$. Then by definition of $k$-corpus, $A, B, C$ cannot be interior bags of any constituent path of $\F$. Thus $A, B, C, D_{AB}, D_{BC}, D_{AC}$ are all contained in the same terminating set of $\F$. It follows that $H$ is a threshold graph. But the only way to partition $V(H)$ into the vertex set of a clique and a stable set is $\{A, B, C\}$ and $\{D_{AB}, D_{BC}, D_{BA} \}$ and the graph between the two sets is not a half graph, a contradiction. This proves (\ref{corp:nbrs-or-v-in-F-are-in-some-bag-and-itsnbrs}).

		\stmt{There is some bag $J$ in a elemental side of $\F$ such that every neighbor of $v$ in $\F$ is contained in $J$ or a neighboring bag of $J$. \label{corp:nbrs-of-V-in-F-are-in-some-elemental-bag-and-its-nbrs}}
		By (\ref{corp:nbrs-or-v-in-F-are-in-some-bag-and-itsnbrs}) there is some bag $J$ of $\F$ such that every neighbor of $v$ in $V(\F)$ is contained in $J$ or in a neighboring bag of $J$ in $\F$.
		
		Since $v$ has neighbors in the interior of an elemental path of $\F$, $J$ cannot be an apex of $\F$. 
		Suppose $J$ is an interior bag of some elemental path $\q_i$ of $\F$.
		Then by definition, $N(v) \cap V(\F) \subseteq V(\p_i)$. Then by definition of $k$-corpus there is some inflated hole $\C$ contained in $\F$ such that $\p_i$ is a sub-inflated graph of $\C$. 
		By applying Lemma \ref{lem:inflated-hole-vertex-neighbors-facts} we obtain that the graph obtained from $\C$ by adding $v$ to $J$ is another inflated $C_\ell$. But then the graph obtained from $\mathcal{R}$ by adding $v$ to $J$ is another crowned $k$-corpus, a contradiction. This proves (\ref{corp:nbrs-of-V-in-F-are-in-some-elemental-bag-and-its-nbrs}).
		\\
		\\
		Without loss of generality $J = X_1$. 
		Let $W_1$ denote the neighbor of $X_1$ in $\q_1$. By assumption, $v$ has a neighbor in $W_1$ and a neighbor in one of $X_2 \cup X_3 \cup \dots \cup X_k$. 
		For each $i \in [k]$ let $\p_i$ be the constituent path of $\F$ containing $\q_i$ and let $T_i$ denote the end of $\p_i$ that is equal or adjacent to $X_i$. Note that $X_i = T_i$ unless $T_i$ is an apex of $\F$.
		It follows that $X_1 = T_1$. From (\ref{corp:nbrs-of-V-in-F-are-in-some-elemental-bag-and-its-nbrs}), for any $i \in [2,k]$ if $T_i$ is an apex, $X_1$ and $X_i$ are not adjacent bags.
		\stmt{The graph $\mathcal{F'}$ obtained from $\F$ by adding $v$ to $X_1$ is an inflated graph and the underlying graphs of $\mathcal{F'}$ and $\F$ are isomorphic. \label{corp:adding-v-to-F-yields-an-infalted-graph} }
		Let $T_i$ be a neighboring bag of $X_1$ for some $i\in [2, k]$.
		By definition of $k$-corpus there is an induced inflated hole $\mathcal{C}$ in $\F$ containing $\q_1, \q_i$.
		Hence by Lemma \ref{lem:inflated-hole-vertex-neighbors-facts} and (\ref{corp:nbrs-in-F-in-a-P3}) the graph obtained from $\mathcal{C}$ by adding $v$ to $X_1$ is an inflated hole.
		Hence we need only show for any two $i,j \in [k]$ if $T_i$ and $T_j$ are neighboring bags of $X_1$ then the graph between $X_1 \cup \{ v\}$ and $T_i$ and the graph between $X_1 \cup \{ v\}$ and $T_j$ are both half graphs and they are compatible with respect to $X_1 \cup \{ v\}$.
		
		If $T_i$ and $T_j$ are not adjacent bags, $\F$  contains some inflated hole $\mathcal{C}'$ such that $\p_i, \p_j, X_1 \subseteq \mathcal{C'}$ so the result follows from Lemma \ref{lem:inflated-hole-vertex-neighbors-facts}. Hence, we may assume $T_i$ and $T_j$ are adjacent bags. But $X_1$ is complete to $T_i \cup T_j$ and the result follows.
		This proves (\ref{corp:adding-v-to-F-yields-an-infalted-graph}).
		\stmt{$X_1 \cup \{ v\}$ is mixed on one of $T_2, T_3, \dots, T_k$. \label{corp:A1-union-v-is-mixed-on-an-end-bag-of-constituent-path}}
		Since $|V(\F)|$ was chosen to be maximum, $\F'$ is not a $k$-corpus.
		By definition of $k$-corpus and (\ref{corp:nbrs-in-F-in-a-P3}), 
		There exists some $i \in [2,k]$ such that $X_1 \cup \{ v\}$ is mixed on $T_i$.
		This proves (\ref{corp:A1-union-v-is-mixed-on-an-end-bag-of-constituent-path})
		\stmt{Suppose $T_i$ is a neighboring bag of $X_1$ and for some $i \in [2,k]$ and $v$ has a non-neighbor in $x_i \in T_i$. 
			For every $j \in [2,k] \setminus \{ i\}$, if $T_j$ is a neighboring bag of $T_i$ then $v$ is anticomplete to $X_j$. \label{corp:v-has-restrictions-on-top-nbrs}}
		Suppose for some $j \in [2,k]$, $v$ has a neighbor $t_j \in T_j$ and $T_j$ is a neighboring bag of $T_i$. 
		Let $\C$ be an inflated $C_\ell$ contained in $\F$ as a sub-inflated graph such that $\p_1, \p_i \subseteq \C$. Then by definition of inflated graph $G[V(\p_1 \cup \p_i) \cup \{ v\}]$ contains a $vt_i$-path $R$ of length $\ell -1$.
		Since $T_j$ is a neighboring bag of $T_i$, they are complete to each other by definition of $k$-corpus. But then $V(R) \cup \{ t_j\}$ induces a hole of length $\ell + 1$, a contradiction. This proves (\ref{corp:v-has-restrictions-on-top-nbrs}).
		\\
		\\
		Let $\mathcal{T}$ denote the set $\{T_1, T_2, \dots, T_k\}$. By (\ref{corp:nbrs-of-V-in-F-are-in-some-elemental-bag-and-its-nbrs}), $X_1$ is not an apex and and since $T_1 = X_1$ this implies $|\mathcal{T}| > 1$.
		Without loss of generality $v$ has a neighbor in $T_2$.
		Let $F$ be the underlying graph of $\F$.
		Then since $|\z| > 1$,  the subgraph $W$ of $F$ induced by vertices corresponding to bags in $\z$ is a 2-connected threshold graph by definition of $k$-spine.
		Hence there is some $m \in [2,k]$, such that $T_m$ is complete to every other bag in $\mathcal{Z}$.
		Then by (\ref{corp:v-has-restrictions-on-top-nbrs}), $v$ must have a neighbor in $T_m$ because $v$ has a neighbor in $T_2$ and if $2 \neq m$ $T_2$ and $T_m$ are neighboring bags.
		Thus, it follows from (\ref{corp:v-has-restrictions-on-top-nbrs}) that if $T_i$ is a neighboring bag of $X_1$ for some $i \in [k] \setminus \{ m\}$, then $v$ is complete to $V(T_i)$.
		Thus by (\ref{corp:A1-union-v-is-mixed-on-an-end-bag-of-constituent-path}) and (\ref{corp:nbrs-of-V-in-F-are-in-some-elemental-bag-and-its-nbrs}), $v$ has a non-neighbor $t_m \in T_m$. Since $W$ is a two connected threshold graph, $T_1$ has a neighboring bag $T_i$ for some $i \in [k] \setminus \{ 1, m\}$ and so $v$ is complete to $T_i$. But then $Z_m$, $T_i$ contradict (\ref{corp:v-has-restrictions-on-top-nbrs}).
	\end{proof}
	
	\subsection{Paths with neighbors in a maximal crowned corpus}
	\begin{thm}
		Let $G$ be an $\ell$-monoholed graph for some $\ell \geq 7$. Suppose that $G$ contains a $k$-spine $F$ for some $k \geq 3$. Choose $F$ to maximize $k$. Let $A, B$ be the terminating sets of $S$. Let $W$ be an induced path $w_1 \dd w_2 \dd \dots \dd w_n$ in $G \setminus S$ of length at least one satisfying:
		\begin{itemize}
			\item $W^*$ is anticomplete to $V(F)$
			\item $N(q_1) \cap V(F)$ and $N(q_2) \cap V(F)$ are both vertex sets of cliques.
		\end{itemize}
		Then $q_1, q_n$ are both anticomplete to one of $A, B$.
		\label{thm:pathcases:no-top-bottom-path}
	\end{thm}

	\begin{proof}
		Suppose neither $A$ nor $B$ is anticomplete $\{ q_1, q_n \}$.
		By definition of $k$-spine, $A$ and $B$ are anticomplete to each other. We may assume that $q_1$ has a neighbor in $A$ and $q_2$ has a neighbor in $B$. 	For every two distinct $i,j \in [4]$ let $A_{ij}$ be a shortest $a_ia_j$-path in $G[A]$ and let $B_{ij}$ be a shortest $b_ib_j$-path in $G[B]$. Since we will only consider $i, j \leq 4$ in this proof this notation is not ambiguous.
		
		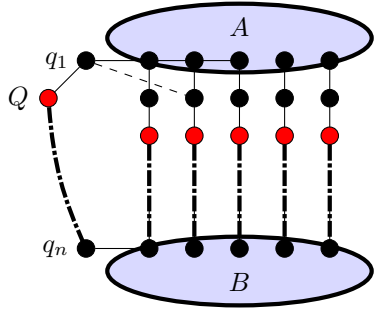
\begin{figure}[!h]
			\centering
			\input{Tikz/temp_fig.tex}
			\caption{An illustration of the a case where (i) and (ii) from (\ref{paths:topbottomcases}) both do not hold and $k = 5$. The figure contains a mated $(k+1)$-spider where the toes are the red vertices.
			}
			\label{fig:addinglongpaths}
		\end{figure}
		
		\stmt{\label{paths:topbottomcases} 
			One of the following statement holds:
			\begin{enumerate}[(i)]
				\item There exists $i  \in [k]$ and $j \in \{1, n\}$ such that $P_i$ has length two and $q_j$ is adjacent to the internal vertex of $P_i$ or
				\item $n=2$ and $q_1, q_2$ both have at least two neighbors in $V(S)$.
			\end{enumerate}
		}
		
		We may assume that $q_1$ is adjacent to $a_1$ and $q_n$ is adjacent to $b_j$ for some $j \in [k]$. Suppose, $G[(V(Q \cup S)]$ contains a pair of mated $(k+1)$-spiders. Then by Theorems \ref{thm:if-l-odd-G-has-pyramid} and \ref{thm:orderly-spines-even-case}, $G[V(Q \cup Q)]$ is a $(k+1)$-spine, a contradiction. Thus $G[V(Q \cup S)]$ does not contain a pair of mated $(k+1)$-spiders.

		For each $i \in [k]$ choose $t_i \in P_i^*$.
		Let $X_i, Y_i$ be the $a_it_i$ and $t_ib_i$-paths of $P_i$, respectively. Choose $t_i$ to be non-adjacent to both $q_1, q_n$, if possible.
		Let $X_i, Y_i$ denote the $t_ia_i$ and $t_ib_i$-paths of $P_i$, respectively. Then $V(\cup_{i=1}^k X_i)$ induces a $k$-spider $S_X$ and $V(\cup_{i=1}^k Y_i)$ induces a $k$-spider $S_Y$. Moreover, $S_X$ and $S_Y$ are mated to each other and have toes 	$t_1, t_2, \dots, t_k$.
		
		Suppose (i) does not hold. Then none of $t_1, t_2, \dots, t_k$ have a neighbor in $V(Q)$. 
		Suppose $n > 2$. Then $G[\{q_1, q_2\} \cup V(S_X)]$ contains a $(k+1)$-spider $S_X'$ with toes $q_2, t_1, t_2, \dots, t_k$ and $G[\{q_2, q_3, \dots, q_n\} \cup V(S_Y)]$ contains a $(k+1)$-spider $S_Y'$ with toes $q_2, t_1, t_2, \dots t_k$.
		$S_X'$ and $S_Y'$ cannot be mated to each other so $V(S_X') \setminus \{q_2, t_1, t_2, \dots , t_k\}$ and $V(S_Y') \setminus \{t_1, t_2, \dots, t_k\}$ are not anticomplete to each other. Since $Q^*$ is anticomplete to $V(S)$, we may assume $q_1$ has a neighbor in both $V(S_X') \setminus \{ t_1, t_2, \dots, t_k\}$ and $V(S_Y') \setminus \{t_1, t_2, \dots, t_k\}$. But then $N(q_1) \cap V(S)$ contains two non-adjacent vertices, a contradiction. Hence $n = 2$. See Figure \ref{fig:addinglongpaths} for an illustration of this case.
		
		Suppose $q_1$ has exactly one neighbor in $V(S)$. Then $G[V(S_X) \cup \{ q_1\}]$ and $G[V(S_Y) \cup \{ q_1, q_2\}]$ contain $(k+1)$-spider $S_X''$ and $S_Y''$ respectively. Moreover since $N(q_2) \cap V(S)$ is the vertex set of a clique $S_X''$ and $S_Y''$ are mates, a contradiction. This proves (\ref{paths:topbottomcases}).
		\\
		\\
		We will first show that statement (i) does not hold and then we will show (ii) cannot hold for a contradiction.
		\stmt{If some constituent path of $S$ has length two then $\ell = 8$ and every constituent path has length two or three.\label{pathcases:constituent-path-length-two-then-ell-is-8}}
		By definition of $k$-spine every constituent path has a length equal to $\frac{\ell -1}{2}$ if $\ell$ is odd or $\frac{\ell}{2}, \frac{\ell}{2} -1, \frac{\ell}{2} -2$ if $\ell$ is even and the difference in length between any two constituent paths it at most one.  Since $\ell \geq 7$ it follows that $\ell = 8$. Hence, every constituent path of $S$ has length $2$ or $3$. This proves (\ref{pathcases:constituent-path-length-two-then-ell-is-8}).
		\begin{figure}[!h]
			\centering
			\input{Tikz/pathcases_some_path_has_length_2.tex}
			\caption{An illustration of (\ref{pathcases:no-path-with-q1-qn-covering-all-vertices-in-it}). $C$ is drawn as the outer face. Every hole in the picture must have length $\ell$, so we reach a contradiction.}
			\label{fig:pathscase-qn-adjacent-to-b1}
		\end{figure}
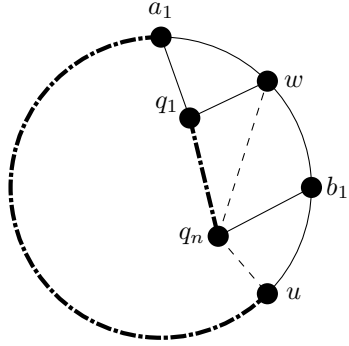
		\stmt{If for some $i \in [k]$, $P_i$ has length two and $q_1$ is adjacent to the interior vertex of $P_i$ then $q_n$ is anticomplete to $V(P_i)$. The same statement holds with $q_1$ and $q_n$ exchanged.. \label{pathcases:no-path-with-q1-qn-covering-all-vertices-in-it}}
		Suppose $P_1$ has length two. Then by (\ref{pathcases:constituent-path-length-two-then-ell-is-8}), $\ell = 8$. Let the vertices of $P_1$ be $a_1 \dd w \dd b_1$, in order.
		Suppose $q_1$ is adjacent to $a_1, w$ and $q_n$ has a neighbor in $V(P_1)$. Then $q_n$ is adjacent to $b_1$ since $N(q_n) \cap V(S)$ is the vertex set of a clique and $q_n$ has a neighbor in $B$. Since $N(q_1) \cap V(S)$ is a vertex set of a clique, $N(q_1) \cap S = \{ a_1, w\}$.
		Then $q_1 \dd q_2 \dots \dd q_n \dd b_1 \dd w \dd q_1$ is a cycle of length $|E(Q)|+ 3$.
		
		Suppose $G[V(Q) \cup \{ w, b_1\}]$ contains a hole. $|E(Q)| \geq \ell - 3 > 3$ or $G[V(Q) \cup \{ w, b_1\}]$ does not contain a hole.
		Let $C$ be a hole of $S$ containing $P_1$. Let $u$ be the neighbor of $b_1$ in $C \setminus \{ w \}$. 
		Then if $q_n$ is not adjacent to $u$, $G[V(C \cup Q) \setminus \{ x\}]$ is a hole of length $\ell + |E(Q)| -2$ and if $b_1$ is adjacent to $u$, $G[V(C \cup Q) \setminus \{ x, u\}]$ is a hole of length $\ell + |E(Q)| - 3$.
		In either case $G$ contains a hole of length greater than $\ell$, a contradiction.
		
		Hence $G[V(Q) \cup \{ w, b_1\}]$ is chordal. Hence $n = 2$ and $w$ is adjacent to $q_2$.
		Then $q_1$ is not adjacent to $u$. Then union of $a_1 \dd q_1 \dd q_2 \dd b_1$ and $C \setminus w$ is a hole of length at least $\ell + 1$, a contradiction.
		See Figure \ref{fig:pathscase-qn-adjacent-to-b1} for an illustration of this argument.
		This proves (\ref{pathcases:no-path-with-q1-qn-covering-all-vertices-in-it}).
		\begin{figure}
			\centering
			\input{Tikz/pathcases_not_i.tex}
			\caption{An illustration of the proof of (\ref{pathcases:not_i}). The vertices $a \in A$ and $b \in B$ exist because $G[A]$ and $G[B]$ both do not contain a cut vertex.}
		\end{figure}
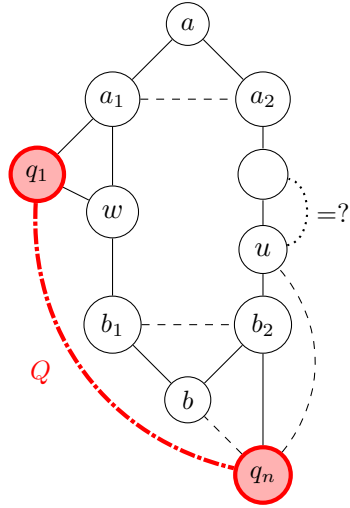
		\stmt{(i) does not hold. \label{pathcases:not_i}}
		Suppose for a contradiction that $P_1$ has length two and $q_1$ is adjacent to $a_1$ and the interior vertex of $P_1$. Then by (\ref{pathcases:constituent-path-length-two-then-ell-is-8}), $\ell = 8$.
		Let the vertices of $P_1$ be $a_1 \dd w \dd b_1$, in order. Then $N(q_1 )\cap S = \{ a_1, w \}$. 
		By (\ref{pathcases:no-path-with-q1-qn-covering-all-vertices-in-it}), $q_n$ is not adjacent to $b_1$. So we may assume $q_n$ is adjacent to $b_2$ and $a_2 \neq a_1$. Thus $q_n$ is not adjacent to $w$.
		The union $B_{12}$ and the path $b_1 \dd w \dd q_1 \dd q_2 \dd \dots \dd q_n \dd b_2$ is a cycle $C$ of length $|E(Q)| + |E(B_{12})| + 3$ and $G[V(C)]$ contains a hole. Since $B_{12}$ is a shortest $b_1b_2$-path in $G[B]$, $|E(B_{12})| \leq 2$. Thus since $\ell= 8$ it follows that $|E(Q)| \geq 3$. See Figure \ref{pathcases:not_i}.

		Let $C'$ be the cycle consisting of $P_2 \cup A_{12} \cup Q$ and the edges $q_1a_1, q_nb_2$. Then $C$ has length $|E(Q)| + |E(P_2)| + |E(A_{12})| + 2$. 
		Since $\ell = 8$, by definition of $k$-spine if $P_2$ has length two $R'$ has length two. $|E(A_{12})| \geq 1$ because $a_1 \neq a_2$. Thus $C'$ has length at least $|E(Q)| + 6 \geq 9 > \ell$, so it is not a hole. Let $u$ denote the neighbor of $b_2$ in $P_2$. Since $C'$ is not a hole, $q_n$ is adjacent to $u$ and so $N(q_n) \cap V(S) = \{b_2, u\}$.
		Then $G[V(C') \setminus \{ u\}]$ is a hole of length $|E(C')| - 1$, so $|E(Q)| \leq 3$. Thus, $|E(Q)| = 3$.
		
		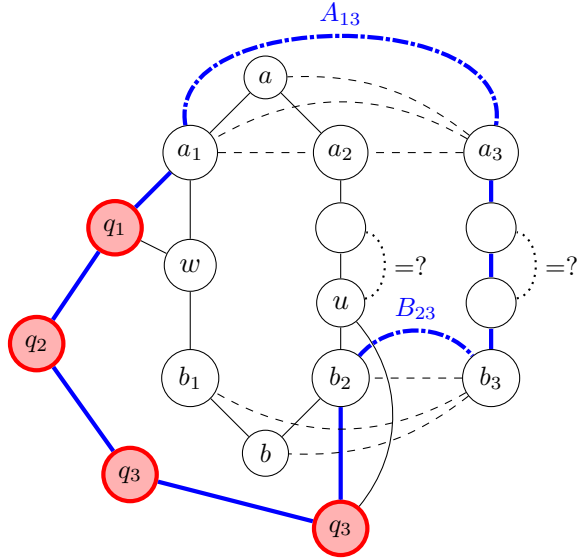
\begin{figure}
			\centering
			\input{Tikz/pathcases_not_i_part2.tex}
			\caption{An illustration for the proof of statement \ref{pathcases:not_i} of Theorem \ref{thm:pathcases:no-top-bottom-path}. $C''$ is drawn with blue edges. Note $a,a_2$ maybe equal or adjacent to vertices in $V(A_{13})$ and $b$ maybe equal or adjacent to vertices in $V(B_{23})$. $A_{23}$ and $B_{23}$ each have length at least one so $C''$ has length at least $9$, a contradiction.}
			\label{fig:pathcases:not_i_pt2}
		\end{figure}
		
		The union of the path $a_1 \dd q_1 \dd q_2 \dd q_3 \dd \dots \dd q_n \dd b_2$ and $B_{23} \cup P_3 \cup A_{13}$ is a hole $C''$ of length $|E(Q)| + |E(A_{23})| + |E(P_3)| + |E(A_{13})| + 2$.
		Since $b_2$ is the only neighbor of $q_n$ in $B$, $C$ is a hole so $\ell = E(Q)| + |E(B_{12})| + 3$. Thus $|E(B_{12}| = 2$ and in particular $b_2$ is not adjacent to $b_1$ so $b_1, b_2$ are not multipurpose from the definition of $k$-skeleton.
		Since $P_1$ has length two and $\ell = 8$, $a_1$ cannot be multipurpose from the definition of $k$-skeleton. Thus $A_{13}$ and $B_{23}$ both have length one or two.
		It follows that $|E(C'')| \geq |E(Q)| + 6 = 9 > \ell$, a contradiction. See Figure \ref{fig:pathcases:not_i_pt2}.
		This proves (\ref{pathcases:not_i}). 
		\\
		\\
		By (\ref{paths:topbottomcases}), it follows that $n = 2$ and $q_1, q_2$ both have at least two neighbors in $V(S)$.
		\stmt{Every constituent path of $S$ has length at most $\ell -4$. \label{pathcases:paths-are-at-most-l-minus-4}}
		Suppose $P_1$ has length $\ell -3$. Then by definition of $k$-spine every path of $S$ has length at least $\ell -4$. 
		$P_1 \cup A_{12} \cup P_2 \cup B_{12}$ is a hole of length $\ell - 3 + |E(P_2)| + |E(A_{12})| + |E(B_{12})|$.
		So $\ell \geq 2\ell - 7 + |E(A_{12})| + |E(B_{12})|$. So, $$7 - |E(A_{12})| - |E(B_{12})| \geq \ell = \ell - 3 + |E(P_2)| + |E(A_{12})| + |E(B_{12})| \geq 7$$ 
		Hence $|E(A_{12}| = |E(B_{12})| = 0$ and $|E(P_2)| = \ell -4$.
		Since $P_1$ and $P_2$ have different lengths $S$ is not a $k$-theta and in particular for each $i \in [k]$ at most one of $a_i, b_i$ is a multipurpose-vertex by definition of $k$-spine.
		But $|E(A_{12})| = |E(B_{12})| = 0$,  so $a_1 = a_2$ and $b_1 = b_2$, a contradiction. This proves (\ref{pathcases:paths-are-at-most-l-minus-4}).
		\stmt{For every $i \in [k]$, $q_1$ is not adjacent to $a_i$ or $q_2$ is not adjacent to $b_i$. \label{pathcases:Q-cup-Pi-not-a-cycle}}
		Suppose $a_1$ is adjacent to $q_1$ and $b_1$ is adjacent to $q_2$. Then $G[V(Q \cup P_1)]$ contains a hole so $|E(P_1)| \leq \ell -1$, contradicting (\ref{pathcases:paths-are-at-most-l-minus-4}). 
		This proves (\ref{pathcases:Q-cup-Pi-not-a-cycle}).
		\\
		\\
		Since $q_1$ has a neighbor in $A$ and $b_1$ has a neighbor in $B$ it follows from (\ref{pathcases:Q-cup-Pi-not-a-cycle}) that $|A|, |B| \geq 2$.
		\stmt{$q_1, q_2$ are both anticomplete to $P_1^* \cup P_2^* \cup \dots \cup P_k^*$.\label{pathcases:no-interior-neighbors}}
		Let $a_1'$ denote the neighbor of $a_1$ in $P_1$. Suppose $q_1$ is adjacent to both $a_1$ and $a_1'$.
		We may assume $q_2$ is adjacent to $b_2$ and thus $b_2 \neq b_1$ and $a_2 \neq a_1$ by (\ref{pathcases:Q-cup-Pi-not-a-cycle}).
		Suppose $b_1$ is adjacent to $b_2$. Let $C$ denote the union of $a_1' \dd q_1 \dd q_2 \dd b_2 \dd b_1$ and $P_1 \setminus a_1$. Then $C$ is a hole of length $|E(P_1)| + 3$. So $|E(P_1)| = \ell -3$, contradicting (\ref{pathcases:paths-are-at-most-l-minus-4}) . 
		Hence $b_1$ and $b_2$ are not adjacent. Moreover by the same argument $q_2$ is not adjacent to any neighbor of $b_1$ in $B$.
		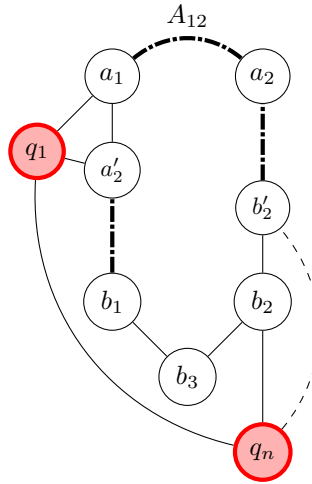
\begin{figure}[!h]
			\centering
			\input{Tikz/pathcases_anticomplete_to_interior_paths.tex}
			\caption{An illustration of $G[V(C' \cup Q)]$ from the the proof of Statement \ref{pathcases:no-interior-neighbors} of Theorem \ref{thm:pathcases:no-top-bottom-path}.}
			\label{pathcases:fig_no_interior_nbrs}
		\end{figure}
		
		Since $|B| \geq 2$ and $G[B]$ is a 2-connected threshold graph we may assume $b_3$ is adjacent to both $b_1$ and $b_2$. Hence $b_3$ is not adjacent to $q_2$. 
		Let $C'$ denote the union of $P_1 \cup A_{12} \cup P_2$ and $b_1 \dd b_3 \dd b_2$. Then $C'$ is a hole.
		Since $N(q_1) \cap V(S)$ and $N(q_2) \cap V(S)$ are vertex sets of cliques $a_1', a_1$ are the only neighbors of $q_1$ in $V(C')$.
		Let $b_2'$ denote the neighbor of $b_2$ in $P_2$.
		Then $b_2, b_2'$ are the only possible neighbors of $q_2$ in $V(C')$.
		
		Suppose $q_2$ is not adjacent to $b_2'$.
		$G[V(C \cup Q)]$ is a pyramid, so it contains an odd hole
		Thus $\ell$ is odd and $q_1 \dd q_2 \dd b_2$ is a constituent path of $G[V(C \cup Q)]$ so $\ell = 6$.
		Hence $q_2b_2 \in E(G)$.
		
		Then $G[V(C \cup Q)]$ is a prism so $\ell$ is even and at least eight. Hence, every constituent path of $G[V(C \cup q]$ has length $n$ for some $n \geq 3$.
		But $q_1 \dd q_2$ is a constituent path of $G[V(C \cup Q)]$, a contradiction. See Figure \ref{pathcases:fig_no_interior_nbrs} for an illustration of this argument.
		This proves (\ref{pathcases:no-interior-neighbors}).
		\\
		\\
		Since $q_1, q_2$ both have at least two neighbors in $V(S)$ we may assume all of the following:
		\begin{itemize}
			\item $k \geq 4$,
			\item  $q_1$ is adjacent to $a_1, a_2$, 
			\item $q_1$ is not adjacent to $a_3, a_4$, 
			\item $q_2$ is adjacent to $b_3, b_4$,
			\item $q_2$ is not adjacent to $b_1, b_2$ and 
			\item $a_1, a_2, a_3, a_4, b_1, b_2, b_3, b_4$ are all distinct vertices.
		\end{itemize}
		Hence $A_{ij}$ and $B_{ij}$ both have length one or two for every two distinct $i, j \in [k]$. See Figure \ref{fig:pathcases_ending} for an illustration.
				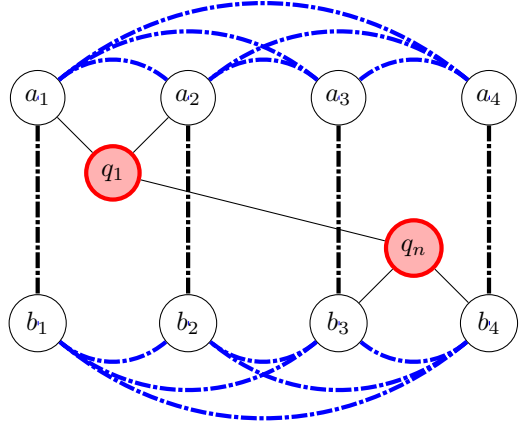
\begin{figure}[!h]
			\centering
			\input{Tikz/pathcases_ending_case.tex}
			\caption{An illustration for the conclusion of the proof of Theorem \ref{thm:pathcases:no-top-bottom-path}. For distinct $i,j \in [4]$, $A_{ij}$ and $B_{ij}$ are drawn in blue. Note this is a simplified drawing; for distinct $i',j' \in [4]$, vertices in $V(A_{ij})$ and $V(B_{ij})$ may be equal or adjacent to vertices in $V(A_{i'j'})$ and $V(B_{i'j'})$, respectively.}
			\label{fig:pathcases_ending}
		\end{figure}
		
		Let $i \in \{1,2\}$ and $j \in \{ 3, 4\}$. 
		Then the graphs $a_i \dd P_i \dd b_i \dd B_{ij} \dd b_j \dd q_2 \dd q_1 \dd a_i$ and $b_j \dd P_j \dd a_j \dd A_{ij} \dd a_i \dd q_1 \dd q_2 \dd b_j$ are holes of length $|E(P_i)| + |E(B_{ij})| + 3$ and $|E(P_j)| + |E(A_{ij})| + 3$ respectively.
		Hence, $\ell -3 = |E(P_i)| + |E(B_{ij})| = |E(P_j)| + |E(A_{ij})|$. 
		But then, $P_1 \cup A_{12} \cup P_2 \cup B_{12}$ is a hole of length $2 \ell -6$. So $2\ell - 6 \leq \ell$. But $\ell > 6$, a contradiction.	
	\end{proof}
	\begin{corr}
		Let $G$ be an $\ell$-monoholed graph for some $\ell \geq 7$. Suppose $G$ does not contain a clique cut-set and suppose $G$ contains a $k$-spine for some $k \geq 3$. Let $\R$ be a crowned $k$-corpus in $G$ chosen to maximize $k$ and with respect to that to maximize $V(\R)$.
		Let $H, I$ be the crowns of $\R$.
		Let $W$ be an induced path $w_1 \dd w_2 \dd \dots \dd w_n$ in $G \setminus V(\R)$ of length at least one satisfying:
		\begin{itemize}
			\item $W^*$ is anticomplete to $V(\R)$
			\item $N(w_1) \cap V(\R)$ and $N(w_n) \cap V(\R)$ are both vertex sets of cliques.
		\end{itemize}
		Then $w_1, w_n$ are both anticomplete to one of $V(H), V(I)$.
		\label{corr:fml}
	\end{corr}
	\begin{proof}
		Suppose $V(H)$, $V(I)$ both have neighbors in $\{ w_1, w_n\}$.
		Then we may assume $w_1$ has a neighbor in $V(H)$ and $w_n$ has a neighbor in $V(I)$.
		Let $\F$ be the $k$-corpus of $\mathcal{R}$.
		Let $F$ be the $k$-corpus underlying $\mathcal{F}$. 
		Let $A, B$ be the terminating sets of $F$. Then by definition, we may assume $A \subseteq V(H)$ and $B \subseteq V(I)$. Let $P_1, P_2, \dots, P_k$ denote the constituent paths of $F$. For each $i\in [k]$, let the ends of $P_i$ be $a_i, b_i$ where $a_i \in A$ and $b_i \in B$.
		
		By Theorem \ref{thm:pathcases:no-top-bottom-path}, we may assume $w_1$ is anticomplete to $A$. Hence, $w$ is anticomplete to $V(F)$.
		$G[V(\R \cup W)]$ does not contain a mated $(k+1)$-spider.
		Thus by definition of crowned corpus we may assume $w_n$ is complete to interior of $V(P_1)$.
		Hence $|E(P_1)| = 2$ and thus $\ell =7$ or $\ell = 8$. Then, $a_1, b_1$ cannot be apexes of $F$.
		
		By definition $w_1$ has some neighbor $h \in V(H)\setminus V(A)$ such that $h$ has a neighbor in $A$.
		By Theorem \ref{thm:pathcases:no-top-bottom-path}, $N(h) \cap V(A)$ is not the vertex set of a clique. In particular, we may assume $h$ is adjacent to $a_2$.
		Let $C$ be a hole in $F$ containing $P_1, P_2$.
		Consider the graph $G[V(W \cup C) \cup \{ h\}]$ depicted in Figure \ref{fig:fml}.
		By definition $N(w_n) \cap V(C) = \{ v_1, b_1\}$.
		Since $a_1$ is not an apex $a_1 \neq a_2$.
		
		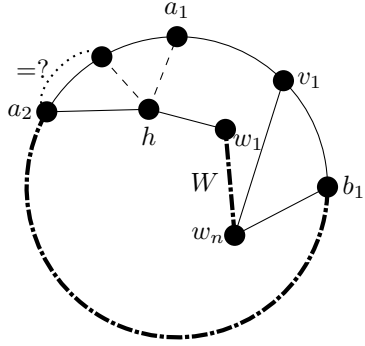
\begin{figure}[!h]
			\centering
			\input{Tikz/fml.tex}

			\caption{The graph $G[V(W \cup C) \cup \{ h\}]$ from Corollary \ref{corr:fml} is drawn with $C$ as the outer face.}
			\label{fig:fml}
			
		\end{figure}

		If $a_1$ and $a_2$ are not adjacent by definition of $k$-spine $a_1$ and $a_2$ have a common neighbor $a_i$.
		If $a_1$ and $a_2$ are adjacent let $i =1$.
		By definition, $N(h) \cap V(C) \subseteq \{ a_1, a_2, a_i\}$.
		Then the union of $C \setminus \{ a_i , a_1, v_1, b_1\}$ and the path $a_2 \dd h \dd w_1 \dd w_2 \dd \dots \dd w_n \dd b_1$ is a hole of length $|E(C)| + |E(W)| -1 + \mathbb{X}_{i =1}$ where $\mathbb{X}_{i=1}$ is equal to $1$ if $i =1$ and zero otherwise. Since $G$ is $\ell$-monoholed it follows that $n =2 - \mathbb{X}_{i=1}$. Hence $n \leq 2$, so $n =2$. Thus, $\mathbb{X}_{i=1}  = 0$ so $a_1$ and $a_2$ are not adjacent.
		Hence $h$ is not adjacent to $a_1, a_i$ for otherwise $G[\{ a_i, a_1, v_1, b_1, w_2, w_1, h \}]$ includes a hole of length at most six, a contradiction.
		Thus $G[V(W \cup C) \cup \{ h\}]$ is a pyramid $Z$, so $\ell$ is odd.
		Let $R$ denote the $b_2b_1$-path of $C$ not containing $v_1$.
		Then $P_2 \cup R$ is a constituent path of $Z$ so it has length three.
		Since $P_2$ has length at least two, $R$ has length at most one.
		
		The neighbors of $h$ in $V(A)$ do not induce a clique, so we may assume $h$ is adjacent to $a_3$ and $a_3$ is not adjacent to $a_2$.
		Let $C'$ be a hole of $F$ containing $P_2$ and $P_3$.
		Let $a_j$ be the common neighbor of $a_2, a_3$ in $C'$.
		Then $V(C' \setminus a_j) \cup \{ h\}$ induces a hole $C''$.
		Let $R$ be a shortest path in $G[B]$ from $b_1$ to $\{ b_2, b_3\}$.
		Then $V(C' \cup B \cup W)$ induces a pyramid or a theta $H$.
		So since $\ell$ is odd, $H$ is a pyramid.
		Then $P_2$ is a constituent path of $H$, so it has length three.
		Hence $R$ has length zero, so $b_2 = b_1$.
		But then the union of $P_1, P_2, P_i$ is a theta, contradicting that $\ell$ is odd.
	\end{proof}

	\begin{thm}\label{thm:cant-add-path-to-maximal-skeleton}
	Let $\ell \geq 7$ be an integer.
		Let $G$ be an $\ell$-monoholed graph and suppose $G$ contains a $k$-spine $S$ for some $k \geq 3$. Choose $S$ to maximize $k$. Suppose there is some path $Q$ in $G \setminus V(S)$ such that there are two non-adjacent vertices in $N(V(Q)) \cap V(S)$. Then there exists some $v \in V(Q)$ such that $N(q) \cap V(S)$ contains two non-adjacent vertices.
	\end{thm}

	\begin{proof}
		Choose $Q$ to be a minimal path in $G \setminus V(S)$ such that there are two non-adjacent vertices in $V(S)$ with neighbors in $V(Q)$.
		Let the vertices of $Q$ be $q_1 \dd q_2 \dd \dots \dd q_n$, in order. Suppose for each $i \in [n]$ that $N(q_i) \cap V(S)$ is empty or the vertex set of a clique. Then, $n > 1$. 
		Let $P_1, P_2, \dots, P_k$ denote the constituent paths of $S$ and let $A, B$ denote the terminating sets of $S$. For each $i \in [k]$ let $a_i \in A$, $b_i \in B$ be the ends of $P_i$.
		By the minimality of $P$ we may assume there exist non-adjacent $u, v \in V(S)$ such that $u$ is adjacent to $q_1$ and no other vertex in $V(Q)$ and $v$ is adjacent to $q_n$ and no other vertex in $V(Q)$.
		\stmt{$q_1$ and $q_n$ do not have a common neighbor in $V(S)$ and $Q^*$ is anticomplete to $V(S)$. \label{paths:ends-have-no-common-nbr}}
		Let $X$ be the set of common neighbors of $q_1$ and $q_n$ in $V(S)$ and suppose $X \neq \emptyset$, then $X$ is the vertex set of a clique. Let $x \in X$. Then $u, v$ are both adjacent to $x$. By definition of $k$-spine there is some hole $C$ of length $\ell$ containing the path $u \dd x \dd v$. It follows that $C$ contains no vertex in $X \setminus \{ x\}$. By minimality of $Q$, for every $i \in [2, n-1]$, the set of neighbors of $q_i$ in $V(S)$ is contained in $X$. But then $G[V(C \cup Q) \setminus \{x \}]$ is a hole of length greater than $\ell$, a contradiction. Thus $X = \emptyset$ and it follows that $Q^*$ is anticomplete to $V(S)$. This proves (\ref{paths:ends-have-no-common-nbr}).
		
		\stmt{ $N(q_1) \cap V(S)$ and $N(q_n)\cap V(S)$ are anticomplete.\label{paths:anticomplete-nbrs}}
		By (\ref{paths:ends-have-no-common-nbr}) we need only show that there is no edge between a neighbor of $q_1$ in $V(S)$ and a neighbor of $q_n$ in $V(S)$.
		Suppose there are some adjacent $x, y \in N(S)$ with $x$ adjacent to $q_1$ and $y$ adjacent to $q_n$.
		Then by (\ref{paths:anticomplete-nbrs}), $V(Q) \cup \{ x,y\}$ induces a hole. Hence, $Q$ has length $\ell -3$.
		Let $C$ be a hole of $S$ containing $u$ and $v$. Let $R_1, R_2$ be the two $u,v$-paths of $C$. By (\ref{paths:anticomplete-nbrs}) for $i \in \{1,2\}$ $G[V(Q \cup  R_i)]$ contains a hole $C_i$ of length $\ell$. Since $N(q_1) \cap V(C)$ and $N(q_2) \cap V(C)$ both induce cliques, $|E(C_1) \setminus E(R_1)| + |E(C_2) \setminus E(R_2) \leq \ell -2$.
		Then, $$2 \ell = |E(C_1)| + |E(C_2)| \geq 2|E(Q)| + |E(R_1)| + |E(R_2)| + 2 = 2|E(Q)| + \ell + 2 = 3 \ell -4$$ But $\ell > 4$, a contradiction. This proves (\ref{paths:anticomplete-nbrs}).
		\\
		\\
		Thus by Theorem \ref{thm:pathcases:no-top-bottom-path}, we may assume $q_1$ is anticomplete to $A \cup B$. Without loss of generality, $N(q_1) \cap V(S) \subseteq P_1^*$.
		Let $\alpha_1, \beta_1$ denote the neighbors of $q_1$ in $V(P_1)$ with minimum $P_1$-distance to $a_1$ and $b_1$, respectively.
		\stmt{$q_n$ is anticomplete to $V(P_1)$ \label{paths:qn-does-not-have-a-nbr-in-P1}}
		Suppose $q_n$ has a neighbor in $V(P_1)$. Let $x,y \in V(P_1)$ such that $q_1$ is adjacent to $x$, $q_n$ adjacent to $y$. Choose $x,y$ to maximize the $P_1$-distance between $x$ and $y$. Let $R$ be the $xy$-path contained in $P_1$.
		Let $C$ be the $\ell$-hole in $S$ containing $P_1$ and $P_2$. Then $P_1 \setminus R^* \cup Q$ is a hole of length $\ell - |E(R)| + |E(Q)| + 2$. By definition of $k$-spine, $|E(P_1) \leq \frac{\ell}{2}$. Hence, $|E(Q)| < \frac{\ell}{2} $. But $G[V(R \cup Q)]$ contains a hole of length at most $|E(R)| + |E(Q)|$ and $|E(R)| + |E(Q)| < \ell$, a contradiction. This proves (\ref{paths:qn-does-not-have-a-nbr-in-P1}).
		\\
		\\
		Without loss of generality $q_n$ has a neighbor in $V(P_2)$. Let $\alpha_2, \beta_2$ denote the neighbors of $q_n$ in $V(P_2)$ with minimum $P_2$-distance to $a_2$ and $b_2$, respectively. For $i \in \{ 1,2\}$,
		let $A_i$, $B_i$ denote the paths of $P_i$ with ends $a_i, \alpha_i$ and ends $\beta_i, b_i$, respectively.
		Hence, $|E(A_i)| + |E(B_i)|$ is equal to $|E(P_i)| -1$ or $|E(P_i)|$.
		
		Let $C$ be a hole in $S$ containing $P_1$ and $P_2$ and let $J$ be the graph induced by $V(C \cup Q)$.
		Then $J$ is a theta, prism or pyramid and the constituent paths of $J$ have length $\frac{\ell}{2}$, $\frac{\ell-1}{2}$ or $\frac{\ell}{2} -1$. Since there is a constituent path of $J$ consisting of $Q$ and at most two more edges, $|E(Q)| \geq \frac{\ell}{2}-3$. 

		Let $X$ be the path $A_1 \cup Q \cup B_2 \cup \{\alpha_1q_1, \beta_2q_n\}$ and let $Y$ be the path $A_2 \cup Q \cup B_1 \cup \{ \beta_1q_1, \alpha_1q_1\}$.
		For $i \in \{1,2\}$,  $|E(A_i)| + |E(B_i)|$ is equal to $|E(P_i)| -1$ or $|E(P_i)$ and by definition of $k$-spine  $P_1, P_2$ each have length at least $\frac{\ell}{2} -1$. Thus, one of $|E(A_1 \cup B_2)|$ or $|E(A_2 \cup B_1)|$ is at least $\frac{\ell}{2}-2$.
		Thus we may assume $|E(X)| \geq \ell -3$.
		\stmt{$q_n$ is adjacent to $a_i$ for every $i \in [2,k]$ \label{paths:qn-has-lots-of-nbrs-on-top}}
		Suppose $q_n$ is not adjacent to $a_3$.
		Since $G[A]$ is a connected threshold graph there is a $a_1a_3$-path $M$ of length at most two in $G[A]$.
		Since $G[B]$ is a connected threshold graph there is a $b_2b_3$-path $M'$ of length at most two in $G[A]$.
		Then $Q \cup X \cup P_3$ is a hole of length greater than $\ell$, since $\ell > 7$, a contradiction. 
		It follows from the fact that $N(q_n) \cap V(S)$ does not contain two non-adjacent vertices and the fact that $q$ has a neighbor in $V(P_2)$ that $q_n$ is adjacent to $a_1$. This proves (\ref{paths:qn-has-lots-of-nbrs-on-top}).
		\\
		\\
		Since $N(q_n)$ does not contain two non-adjacent vertices in $V(S)$, it follows that $N(q_n) \cap V(S) = \{ a_2, a_3, \dots, a_k\}$ and $\{a_2, a_3, \dots, a_k\}$ is the vertex set of a clique. Moreover, since $q_n$ is not adjacent to $a_1$, $|A| \geq 2$. Thus, $S$ is not a $k$-theta. Thus $S$ is a generalized $k$-prism or a generalized $k$-pyramid.
		Since $q_n$ is non-adjacent to $a_1$ and $q_n$ is adjacent to $a_2, \dots, a_k$, $a_1$ does not equal any of $a_2, a_3, \dots, a_k$.
		By Fact \ref{lem:leg-fact}, we may assume $a_1$ is adjacent to $a_3$.
		Let $C$ be a hole contained in $S$ with $P_1, P_2 \subseteq S$.
		Then $G[V(C \cup Q)]$ is a prism or a theta depending on whether $\alpha_1 = \beta_1$ and the union of $A_1$ and the edge $a_1a_3$ is a constituent path of $G[V(C \cup Q)]$.
		\stmt{$\alpha_1 \neq \beta_1$ \label{paths:end-I-hope}}
		Suppose $\alpha_1 = \beta_1$. Then $G[V(C \cup Q)]$ is a theta with constituent paths of length $\frac{\ell}{2}$ and $P_1 = A_1 \cup B_1$.
		Hence, $|E(A_1)| = \frac{\ell}{2} -1$.
		$B_1$ has length at least one, because $N(q_1) \subseteq P_1^*$.
		Hence, $P_1$ has length at least $\frac{\ell}{2}$.
		Since $G$ contains a theta, $\ell$ is even so $S$ is a generalized $k$-prism.
		But then $P_1$ has length at most $\frac{\ell}{2} -1$, a contradiction. This proves (\ref{paths:end-I-hope}).
		\\
		\\
		Thus $G[V(C \cup Q)]$ is pyramid and so $\ell$ is odd.
		Moreover, $|E(P_1)| = |E(A_1)| + |E(B_1)| + 1$.
		Since the union of $A_1$ and the edge $a_1a_3$ is a constituent path of $G[V(C) \cup Q]$, it follows that $|E(A_1)| = \frac{\ell -1}{2} -1$.
		Thus $P_1$ has length $\frac{\ell -1}{2} + 1$.
		But since $\ell$ is odd $S$ is a generalized $k$-prism and so $P_1, P_2, \dots P_k$ all have length at most $\frac{\ell -1}{2}$, a contradiction.
	\end{proof}
	
	\begin{thm} \label{thm:k-corpus-if-path-some-vertex-has-two-non-adjacent-nbrs-in-corpus}
		Let $G$ be an $\ell$-monoholed graph for some $\ell \geq 7$ and suppose $G$ contains a $k$-spine.
		Let $\R$ be a crowned $k$-corpus, chosen to maximize $k$ and with respect to that maximize $|V(\R)|$.	
		Suppose there is some path $W$ in $G \setminus V(\R)$ such that $N(V(W)) \cap V(\R)$ contains vertices from two non-adjacent bags of $\R$.
		Then there exists some $w \in V(W)$ such that $N(w) \cap V(\R)$ contains vertices from two non-adjacent bags.
	\end{thm}
	\begin{proof}
		Choose $W$ to be a minimal path in $G \setminus V(\R)$ such that $V(W)$ has neighbors in two non-adjacent bags of $\R$.
		Let the vertices of $W$ be $w_1 \dd w_2 \dd \dots \dd w_n$, in order. Suppose for each $i \in [n]$ that $N(w_i) \cap V(\R)$ is empty or the vertex set of a clique. Then, $n > 1$.
		Let $X, Y$ be the elemental sides of $\R$ and let $C_X, C_Y$ be the crowns of $\R$ where $X \subseteq C_X$, $Y \subseteq C_Y$. Let $\q_1, \q_2, \dots , \q_k$ be the elemental paths of $\R$ and lets $X_i, Y_i$ be the end bags of $\q_i$ where $X_i \in X, Y_i \in Y$.
		\stmt{The neighbors of $V(W)$ in $V(\R)$ are not contained in $V(C_X)$ and they are not contained in $V(C_Y)$. \label{omfg:not-just-in-one-crown}}
		Suppose the neighbors of $V(W)$ in $V(\R)$ are contained in $V(C_X)$.
		There exist some two nonadjacent vertices $c_1, c_2 \in C_X$ such that $w_1$ is adjacent to $c_1$ and $w_n$ is adjacent to $c_n$. By definition of crowned $k$-corpus for each $i\in {1,2}$, $c_i$ is equal or adjacent to some vertex $x_i$ in a bag of $X$.
		Then, $x_1$ and $x_2$ are non-adjacent. Then the graph induced by $V(W) \cup \{ c_1, c_2, x_1, x_2\}$ contains a $x_1x_2$-path $Z$ of length at least three.
		By \ref{fact:corpus-two-elemental-vertices-ell-minus-2-path-between-them}, $\R$ contains an $x_1x_2$-path $P$ such that $P^*$ is anticomplete to $V(C_X)$. But then $M \cup P$ is a hole of length greater than $\ell$, a contradiction. This proves (\ref{omfg:not-just-in-one-crown}).
		\stmt{$W^*$ is anticomplete to $V(\R)$ \label{omfg:no-nbrs-in-interior}}
		By minimality of $W$, the set $N(W^*) \cap V(\R)$ is complete to $N(w_1) \cap V(\R)$ and $N(w_n) \cap V(\R)$.
		Suppose some $w_i \in W^*$ has a neighbor $r \in V(\R)$. 
		
		By definition, there are some two vertices $u,v \in N(r) \cap V(\R)$ such that $u$ and $v$ are not adjacent and $w_1$ is adjacent to $u$ and $w_n$ is adjacent to $v$.
		Suppose there is some hole $C$ contained in $\R$ such that $u, v, r \in V(C)$.
		Then since $\ell \geq 7$, $W^*$ is anticomplete to $V(C \setminus v)$. Thus $W \cup C \setminus v$ is a hole of length greater than $\ell$, a contradiction. Thus $u, v, r$ are not all contained in any hole of $\R$.
		
		Let $\F$ be the $k$-corpus of $\R$. Then by definition $r \not \in V(\F)$.
		Thus we may assume $r \in V(C_X) \setminus X$.
		It follows from the minimality of $W$ that $N(V(W)) \cap V(\R)$ is contained in a single crown of $\R$, contradicting (\ref{omfg:not-just-in-one-crown}). This proves (\ref{omfg:no-nbrs-in-interior}).
		\stmt{$V(W)$ is anticomplete to one of $C_X$, $C_Y$.\label{omfg:not-int-both-crowns}}
		Suppose $N(V(W))$ contains vertices in both $C_X$, $C_Y$.  
		Then, by minimality of $W$, we may assume that $w_1$ has a neighbor in $V(C_X)$ and $w_n$ has a neighbor in $V(C_Y)$. Then, by Corollary \ref{corr:fml}, for some $i \in [2,n-1]$, $w_i$ has a neighbor in $v \in V(\R)$, contradicting (\ref{omfg:no-nbrs-in-interior}). This proves (\ref{omfg:not-int-both-crowns}).
		\stmt{Let $\F$ be the corpus of $\R$. Then $N(V(W)) \cap V(\F)$ is the vertex set of a clique.\label{omfg:applying-the-previous-path-theorem}}
		Suppose $N(V(W)) \cap V(\F)$ contains two non-adjacent vertices $u, v$ . Then by definition of inflated graph there is some graph $F$ underlying $\mathcal{F}$ with $u,v \in V(F)$. But then $W, F$ contradicts Theorem \ref{thm:cant-add-path-to-maximal-skeleton}. This proves (\ref{omfg:applying-the-previous-path-theorem}).
		\\
		\\
		By (\ref{omfg:applying-the-previous-path-theorem}), we may assume $N(w_1) \cap V(\R) \subseteq V(C_X) \setminus X$. Then by (\ref{omfg:not-int-both-crowns}), $w_n$ is anticomplete to $V(C_Y)$. Then by (\ref{omfg:not-just-in-one-crown}), we may assume $w_n$ has a neighbor $j$ in an interior bag of $\q_1$.
		Let $h$ be a neighbor of $w_1$ in $V(C_X) \setminus X$.
		By definition of crowned $k$-corpus, we may assume $h$ has a neighbor $x_2 \in X_2$.
		Let $\mathcal{C}$ be an inflated hole in $\R$ containing $\q_1, \q_2$.
		Then by definition of inflated graph there is a hole $C$ underlying $\mathcal{C}$ such that $x_2, j \in V(C)$.
		Let $x_1$ be the vertex in $V(C)$ corresponding to the bag $X_1$ and let $y_1, y_2$ be the vertices in $V(C)$ corresponding to the bags $Y_1, Y_2$.
		Then there is an $x_1x_2$-path $P_{12}$ of length at most two contained in $V(C)$.
		Moreover, by definition of crowned $k$-corpus $V(C) \cap V(C_X) \subseteq V(P_{12})$.
		Hence the neighbors of $h$ in $V(C)$ are contained in $V(P_{12})$.
		By definition $w_n$ has at most two neighbors in $V(C)$. Let $j, j'$ denote the neighbors of $w_n$ in $V(C)$. Then either $j = j'$ or $j$ is adjacent to $j'$.
		Without loss of generality $x_1, j, j', y_1, y_2, x_2, v, x_1$ occur in order in $V(C)$.
		Let $L_1$ be the $x_1j$-path of $C$ not containing $y_1$.
		Let $L_2$ be the $x_2j'$ path of $C$ not containing $x_1$.
		Then the graph induced by $V(L_1 \cup W \cup P_{12})$ includes a hole of length $|E(W)| + |E(L_1)| + 3$ or $ |E(W)| + |E(L_1)| + 4, |E(W)| + |E(L_1)| + 5$ depending on the length of $P_{12}$ and the neighbors of $h$ in $V(P_{12})$.
		The  union of $L_2$ and the path $x_2 \dd h \dd w_1 \dd w_2 \dd \dots \dd w_n \dd j'$ is a hole of length $|E(W)| + |E(L_2)| + 3$.
		
		Hence 	$|(L_1)| + \zeta = |E(L_2)|$ for some $\zeta \in \{0, 1, 2\}$.
		Let $Q_1, Q_2$ be underlying paths of $\q_1, \q_2$ contained in $C$.
		Then,
		$|E(L_2)| = |E(Q_2)| + d_{L_2}(y_1, y_2) + d_{L_2} (y_2, j')$ and $|E(L_1)| = |E(Q_1)| - d_{L_2}(y_2, j') - d_C(j, j')$.
		It follows that, 
		$$|E(Q_1)| - d_{L_2}(y_2, j') - d_C(j, j')  + \zeta = |E(Q_2)| + d_{L_2}(y_1, y_2) + d_{L_2} (y_2, j')$$ 
		So, $$\zeta = |E(Q_2)| - |E(Q_1)| + 2d_{L_2}(y_2, j') + d_{L_2}(y_1, y_2) - d_C(j,j')$$
		By definition of elemental path, $Q_1, Q_2$ differ in length by at most one.
		Hence $\zeta \geq 2$ so $\zeta = 2$.
		
		It follows that $j = j'$, $y_2j \in E(G)$ and $y_1y_2 \in E(G)$.
		Also, $P_{12}$ has length two and the only neighbor of $h$ in $V(C)$ is $x_2$.
		Then, $G[V(C \cup W) \cup \{ h\}]$ is a theta and each of paths $h_1 \dd w_1 \dd w_2 \dd \dots \dd w_n \dd j$, $ G[V(Q_1 \cup P_{12}) \setminus y_1]$ and $j \dd y_1 \dd y_2 \dd Q_2$ each have length $\frac{\ell}{2}$.
		Thus $|E(Q_2)| = \frac{\ell}{2} - 3$.
		
		Let $v$ be the central vertex of $P_{12}$.
		Since $h$ and $v$ have a common neighbor in a $X$ and $hv$ is not an edge in $G$, it follows from the definition of crown that $v \not \in V(C_X) \setminus X$.
		Hence, $v$ is in a bag $X_3$ of $X$. Let $Q_3$ be a path underlying $\q_3$ such that $v \in V(Q_3)$. Let $y_3$ be the end of $Q_3$ not equal to $v$.
		Let $Z_{23}$ be a shortest $y_2y_3$-path in $C_Y$.
		Then the union of $Q_2$, $Q_3$, $Z_{23}$, and the edge $x_2x_3$ is a hole.
		So $|E(Q_3)| + |E(Z_{23})| = \frac{\ell}{2} + 2$.
		Let $Z_{13}$ be a shortest $y_1y_2$-path in $C_Y$.
		Then the union of $Z_{13}$, $Q_3$, and the path $x_3 \dd x_2 \dd h \dd w_1 \dd \dots \dd w_n \dd j \dd y_1$ is a hole.
		So $|E(Q_3)| + |Z_{13}| = \frac{\ell}{2}$.
		But by definition of $C_Y$, the lengths of $Z_{13}, Z_{23}$ differ by at most one.
	\end{proof}

	\subsection{Everything is a crowned $k$-corpus}
	\begin{lemma}\label{lem:not-neighbors-in-just-caps-or-2-caps}
		Let $G$ be an $\ell$-monoholed graph. Suppose $G$ does not contain a clique cut-set and suppose $G$ contains a $k$-spine for some $k \geq 3$. Let $\R$ be a crowned $k$-corpus in $G$ chosen to maximize $k$ and with respect to that to maximize $V(\R)$. Let $\F$ be the $k$-corpus of $\R$.
		Suppose for some vertex $v$ in $V(G) \setminus V(\R)$, the set of vertices $N(v) \cap V(\R)$ contains two non-adjacent vertices. Then $v$ contains two non-adjacent vertices in $V(\F)$ and $v$ has a neighbor in the core of $\R$.
	\end{lemma}
	\begin{proof}
		Let $X, Y$ be the elemental sides of $R$ and let $C_X, C_Y$ be the crowns of $R$ where $V(X) \subseteq V(C_X)$ and $V(Y) \subseteq V(C_Y)$.
		Let $C_X \setminus X$ denote the graph formed from $C_X$ by deleting any vertices in a bag in $X$.
		Suppose $N(v) \cap V(\R)$ contains two non-adjacent vertices in $V(\R) \cap V(\F)$.
		Suppose $v$ is anticomplete to $V(\F)$ or $N(v) \cap V(\F)$ is the vertex set of a clique.
		\stmt{$N(v) \cap V(\R)$ is not contained in $V(C_X)$ or $V(C_Y)$. \label{bleh:easy}}
		Suppose every neighbor of $v$ in $V(\R)$ is contained in $V(C_X)$.
		Then $v$ has a neighbor $u \in V(C_X \setminus X)$ and a neighbor $w$ in $V(C_X)$ and $u$ and $w$ are non-adjacent. By definition of crown there are some two non-adjacent vertices in $x_1, x_2 \in X$ such that $u$ is adjacent to $x_1$ and $w$ is equal or adjacent to $x_2$.
		It follows that $G[\{v, u, w, x_1, x_2\}]$ is an $x_1x_2$-path of length at least three.
		But by Fact \ref{fact:corpus-two-elemental-vertices-ell-minus-2-path-between-them} $\F$ contains an $x_1x_2$-path $M$ of length $\ell -2$ such that $M^*$ is anticomplete to $V(C_X)$.
		Hence the union of $G[\{v, u, w, x_1, x_2\}]$ and $M$ is a hole of length greater than $\ell$, a contradiction. This proves (\ref{bleh:easy}).
		\stmt{$v$ does not have a neighbor in the interior of an elemental path of $\F$.\label{bleh:two}}
		Suppose $v$ has a neighbor in the interior of some elemental path of $\F$.
		Then for some graph $R$ underlying $\R$, the set of vertices $N(v) \cap V(R)$ contains two non-adjacent vertices and $v$ has a neighbor in the interior of an elemental path of $R$. By definition, $C_X \setminus X \subseteq R$. Without loss of generality, $v$ has a neighbor $h \in V(C_X \setminus X)$.
		Let the elemental paths of $R$ be $Q_1, Q_2, \dots, Q_k$ where for each $i \in [k]$, $Q_i \subseteq \q_i$ and $v$ has a neighbor $j \in Q_1^*$. For each $i \in [k]$ let $x_i, y_i$ denote the ends of $Q_i$ contained in $X, Y$, respectively. Let $Q_1^x$ be the $jx_1$-path of $Q_1$. 

		Let $M$ be a shortest path from $x_1$ to a vertex in $N(v) \cap V(C_X)$ in $C_X$.
		Then, $M$ has length one or two.
		$G[V(Q_1^x \cup M) \cup \{ v\}]$ is a hole of length $|E(Q_1^x)| + 3$ or $|E(Q_1^x)| + 4$.
		So $|E(Q_1^x)| \in \{ \ell -3, \ell -4\}$ and so $|E(Q_1)| \geq \ell - 3$.
		By definition, every elemental path of $R$ has length at most $\frac{\ell}{2} -1$.
		Hence $\ell \leq 4$, a contradiction. This proves (\ref{bleh:two}).
		\stmt{$G$ contains a ($k+1$)-spine. \label{bleh:contradiction}}
		We have assumed $G$ does not contain a $(k+1)$-spine.
		We may assume $v$ has a neighbor in $h \in V(C_X \setminus X)$ and a neighbor in $C_Y$.
		For each $i \in [k]$ let $t_i$ be a vertex from an interior bag of a constituent path of $\mathcal{R}$.
		Thus, if $v$ neighbors in at most one bag of each of the terminating sets of $Y$ then $\R$ then $G[V(\R) \cup \{ v\}]$ contains a pair of mated $(k+1)$-spiders with toes $t_1, t_2, \dots, t_k, v$. Hence, for some terminating set $\mathcal{B}$ of $\R$, $v$ has neighbors in two bags of $\mathcal{B}$.
		
		Let $\F$ be the corpus of $\R$ it follows that, $v$ has neighbors in two bags that are contained in $V(C_Y) \cap V(\F)$.
		Let $F$ be an underlying graph of $\F$ such that $N(h) \cap V(F)$ contains two non-adjacent vertices and $v$ has two neighbors in one of the terminating sides of $F$. This is possible by definition of inflated graph and of $\F$.
		Let the constituent paths of $F$ be $P_1, P_2, \dots, P_k$ and let $Q_1, Q_2, \dots, Q_k$ be the elemental paths of $F$ where $Q_i \subseteq P_i$ for every $i \in [k]$.
		For each $i \in [k]$ let the ends of $P_i$ be $a_i, b_i$ such that $a_i \in V(C_X)$ and $b_i \in V(C_Y)$. For each $i \in [k]$, let $x_i, y_i$ be the vertices in the elemental sides of $F$ that are contained in $V(P_i)$ where $x_i \in V(C_X)$ and $y_i \in V(C_Y)$.
		
		We may assume $h$ is adjacent to $x_1, x_2$ and $x_1$ is not adjacent to $x_2$.
		Let $B$ be the set $\{ b_1, b_2, \dots, b_k\}$. Thus $v$ has two neighbors in $B$.
		Let $M$ be the shortest path from a neighbor of $v$ in $V(B)$ to $b_1$ or $b_2$.
		Then $V(Q_1 \cup Q_2 \cup M \cup \{ h, v\})$ induces a theta or a pyramid with constituent paths of length at most three.
		Hence it is a pyramid and $\ell = 7$, $M$ has length two and $b_1$ is adjacent to $b_2$.
		Thus $F$ is a generalized $k$-pyramid and $P_1, P_2, \dots, P_k$ each have lengths in $\{ 2, 3\}$ and in particular $P_1, P_2$ both have length two.
		
		Suppose $b_3$ is a neighbor of $v$ in $B$.
		Then $h$ is anticomplete to $V(P_3)$ for otherwise $G[V(P_3) \cup \{ v, h\}]$ contains a hole of length six or less, a contradiction.
		
		Since $M$ has length two, $v$ is not adjacent to $b_1, b_2$ or any of their neighbors. Thus $b_3$ is not adjacent to $b_1, b_2$.
		Thus by definition of generalized $k$-prism, $P_3$ has length two, $a_3$ is adjacent to $a_1, a_2$.
		But then the union of $P_3$ and the path $a_3 \dd a_2 \dd h \dd v \dd b_3$ is a path of length six, a contradiction.
		Thus $G$ contains a mated $(k+1)$-spider and so it contains a $(k+1)$-spine.
		This proves (\ref{bleh:contradiction}).
		\\
		\\
		(\ref{bleh:contradiction}) contradicts the maximality of $k$.
	\end{proof}

	\begin{thm}
		Let $G$ be an $\ell$-monoholed graph. Suppose $G$ does not contain a clique cut-set and $G$ does not contain a vertex $v$ that is adjacent to every other vertex in $V(G)$.	
		Suppose $G$ contains a $k$-spine for some $k \geq 3$. Let $\R$ be a crowned $k$-corpus in $G$ chosen to maximize $k$ and with respect to that to maximize $|V(\R)|$.
		Suppose $G$ does not contain a clique cut-set and suppose $G$ contains a $k$-spine for some $k \geq 3$.
		Then $G = \R$
		\label{thm:everything-is-a-crowned-corpus}
	\end{thm}
	\begin{proof}[Proof]
		Suppose $G \neq \R$.
		Let $H$ be the set of vertices that are complete to $V(\R)$. Then $H$ is a clique since $G$ does not contain a $C_4$.
		By assumption $V(G) \setminus (V(\R) \cup H) \neq \emptyset$.
		
		Then since $G$ does not contain a clique cut-set there is connected induced subgraph $W$ of $G \setminus (V(\R) \cup H)$ such that $V(W) \cap V(\R)$ contains two non-adjacent vertices. Choose $W$ to be minimal. Then $W$ is a path.
		By Lemma \ref{lem:not-neighbors-in-just-caps-or-2-caps} and Theorem \ref{thm:crowned-k-corpus-extra-vertex-only-can-have-nbrs-in-one-end}, $W$ does not consist of a single vertex.
		Then, by Theorem \ref{thm:k-corpus-if-path-some-vertex-has-two-non-adjacent-nbrs-in-corpus}, $N(W) \cap V(\F)$ does not contain vertices from two non-adjacent bags.
		Hence there are some two adjacent bags $J, J'$ of $\R$ and $j \in J$ and $j \in J'$ such that $N(W)$ contains both $j$ and $j'$.
		Let the vertices of $W$ be $w_1 \dd w_w \dd \dots \dd w_n$, in order. Then we may assume $w_1$ is adjacent to $j$ and $w_n$ is adjacent to $j'$. By minimality $W^*$ is anticomplete to $j, j'$.
		\stmt{There is no inflated hole $\mathcal{C}$ contained in $\F$ with $J, J' \subseteq \mathcal{C}.$ \label{grrr:easy}}
		Suppose $\mathcal{C}$ is an inflated hole in $\F$ containing both $J$ and $J'$ as bags.
		Then there is some $jj'$-path $R$ contained in $\mathcal{C}$ of length $\ell -1$. But then the union of $R$ and path $j \dd w_1 \dd w_2 \dd \dots \dd w_n \dd j'$ is a hole of length greater than $\ell$, a contradiction. This proves (\ref{grrr:easy}).
		\\
		\\
		Let $\F$ be the $k$-corpus of $\R$.
		Then it follows from (\ref{grrr:easy}) and the definition of $k$-corpus that $J, J'$ cannot both be in $\F$.
		Since the bags of $\F \setminus \R$ are single vertices, $J$ and $J'$ cannot both be in $\R \setminus \F$.
		Then we may assume $J$ is a bag of an elemental side $X$ of $\R$ and $J'$ is in $\R \setminus \F$. By (\ref{grrr:easy}), $J$ is complete to every neighboring bag of $J'$ in $X$.
		By definition of crowned $k$-corpus $J'$ consists of a single vertex $j'$ and $N(j')$ contains nonadjacent vertices $x_1, x_2$ from bags in $X$. Then $j \dd x_1 \dd j' \dd x_2 \dd j$ is a hole of length four, a contradiction.
	\end{proof}

	\section{Analyzing the structure of a maximal crowned $k$-corpus}
	In this section we will prove that the crowns of a crowned $k$-corpus are ``transitive closures of trees'' for any $\ell \geq 7$ and $k \geq 3$. We begin by introducing transitive closure of trees.
	\subsection{Transitive closures of trees}
	Given a tree $T$ with root $r \in V(T)$ we call the graph obtained from $T$ by adding edges between every vertex $v \in V(T)$ all descendants of $v$ in $V(T)$ the \textit{transitive closure of $T$}. See Figure \ref{fig:trans-closure-of-tree} for an illustration. We need the following lemma about transitive closures of trees. 
	
	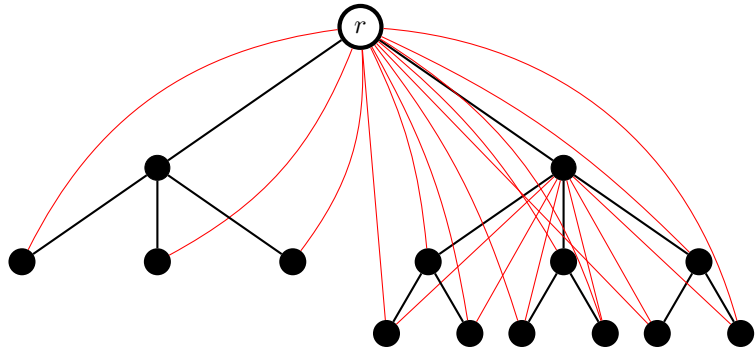
\begin{figure}[!h]
		\centering
		\input{Tikz/transitive_closure_of_tree_example.tex}
		\caption{An example of a transitive closure of a tree. Let $T$ be the tree drawn in black and $r$ the root of $T$. Then the transitive closure is the union of $T$ and the red edges. }
		\label{fig:trans-closure-of-tree}
	\end{figure}
	
	\begin{lemma}
		Let $H$ be a connected graph such that $V(H)$ can be partitioned into a stable set $S$ and a set $X$ satisfying the following conditions:
		\begin{enumerate}[(a)]
			\item Every $x \in X$ has a neighbor in $S$.
			\item For any $x_1, x_2 \in X$, $x_1$ is adjacent to $x_2$ if and only if $N(x_1) \cap S \subseteq N(x_2) \cap S$ or $N(x_2) \cap S \subseteq N(x_1) \cap S$,
			\item For any $x_1, x_2 \in X$, $x_1$ is not adjacent to $x_2$ if and only if $N(x_1) \cap S$ and $N(x_2) \cap S$ are disjoint.
		\end{enumerate} 
		Then $H$ is a transitive closure of a tree and $S$ is the set of leaves of $H$.
		\label{lem:trans-closure-tree-nbr-characterization}
	\end{lemma}
	\begin{proof}
		We proceed by induction on $|X|$.
		\stmt{ $G[X]$ is connected. \label{transylvania:connected}}
		Suppose $C_1, C_2$ are two components of $G[X]$. Then since $H$ is connected, there is some vertex  in $H$ with a neighbor in both $V(C_1)$ and $V(C_2)$, contradicting (c). This proves (\ref{transylvania:connected}).
		\stmt{If $x_1 \dd x_2 \dd x_3$ is a $P_3$ contained in $X$, then $N(x_i) \cap S \subsetneq N(x_2) \cap S$ for $i \in \{ 1,3\}$.\label{transylvania:P3}}
		Suppose $N(x_1) \cap S$ is not a proper subset of $N(x_2) \cap S$. Then by (b), $N(x_2) \cap S \subseteq N(x_1) \cap S$. But then by (b), $N(x_1) \cap S$ and $N(x_3) \cap S$ are not disjoint, contradicting (c). This proves (\ref{transylvania:P3})
		\\
		\\
		Choose $x \in X$ to maximize $|N(x) \cap S|$. 
		\stmt{$x$ is complete to $V(H) \setminus \{ x\}$. \label{transylvania:complete}}
		Suppose $x$ is not complete to $X \setminus \{ x\}$. Then since $G[X]$ is connected for some $x', x'' \in X$, the path $x \dd x' \dd x''$ is an induced $P_3$. But then by (\ref{transylvania:P3}), $|N(x') \cap S| > |N(x) \cap X|$, a contradiction. Hence $x$ is adjacent to every other vertex in $X$. Since $H$ is connected every $s \in S$ has a neighbor in $X$, so $x$ is complete to $S$ by (b) and our choice of $x$. 
		\\
		\\
		Let $H_1, \dots, H_k$ be the components of $H \setminus \{ x\}$. By induction for each $i \in [k]$, $H_i$ is the transitive closure of some tree $T_i$ with root $r_i \in V(T_i)$ and leaves $V(H_i) \cap S$. Let $T$ be the tree obtained from the union of $T_1, T_2, \dots, T_k$ by adding a new vertex $r$ adjacent to $r_1, r_2, \dots, r_k$. Then $H$ is the transitive closure of $T$ since $x$ is adjacent to $V(H) \setminus \{ x\}$.
	\end{proof}
	We need the following definition. Let $T$ be a tree with root $r$. Let $L$ denote the set of leaves of $T$ and $S$ denote the set of parents of vertices in $L$.
	Let $L_1, L_2, L_3, \dots, L_k$ be a partition of $L$.
	We say $T$ is $\{L_1, L_2, L_3, \dots, L_k\}$-friendly if all of the following statements hold:
	\begin{enumerate}[(i)]
		\item Every vertex in $S$ has neighbors in at least two of $X_1, X_2, \dots, X_k$ and
		\item For each $i \in [k]$ there is a path $P$ of $T$ from the root to a vertex in $S$ such that $N(X_i) \subseteq V(P)$.
	\end{enumerate}
	
		\begin{figure}
		\centering
		\input{Tikz/pyramidoid.tex}
		\caption{An example of a crowned 4-corpus where the underlying graph of the corpus is a 4-pyramid. The top is the transitive closure of a type of tree called a caterpillar.}
		\label{fig:pyramidoid}
	\end{figure}
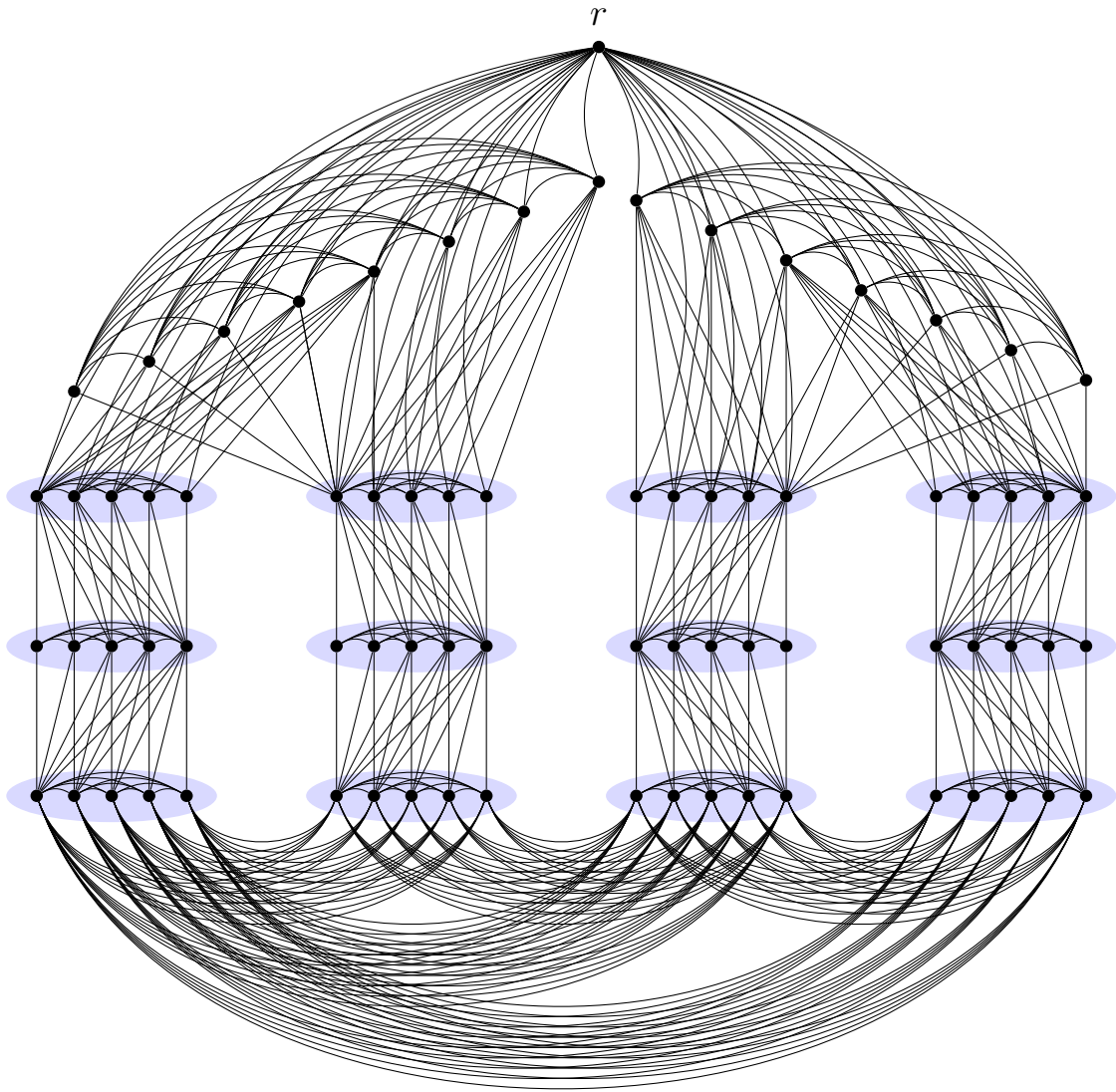

\subsection{Crowns and transitive closures of trees}
In this subsection we will prove that the crowns of a crowned $k$-corpus are transitive closures of trees for any $k \geq 3$ and $\ell \geq 7$. We begin with an observation about elemental sides of a $k$-spine.

Let $F$ be a $k$-spine for some $k \geq 3$.
Let $Q_1, Q_2, \dots, Q_k$ denote the elemental paths of $F$.
Let $X, Y$ be the elemental sides of $F$.
For $L \subseteq [k]$, let $X(L), Y(L)$ denote the set of ends of $\{ Q_i \ | \ i \in L\}$ that lie in $X$ and $Y$, respectively.
We call $I, J$ a \textit{helpful} partition of $[k]$ with respect $X$ if all of the following conditions hold.
\begin{itemize}
	\item $X(I)$ is a stable set,
	\item $X(J)$ is a clique and
	\item If $I \neq \emptyset$, every $x\in X(J)$ has at least one neighbor in $X(I)$.
\end{itemize}

\begin{lemma}
		If $F$ a $k$-spine for some $k \geq 3$,  $X$ is an elemental side of $F$ then there is a helpful partition with respect to $X$.
		\label{lem:spine-helpful}
\end{lemma}
\begin{proof}
We may assume $X$ is not the vertex set of a clique or a stable set.
\end{proof}

Let $\R$ be a $k$-corpus.
Let $C_X$ be a crown of $\R$.
Let $\mathcal{X}$ be an elemental side of $\R$ where $V(\mathcal{X}) \subseteq V(C_X)$.

We call $\mathcal{I}, \mathcal{J}$ a \textit{helpful} partition of the bags of $C_X$ if all of the following hold:
\begin{itemize}
	\item $\mathcal{I}$ is a (possibly empty) set of pairwise adjacent bags,
	\item $\mathcal{I} \subseteq \mathcal{X}$,
	\item $\mathcal{J} \cap \mathcal{X}$ is a set of pairwise adjacent bags,
	\item If $\mathcal{I} \neq \emptyset$, every bag $J \in \mathcal{J}$ is adjacent to at least one bag of $\mathcal{I}$.
\end{itemize}

\begin{lemma}
	Let $G$ be an $\ell$-monoholed graph for some $\ell \geq 7$.
Let $\R$ be a crowned $k$-corpus in $G$. 
	Let $C_X$ be a crown of $\R$.
	Then there is a helpful partition of the bags of $C_X$.
\end{lemma}
\begin{proof}
	Let $\mathcal{X}$ be an elemental side of $\R$ where $V(\mathcal{X}) \subseteq V(C_X)$.
	Suppose $\mathcal{X}$ is a set of pairwise adjacent bags. Then by definition, $V(C_X) = V(\mathcal{X})$ so $\emptyset, \mathcal{X}$ is a helpful partition.
	Hence we may assume some two bags of $\mathcal{X}$ are adjacent.
	
	Let $W$ be the graph with the vertex set $\mathcal{X}$ and edges between pairs of vertices if and only if they are adjacent bags.
	Let $\F$ be the corpus of $\R$ and let $F$ be the underlying graph of $\F$. Then by definition, $W$ is isomorphic to the subgraph induced by the vertices of the elemental side of $F$ corresponding to $\mathcal{X}$.
	Thus by Lemma \ref{lem:spine-helpful}, there is a helpful partition $I, J$ of $\mathcal{X}$.
	Let $S$ be the set of bags of $C_X$ that are not elements of $\mathcal{X}$.
	Then since $C_X$ is a crown, every vertex in $s \in V(S)$ has the property that $N(s) \cap V(\mathcal{X})$ contains two non-adjacent vertices. So by definition of $k$-corpus, $s$ has a neighbor in a bag of $I$.
	Thus, $I, J \cup S$ is a helpful partition of the bags of $C_X$.
\end{proof}

\begin{lemma}
	Let $G$ be an $\ell$-monoholed graph for some $\ell \geq 7$.
Let $\R$ be a crowned $k$-corpus in $G$.  Let $\mathcal{X}$ be an elemental side of $\mathcal{R}$ and let $C_X$ be the crown of $\R$ containing vertices in bags $\mathcal{X}$. Suppose $X_1, X_2$ are non-adjacent bags in $\mathcal{X}$.
	Let $P$ be a path from a vertex in $X_1$ to a vertex in $X_2$ with $P^* \subseteq V(C_X) \setminus (X_1 \cup X_2)$. Then $P$ has length two.
	\label{lem:no-long-crown-path}
\end{lemma}
\begin{proof}
	$P$ has length at least two by definition of non-adjacent bags.
	Let $x_1 \in X_1$ and $x_2 \in X_2$ be the ends of $P$.
	By Fact \ref{fact:corpus-two-elemental-vertices-ell-minus-2-path-between-them}, there is an $x_1x_2$-path $R$ of length $\ell -2$ such that $R^* \subseteq V(\R) \setminus V(C_X)$.
	Thus $P \cup R$ is a hole and so it has length $\ell$.
	Hence, $P$ has length two.
\end{proof}

\begin{lemma}
	Let $G$ be an $\ell$-monoholed graph for some $\ell \geq 7$.
	Let $\R$ be a crowned $k$-corpus in $G$. Let $\mathcal{X}$ be an elemental side of $\mathcal{R}$ and let $C_X$ be the crown of $\R$ containing vertices in bags $\mathcal{X}$.
	Let $X_1, X_2, X_3$ be distinct bags of $\mathcal{X}$ such that $X_2$, $X_3$ are adjacent bags and $X_1, X_2$ are non-adjacent bags.
	Then there does not exist any $z \in V(C_X)$ such that $z$ has a neighbor in $X_1$, a neighbor in $X_2$ and a non-neighbor in $X_3$.
	\label{lem:no-weird-P4}
\end{lemma}
\begin{proof}
Suppose $z \in V(C_X)$ such that $z$ has a neighbor in $X_1$, a neighbor in $X_2$ and a non-neighbor in $X_3$.
Then $z \not \in V(\mathcal{X})$ by definition of $k$-corpus.

Suppose $X_1$ and $X_3$ are adjacent bags of $\mathcal{X}$.
Then $G[V(X_1 \cup X_2 \cup X_3) \cup \{ z\}]$ contains a hole of length four, a contradiction.
So $X_1$ is anticomplete to $X_3$.

Let $x_1, x_2$  be a neighbor of $z$ in $X_1$ and $X_2$ respectively. Let $x_3 \in X_3$.
	By Fact \ref{fact:corpus-two-elemental-vertices-ell-minus-2-path-between-them}, there is an $x_1x_3$-path $R$ of length $\ell -2$ such that $R^* \subseteq V(\R) \setminus V(C_X)$.
	Hence the union of $R$ and the path $x_1 \dd z \dd x_2 \dd x_3$ is a hole of length $\ell + 1$, a contradiction.
\end{proof}

\begin{theorem}
	Let $\R$ be a crowned $k$-corpus. Let $\mathcal{X}$ be an elemental side of $\mathcal{R}$ and let $C_X$ be the crown of $\R$ containing vertices in bags $\mathcal{X}$. Let $\mathcal{I}, \mathcal{J}$ be a helpful partition of the bags of $C_X$.
	Then the graph obtained from $C_X$ by removing all edges between vertices in bags of $\mathcal{I}$ is a transitive closure of some $\mathcal{I}$-friendly tree $T$.
	\label{thm:crowns-are-trans}
\end{theorem}
\begin{proof}
	Let $J_X$ denote $\mathcal{J} \cap V(\mathcal{X})$ and let $J_C$ denote $V(\mathcal{J}) \setminus V(\mathcal{X})$.
	\stmt{Suppose $u, v$ are non-adjacent vertices in $V(C_X)$. Then there is no bag $X \in \mathcal{I}$ such that $u, v$ both have a neighbor in $V(X)$. \label{trans:1}}
	Suppose $u,v$ are both in $J_X$.
	Then (\ref{trans:1}), holds by definition of $k$-corpus.
	
	Hence, we may assume that $u \in J_C$. Suppose $u, v$ both have a neighbor in some bag $X$ of $\mathcal{I}$.
	Let $x_u, x_v$ be neighbors of $u, v$ in $X$.
	Then by definition of crowned $k$-corpus and Lemma \ref{lem:crown-tops-have-two-no-adjacent-nbrs-in-bottom}, $N(u) \cap V(\mathcal{X})$ contains two non-adjacent vertices.
	Thus, $x_u$ is a good child of $u$ and there exist $X_1 \in \mathcal{X}$ such that $u$ has a neighbor $x_1 \in X_1$ and $X_1, X$ are non-adjacent bags.
	
	Suppose $v \in J_X$. Then, $x_v$ is a good child of $v$ and there exists $X_2 \in \mathcal{X}$ such that $v$ has a neighbor $x_2 \in X_2$ and $X_2, X$ are non-adjacent bags. By Lemma \ref{lem:crown-tops-have-two-no-adjacent-nbrs-in-bottom}, $x_u \neq x_v$.
	If $x_1$ is equal or adjacent to $x_2$, then $G[\{ u, v, x_u, x_v, x_1, x_2 \} ]$ contains a hole of length at most six, a contradiction.
	Hence $X_1$ and $X_2$ are non-adjacent bags by definition of $k$-corpus.
	But then the path $x_1 \dd u \dd x_u \dd x_v \dd v \dd x_2$ violates Lemma \ref{lem:no-long-crown-path}.
	
	Thus $v \in J_C$. Then $v$ is complete to $X$ by definition of $k$-corpus.
	Suppose $v$ has a neighbor in $X_1$. Then $v$ is adjacent to $x_1$ and $x_1 \dd v \dd x_u \dd u \dd x_1$ is a hole of length four a contradiction.
	So $v$ is anticomplete to $X_1$ and thus the bag $B$ containing $v$ in $\R$ is non-adjacent to $X_1$. But then $X_1, X, B, u$ violate Lemma \ref{lem:no-weird-P4}.
	This proves (\ref{trans:1}).
	
	\stmt{Suppose $u, v$ are adjacent vertices in $V(C_X)$. Then $N(u) \cap V(\mathcal{I}) \subseteq N(v) \cap V(\mathcal{I})$ or $N(v) \cap V(\mathcal{I}) \subseteq N(u) \cap V(\mathcal{I})$. \label{trans:adjacent-nested}}
	Suppose there exist adjacent $u,v \in V(C_X)$ and there exists $x_1, x_2 \in V(I)$ such that $x_1u, x_2v \in E(G)$ and $x_1v, x_2u \not \in E(G)$.
	
	By definition of $k$-corpus $u, v$ are not both in $J_X$.
	By Lemma \ref{lem:crown-good-bag-children-properties}, $u,v$ are not both in $J_C$. Hence we may assume $u \in J_C$ and $v \in J_X$.
	Let $X_1, X_2$ be the bags in $\mathcal{I}$ containing $x_1$ and $x_2$, respectively. Then $v$ is complete to $X_2$. So $X_1 \neq X_2$.
	Then by Lemma \ref{lem:no-long-crown-path}, $X_1$ must be an adjacent bag to $X_2$.
	But then $G[\{ u, v\} \cup X_1 \cup X_2]$ contains a $C_4$, a contradiction.
	This proves (\ref{trans:adjacent-nested}).
	\\
	\\
	Hence the result follows from Lemma \ref{lem:trans-closure-tree-nbr-characterization}.
\end{proof}

%% file: Tikz/halfgraph.tex
\begin{tikzpicture}[scale=.75]
	\def\height{5}
	\def\width{2}
	
	\foreach \i in {1, ..., 10}{
		\draw node[normal node] (x\i) at (\width*\i, 0){\small $x_{\i}$};
		\draw node[normal node] (y\i) at (\width*\i, -\height){\small $y_{\i}$};
		\foreach \j in {1, ..., \i}{
			\draw (x\i) to (y\j);
		}
		
	}
	
\end{tikzpicture}

%% file: Tikz/fig_spider.tex
\begin{tikzpicture}[scale=.75]
	\def\width{2}
	\newcounter{test}
	
	\foreach \x in {2,4,1,3,5}{
		\stepcounter{test}
		\foreach \y in {1, ..., \value{test}}{
			\draw node[dot] (L\x\y) at (\width*\x, -\y){};
			\foreach \z in { 1, ..., \y}{
				\draw (L\x\y) to (L\x\z);
			}
		}
		\draw node[normal node] (t\x) at (\width*\x, -\value{test} - 1){$t_\x$};
		\def\n{\value{test}};
		\draw node (L\x) at (\width*\x, -\n){};
		\draw (L\x) to (t\x);	
	}	
	
	\draw node [red node] (a1) at (\width *1.5, 0) {$a_1$};
	\draw node[red node] (a5) at (\width*4.5, 0){$a_4$};
	\draw node [red node] (v) at (\width * 2.25, 1){ };
	\draw node [red node] (a3) at (\width*3, 0){$a_3$};
	
	\draw (L11) to (a1);
	\draw (L21) to (a1);
	\draw (L31) to (a3);
	\draw (L41) to (a5);
	\draw (L51) to (a5);
	\draw (a3) to (a5);
	\draw (v) to (a3);
	\draw (v) to (a1);
\end{tikzpicture}

%% file: Tikz/path_with_edges_to_hole.tex
\begin{tikzpicture}
	
	\coordinate (O) at (1,2);
	\def\radius{2.5cm}
	\def\diff{10}

	\pgfmathsetmacro{\angleA}{80 - \diff}
	\path (O) ++(\angleA:\radius) coordinate (b1);
	\fill[black] (b1) circle[radius=4pt] ++(\angleA:1em) node {$b_1$};
	
	\path (O) ++(110:\radius) coordinate (a1);
	\fill[black] (a1) circle[radius=4pt] ++(110:1em) node {$a_1$};
	
	\pgfmathsetmacro{\angleb}{260 - \diff}
	\path (O) ++(\angleb:\radius) coordinate (a2);
	\fill[black] (a2) circle[radius=4pt] ++(\angleb:1em) node {$a_2$};
	
	\pgfmathsetmacro{\angleb}{280 + \diff}
	\path (O) ++(\angleb:\radius) coordinate (b2);
	\fill[black] (b2) circle[radius=4pt] ++(\angleb:1em) node {$b_2$};
	
	\begin{scope}[bend right = 100]
		\path
		(a2) edge node[below] {$=$?} (b2) [maybe equal ];
	\end{scope}
	
	\begin{scope}[bend left = 100]
		\path
		(a1) edge node[above] {$=$?} (b1) [maybe equal, thick];
	\end{scope}
	
	\coordinate (start) at (1,3.5);
	\fill (start) circle[radius=4pt] node[left] {$x_1$};
	
	\coordinate (end) at (1,.5);
	\fill (end) circle[radius=4pt] node[left] {$x_n$};
	
	\path 
	(a1) edge (start)
	(b1) edge (start)
	(a2) edge (end)
	(b2) edge (end);
	
	\fill (O) circle[radius=4pt] node[above left] {$x_i$};
	\draw[path] (start) to (O);
	\draw[path] (end) to (O);
	
	\draw[path] (b1) arc[start angle=\angleA, end angle=-70,radius=2.5cm];
	\draw (b1) arc[start angle= \angleA, end angle = \angleb, radius= 2.5cm];

	\pgfmathsetmacro{\angled}{180}
	\path (O) ++(\angled:\radius) coordinate (z);
	\fill[black] (z) circle[radius=4pt] ++(\angled:1em) node[below right = .1cm and .4cm] {$z$};
	
	\draw (z) to (O);

	\coordinate[left = .5cm of z] (A);
	\node at (A) {\large $A$};
	
	\pgfmathsetmacro{\angled}{0}
	\path (O) ++(\angled:\radius) coordinate (zero);
	\coordinate[right = .5cm of zero] (left);
	\node at (left) {\large $B$};
	
\end{tikzpicture}

%% file: Tikz/standardconstruction.tex
\begin{tikzpicture}
	\def\width{3};
	\def\height{4};
	\coordinate (one) at (1,\width);
	\def\radius{.75cm};
	\def\diff{10};
	\def\smallradius{4pt}
	
	\foreach \i in {1,2,3,4}{
		\foreach \j in {0,1}{
			\coordinate (c\i\j) at (\width*\i, -1* \j*\height);
		}
	}

	\draw[fill = black!20] (c20) circle[radius = \radius];
	\node at (c20) {\Large $R$};
	
	\draw (c21) circle[radius = \radius];
	\fill[black!20] (c21) circle[radius = \radius];
	\node at (c21) {\Large $R$};
	
	\draw[fill = white] (c10) circle[radius = \radius];
	\node at (c10) {\Large $Q$};
	
	\draw[fill = white] (c11) circle[radius = \radius];
	\node at (c11) {\Large $Q$}; 
	
	\draw[fill = white] (c30) circle[radius = \radius];
	\node at (c30) {\Large $S$}; 
	
	\draw[fill = white] (c41) circle[radius = \radius];
	\node at (c41) {\Large $T$}; 
	
	\fill (c40) circle[radius = \smallradius] node[above = 3pt] {\Large $T$};
	\fill (c31) circle[radius = \smallradius] node[below = 3pt] {\Large $S$};
	
	\def\anglediff{30};
	\foreach \k in {1,2,3,4}{
		\def\m{\k -3};
		\pgfmathsetmacro{\angle}{270 + (\anglediff * (\k -3))};
		\path (c10) ++(\angle:\radius) coordinate (b\k);	
	}
	
	\foreach \i in {1,2,3,4}{
		\foreach \j in {0}{
			\foreach \k in {1,2,3,4,5}{
				\pgfmathsetmacro{\angle}{270 + (\anglediff * (\k -3) };
				\path (c\i\j) ++(\angle:\radius) coordinate (b\i\j\k);
			}
		}
	}
	
	\foreach \i in {1,2,3,4}{
		\foreach \j in {1}{
			\foreach \k in {1,2,3,4,5}{
				\pgfmathsetmacro{\angle}{90 - (\anglediff * (\k -3) };
				\path (c\i\j) ++(\angle:\radius) coordinate (b\i\j\k);
			}
		}
	}
	
	\foreach \i in {1,2}{
		\foreach \k in {1,2,3,4,5}{
			\draw[path] (b\i0\k) to (b\i1\k);
		}
	}
	
	\foreach \k in {1, ..., 5}{
		\draw[path] (b30\k) to (c31);
		\draw[path] (b41\k) to (c40);
	}
	\foreach \i in {1,2,3,4}{
		\foreach \j in {0,1}{
			\pgfmathsetmacro{\top}{90};
			\pgfmathsetmacro{\bottom}{270};
			
			\path (c\i\j) ++ (\top:\radius) coordinate (top\i\j);
			\path (c\i\j) ++ (\bottom:\radius) coordinate (bot\i\j);
			\path (c\i\j) ++ (0:\radius) coordinate (l\i\j);
			\path (c\i\j) ++ (180:\radius) coordinate (r\i\j);
		}
	}
	\def\blackamount{60}
	\begin{scope}[on behind layer]
		\draw[fill =black!\blackamount] (top20) to (bot20) to (bot30) to (top30) to (top20);
		\draw[fill = black!\blackamount] (c40) to (top30) to (bot30) to (c40);
		\draw[fill =black!\blackamount] (c31) to (top41) to (bot41) to (c31);
		\draw[fill =black!\blackamount] (c31) to (top21) to (bot21) to (c31);
		\draw[fill = black!\blackamount] (c40) to [bend right = 40] (l20) to (r20) to [bend left = 70] (c40);
		\draw[fill = black!\blackamount] (c40) to [bend right = 50] (l10) to (r10) to [bend left = 70] (c40);
		\draw[fill = black!\blackamount] (l41) to [bend left = 50] (bot11) to (l11) to [bend right = 40] (r41) to (l41);
		\draw[fill = black!\blackamount] (c31) to [bend left = 70] (r11) to (l11) to [bend right = 40] (c31);
		\draw[pattern=north west lines, pattern color=black] (top10) rectangle (bot20);
		\draw[pattern=north east lines, pattern color=black] (top11) rectangle (bot21);
	\end{scope}
\end{tikzpicture}

%% file: fig_mated_four_spiders.tex
\begin{tikzpicture}[scale=.25]
	\def\width{1};
	\def\height{3};
	\def\between{4cm}
	\def\k{30} 
	\node(zero) at (0,0){
		
		\begin{tikzpicture}
			\def\bend{20}
			\draw node[dot] (a) at (0,0){};
			\draw node[dot] (b) at (0, -\height){};
			\foreach \x in {-1,1,3}{
				\begin{scope}[bend left=\x * \bend]
					\path[path]
					(a) edge  (b);
				\end{scope}
			};
			\begin{scope}[bend left=-3* \bend]
				\path[path]
				(a) edge node[left]{$\frac{\ell}{2}$} (b);
			\end{scope}

		\end{tikzpicture}
	};
	\node (two) at (zero.east)[anchor=east,xshift=\between]{
		\begin{tikzpicture}
			\node[dot] (a1) at (0, -.20 * \height){};
			\node[dot] (a3) at (-\width, 0){};
			\node[dot] (a4) at (\width, 0){};
			
			\node[dot] (b1) at (0, -.80 * \height){};
			\node[dot] (b3) at (-\width, -\height){};
			\node[dot] (b4) at (\width, -\height){};
			\node[dot](b2) at (0, -1.2*\height){};
			
			\path[path]
			(a3) edge node[left, above, rotate=90]{\small $\frac{\ell}{2} -2$} (b3)
			(a4) edge node[left, below, rotate=90]{\small $\frac{\ell}{2} -2$} (b4)
			(a1) edge node[left, above,  rotate=90]{\small $\frac{\ell}{2}-1$} (b1);
			\begin{scope}[bend left =60]
				\path[path]
				(a1) edge node[rotate=90, left, above]{\small $\frac{\ell}{2}-1$} (b2);
			\end{scope};
			\path
			(a1) edge (a3)
			(a3) edge (a4)
			(a1) edge (a4)
			(b1) edge (b3)
			(b1) edge (b4)
			(b2) edge (b3)
			(b2) edge (b4)
			(b3) edge (b4);

		\end{tikzpicture}
	};
	\node (three) at (two.east)[anchor= east, xshift=\between]{
		\begin{tikzpicture}
			\foreach \x in {0, ..., 3}{
				\def\name{a\x};
				\draw node[dot] (\name) at (\width * \x, 0){};
				\foreach \y in {0, ..., \x}{
					\def\othername{a\y}
					\draw[bend right = \k] (\name) to (\othername);
				}
			}
			\foreach \x in {0, ..., 3}{
				\def\name{b\x};
				\draw node[dot] (\name) at (\width * \x, -\height){};
				\def\topname{a\x};
				\draw[path] (\name) to (\topname);
				\foreach \y in {0, ..., \x}{
					\def\othername{b\y}
					\draw[bend left = \k] (\name) to (\othername);
				}
			}
			\draw node at (.9*\width, -0.5 * \height){\small $\frac{\ell}{2}-1$};
		\end{tikzpicture}
	};
	
\end{tikzpicture}

\begin{tikzpicture}
	\def\width{1};
\def\height{3};
\def\between{4cm}
\def\k{30} 
	\node (zero) at (0,0){
	\begin{tikzpicture}
	\draw node[dot] (a) at (1.5*\width, 0){};
	\foreach \x in {0, ..., 3}{
	\draw node[dot] (\x) at (\width * \x, -\height){};
	\path[path]
	(a) edge (\x);
	\foreach \y in {0, ..., \x}{
	\draw[bend left = \k] (\x) to (\y);
	}	
	}
	\draw node at (.5*\width, -0.5 * \height){$\frac{\ell -1}{2}$};			
	\end{tikzpicture}
	};

	\node (two) at (zero.east)[anchor=east,xshift=\between]{
	\begin{tikzpicture}
		\node[dot] (a1) at (0, -.20 * \height){};
		\node[dot] (a3) at (-\width, 0){};
		\node[dot] (a4) at (\width, 0){};
		
		\node[dot] (b1) at (0, -.80 * \height){};
		\node[dot] (b3) at (-\width, -\height){};
		\node[dot] (b4) at (\width, -\height){};
		\node[dot](b2) at (0, -1.2*\height){};
		
		\path[path]
		(a3) edge node[left, above, rotate=90]{\small $\frac{\ell-1}{2} -1$} (b3)
		(a4) edge node[left, below, rotate=90]{\small $\frac{\ell-1}{2} -1$} (b4)
		(a1) edge node[left, above,  rotate=90]{\small $\frac{\ell -1}{2}$} (b1);
		\begin{scope}[bend left =60]
			\path[path]
			(a1) edge node[rotate=90, left, above]{\small $\frac{\ell-1}{2}$} (b2);
		\end{scope};
		\path
		(b1) edge (b2)
		(a1) edge (a3)
		(a3) edge (a4)
		(a1) edge (a4)
		(b1) edge (b3)
		(b1) edge (b4)
		(b2) edge (b3)
		(b2) edge (b4);
		
	\end{tikzpicture}
};
\end{tikzpicture}

%% file: Tikz/crown_mean_graphs.tex
\begin{tikzpicture}
	\def\width{1}
	
	\node[dot, fill =blue] (i1) at (0,0) {};
	\node[dot, fill =red] (i2) at (\width,0) {};
	\node[dot, fill =blue] (i3) at (2*\width,0) {};
	\node[dot, fill =blue] (i4) at (3*\width, .5*\width) {};
	\node[dot, fill =blue] (i5) at (3*\width, -.5*\width) {};
	
	\draw (i1) to (i2);
	\draw (i2) to (i3);
	\draw (i3) to (i4);
	\draw (i3) to (i5);
	
	\def\offset{5}
	
	\node[dot, fill = blue] (p1) at (\offset, 0){};
	\node[dot, fill = red] (p2) at (\offset + \width, 0){};
	\node[dot, fill = blue] (p3) at (\offset + 2*\width, 0){};
	\node[dot, fill = red] (p4) at (\offset + 3*\width, 0){};
	\node[dot, fill = blue] (p5) at (\offset + 4*\width, 0){};
	
	\draw (p1) to (p2);
	\draw (p2) to (p3);
	\draw (p3) to (p4);
	\draw (p4) to (p5);
\end{tikzpicture}

%% file: Tikz/forbidden_induced_subgraph_of_a_corpus.tex
\begin{tikzpicture}
	\foreach \x in {0,...,5}{
		\path (0,0) ++(60*\x:1cm) coordinate (a\x);
		\node[dot] at (a\x) {};
	}
	\foreach \x in {0,...,5}{
		\path (0,0) ++(60*\x:1.5cm) coordinate (b\x);
	}
	\draw (a0) -- (a1) -- (a2) -- (a3) -- (a4) -- (a5) -- (a0) -- (a2) -- (a4)-- (a0);
	
	\node at (b1) {$A$};
	\node at (b0) {$D_{AB}$};
	\node at (b5) {$B$};
	\node at (b4) {$D_{BC}$};
	\node at (b3) {$C$};
	\node at (b2) {$D_{AC}$};
\end{tikzpicture}

%% file: Tikz/temp_fig.tex
\begin{tikzpicture}
	\def\height{-2.5};
	\def\ewidth{50};
	\def\eheight{13};
	\coordinate (A) at (0,0);
	\coordinate (A') at (0, .3);
	\coordinate (B) at (0, \height);
	\coordinate (B') at (0, \height - .3);
	\draw[ultra thick, bag] (A') ellipse (\ewidth pt and \eheight pt);
	\draw[ultra thick, bag] (B') ellipse (\ewidth pt and \eheight pt);
	\node[above= .2 of A] {$A$};
	\node[below = .2 of B]  {$B$};
	
	\foreach \x in {-2,...,2}{
		\coordinate[left=.6*\x of A] (start);
		\node (n0) at (start) {};
		\foreach \y in {1,2,3,4,5}{
			\coordinate[below = .5*\y of start] (c\y\x);
			\coordinate[below = .5*\y of start] (c\y);
		}
		\draw (c1) to (start);
		\draw (c1) to (c2);
		\node[dot] (n0) at (start){};
		\node[dot] (n1) at (c1){};
		\draw (c4) to (c5);
		\node[dot, fill = red] (n2) at (c2){};
		\draw[path] (n2) to (c5);
		\node[dot] (n5) at (c5){};
	}

	\coordinate[right =.6*-3.4 of A] (qstart);
	\node[dot] (q0) at (qstart){};
	\draw (q0) to (A);
	\node[left=.3em of qstart] {$q_1$};
	\draw[maybe edge] (q0) to (c11);
	
	\coordinate[above =\height of qstart] (qend);
	\node[dot] (bottomq)at (qend){};
	\node[left=.3em of qend] {$q_n$};
	\draw (bottomq) to (n5);
	
	\coordinate[below left = .5 and .5 of qstart] (q2); 
	\draw (q0) to (q2);
	\node [dot, fill = red] (middle) at (q2){};
	\draw[path, bend right = 10] (middle) to (qend);
	\coordinate[left = .4 of q2] (temp);
	\node at (temp) {$Q$};
	
\end{tikzpicture}

%% file: Tikz/pathcases_some_path_has_length_2.tex
\begin{tikzpicture}
	\path (0,0) ++(90:2cm) coordinate (a1);
	\draw[path] (a1) arc[start angle=90, end angle=315,radius=2cm];
	\draw (a1) arc[start angle = 90, end angle= -45, radius = 2cm];
	\fill[black] (a1) circle[radius=4pt] ++(90:1em) node {$a_1$};
	
	\path (0,0) ++(0:2cm) coordinate (b1);
	\fill[black] (b1) circle[radius=4pt] ++(0:1em) node {$b_1$};
	
	\path (0,0) ++(45:2cm) coordinate (w);
	\fill[black] (w) circle[radius=4pt] ++(0:1em) node {$w$};
	
	\path (0,0) ++(67.5:1cm) coordinate (qstart);
	\fill[black] (qstart) circle[radius=4pt] ++(150:1em) node {$q_1$};
	
	\path (0,0) ++(-40.5:1cm) coordinate (qend);
	\fill[black] (qend) circle[radius=4pt] ++(180:1em) node {$q_n$};

	\path (0,0) ++(315:2cm) coordinate (u);
	\fill[black] (u) circle[radius=4pt] ++(0:1em) node {$u$};
	
	\draw[path] (qstart) -- (qend);
	\draw (qstart) to (w);
	\draw (qstart) to (a1);
	\draw (qend) to (b1);
	\draw[maybe edge] (qend) to (u);
	\draw[maybe edge] (qend) to (w);
	
\end{tikzpicture}

%% file: Tikz/pathcases_not_i.tex
\begin{tikzpicture}[scale = .5]
	\def\height{2};
	\def\width{2};
	\node[normal node] (a1) at (-\width, 0) {$a_1$};
	\node[normal node] (a2) at (\width, 0) {$a_2$};
	\node[normal node] (a) at (0, \height){$a$};
	\node[normal node, scale=1.1] (w) at (-\width, -1.5*\height){$w$};
	\node[normal node] (b1) at (-\width, -3*\height) {$b_1$};
	\node[normal node] (b2) at (\width, -3*\height){$b_2$};
	\node[normal node, scale = 2] (a21) at (\width, -\height){};
	\node[normal node, scale = 1.1] (a22) at (\width, -2*\height){$u$};
	\node[normal node] (b) at (0, -4*\height){$b$};
	\draw (a1) to (a);
	\draw (a) to (a2);
	\draw[maybe edge] (a1) to (a2);
	\draw (a1) to (w);
	\draw (b1) to (w);
	\draw (b1) to (b);
	\draw (b2) to (b);
	\draw[maybe edge] (b1) to (b2);
	\draw (b2) to (a22);
	\draw (a22) to (a21);
	\draw (a2) to (a21);
	
	\begin{scope}[bend right = 80]
		\path
		(a22) edge node[right] {$=$?} (a21) [maybe equal, thick ];
	\end{scope}
	
	\node[red node] (q1) at (-2*\width, -1*\height){$q_1$};
	\draw (q1) to (w);
	\draw (q1) to (a1);
	\node[red node] (qn) at (\width, -5*\height){$q_n$};
	
	\begin{scope}[path, red, bend right = 40]
		\path
		(q1) edge node[left=.3] {${Q}$} (qn);
	\end{scope}
	\draw (qn) to (b2);
	\draw[maybe edge] (qn) to (b);
	\draw[maybe edge, bend right = 40] (qn) to (a22);
	
\end{tikzpicture}

%% file: Tikz/pathcases_not_i_part2.tex
\begin{tikzpicture}[scale = .5]
	\def\height{2};
	\def\width{2};
	\node[normal node] (a1) at (-\width, 0) {$a_1$};
	\node[normal node] (a2) at (\width, 0) {$a_2$};
	\node[normal node] (a) at (0, \height){$a$};
	\node[normal node, scale=1.1] (w) at (-\width, -1.5*\height){$w$};
	\node[normal node] (b1) at (-\width, -3*\height) {$b_1$};
	\node[normal node] (b2) at (\width, -3*\height){$b_2$};
	\node[normal node, scale = 2] (a21) at (\width, -\height){};
	\node[normal node, scale = 1.1] (a22) at (\width, -2*\height){$u$};
	\node[normal node] (b) at (0, -4*\height){$b$};
	\draw (a1) to (a);
	\draw (a) to (a2);
	\draw[maybe edge] (a1) to (a2);
	\draw (a1) to (w);
	\draw (b1) to (w);
	\draw (b1) to (b);
	\draw (b2) to (b);
	
	\draw (b2) to (a22);
	\draw (a22) to (a21);
	\draw (a2) to (a21);
	
	\begin{scope}[bend right = 80]
		\path
		(a22) edge node[right] {$=$?} (a21) [maybe equal, thick ];
	\end{scope}
	
	\node[red node] (q1) at (-2*\width, -1*\height){$q_1$};
	\draw (q1) to (w);
	\node[red node] (qn) at (\width, -5*\height){$q_3$};
	
	\node[red node, below left= 1 and .5 of q1] (q2) {$q_2$};
	\node[red node, scale = 1, below right= 1.2 and .7 of q2] (q3) {$q_3$};
	
	\node[normal node] (a3) at (3*\width, 0){$a_3$};
	\node[normal node] (b3) at (3*\width, -3*\height){$b_3$};
	\node[normal node, scale = 2] (a31) at (3*\width, -\height){};
	\node[normal node, scale = 2] (a32) at (3*\width, -2*\height){};
	\begin{scope}[bend right = 80]
		\path
		(a32) edge node[right] {$=$?} (a31) [maybe equal, thick ];
	\end{scope}
	\draw (b3) to (a32);
	\draw (a32) to (a31);
	\draw (a31) to (a3);

	\draw (qn) to (b2);
	\draw[bend right = 40] (qn) to (a22);
	
	\begin{scope} [blue, ultra thick]
		\draw (q1) to (a1);
		\draw (q1) to (q2);
		\draw (q2) to (q3);
		\draw (q3) to (qn);
		\draw (qn) to (b2);
		\draw (b3) to (a32);
		\draw (a32) to (a31);
		\draw (a31) to (a3);
		\begin{scope}[bend left = 50]
			\path
			(b2) edge node[above] {\textbf{$B_{23}$}} (b3) [path ];
		\end{scope}
		\begin{scope}[bend left = 100]
			\path
			(a1) edge node[above] {\textbf{$A_{13}$}} (a3) [path ];
		\end{scope}

	\end{scope}
	
	\draw[maybe edge] (b3) to (b2);
	\draw[maybe edge, bend left = 30] (b3) to (b1);
	\draw[maybe edge, bend left = 20] (b3) to (b);
	
	\draw[maybe edge] (a3) to (a2);
	\draw[maybe edge, bend right = 30] (a3) to (a1);
	\draw[maybe edge, bend right = 20] (a3) to (a);

\end{tikzpicture}

%% file: Tikz/pathcases_anticomplete_to_interior_paths.tex
\begin{tikzpicture}[scale = .5]
	\def\height{2};
	\def\width{2};
	\node[normal node] (a1) at (-\width, 0) {$a_1$};
	\node[normal node] (a2) at (\width, 0) {$a_2$};
	
	\node[normal node, scale=.9] (w) at (-\width, -1.25*\height){$a_2'$};
	\node[normal node] (b1) at (-\width, -3*\height) {$b_1$};
	\node[normal node] (b2) at (\width, -3*\height){$b_2$};
	\node[normal node, scale = .9] (a22) at (\width, -1.75*\height){$b_2'$};
	\node[normal node] (b) at (0, -4*\height){$b_3$};

	\draw (a1) to (w);
	\draw[path] (b1) to (w);
	\draw (b1) to (b);
	\draw (b2) to (b);
	\draw (b2) to (a22);
	\draw[path] (a22) to (a2);

	\node[red node] (q1) at (-2*\width, -1*\height){$q_1$};
	\draw (q1) to (w);
	\draw (q1) to (a1);
	\node[red node] (qn) at (\width, -5*\height){$q_n$};
	
	\begin{scope}[bend right = 40]
		\path
		(q1) edge (qn);
	\end{scope}

	\begin{scope}[path, bend left = 40]
		\path
		(a1) edge node[above] {${A_{12}}$} (a2);
	\end{scope}
	
	\draw (qn) to (b2);
	\draw[bend right = 40, maybe edge ] (qn) to (a22);
	
\end{tikzpicture}

%% file: Tikz/pathcases_ending_case.tex
\begin{tikzpicture}[scale = .5]
	\def\height{6};
	\def\width{4};
	
	\foreach \x in {1,2,3,4}{
		\node[normal node] (a\x) at (\width*\x, 0) {$a_\x$};
		\node[normal node] (b\x) at (\width*\x, -\height) {$b_\x$};
		\draw[path] (a\x) to (b\x);
		\foreach \y in {1,...,\x}{
			\draw[path, blue, bend right = 40] (a\x) to (a\y);
			\draw[path, blue, bend left = 40] (b\x) to (b\y);
		}
	}

	\node[red node] (q1) at (1.5*\width, -2){$q_1$};
	\draw (q1) to (a1);
	\draw (q1) to (a2);
	\node[red node] (qn) at (3.5*\width, -4){$q_n$};
	\draw (q1) to (qn);
	\draw (qn) to (b3);
	\draw (qn) to (b4);

\end{tikzpicture}

%% file: Tikz/fml.tex
	\begin{tikzpicture}
	\path (0,0) ++(90:2cm) coordinate (a1);
	\fill[black] (a1) circle[radius=4pt] ++(90:1em) node {$a_1$};
	
	\path (0,0) ++(0:2cm) coordinate (b1);
	\fill[black] (b1) circle[radius=4pt] ++(0:1em) node {$b_1$};
	
	\path (0,0) ++(45:2cm) coordinate (z);
	\fill[black] (z) circle[radius=4pt] ++(0:1em) node {$v_1$};
	
	\path (0,0) ++(110:1.1cm) coordinate (qstart);
	\fill[black] (qstart) circle[radius=4pt] ++(270:1em) node {$h$};
	
	\path (0,0) ++(-40:1cm) coordinate (qend);
	\fill[black] (qend) circle[radius=4pt] ++(180:1em) node {$w_n$};
	
	\path (0,0) ++(120:2cm) coordinate (u);
	\fill[black] (u) circle[radius=4pt] ++(90:1em) node {};
	
	\path (0,0) ++(150:2cm) coordinate (w);
	\fill[black] (w) circle[radius=4pt] ++(180:1em) node {$a_2$};
	
	\path (0,0) ++(50:1cm) coordinate (w1);
	\fill[black] (w1) circle[radius=4pt] ++(150:-1em) node {$w_1$};
	
	\path[path] (qend) edge node[left] {$W$} (w1);
	\draw (w1) to (qstart);
	\draw[maybe edge] (qstart) to (a1);
	\draw (qend) to (b1);
	\draw[maybe edge] (qstart) to (u);
	\draw (qstart) to (w);
	\draw[path] (w) arc[start angle=150, end angle=360,radius=2cm];
	\draw (b1) arc[start angle = 0, end angle= 150, radius = 2cm];
	\draw (qend) to (z);
	
	\begin{scope}[bend right = 100]
		\path
		(u) edge node[left] {$=$?} (w) [maybe equal, thick];
	\end{scope}
\end{tikzpicture}

%% file: Tikz/transitive_closure_of_tree_example.tex
\begin{tikzpicture}[scale=0.5]
	\def\width{3.6};
	\def\height{2.5};
	\def\bend{30};
	\newcounter{c};
	\node[normal node, ultra thick] (r) at (5.5*\width, 2.5*\height) {$r$};
	\foreach \x in {4,7}{
		\node[normal node, fill = black] (a\x) at (\x*\width, 1*\height){};
		\draw[thick] (a\x) to (r);
	}
	\foreach \x in {3}{
		\node[normal node, thick, fill = black] (b\x) at (\x*\width, 0){};
		\draw[thick] (b\x) to (a4);
		\draw[red, bend left = 28] (b\x) to (r);
	}
	\foreach \x in {4,5}{
		\node[normal node, fill = black, thick] (b\x) at (\x*\width, 0){};
		\draw[thick] (b\x) to (a4);
		\draw[red, bend right = 20] (b\x) to (r);
	}

	\foreach \x in {6,7,8}{
		\node[normal node, thick, fill = black] (b\x) at (\x*\width, 0){};
		\draw[thick] (b\x) to (a7);
		\draw[red, bend right = 10] (b\x) to (r);
	}

	\foreach \x in {6,7,8}{
		\node[normal node, fill = black, below left = .7cm and .3cm of b\x, thick] (c\x){};
		\draw[thick] (c\x) to (b\x);
		\draw[red](a7) to (c\x);
		
		\node[normal node, fill = black, below right = .7cm and .3cm of b\x, thick] (d\x){};
		\draw[thick] (d\x) to (b\x);
		\draw[red] (a7) to (d\x);
	
	}
	
	\draw[red] (c6) to (r);
	\draw[red, bend right = 10] (d6) to (r);
	
	\draw[red, bend right = 10] (c7) to (r);
	\draw[red, bend right = 20] (d7) to (r);
	
	\draw[red] (c8) to (r);
	\draw[red, bend right = 35] (d8) to (r);

\end{tikzpicture}

%% file: Tikz/pyramidoid.tex
\begin{tikzpicture}
	\def\ewidth{40}
	\def\eheight{10}
	\def\rowwidth{4};
	\def\rowheight{-2};
	\def\bendamount{30};
	
	\coordinate (apex) at (2*\rowwidth,1);
	\foreach \i in {1,2}{
		\foreach \j in {1,2,3}{
			\coordinate (p\i\j) at (\i*\rowwidth, \j*\rowheight);
			\fill[bag] (p\i\j) ellipse (\ewidth pt and \eheight pt);
			\foreach \k in {1, ..., 5}{
				\coordinate[left=.5*\k of p\i\j] (temp);
				\coordinate [right = 1.5 of temp] (temp);
				\node[circle, fill = black, scale=.5] (h\i\j\k) at (temp) {};
				\foreach \x in {1, ..., \k}{
					\draw[bend left = \bendamount] (h\i\j\k) to (h\i\j\x);
				}
				
			}
		}
	}

\foreach \i in {3,4}{
	\foreach \j in {1,2,3}{
		\coordinate (p\i\j) at (\i*\rowwidth -3, \j*\rowheight);
		\coordinate (ellspot) at (\i*\rowwidth, \j*\rowheight);
		\fill[bag] (ellspot) ellipse (\ewidth pt and \eheight pt);
		\foreach \k in {1, ..., 5}{
			\coordinate[right=.5*\k of p\i\j] (temp);
			\coordinate [right = 1.5 of temp] (temp);
			\node[circle, fill = black, scale=.5] (h\i\j\k) at (temp) {};
			\foreach \x in {1, ..., \k}{
				\draw[bend right = \bendamount] (h\i\j\k) to (h\i\j\x);
			}
			
		}
	}
}

	\foreach \i in {1}{
		\foreach \l in {2,3,4}{
			\foreach \k in {1, 2, 3, 4, 5}{
				\foreach \m in {1, ..., 5}{
					\draw[bend right = 70] (h\i3\k) to (h\l3\m);
				}
				
			}
		}
	}
	\foreach \i in {2}{
		\foreach \l in {3}{
			\foreach \k in {1, 2, 3, 4, 5}{
				\foreach \m in {1, ..., 5}{
					\draw[bend right = 70] (h\i3\k) to (h\l3\m);
				}	
			}
		}
	}
	
	\foreach \i in {3}{
		\foreach \l in {4}{
			\foreach \k in {1, 2, 3, 4, 5}{
				\foreach \m in {1, ..., 5}{
					\draw[bend right = 70] (h\i3\k) to (h\l3\m);
				}	
			}
		}
	}

	\foreach \i in {1,2,3}{
		\foreach \z in {1,...,5}{
			\foreach \k in {\z,...,5}{
			}
			
		}
	}
	\foreach \i in {1,2,3,4}{
		\foreach \k in {1,...,5}{
			\foreach \z in {\k,...,5}{
				\draw (h\i1\z) to (h\i2\k);
				\draw (h\i3\z) to (h\i2\k);
			}
			
		}
	}
	%
	%
	%
	%
	%
	
	\def\smallstep{.4}
	\coordinate (p0) at (\rowwidth, -1);
	\foreach \k in {1, ..., 8}{
		\coordinate[above right = \smallstep*\k and \k of p0] (temp);
		\coordinate [left = 1.5 of temp] (temp);
		\node[circle, fill = black, scale=.5] (t\k) at (temp) {};
		\foreach \x in {1, ..., \k}{
			\draw[bend right = 50] (t\k) to (t\x);
		}
	}
	\newcounter{viscount}
	\foreach \k in {5,4,3,2,1}{
		\stepcounter{viscount}
		\foreach \z in {5,...,\value{viscount}}{
			\draw (h11\k) to (t\z);
		}
		
	}
	
	\newcounter{counting}
	\foreach \k in {5,4,3,2,1}{
		\stepcounter{counting}
		\foreach \z in {5,...,\value{counting}}{
			\pgfmathparse{int(\z + 3)}
			\edef\z{\pgfmathresult}
			\draw (h21\k) to (t\z);
		}
		
	}

	\def\smallstep{.4}
	\coordinate (p3) at (3*\rowwidth, -.85);
	\foreach \k in {1, ..., 7}{
		\coordinate[above left = \smallstep*\k and \k of p3] (temp);
		\coordinate [right = 6 of temp] (temp);
		\node[circle, fill = black, scale=.5] (r\k) at (temp) {};
		\foreach \x in {1, ..., \k}{
			\draw[bend left = 50] (r\k) to (r\x);
		}
	}
	\newcounter{con}
	\foreach \k in {5,4,3,2,1}{
		\stepcounter{con}
		\foreach \z in {5,...,\value{con}}{
			\draw (h41\k) to (r\z);
		}
		
	}
	
	\newcounter{monkey}
	\foreach \k in {5,4,3,2,1}{
		\stepcounter{monkey}
		\foreach \z in {5,...,\value{monkey}}{
			\pgfmathparse{int(\z + 2)}
			\edef\z{\pgfmathresult}
			\draw (h31\k) to (r\z);
		}
		
	}
	
	\node[circle, fill = black, scale=.5] (apex) at (10.5,4) {};
	\node [above = .25em of apex] {\Large $r$};
	\foreach \k in {1, ..., 8}{
		\draw[bend right = 20] (apex) to (t\k);
	}
	\foreach \k in {1, ..., 7}{
		\draw[bend left = 20] (apex) to (r\k);
	}
	\foreach \k in {1,...,5}{
		\foreach \i in {1,2}{
			\draw[bend left = 30] (h\i1\k) to (apex);
		}
		\foreach \i in {3,4}{
			\draw[bend right = 30] (h\i1\k) to (apex);
		}
	}

\foreach \i in {1,2,3,4}{
	\draw (t\i) to (h215);
}

\foreach \i in {1,2}{
	\draw (r\i) to (h315);
}

\end{tikzpicture}